\documentclass[11pt]{article}

\usepackage{graphicx}

\usepackage{enumerate,xspace}
\usepackage{amsmath,amssymb,wasysym}
\usepackage[all]{xy}
\usepackage{proof}
\usepackage[svgnames]{xcolor}
\usepackage{tikz}

\usepackage{stmaryrd} 
\usepackage{mathtools}
\usepackage{latexsym}
\usepackage{dsfont}
\usepackage{multicol}
\usepackage{lscape}

\usepackage{shuffle}
\usepackage{cmll}
\usepackage{mathtools}
\allowdisplaybreaks

\DeclareFontFamily{U}{shuffle}{}
\DeclareFontShape{U}{shuffle}{m}{n}{ <-8>shuffle7 <8->shuffle10}{}

\newcommand\abs[1]{\left|#1\right|}

\newtheorem{observation}{Remark}[section]
\newtheorem{lemma}[observation]{Lemma}  
\newtheorem{theorem}[observation]{Theorem}
\newtheorem{definition}[observation]{Definition}
\newtheorem{example}[observation]{Example}

\newtheorem{proposition}[observation]{Proposition} 
\newtheorem{corollary}[observation]{Corollary}

\newcommand{\proof}{\noindent{\sc Proof:}\xspace}
\def\endproof{~\hfill$\Box$\vskip 10pt}

\usetikzlibrary{shapes,snakes}
\pgfdeclarelayer{edgelayer}
\pgfdeclarelayer{nodelayer}
\pgfsetlayers{edgelayer,nodelayer,main}

\definecolor{hexcolor0xff0000}{rgb}{1.000,0.000,0.000}
\definecolor{hexcolor0x000000}{rgb}{0.000,0.000,0.000}
\definecolor{hexcolor0x00ff00}{rgb}{0.000,1.000,0.000}
\definecolor{hexcolor0xffff00}{rgb}{1.000,1.000,0.000}
\definecolor{hexcolor0x000000}{rgb}{0.000,0.000,0.000}
\definecolor{hexcolor0x000000}{rgb}{0.000,0.000,0.000}
\definecolor{white}{rgb}{1.000,1.000,1.000}

\tikzstyle{none}=[inner sep=0pt]
\tikzstyle{port}=[inner sep=0pt]
\tikzstyle{component}=[circle,fill=white,draw=black,inner sep=1pt]
\tikzstyle{integral}=[inner sep=0pt]
\tikzstyle{differential}=[inner sep=0pt]
\tikzstyle{codifferential}=[inner sep=0pt]
\tikzstyle{function}=[rectangle,fill=white,draw=black,inner sep =1pt]
\tikzstyle{duplicate}=[inner sep=1pt]
\tikzstyle{wire}=[-,draw=black,line width=1.000]
\tikzstyle{dwire}=[-,dotted,draw=black,line width=0.5]
\tikzstyle{object}=[inner sep=2pt]

\makeatletter

\newdimen\w@dth

\def\setw@dth#1#2{\setbox\z@\hbox{\scriptsize $#1$}\w@dth=\wd\z@
\setbox\@ne\hbox{\scriptsize $#2$}\ifnum\w@dth<\wd\@ne \w@dth=\wd\@ne \fi
\advance\w@dth by 1.2em}

\def\t@^#1_#2{\allowbreak\def\n@one{#1}\def\n@two{#2}\mathrel
{\setw@dth{#1}{#2}
\mathop{\hbox to \w@dth{\rightarrowfill}}\limits
\ifx\n@one\empty\else ^{\box\z@}\fi
\ifx\n@two\empty\else _{\box\@ne}\fi}}
\def\t@@^#1{\@ifnextchar_ {\t@^{#1}}{\t@^{#1}_{}}}

\def\t@left^#1_#2{\def\n@one{#1}\def\n@two{#2}\mathrel{\setw@dth{#1}{#2}
\mathop{\hbox to \w@dth{\leftarrowfill}}\limits
\ifx\n@one\empty\else ^{\box\z@}\fi
\ifx\n@two\empty\else _{\box\@ne}\fi}}
\def\t@@left^#1{\@ifnextchar_ {\t@left^{#1}}{\t@left^{#1}_{}}}

\def\two@^#1_#2{\def\n@one{#1}\def\n@two{#2}\mathrel{\setw@dth{#1}{#2}
\mathop{\vcenter{\hbox to \w@dth{\rightarrowfill}\kern-1.7ex
                 \hbox to \w@dth{\rightarrowfill}}%
       }\limits
\ifx\n@one\empty\else ^{\box\z@}\fi
\ifx\n@two\empty\else _{\box\@ne}\fi}}
\def\tw@@^#1{\@ifnextchar_ {\two@^{#1}}{\two@^{#1}_{}}}

\def\tofr@^#1_#2{\def\n@one{#1}\def\n@two{#2}\mathrel{\setw@dth{#1}{#2}
\mathop{\vcenter{\hbox to \w@dth{\rightarrowfill}\kern-1.7ex
                 \hbox to \w@dth{\leftarrowfill}}%
       }\limits
\ifx\n@one\empty\else ^{\box\z@}\fi
\ifx\n@two\empty\else _{\box\@ne}\fi}}
\def\t@fr@^#1{\@ifnextchar_ {\tofr@^{#1}}{\tofr@^{#1}_{}}}

\newdimen\W@dth
\def\setW@dth#1#2{\setbox\z@\hbox{$#1$}\W@dth=\wd\z@
\setbox\@ne\hbox{$#2$}\ifnum\W@dth<\wd\@ne \W@dth=\wd\@ne \fi
\advance\W@dth by 1.2em}

\def\T@^#1_#2{\allowbreak\def\N@one{#1}\def\N@two{#2}\mathrel
{\setW@dth{#1}{#2}
\mathop{\hbox to \W@dth{\rightarrowfill}}\limits
\ifx\N@one\empty\else ^{\box\z@}\fi
\ifx\N@two\empty\else _{\box\@ne}\fi}}
\def\T@@^#1{\@ifnextchar_ {\T@^{#1}}{\T@^{#1}_{}}}

\def\T@left^#1_#2{\def\N@one{#1}\def\N@two{#2}\mathrel{\setW@dth{#1}{#2}
\mathop{\hbox to \W@dth{\leftarrowfill}}\limits
\ifx\N@one\empty\else ^{\box\z@}\fi
\ifx\N@two\empty\else _{\box\@ne}\fi}}
\def\T@@left^#1{\@ifnextchar_ {\T@left^{#1}}{\T@left^{#1}_{}}}

\def\Tofr@^#1_#2{\def\N@one{#1}\def\N@two{#2}\mathrel{\setW@dth{#1}{#2}
\mathop{\vcenter{\hbox to \W@dth{\rightarrowfill}\kern-1.7ex
                 \hbox to \W@dth{\leftarrowfill}}%
       }\limits
\ifx\N@one\empty\else ^{\box\z@}\fi
\ifx\N@two\empty\else _{\box\@ne}\fi}}
\def\T@fr@^#1{\@ifnextchar_ {\Tofr@^{#1}}{\Tofr@^{#1}_{}}}

\def\Two@^#1_#2{\def\N@one{#1}\def\N@two{#2}\mathrel{\setW@dth{#1}{#2}
\mathop{\vcenter{\hbox to \W@dth{\rightarrowfill}\kern-1.7ex
                 \hbox to \W@dth{\rightarrowfill}}%
       }\limits
\ifx\N@one\empty\else ^{\box\z@}\fi
\ifx\N@two\empty\else _{\box\@ne}\fi}}
\def\Tw@@^#1{\@ifnextchar_ {\Two@^{#1}}{\Two@^{#1}_{}}}

\def\to{\@ifnextchar^ {\t@@}{\t@@^{}}}
\def\from{\@ifnextchar^ {\t@@left}{\t@@left^{}}}
\def\tofro{\@ifnextchar^ {\t@fr@}{\t@fr@^{}}}
\def\To{\@ifnextchar^ {\T@@}{\T@@^{}}}
\def\From{\@ifnextchar^ {\T@@left}{\T@@left^{}}}
\def\Two{\@ifnextchar^ {\Tw@@}{\Tw@@^{}}}
\def\Tofro{\@ifnextchar^ {\T@fr@}{\T@fr@^{}}}

\makeatother

\title{Integral Categories and Calculus Categories}
\author{J.R.B. Cockett         \and
        J-S. Lemay}

\begin{document}
\allowdisplaybreaks

\maketitle



\begin{abstract}
Differential categories are now an established abstract setting for differentiation. However not much attention has been given to the process which is inverse to differentiation: integration. This  paper presents the parallel development for integration by axiomatizing an integral transformation, $\mathsf{s}_A: \oc A \to \oc A \otimes A$,  in a symmetric monoidal category with a coalgebra modality. When integration is combined with differentiation, the two fundamental theorems of calculus are expected to hold (in a suitable sense):  a differential category with integration which satisfies these two theorem is called a {\em calculus category\/}.  

Modifying an approach to antiderivatives by T. Ehrhard,  we define having antiderivatives as the demand that a certain natural 
transformation, $\mathsf{K}: \oc A \to \oc A$, is invertible. We observe that a differential category having antiderivatives, in this 
sense, is always a calculus category. 

When the coalgebra modality is monoidal, it is natural to demand an extra coherence between integration and the coalgebra modality.   In the presence of this extra coherence we show that a calculus category with a monoidal coalgebra modality has its integral transformation given by antiderivatives and, thus, that the integral structure is uniquely determined by the differential structure.

The paper finishes by providing a suite of separating examples.  Examples of differential categories, integral categories, and calculus categories based on both monoidal and (mere) coalgebra modalities are presented.  
In addition, differential categories which are {\em not\/} integral categories are discussed and vice versa.
\end{abstract}

\section{Introduction}\label{intro}

The two fundamental theorems of calculus relate the two most important operations of calculus: differentiation and integration. The first theorem states that the derivative of the integral of a function $f: \mathbb{R} \to \mathbb{R}$ is the original function $f$:
\begin{align*}
\frac{{\sf d} (\int_a^t f(u)~{\sf d}u)}{{\sf d}t}(x)=f(x) \tag{\sf FTC1}
\end{align*}
while the second states that the integral of the derivative of a function $f: \mathbb{R} \to \mathbb{R}$ on a closed interval $[a,b]$ is equal to the difference of $f$ evaluated at the end points:
\begin{align*}
\int_a^b\frac{{\sf d} f(t)}{{\sf d}t}(x)~{\sf d}x=f(b)-f(a). \tag{\sf FTC2}
\end{align*}
They are called ``fundamental" theorems because they are absolutely fundamental to the development of classical calculus. 

There has been significant progress in the abstract understanding of differentiation in category theory.  In contrast, the abstract formulation of integration has not received the same level of attention. Nonetheless, one might expect that, when suitably adjoined to the formulation of differentiation, an abstract form for integration should encompass these fundamental theorems.   The purpose of this paper is to explore the extent to which this expectation is realized.

In the early 2000's, T. Ehrhard and L. Regnier introduced the differential $\lambda$-calculus \cite{ehrhard2003differential} and differential proof nets \cite{ehrhard2006differential}. A few years later, R. Blute, R. Cockett and R. Seely introduced differential categories \cite{dblute2006differential}, which were the appropriate categorical structure for modelling Ehrhard and Regnier's differential linear logic. Differential categories now have a rich literature of their own \cite{blute2011kahler, blute2015cartesian, blute2009cartesian, blute2015derivations, cockett2011differential, cockett2014differential, fiore2007differential, laird2013constructing} and there are many examples which have been extensively studied \cite{blute2010convenient, dblute2006differential, ehrhard2017introduction}. However, as mentioned before, little attention has been given to abstracting the process of integration.

In 2014, T. Ehrhard observed that in certain $*$-autonomous categories which had the appropriate structure to be a differential category, it was possible with one additional assumption to produce 
{\em antiderivatives\/} \cite{ehrhard2017introduction}. The additional assumption was that a certain natural transformation -- which he called $\mathsf{J}$ -- constructed from the deriving transformation was a natural 
isomorphism.  With this assumption, Ehrhard constructed an integral transformation with an inverse behaviour to the deriving transformation, in the sense that he gave necessary and sufficient conditions 
for a {\em map\/} to satisfy the first fundamental theorem of calculus -- by proving Poincar\'e's Lemma.  Furthermore, when the deriving transformation satisfied an extra -- non-equational -- condition, 
which he called the ``Taylor Property", he then showed that every differentiable function satisfied the second fundamental theorem of calculus.  

While much of the inspiration for our approach to integration derives from these observations, Ehrhard made no attempt to axiomatize integration separately from differentiation.  Here we
introduce integral categories as a notion which stands on its own (i.e. in the absence of differentiation).  The inspiration for this independent axiomatization of integral 
categories comes from the much older notion of a Rota-Baxter algebra \cite{baxter1960analytic, guo2012introduction, rota1969baxter}, the classical algebraic abstraction of integration. Briefly, for a commutative ring $R$, a Rota-Baxter algebra of weight $\lambda$ is an $R$-algebra, $A$, with an $R$-linear morphism, $P: A \to A$, which satisfies the Rota-Baxter rule: 
\[P(a)P(b)=P(aP(b))+P(P(a)b)+\lambda P(ab) \quad \forall a,b \in A\]
where $\lambda \in R$. The map $P$ is called a Rota-Baxter operator of weight $\lambda$. A particular example of a Rota-Baxter algebra of weight zero is the $\mathbb{R}$-algebra of real continuous functions $\mathsf{Cont}(\mathbb{R})$, where the Rota-Baxter operator $P: \mathsf{Cont}(\mathbb{R}) \to \mathsf{Cont}(\mathbb{R})$ is defined as the integral of the function centred at zero:
\[P(f)(x)=\int_0^x f(t)~{\sf d}t\]
The Rota-Baxter rule for this example is the expression of the integration by parts rule without the use of derivatives (see \cite{guo2012introduction} for more details): 
\[  \int^x_0 f(t)~\mathsf{d}t \cdot \int^x_0 g(t)~\mathsf{d}t=  \int^x_0 f(t) \cdot (\int^t_0 g(u)~\mathsf{d}u)~\mathsf{d}t + \int^x_0  (\int^t_0 f(u)~\mathsf{d}u) \cdot g(t)~\mathsf{d}t\]
This example motivates the idea that the Rota-Baxter rule is one of the fundamental axioms of integration. In particular, a formalization of the Rota-Baxter rule of weight zero will be one of the axioms of the integral category structure.

When differentiation and integration are combined into what we call here a 
{\em calculus category\/}, we demand that the two fundamental theorems hold.  The second fundamental theorem is assumed to hold verbatim.  However, the first fundamental theorem, 
as above, has to be interpreted as being on {\em maps\/} -- rather than objects -- and, under this interpretation, becomes the Poincar\'e condition, a conditional property which provides necessary 
and sufficient conditions for a map to be the differential of its integral.   

To obtain the notion of integration as an {\em antiderivative\/}, we insist that a slightly different natural transformation, which we call $\mathsf{K}$,  should be invertible.  We show this is equivalent to 
requiring both that Ehrhard's transformation $\mathsf{J}$ is invertible and that the ``Taylor Property'' -- which  Ehrhard had suggested was desirable -- holds.  This improvement is easily underestimated: the  ``Taylor Property''  is a conditional requirement,  replacing a conditional requirement by a purely equational requirement is {\em always}, mathematically, a significant step. 
Demanding that, $\mathsf{K}$, is invertible not only produces an integral transformation, but also secures the first {\em and\/} second fundamental theorem of calculus.  

In fact, an integral transformation is already obtained when just Ehrhard's transformation, $\mathsf{J}$, is invertible (see Proposition \ref{Jint}). Furthermore, this integral satisfies the Poincar\'e condition -- our interpretation of the first fundamental theorem.  However, this integral need not satisfy the second fundamental theorem of calculus -- for a counterexample see Section \ref{examplesec}, Example \ref{JnotKexample} -- and, thus, may not yield a calculus category.  The  ``Taylor Property''  is required in addition to the invertibility of $\mathsf{J}$ to secure a calculus category.  Of course, assuming invertibility of $\mathsf{J}$ and the Taylor property, guarantees the invertibility of $\mathsf{K}$.  

To prove that the antiderivative produced using the inverse of $\mathsf{K}$ yields a calculus category, we use the fact that, when $\mathsf{K}$ is invertible, $\mathsf{J}$ is also invertible and that the antiderivative produced by the inverse of $\mathsf{K}$, is precisely the antiderivative produced by the inverse of $\mathsf{J}$.  This immediately gives that the invertibility of $\mathsf{K}$ gives an integral transformation satisfying the Poincar\'e condition; it then remains only to show that invertibility of $\mathsf{K}$ implies the second fundamental theorem of calculus (see Theorem \ref{anticalc}).

\subsection{Conventions and the Graphical Calculus}

Before beginning, we should mention the conventions that we use in the paper. We shall use diagrammatic order for composition: explicitly, this means that the composite map $fg$ is the map which first does $f$ then $g$.   Furthermore, to simplify working in symmetric monoidal categories, we will allow ourselves to work in strict symmetric monoidal categories and so will generally suppress the associator and unitor isomorphisms.  For a symmetric monoidal category we will use $\otimes$ for the tensor product, $K$ for the unit, and $\sigma: A \otimes B \to B \otimes A$ for the symmetry isomorphism.  

We shall make extensive use of the graphical calculus (see \cite{joyal1991geometry}) for symmetric monoidal categories as this makes proofs easier to follow.  We refer the reader to \cite{selinger2010survey} for an introduction to the graphical calculus in monoidal categories and its variations.  Note, however, that our diagrams are to be read down the page -- from top to bottom -- and we shall often omit labelling wires with objects. 

We will be working with coalgebra modalities: these are based on a comonad $(\oc, \delta, \varepsilon)$ where $\oc$ is the functor, $\delta$ is the comultiplicaton and $\varepsilon$ is the counit.  As in  \cite{dblute2006differential}, we will use functor boxes when dealing with string diagrams involving the functor $\oc$:  a mere map $f: A \to B$ will be encased in a circle while $\oc(f): \oc A \to \oc B$ will be encased in a box: 
$$ f= \begin{array}[c]{c}\resizebox{!}{2cm}{%
\begin{tikzpicture}
	\begin{pgfonlayer}{nodelayer}
		\node [style=circle] (0) at (0, 3) {$A$};
		\node [style=circle] (1) at (0, 0) {$B$};
		\node [style={circle, draw}] (2) at (0, 1.5) {$f$};
	\end{pgfonlayer}
	\begin{pgfonlayer}{edgelayer}
		\draw [style=wire] (0) to (2);
		\draw [style=wire] (2) to (1);
	\end{pgfonlayer}
\end{tikzpicture}
  }%
\end{array} ~~~~~   \oc(f)= \begin{array}[c]{c}\resizebox{!}{2cm}{%
\begin{tikzpicture}
	\begin{pgfonlayer}{nodelayer}
		\node [style=circle] (0) at (0, 3) {$\oc A$};
		\node [style=circle] (1) at (0, 0) {$\oc B$};
		\node [style={regular polygon,regular polygon sides=4, draw}] (2) at (0, 1.5) {$f$};
	\end{pgfonlayer}
	\begin{pgfonlayer}{edgelayer}
		\draw [style=wire] (0) to (2);
		\draw [style=wire] (2) to (1);
	\end{pgfonlayer}
\end{tikzpicture}
  }%
\end{array} $$

\section{Coalgebra Modalites and the Coderiving Transformation}\label{coalgsec}

\subsection{Coalgebra Modalities}

In a symmetric monoidal category, a cocommutative comonoid is a triple $(C, \Delta, e)$ consisting of an object $C$, a map $\Delta: C \to C \otimes C$ called the comultiplication and a map $e: C \to K$ called the counit such that the following diagrams commute: 
        $$  \xymatrixcolsep{2pc}\xymatrix{C  \ar[r]^-{\Delta} \ar[d]_-{\Delta} & C \otimes C \ar[d]^-{\Delta \otimes 1}& & C \ar[d]^-{\Delta} \ar@{=}[dr]\ar@{=}[dl] &      C   \ar[r]^-{\Delta}  \ar[dr]_-{\Delta} & C \otimes C \ar[d]^-{\sigma}\\
      C \otimes C \ar[r]_-{1 \otimes \Delta} & C \otimes C \otimes C& C & C \otimes C \ar[l]^-{e \otimes 1} \ar[r]_-{1 \otimes e} & C &  C \otimes C}$$
A morphism of comonoids $f: (C, \Delta, e) \to (D, \Delta^\prime, e^\prime)$ is a map $f: C \to D$ which preserves the comultiplication and counit, that is, the following diagrams commute: 
  $$\xymatrixcolsep{2pc}\xymatrix{
        C   \ar[r]^{\Delta} \ar[d]_-{f} & C \otimes C \ar[d]^-{f \otimes f} & C  \ar[r]^-{f} \ar[dr]_-{e} & D \ar[d]^{e}\\
        D  \ar[r]_-{\Delta} &  D \otimes D & & K}$$
Commutative monoids and their morphisms are defined dually. 

\begin{definition} \normalfont A \textbf{coalgebra modality} \cite{dblute2006differential} on a symmetric monoidal category is a quintuple $(\oc, \delta, \varepsilon, \Delta, e)$ consisting of a comonad $(\oc, \delta, \varepsilon)$, a natural transformation $\Delta$ with components $\Delta_A: \oc A \to \oc A \otimes \oc A$, and a natural transformation $e$ with components $e_A: \oc A \to K$ such that for each object $A$, $(\oc A, \Delta_A, e_A)$ is a cocommutative comonoid, and $\delta$ preserves the comultiplication, that is, $\delta\Delta=\Delta(\delta \otimes \delta)$. Algebra modalities are defined dually. 
\end{definition}

Requiring that $\Delta$ and $e$ be natural transformations is equivalent to asking that for each map $f: A \to B$, $\oc(f): \oc A \to \oc B$ is a comonoid morphism. This can be used to show that $\delta$ is a comonoid morphism. 

\begin{lemma} For any coalgebra modality $(\oc, \delta, \varepsilon, \Delta, e)$, $\delta$ also preserves the counit, that is, $\delta e = e$, and so $\delta$ is a comonoid morphism.
\end{lemma}
\begin{proof} By the above remark, $\oc(\varepsilon)$ is a comonoid morphism which implies that $\oc(\varepsilon)e=e$. Therefore, since $(\oc, \delta, \varepsilon)$ is a comonad, we obtain the following: 
\[\delta e=\delta\oc(\varepsilon)e=e\]
\end{proof} 

\subsection{Coderiving Transformation}

\begin{definition} \normalfont For a coalgebra modality $(!, \delta, \varepsilon, \Delta, e)$ on an additive symmetric monoidal category, the \textbf{coderiving transformation} is the natural transformation $\mathsf{d}^\circ_A:  \oc A \to \oc A \otimes A$ defined as:
\[\mathsf{d}^\circ_A= \Delta_A (1_{\oc A} \otimes \varepsilon_A)\] \end{definition}
 In the graphical calculus, we represent the coderiving transformation as follows:
    $$\mathsf{d}^\circ \coloneqq \begin{array}[c]{c} \resizebox{!}{1.5cm}{%
 \begin{tikzpicture}
	\begin{pgfonlayer}{nodelayer}
		\node [style=port] (0) at (-0.75, 0.75) {};
		\node [style=port] (1) at (0.75, 0.75) {};
		\node [style=port] (2) at (0, 3) {};
		\node [style=codifferential] (3) at (0, 2) {{\bf =\!=\!=\!=}};
	\end{pgfonlayer}
	\begin{pgfonlayer}{edgelayer}
		\draw [style=wire, bend left, looseness=1.00] (3) to (1);
		\draw [style=wire] (2) to (3);
		\draw [style=wire, bend right, looseness=1.00] (3) to (0);
	\end{pgfonlayer}
\end{tikzpicture}}
   \end{array} :=
   \begin{array}[c]{c} \resizebox{!}{1.5cm}{%
\begin{tikzpicture}
	\begin{pgfonlayer}{nodelayer}
		\node [style=duplicate] (0) at (0, 2.25) {$\Delta$};
		\node [style=port] (1) at (0.75, 0.75) {};
		\node [style={circle, draw}] (2) at (0.75, 1.25) {$\varepsilon$};
		\node [style=port] (3) at (0, 3) {};
		\node [style=port] (4) at (-0.75, 0.75) {};
	\end{pgfonlayer}
	\begin{pgfonlayer}{edgelayer}
		\draw [style=wire] (3) to (0);
		\draw [style=wire, bend left, looseness=1.00] (0) to (2);
		\draw [style=wire] (2) to (1);
		\draw [style=wire, in=90, out=-156, looseness=1.00] (0) to (4);
	\end{pgfonlayer}
\end{tikzpicture}}
   \end{array}$$

\begin{proposition}\label{codevprop} The coderiving transformation $\mathsf{d}^\circ$ satisfies the following properties: 
 \begin{enumerate}[{\bf [cd.1]}]
\item $\mathsf{d}^\circ (e \otimes 1) = \varepsilon$
\item $\mathsf{d}^\circ (\varepsilon \otimes 1)= \Delta (\varepsilon \otimes \varepsilon)$
\item $\mathsf{d}^\circ(\Delta \otimes 1)=\Delta(1 \otimes \mathsf{d}^\circ)$
\item $\mathsf{d}^\circ(\Delta \otimes 1)(1 \otimes \sigma)=\Delta(\mathsf{d}^\circ \otimes 1)$
\item $\mathsf{d}^\circ(\delta \otimes 1)=\delta\mathsf{d}^\circ(1 \otimes \varepsilon)$
\item $\mathsf{d}^\circ (\mathsf{d}^\circ \otimes 1)=\mathsf{d}^\circ (\mathsf{d}^\circ \otimes 1)(1 \otimes \sigma)$
\item $\delta\mathsf{d}^\circ(\oc(\varepsilon) \otimes e) = 1$
\end{enumerate}
 \end{proposition} 
 
 \begin{proof} Most of these properties follow from the coalgebra modality structure: \\
{\bf[cd.1]}: Here we use the counit of the comultiplication:  
\begin{align*}
\begin{array}[c]{c}\resizebox{!}{2cm}{%
\begin{tikzpicture}
	\begin{pgfonlayer}{nodelayer}
		\node [style=port] (0) at (0, 3) {};
		\node [style=codifferential] (1) at (0, 2) {{\bf =\!=\!=}};
		\node [style={circle, draw}] (2) at (-0.5, 0.75) {$e$};
		\node [style=port] (3) at (0.5, 0) {};
	\end{pgfonlayer}
	\begin{pgfonlayer}{edgelayer}
		\draw [style=wire, bend right=15, looseness=1.25] (1) to (2);
		\draw [style=wire, bend left=15, looseness=1.00] (1) to (3);
		\draw [style=wire] (0) to (1);
	\end{pgfonlayer}
\end{tikzpicture}}%
   \end{array} =
      \begin{array}[c]{c}\resizebox{!}{2cm}{%
\begin{tikzpicture}
	\begin{pgfonlayer}{nodelayer}
		\node [style={circle, draw}] (0) at (0.75, 1.25) {$\varepsilon$};
		\node [style=duplicate] (1) at (0, 2.25) {$\Delta$};
		\node [style=port] (2) at (0, 3) {};
		\node [style=port] (3) at (0.75, 0.25) {};
		\node [style={circle, draw}] (4) at (-0.75, 1) {$e$};
	\end{pgfonlayer}
	\begin{pgfonlayer}{edgelayer}
		\draw [style=wire] (2) to (1);
		\draw [style=wire, bend left, looseness=1.00] (1) to (0);
		\draw [style=wire] (0) to (3);
		\draw [style=wire, in=90, out=-156, looseness=1.00] (1) to (4);
	\end{pgfonlayer}
\end{tikzpicture}}%
   \end{array}=
   \begin{array}[c]{c}\resizebox{!}{2cm}{%
\begin{tikzpicture}
	\begin{pgfonlayer}{nodelayer}
		\node [style=port] (0) at (0, 3) {};
		\node [style=port] (1) at (0, 0) {};
		\node [style={circle, draw}] (2) at (0, 1.5) {$\varepsilon$};
	\end{pgfonlayer}
	\begin{pgfonlayer}{edgelayer}
		\draw [style=wire] (0) to (2);
		\draw [style=wire] (2) to (1);
	\end{pgfonlayer}
\end{tikzpicture}}%
   \end{array}   
\end{align*} 
{\bf[cd.2]}: By definiton of the coderiving transformation: 
\begin{align*}
\begin{array}[c]{c}\resizebox{!}{2cm}{%
   \begin{tikzpicture}
	\begin{pgfonlayer}{nodelayer}
		\node [style=port] (0) at (0, 3) {};
		\node [style={circle, draw}] (1) at (-0.5, 0.75) {$\varepsilon$};
		\node [style=codifferential] (2) at (0, 2) {{\bf =\!=\!=}};
		\node [style=port] (3) at (0.5, 0) {};
		\node [style=port] (4) at (-0.5, 0) {};
	\end{pgfonlayer}
	\begin{pgfonlayer}{edgelayer}
		\draw [style=wire, bend right=15, looseness=1.25] (2) to (1);
		\draw [style=wire, bend left=15, looseness=1.00] (2) to (3);
		\draw [style=wire] (0) to (2);
		\draw [style=wire] (1) to (4);
	\end{pgfonlayer}
\end{tikzpicture}}%
   \end{array} =
   \begin{array}[c]{c}\resizebox{!}{2cm}{%
 \begin{tikzpicture}
	\begin{pgfonlayer}{nodelayer}
		\node [style=port] (0) at (0, 3) {};
		\node [style=port] (1) at (-0.5, 0) {};
		\node [style={circle, draw}] (2) at (-0.5, 0.75) {$\varepsilon$};
		\node [style=duplicate] (3) at (0, 2) {$\Delta$};
		\node [style=port] (4) at (0.5, 0) {};
		\node [style={circle, draw}] (5) at (0.5, 0.75) {$\varepsilon$};
	\end{pgfonlayer}
	\begin{pgfonlayer}{edgelayer}
		\draw [style=wire, bend right=15, looseness=1.25] (3) to (2);
		\draw [style=wire] (0) to (3);
		\draw [style=wire] (2) to (1);
		\draw [style=wire] (5) to (4);
		\draw [style=wire, in=97, out=-53, looseness=1.00] (3) to (5);
	\end{pgfonlayer}
\end{tikzpicture}}%
   \end{array}
\end{align*} 
{\bf[cd.3]}: Here we use the coassociativity of the comultiplication:
\begin{align*}
\begin{array}[c]{c}\resizebox{!}{2cm}{%
\begin{tikzpicture}
	\begin{pgfonlayer}{nodelayer}
		\node [style=duplicate] (0) at (-0.5, 1) {$\Delta$};
		\node [style=port] (1) at (-1, -0.25) {};
		\node [style=codifferential] (2) at (0, 2) {{\bf =\!=\!=}};
		\node [style=port] (3) at (0, -0.25) {};
		\node [style=port] (4) at (1, -0.25) {};
		\node [style=port] (5) at (0, 3) {};
	\end{pgfonlayer}
	\begin{pgfonlayer}{edgelayer}
		\draw [style=wire, bend right=15, looseness=1.25] (2) to (0);
		\draw [style=wire, bend left=15, looseness=1.00] (2) to (4);
		\draw [style=wire] (5) to (2);
		\draw [style=wire, bend right=15, looseness=1.00] (0) to (1);
		\draw [style=wire, bend left=15, looseness=1.00] (0) to (3);
	\end{pgfonlayer}
\end{tikzpicture}}%
   \end{array}=
\begin{array}[c]{c}\resizebox{!}{2cm}{%
\begin{tikzpicture}
	\begin{pgfonlayer}{nodelayer}
		\node [style={circle, draw}] (0) at (0.75, 1.25) {$\varepsilon$};
		\node [style=duplicate] (1) at (0, 2.25) {$\Delta$};
		\node [style=port] (2) at (0, 3) {};
		\node [style=port] (3) at (0.75, 0) {};
		\node [style=duplicate] (4) at (-0.75, 1) {$\Delta$};
		\node [style=port] (5) at (-1.25, 0) {};
		\node [style=port] (6) at (-0.25, 0) {};
	\end{pgfonlayer}
	\begin{pgfonlayer}{edgelayer}
		\draw [style=wire] (2) to (1);
		\draw [style=wire, bend left, looseness=1.00] (1) to (0);
		\draw [style=wire] (0) to (3);
		\draw [style=wire, in=90, out=-156, looseness=1.00] (1) to (4);
		\draw [style=wire, in=90, out=-135, looseness=1.00] (4) to (5);
		\draw [style=wire, in=90, out=-45, looseness=1.00] (4) to (6);
	\end{pgfonlayer}
\end{tikzpicture}}%
   \end{array}=
   \begin{array}[c]{c}\resizebox{!}{2cm}{%
\begin{tikzpicture}
	\begin{pgfonlayer}{nodelayer}
		\node [style={circle, draw}] (0) at (0.75, 0.75) {$\varepsilon$};
		\node [style=duplicate] (1) at (-0.5, 2.25) {$\Delta$};
		\node [style=port] (2) at (-0.5, 3) {};
		\node [style=port] (3) at (0.75, 0) {};
		\node [style=duplicate] (4) at (0.25, 1.5) {$\Delta$};
		\node [style=port] (5) at (-1.25, 0) {};
		\node [style=port] (6) at (-0.25, 0) {};
	\end{pgfonlayer}
	\begin{pgfonlayer}{edgelayer}
		\draw [style=wire] (2) to (1);
		\draw [style=wire] (0) to (3);
		\draw [style=wire, in=90, out=-150, looseness=0.75] (1) to (5);
		\draw [style=wire, in=90, out=-135, looseness=1.00] (4) to (6);
		\draw [style=wire, in=90, out=-15, looseness=1.00] (1) to (4);
		\draw [style=wire, in=90, out=-45, looseness=1.25] (4) to (0);
	\end{pgfonlayer}
\end{tikzpicture}}%
   \end{array}=   \begin{array}[c]{c}\resizebox{!}{2cm}{%
\begin{tikzpicture}
	\begin{pgfonlayer}{nodelayer}
		\node [style=duplicate] (0) at (0, 2) {$\Delta$};
		\node [style=port] (1) at (0, -0.25) {};
		\node [style=codifferential] (2) at (0.5, 1) {{\bf =\!=\!=}};
		\node [style=port] (3) at (-1, -0.25) {};
		\node [style=port] (4) at (1, -0.25) {};
		\node [style=port] (5) at (0, 3) {};
	\end{pgfonlayer}
	\begin{pgfonlayer}{edgelayer}
		\draw [style=wire, bend left=15, looseness=1.25] (0) to (2);
		\draw [style=wire, bend right=15, looseness=1.00] (0) to (3);
		\draw [style=wire] (5) to (0);
		\draw [style=wire, bend left=15, looseness=1.00] (2) to (4);
		\draw [style=wire, bend right=15, looseness=1.00] (2) to (1);
	\end{pgfonlayer}
\end{tikzpicture}}%
   \end{array}
\end{align*}
{\bf[cd.4]}: Here we use the cocommutativity of the comultiplication:  
\begin{align*}
\begin{array}[c]{c}\resizebox{!}{2cm}{%
\begin{tikzpicture}
	\begin{pgfonlayer}{nodelayer}
		\node [style=codifferential] (0) at (-0.5, 1) {{\bf =\!=\!=}};
		\node [style=port] (1) at (1, -0.25) {};
		\node [style=port] (2) at (-1, -0.25) {};
		\node [style=port] (3) at (0, -0.25) {};
		\node [style=duplicate] (4) at (0, 2) {$\Delta$};
		\node [style=port] (5) at (0, 3) {};
	\end{pgfonlayer}
	\begin{pgfonlayer}{edgelayer}
		\draw [style=wire, bend right=15, looseness=1.25] (4) to (0);
		\draw [style=wire, bend left=15, looseness=1.00] (4) to (1);
		\draw [style=wire] (5) to (4);
		\draw [style=wire, bend right=15, looseness=1.00] (0) to (2);
		\draw [style=wire, bend left=15, looseness=1.00] (0) to (3);
	\end{pgfonlayer}
\end{tikzpicture}}%
   \end{array}=
\begin{array}[c]{c}\resizebox{!}{2cm}{%
\begin{tikzpicture}
	\begin{pgfonlayer}{nodelayer}
		\node [style={circle, draw}] (0) at (-0.25, 0.75) {$\varepsilon$};
		\node [style=duplicate] (1) at (0, 2.25) {$\Delta$};
		\node [style=port] (2) at (0, 3) {};
		\node [style=port] (3) at (-0.25, 0) {};
		\node [style=duplicate] (4) at (-0.75, 1.5) {$\Delta$};
		\node [style=port] (5) at (0.75, 0) {};
		\node [style=port] (6) at (-1.25, 0) {};
	\end{pgfonlayer}
	\begin{pgfonlayer}{edgelayer}
		\draw [style=wire] (2) to (1);
		\draw [style=wire] (0) to (3);
		\draw [style=wire, in=90, out=-30, looseness=0.75] (1) to (5);
		\draw [style=wire, in=90, out=-135, looseness=1.00] (4) to (6);
		\draw [style=wire, in=90, out=-165, looseness=1.00] (1) to (4);
		\draw [style=wire, in=90, out=-45, looseness=1.25] (4) to (0);
	\end{pgfonlayer}
\end{tikzpicture}}%
   \end{array}=
   \begin{array}[c]{c}\resizebox{!}{2cm}{%
\begin{tikzpicture}
	\begin{pgfonlayer}{nodelayer}
		\node [style={circle, draw}] (0) at (-0.25, 0.75) {$\varepsilon$};
		\node [style=duplicate] (1) at (-0.5, 2.25) {$\Delta$};
		\node [style=port] (2) at (-0.5, 3) {};
		\node [style=port] (3) at (-0.25, 0) {};
		\node [style=duplicate] (4) at (0.25, 1.5) {$\Delta$};
		\node [style=port] (5) at (-1.25, 0) {};
		\node [style=port] (6) at (0.75, 0) {};
	\end{pgfonlayer}
	\begin{pgfonlayer}{edgelayer}
		\draw [style=wire] (2) to (1);
		\draw [style=wire] (0) to (3);
		\draw [style=wire, in=90, out=-150, looseness=0.75] (1) to (5);
		\draw [style=wire, in=90, out=-45, looseness=1.00] (4) to (6);
		\draw [style=wire, in=90, out=-15, looseness=1.00] (1) to (4);
		\draw [style=wire, in=90, out=-135, looseness=1.25] (4) to (0);
	\end{pgfonlayer}
\end{tikzpicture}}%
   \end{array}=   \begin{array}[c]{c}\resizebox{!}{2cm}{%
\begin{tikzpicture}
	\begin{pgfonlayer}{nodelayer}
		\node [style={circle, draw}] (0) at (-0.25, 0.5) {$\varepsilon$};
		\node [style=duplicate] (1) at (-0.5, 2.25) {$\Delta$};
		\node [style=port] (2) at (-0.5, 3) {};
		\node [style=port] (3) at (-0.25, -0.25) {};
		\node [style=duplicate] (4) at (0.25, 1.75) {$\Delta$};
		\node [style=port] (5) at (-1.25, -0.25) {};
		\node [style=port] (6) at (0.75, -0.25) {};
	\end{pgfonlayer}
	\begin{pgfonlayer}{edgelayer}
		\draw [style=wire] (2) to (1);
		\draw [style=wire] (0) to (3);
		\draw [style=wire, in=90, out=-150, looseness=0.75] (1) to (5);
		\draw [style=wire, in=90, out=-15, looseness=1.00] (1) to (4);
		\draw [style=wire, in=-30, out=60, looseness=1.00] (0) to (4);
		\draw [style=wire, in=106, out=-135, looseness=1.25] (4) to (6);
	\end{pgfonlayer}
\end{tikzpicture}}%
   \end{array}=   \begin{array}[c]{c}\resizebox{!}{2cm}{%
\begin{tikzpicture}
	\begin{pgfonlayer}{nodelayer}
		\node [style={circle, draw}] (0) at (-0.25, 0.5) {$\varepsilon$};
		\node [style=duplicate] (1) at (-0.5, 2.25) {$\Delta$};
		\node [style=port] (2) at (-0.5, 3) {};
		\node [style=port] (3) at (-0.25, -0.25) {};
		\node [style=duplicate] (4) at (-1, 1.25) {$\Delta$};
		\node [style=port] (5) at (0.75, -0.25) {};
		\node [style=port] (6) at (-1.5, -0.25) {};
	\end{pgfonlayer}
	\begin{pgfonlayer}{edgelayer}
		\draw [style=wire] (2) to (1);
		\draw [style=wire] (0) to (3);
		\draw [style=wire, in=90, out=-165, looseness=1.00] (1) to (4);
		\draw [style=wire, bend right=15, looseness=1.00] (4) to (6);
		\draw [style=wire, in=90, out=-15, looseness=1.00] (4) to (5);
		\draw [style=wire, in=-30, out=90, looseness=1.25] (0) to (1);
	\end{pgfonlayer}
\end{tikzpicture}}%
   \end{array}=   \begin{array}[c]{c}\resizebox{!}{2cm}{%
\begin{tikzpicture}
	\begin{pgfonlayer}{nodelayer}
		\node [style=port] (0) at (0, 3) {};
		\node [style=port] (1) at (1, -0.25) {};
		\node [style=duplicate] (2) at (-0.5, 1) {$\Delta$};
		\node [style=port] (3) at (0, -0.25) {};
		\node [style=codifferential] (4) at (0, 2) {{\bf =\!=\!=}};
		\node [style=port] (5) at (-1, -0.25) {};
	\end{pgfonlayer}
	\begin{pgfonlayer}{edgelayer}
		\draw [style=wire, bend right=15, looseness=1.25] (4) to (2);
		\draw [style=wire] (0) to (4);
		\draw [style=wire, bend right=15, looseness=1.00] (2) to (5);
		\draw [style=wire, in=-45, out=90, looseness=1.50] (3) to (4);
		\draw [style=wire, bend right=15, looseness=1.00] (1) to (2);
	\end{pgfonlayer}
\end{tikzpicture}}%
   \end{array}
   \end{align*}
{\bf[cd.5]}: Here we use the comonad triangle identity and that $\delta$ is a coalgebra map: 
\begin{align*}
\begin{array}[c]{c}\resizebox{!}{2cm}{%
   \begin{tikzpicture}
	\begin{pgfonlayer}{nodelayer}
		\node [style=port] (0) at (0, 3) {};
		\node [style={circle, draw}] (1) at (-0.5, 0.75) {$\delta$};
		\node [style=codifferential] (2) at (0, 2) {{\bf =\!=\!=}};
		\node [style=port] (3) at (0.5, 0) {};
		\node [style=port] (4) at (-0.5, 0) {};
	\end{pgfonlayer}
	\begin{pgfonlayer}{edgelayer}
		\draw [style=wire, bend right=15, looseness=1.25] (2) to (1);
		\draw [style=wire, bend left=15, looseness=1.00] (2) to (3);
		\draw [style=wire] (0) to (2);
		\draw [style=wire] (1) to (4);
	\end{pgfonlayer}
\end{tikzpicture}}%
   \end{array}=
\begin{array}[c]{c}\resizebox{!}{2cm}{%
\begin{tikzpicture}
	\begin{pgfonlayer}{nodelayer}
		\node [style={circle, draw}] (0) at (0.75, 1) {$\varepsilon$};
		\node [style=duplicate] (1) at (0, 2.25) {$\Delta$};
		\node [style=port] (2) at (0, 3) {};
		\node [style=port] (3) at (0.75, 0) {};
		\node [style={circle, draw}] (4) at (-0.75, 1) {$\delta$};
		\node [style=port] (5) at (-0.75, 0) {};
	\end{pgfonlayer}
	\begin{pgfonlayer}{edgelayer}
		\draw [style=wire] (2) to (1);
		\draw [style=wire, bend left, looseness=1.00] (1) to (0);
		\draw [style=wire] (0) to (3);
		\draw [style=wire] (4) to (5);
		\draw [style=wire, bend right, looseness=1.00] (1) to (4);
	\end{pgfonlayer}
\end{tikzpicture}}%
   \end{array}=
   \begin{array}[c]{c}\resizebox{!}{2cm}{%
\begin{tikzpicture}
	\begin{pgfonlayer}{nodelayer}
		\node [style={circle, draw}] (0) at (0.75, 0) {$\varepsilon$};
		\node [style=duplicate] (1) at (0, 2.25) {$\Delta$};
		\node [style=port] (2) at (0, 3) {};
		\node [style=port] (3) at (0.75, -0.5) {};
		\node [style={circle, draw}] (4) at (-0.75, 0.75) {$\delta$};
		\node [style=port] (5) at (-0.75, -0.5) {};
		\node [style={circle, draw}] (6) at (0.75, 1.5) {$\delta$};
		\node [style={circle, draw}] (7) at (0.75, 0.75) {$\varepsilon$};
	\end{pgfonlayer}
	\begin{pgfonlayer}{edgelayer}
		\draw [style=wire] (2) to (1);
		\draw [style=wire] (0) to (3);
		\draw [style=wire] (4) to (5);
		\draw [style=wire, bend right, looseness=1.00] (1) to (4);
		\draw [style=wire, bend left=15, looseness=1.25] (1) to (6);
		\draw [style=wire] (6) to (7);
		\draw [style=wire] (7) to (0);
	\end{pgfonlayer}
\end{tikzpicture}}%
  \end{array}=   \begin{array}[c]{c}\resizebox{!}{2cm}{%
\begin{tikzpicture}
	\begin{pgfonlayer}{nodelayer}
		\node [style={circle, draw}] (0) at (0.75, 1) {$\varepsilon$};
		\node [style=duplicate] (1) at (0, 1.75) {$\Delta$};
		\node [style={circle, draw}] (2) at (0, 2.5) {$\delta$};
		\node [style=port] (3) at (-0.75, -0.25) {};
		\node [style={circle, draw}] (4) at (0.75, 0.25) {$\varepsilon$};
		\node [style=port] (5) at (0.75, -0.25) {};
		\node [style=port] (6) at (0, 3) {};
	\end{pgfonlayer}
	\begin{pgfonlayer}{edgelayer}
		\draw [style=wire] (2) to (1);
		\draw [style=wire, bend left=15, looseness=1.00] (1) to (0);
		\draw [style=wire] (4) to (5);
		\draw [style=wire] (0) to (4);
		\draw [style=wire] (6) to (2);
		\draw [style=wire, in=90, out=-135, looseness=1.00] (1) to (3);
	\end{pgfonlayer}
\end{tikzpicture}}%
  \end{array}=  \begin{array}[c]{c}\resizebox{!}{2cm}{%
\begin{tikzpicture}
	\begin{pgfonlayer}{nodelayer}
		\node [style=port] (0) at (0, 3) {};
		\node [style={circle, draw}] (1) at (0.5, 0.75) {$\varepsilon$};
		\node [style=codifferential] (2) at (0, 1.5) {{\bf =\!=\!=}};
		\node [style=port] (3) at (0.5, 0) {};
		\node [style=port] (4) at (-0.5, 0) {};
		\node [style={circle, draw}] (5) at (0, 2.25) {$\delta$};
	\end{pgfonlayer}
	\begin{pgfonlayer}{edgelayer}
		\draw [style=wire, bend left=15, looseness=1.25] (2) to (1);
		\draw [style=wire, bend right=15, looseness=1.00] (2) to (4);
		\draw [style=wire] (1) to (3);
		\draw [style=wire] (0) to (5);
		\draw [style=wire] (5) to (2);
	\end{pgfonlayer}
\end{tikzpicture}}%
   \end{array}
   \end{align*}
{\bf[cd.6]}: Here we use {\bf[cd.4]}: 
\begin{align*}
\begin{array}[c]{c}\resizebox{!}{2cm}{%
\begin{tikzpicture}
	\begin{pgfonlayer}{nodelayer}
		\node [style=port] (0) at (0, 3) {};
		\node [style=port] (1) at (1, -0.25) {};
		\node [style=codifferential] (2) at (-0.5, 1) {{\bf =\!=\!=}};
		\node [style=port] (3) at (0, -0.25) {};
		\node [style=codifferential] (4) at (0, 2) {{\bf =\!=\!=}};
		\node [style=port] (5) at (-1, -0.25) {};
	\end{pgfonlayer}
	\begin{pgfonlayer}{edgelayer}
		\draw [style=wire, bend right=15, looseness=1.25] (4) to (2);
		\draw [style=wire] (0) to (4);
		\draw [style=wire, bend right=15, looseness=1.00] (2) to (5);
		\draw [style=wire, in=-45, out=90, looseness=1.50] (3) to (4);
		\draw [style=wire, bend right=15, looseness=1.00] (1) to (2);
	\end{pgfonlayer}
\end{tikzpicture}}%
   \end{array}=
\begin{array}[c]{c}\resizebox{!}{2cm}{%
\begin{tikzpicture}
	\begin{pgfonlayer}{nodelayer}
		\node [style=codifferential] (0) at (0, 2) {{\bf =\!=\!=}};
		\node [style=port] (1) at (1, -0.25) {};
		\node [style=port] (2) at (0, 3) {};
		\node [style=duplicate] (3) at (-0.5, 1) {$\Delta$};
		\node [style=port] (4) at (-1, -0.25) {};
		\node [style=port] (5) at (0, -0.25) {};
		\node [style={circle, draw}] (6) at (1, 0.25) {$\varepsilon$};
	\end{pgfonlayer}
	\begin{pgfonlayer}{edgelayer}
		\draw [style=wire, bend right=15, looseness=1.25] (0) to (3);
		\draw [style=wire] (2) to (0);
		\draw [style=wire, bend right=15, looseness=1.00] (3) to (4);
		\draw [style=wire, in=-45, out=90, looseness=1.50] (5) to (0);
		\draw [style=wire] (6) to (1);
		\draw [style=wire, in=90, out=-30, looseness=1.00] (3) to (6);
	\end{pgfonlayer}
\end{tikzpicture}}%
   \end{array}=   \begin{array}[c]{c}\resizebox{!}{2cm}{%
\begin{tikzpicture}
	\begin{pgfonlayer}{nodelayer}
		\node [style=port] (0) at (-1, -0.5) {};
		\node [style=port] (1) at (0, -0.5) {};
		\node [style=codifferential] (2) at (-0.5, 0.75) {{\bf =\!=\!=}};
		\node [style=port] (3) at (1, -0.5) {};
		\node [style=duplicate] (4) at (0, 1.75) {$\Delta$};
		\node [style=port] (5) at (0, 2.75) {};
		\node [style={circle, draw}] (6) at (1, 0.5) {$\varepsilon$};
	\end{pgfonlayer}
	\begin{pgfonlayer}{edgelayer}
		\draw [style=wire, bend right=15, looseness=1.25] (4) to (2);
		\draw [style=wire] (5) to (4);
		\draw [style=wire, bend right=15, looseness=1.00] (2) to (0);
		\draw [style=wire, bend left=15, looseness=1.00] (2) to (1);
		\draw [style=wire, in=90, out=-30, looseness=1.25] (4) to (6);
		\draw [style=wire] (6) to (3);
	\end{pgfonlayer}
\end{tikzpicture}}%
  \end{array}=  \begin{array}[c]{c}\resizebox{!}{2cm}{%
\begin{tikzpicture}
	\begin{pgfonlayer}{nodelayer}
		\node [style=port] (0) at (0, 3) {};
		\node [style=codifferential] (1) at (-0.5, 1) {{\bf =\!=\!=}};
		\node [style=codifferential] (2) at (0, 2) {{\bf =\!=\!=}};
		\node [style=port] (3) at (1, -0.25) {};
		\node [style=port] (4) at (0, -0.25) {};
		\node [style=port] (5) at (-1, -0.25) {};
	\end{pgfonlayer}
	\begin{pgfonlayer}{edgelayer}
		\draw [style=wire, bend right=15, looseness=1.25] (2) to (1);
		\draw [style=wire, bend left=15, looseness=1.00] (2) to (3);
		\draw [style=wire] (0) to (2);
		\draw [style=wire, bend right=15, looseness=1.00] (1) to (5);
		\draw [style=wire, bend left=15, looseness=1.00] (1) to (4);
	\end{pgfonlayer}
\end{tikzpicture}}%
   \end{array}
   \end{align*}
{\bf[cd.7]}: Here we use that $\delta$ is a comonoid morphism and the counit laws for the comultiplication: 
\begin{align*}
\begin{array}[c]{c}\resizebox{!}{2.5cm}{%
   \begin{tikzpicture}
	\begin{pgfonlayer}{nodelayer}
		\node [style=port] (0) at (0, 2) {};
		\node [style={regular polygon,regular polygon sides=4, draw}] (1) at (-0.5, -0.25) {$\varepsilon$};
		\node [style=port] (2) at (-0.5, -1) {};
		\node [style=differential] (3) at (0, 0.5) {{\bf =\!=\!=}};
		\node [style={circle, draw}] (4) at (0.5, -0.25) {$e$};
		\node [style={circle, draw}] (5) at (0, 1.25) {$\delta$};
	\end{pgfonlayer}
	\begin{pgfonlayer}{edgelayer}
		\draw [style=wire, in=-26, out=94, looseness=1.00] (4) to (3);
		\draw [style=wire, bend right, looseness=1.00] (3) to (1);
		\draw [style=wire] (1) to (2);
		\draw [style=wire] (0) to (5);
		\draw [style=wire] (5) to (3);
	\end{pgfonlayer}
\end{tikzpicture}}%
 \end{array}=   \begin{array}[c]{c}\resizebox{!}{2.5cm}{%
\begin{tikzpicture}
	\begin{pgfonlayer}{nodelayer}
		\node [style={circle, draw}] (0) at (0, 2.25) {$\delta$};
		\node [style={circle, draw}] (1) at (0.5, 0) {$e$};
		\node [style={regular polygon,regular polygon sides=4, draw}] (2) at (-0.5, 0.75) {$\varepsilon$};
		\node [style=duplicate] (3) at (0, 1.5) {$\Delta$};
		\node [style=port] (4) at (0, 3) {};
		\node [style=port] (5) at (-0.5, -0.5) {};
		\node [style={circle, draw}] (6) at (0.5, 0.75) {$\varepsilon$};
	\end{pgfonlayer}
	\begin{pgfonlayer}{edgelayer}
		\draw [style=wire, bend right, looseness=1.00] (3) to (2);
		\draw [style=wire] (2) to (5);
		\draw [style=wire] (4) to (0);
		\draw [style=wire] (0) to (3);
		\draw [style=wire] (6) to (1);
		\draw [style=wire, in=90, out=-30, looseness=1.00] (3) to (6);
	\end{pgfonlayer}
\end{tikzpicture}}%
   \end{array}=    \begin{array}[c]{c}\resizebox{!}{2.5cm}{%
\begin{tikzpicture}
	\begin{pgfonlayer}{nodelayer}
		\node [style={circle, draw}] (0) at (0.5, 0) {$e$};
		\node [style={circle, draw}] (1) at (-0.5, 1.5) {$\delta$};
		\node [style=duplicate] (2) at (0, 2.25) {$\Delta$};
		\node [style=port] (3) at (0, 3) {};
		\node [style={circle, draw}] (4) at (0.5, 1.5) {$\delta$};
		\node [style={regular polygon,regular polygon sides=4, draw}] (5) at (-0.5, 0.5) {$\varepsilon$};
		\node [style=port] (6) at (-0.5, -0.5) {};
		\node [style={circle, draw}] (7) at (0.5, 0.75) {$\varepsilon$};
	\end{pgfonlayer}
	\begin{pgfonlayer}{edgelayer}
		\draw [style=wire, bend right, looseness=1.00] (2) to (1);
		\draw [style=wire, in=90, out=-30, looseness=1.00] (2) to (4);
		\draw [style=wire] (3) to (2);
		\draw [style=wire] (1) to (5);
		\draw [style=wire] (5) to (6);
		\draw [style=wire] (4) to (7);
		\draw [style=wire] (0) to (7);
	\end{pgfonlayer}
\end{tikzpicture}}%
   \end{array}=   \begin{array}[c]{c}\resizebox{!}{2cm}{%
   \begin{tikzpicture}
	\begin{pgfonlayer}{nodelayer}
		\node [style=duplicate] (0) at (1, 2) {$\Delta$};
		\node [style=port] (1) at (0.5, 0) {};
		\node [style={circle, draw}] (2) at (1.5, 0.75) {$e$};
		\node [style=port] (3) at (1, 3) {};
	\end{pgfonlayer}
	\begin{pgfonlayer}{edgelayer}
		\draw [style=wire, bend left=15, looseness=1.25] (0) to (2);
		\draw [style=wire, bend right=15, looseness=1.00] (0) to (1);
		\draw [style=wire] (3) to (0);
	\end{pgfonlayer}
\end{tikzpicture}}%
 \end{array}=  \begin{array}[c]{c}\resizebox{!}{2cm}{%
\begin{tikzpicture}
	\begin{pgfonlayer}{nodelayer}
		\node [style=port] (0) at (0, 3) {};
		\node [style=port] (1) at (0, 0) {};
	\end{pgfonlayer}
	\begin{pgfonlayer}{edgelayer}
		\draw [style=wire] (0) to (1);
	\end{pgfonlayer}
\end{tikzpicture}}%
   \end{array}
\end{align*}
\end{proof} 
 
\subsection{Monoidal Coalgebra Modalities}

A symmetric monoidal functor \cite{mac2013categories} is a triple $(\oc, m_\otimes, m_K)$ consisting of a functor $\oc$ between symmetric monoidal categories, a natural transformation $m_\otimes: \oc A \otimes \oc B \to \oc (A \otimes B)$ and a map $m_K: K \to \oc K$ satisfying certain coherences. This can be extended to defining a symmetric monoidal comonad, which is a quintuple $(\oc, \delta, \varepsilon, m_\otimes, m_K)$ consisting of a symmetric monoidal functor $(\oc, m_\otimes, m_K)$ and a comonad $(\oc, \delta, \varepsilon)$ such that $\delta$ and $\varepsilon$ be monoidal natural transformations. For a full detailed list of the coherences for symmetric monoidal comonads, see \cite{bierman1995categorical}. The string diagrams representations of $m_\otimes$ and $m_K$ are as follows: 
$$  m_\otimes = \begin{array}[c]{c}\resizebox{!}{1cm}{%
\begin{tikzpicture}
	\begin{pgfonlayer}{nodelayer}
		\node [style=port] (0) at (2, 2) {};
		\node [style=port] (1) at (0.5, 2) {};
		\node [style={regular polygon,regular polygon sides=4, draw, inner sep=1pt,minimum size=1pt}] (2) at (1.25, 1) {$\bigotimes$};
		\node [style=port] (3) at (1.25, 0.25) {};
	\end{pgfonlayer}
	\begin{pgfonlayer}{edgelayer}
		\draw [style=wire, in=-90, out=180, looseness=1.50] (2) to (1);
		\draw [style=wire, in=0, out=-90, looseness=1.50] (0) to (2);
		\draw [style=wire] (2) to (3);
	\end{pgfonlayer}
\end{tikzpicture}
  }%
\end{array} ~~~~~ m_K= \begin{array}[c]{c}\resizebox{!}{1cm}{%
\begin{tikzpicture}
	\begin{pgfonlayer}{nodelayer}
		\node [style={circle, draw}] (0) at (1, 1.75) {$m$};
		\node [style=port] (1) at (1, 0.5) {};
	\end{pgfonlayer}
	\begin{pgfonlayer}{edgelayer}
		\draw [style=wire] (0) to (1);
	\end{pgfonlayer}
\end{tikzpicture}}
   \end{array}$$
 
\begin{definition} \normalfont A \textbf{monoidal coalgebra modality} \cite{bierman1995categorical, blute2015cartesian} on a symmetric monoidal category is a symmetric monoidal comonad, $(\oc, \delta, \varepsilon, m_\otimes, m_K)$, and a coalgebra modality, $(\oc, \delta, \varepsilon, \Delta, e)$, satisfying
\begin{enumerate}[{\em (i)}]
\item $\Delta$ and $e$ are monoidal transformations, that is, the following diagrams commute:  
$$ \xymatrixcolsep{2pc}\xymatrix{& K & & K\\
\oc A \otimes \oc B \ar[ur]^-{e \otimes  e} \ar[d]_-{\Delta \otimes \Delta} \ar[r]^-{m_\otimes}  & \oc(A \otimes B) \ar[u]_-{e} \ar[dd]^-{\Delta} & K \ar[r]^-{m_K}  \ar@{=}[ur]^-{} \ar[dr]_-{m_K \otimes m_K} & \oc K \ar[d]^-{\Delta}  \ar[u]_-{e} \\
  \oc A \otimes \oc A \otimes \oc B \otimes \oc B \ar[d]_-{1 \otimes \sigma \otimes 1} & &  & \oc K \otimes \oc K \\
  \oc A \otimes \oc B \otimes \oc A \otimes \oc B \ar[r]_-{m_\otimes \otimes m_\otimes} & \oc(A \otimes B) \otimes \oc(A \otimes B) 
  } $$
  \item $\Delta$ and $e$ are $!$-coalgebra morphisms, that is, the following diagrams commute:
$$\xymatrixcolsep{2pc}\xymatrix{\oc A \ar[d]_-{\Delta} \ar[rr]^-{\delta} & & \oc \oc A \ar[d]^-{\oc(\Delta)} & \oc A \ar[d]_-{e} \ar[r]^-{\delta} & \oc \oc A \ar[d]^-{\oc(e)} \\
    \oc A \otimes \oc A \ar[r]_-{\delta \otimes \delta} & \oc \oc A \otimes \oc \oc A \ar[r]_-{m_\otimes} & \oc(\oc A \otimes \oc A) &   K \ar[r]_-{m_K} & \oc(K)
  } $$
\end{enumerate}
\end{definition}  

\begin{proposition} The coderiving transformation $\mathsf{d}^\circ$ of a monoidal coalgebra modality satisfies the following monoidal rule:  
\begin{description}
\item[\textbf{[cd.m]}] $m_\otimes \mathsf{d}^\circ=(\mathsf{d}^\circ \otimes \mathsf{d}^\circ)(1 \otimes \sigma \otimes 1)(m_\otimes \otimes 1 \otimes 1)$
\[ \begin{array}[c]{c}\resizebox{!}{2.5cm}{%
\begin{tikzpicture}
	\begin{pgfonlayer}{nodelayer}
		\node [style=differential] (0) at (2, -2.25) {$\bigotimes$};
		\node [style=port] (1) at (0.5, -3.25) {};
		\node [style={regular polygon,regular polygon sides=4, draw, inner sep=1pt,minimum size=1pt}] (2) at (1.25, 0) {$\bigotimes$};
		\node [style=port] (3) at (2, 1) {};
		\node [style=port] (4) at (0.5, 1) {};
		\node [style=codifferential] (5) at (1.25, -1) {{\bf =\!=\!=}};
		\node [style=port] (6) at (1.5, -3.25) {};
		\node [style=port] (7) at (2.5, -3.25) {};
	\end{pgfonlayer}
	\begin{pgfonlayer}{edgelayer}
		\draw [style=wire, in=-90, out=180, looseness=1.25] (2) to (4);
		\draw [style=wire, in=0, out=-90, looseness=1.25] (3) to (2);
		\draw [style=wire, bend left, looseness=1.00] (5) to (0);
		\draw [style=wire, in=90, out=-135, looseness=1.00] (5) to (1);
		\draw [style=wire] (2) to (5);
		\draw [style=wire, in=90, out=-33, looseness=1.25] (0) to (7);
		\draw [style=wire, in=90, out=-150, looseness=1.00] (0) to (6);
	\end{pgfonlayer}
\end{tikzpicture}
  }%
\end{array}=    \begin{array}[c]{c}\resizebox{!}{2.25cm}{%
\begin{tikzpicture}
	\begin{pgfonlayer}{nodelayer}
		\node [style={regular polygon,regular polygon sides=4, draw, inner sep=1pt,minimum size=1pt}] (0) at (-0.75, 1) {$\bigotimes$};
		\node [style=port] (1) at (-0.75, 0) {};
		\node [style=port] (2) at (1, 0) {};
		\node [style=codifferential] (3) at (-1, 3.25) {{\bf =\!=\!=}};
		\node [style=port] (4) at (-1, 4.25) {};
		\node [style=codifferential] (5) at (1, 3.25) {{\bf =\!=\!=}};
		\node [style=port] (6) at (2.25, 0) {};
		\node [style=port] (7) at (1, 4.25) {};
	\end{pgfonlayer}
	\begin{pgfonlayer}{edgelayer}
		\draw [style=wire] (0) to (1);
		\draw [style=wire, in=90, out=-28, looseness=1.25] (3) to (2);
		\draw [style=wire] (4) to (3);
		\draw [style=wire, in=90, out=-54, looseness=1.25] (5) to (6);
		\draw [style=wire] (7) to (5);
		\draw [style=wire, in=0, out=-135, looseness=1.25] (5) to (0);
		\draw [style=wire, in=180, out=-135, looseness=1.50] (3) to (0);
	\end{pgfonlayer}
\end{tikzpicture}
  }%
\end{array}\]
\end{description}
\end{proposition}
\begin{proof} Here we use that $\Delta$ and $\varepsilon$ are monoidal transformations
\begin{align*}
\begin{array}[c]{c}\resizebox{!}{2.5cm}{%
\begin{tikzpicture}
	\begin{pgfonlayer}{nodelayer}
		\node [style=differential] (0) at (2, -2.25) {$\bigotimes$};
		\node [style=port] (1) at (0.5, -3.25) {};
		\node [style={regular polygon,regular polygon sides=4, draw, inner sep=1pt,minimum size=1pt}] (2) at (1.25, 0) {$\bigotimes$};
		\node [style=port] (3) at (2, 1) {};
		\node [style=port] (4) at (0.5, 1) {};
		\node [style=codifferential] (5) at (1.25, -1) {{\bf =\!=\!=}};
		\node [style=port] (6) at (1.5, -3.25) {};
		\node [style=port] (7) at (2.5, -3.25) {};
	\end{pgfonlayer}
	\begin{pgfonlayer}{edgelayer}
		\draw [style=wire, in=-90, out=180, looseness=1.25] (2) to (4);
		\draw [style=wire, in=0, out=-90, looseness=1.25] (3) to (2);
		\draw [style=wire, bend left, looseness=1.00] (5) to (0);
		\draw [style=wire, in=90, out=-135, looseness=1.00] (5) to (1);
		\draw [style=wire] (2) to (5);
		\draw [style=wire, in=90, out=-33, looseness=1.25] (0) to (7);
		\draw [style=wire, in=90, out=-150, looseness=1.00] (0) to (6);
	\end{pgfonlayer}
\end{tikzpicture}
  }%
\end{array}&= 
   \begin{array}[c]{c}\resizebox{!}{2.75cm}{%
\begin{tikzpicture}
	\begin{pgfonlayer}{nodelayer}
		\node [style=port] (0) at (0.5, -0.75) {};
		\node [style=duplicate] (1) at (1.25, 2.25) {$\Delta$};
		\node [style={circle, draw}] (2) at (2, 1.25) {$\varepsilon$};
		\node [style={regular polygon,regular polygon sides=4, draw, inner sep=1pt,minimum size=1pt}] (3) at (1.25, 3.25) {$\bigotimes$};
		\node [style=port] (4) at (2, 4.25) {};
		\node [style=port] (5) at (0.5, 4.25) {};
		\node [style=differential] (6) at (2, 0.25) {$\bigotimes$};
		\node [style=port] (7) at (1.5, -0.75) {};
		\node [style=port] (8) at (2.5, -0.75) {};
	\end{pgfonlayer}
	\begin{pgfonlayer}{edgelayer}
		\draw [style=wire, bend left, looseness=1.00] (1) to (2);
		\draw [style=wire, in=90, out=-150, looseness=0.75] (1) to (0);
		\draw [style=wire, in=-90, out=180, looseness=1.25] (3) to (5);
		\draw [style=wire, in=0, out=-90, looseness=1.25] (4) to (3);
		\draw [style=wire] (3) to (1);
		\draw [style=wire, in=90, out=-150, looseness=1.00] (6) to (7);
		\draw [style=wire] (2) to (6);
		\draw [style=wire, in=90, out=-30, looseness=1.25] (6) to (8);
	\end{pgfonlayer}
\end{tikzpicture}
  }%
\end{array} =    \begin{array}[c]{c}\resizebox{!}{3.25cm}{%
\begin{tikzpicture}
	\begin{pgfonlayer}{nodelayer}
		\node [style=port] (0) at (-5.5, 1.75) {};
		\node [style=duplicate] (1) at (-5.5, 0.75) {$\Delta$};
		\node [style=port] (2) at (-3.5, 1.75) {};
		\node [style={regular polygon,regular polygon sides=4, draw, inner sep=1pt,minimum size=1pt}] (3) at (-3.5, -1.25) {$\bigotimes$};
		\node [style={regular polygon,regular polygon sides=4, draw, inner sep=1pt,minimum size=1pt}] (4) at (-5.5, -1.25) {$\bigotimes$};
		\node [style=duplicate] (5) at (-3.5, 0.75) {$\Delta$};
		\node [style={circle, draw}] (6) at (-3.5, -2.25) {$\varepsilon$};
		\node [style=port] (7) at (-5.5, -4.25) {};
		\node [style=port] (8) at (-4, -4.25) {};
		\node [style=port] (9) at (-3, -4.25) {};
		\node [style=differential] (10) at (-3.5, -3.25) {$\bigotimes$};
	\end{pgfonlayer}
	\begin{pgfonlayer}{edgelayer}
		\draw [style=wire] (0) to (1);
		\draw [style=wire, in=180, out=-135, looseness=1.25] (1) to (4);
		\draw [style=wire] (2) to (5);
		\draw [style=wire, in=0, out=-45, looseness=1.25] (5) to (3);
		\draw [style=wire, in=0, out=-150, looseness=1.00] (5) to (4);
		\draw [style=wire, in=180, out=-30, looseness=1.00] (1) to (3);
		\draw [style=wire] (3) to (6);
		\draw [style=wire] (4) to (7);
		\draw [style=wire, in=90, out=-33, looseness=1.25] (10) to (9);
		\draw [style=wire, in=90, out=-150, looseness=1.00] (10) to (8);
		\draw [style=wire] (6) to (10);
	\end{pgfonlayer}
\end{tikzpicture}
  }%
\end{array}=    \begin{array}[c]{c}\resizebox{!}{2.5cm}{%
\begin{tikzpicture}
	\begin{pgfonlayer}{nodelayer}
		\node [style=port] (0) at (-5.5, 1.75) {};
		\node [style=duplicate] (1) at (-5.5, 0.75) {$\Delta$};
		\node [style=port] (2) at (-3.5, 1.75) {};
		\node [style={regular polygon,regular polygon sides=4, draw, inner sep=1pt,minimum size=1pt}] (3) at (-5.5, -1.25) {$\bigotimes$};
		\node [style=duplicate] (4) at (-3.5, 0.75) {$\Delta$};
		\node [style={circle, draw}] (5) at (-2, -1.25) {$\varepsilon$};
		\node [style=port] (6) at (-5.5, -2.25) {};
		\node [style={circle, draw}] (7) at (-3.75, -1.25) {$\varepsilon$};
		\node [style=port] (8) at (-3.75, -2.25) {};
		\node [style=port] (9) at (-2, -2.25) {};
	\end{pgfonlayer}
	\begin{pgfonlayer}{edgelayer}
		\draw [style=wire] (0) to (1);
		\draw [style=wire, in=180, out=-135, looseness=1.25] (1) to (3);
		\draw [style=wire] (2) to (4);
		\draw [style=wire, in=0, out=-150, looseness=1.00] (4) to (3);
		\draw [style=wire] (3) to (6);
		\draw [style=wire, in=90, out=-30, looseness=0.75] (1) to (7);
		\draw [style=wire, in=90, out=-15, looseness=0.75] (4) to (5);
		\draw [style=wire] (7) to (8);
		\draw [style=wire] (5) to (9);
	\end{pgfonlayer}
\end{tikzpicture}
  }%
\end{array}\\&=\begin{array}[c]{c}\resizebox{!}{2.25cm}{%
\begin{tikzpicture}
	\begin{pgfonlayer}{nodelayer}
		\node [style={regular polygon,regular polygon sides=4, draw, inner sep=1pt,minimum size=1pt}] (0) at (-0.75, 1) {$\bigotimes$};
		\node [style=port] (1) at (-0.75, 0) {};
		\node [style=port] (2) at (1, 0) {};
		\node [style=codifferential] (3) at (-1, 3.25) {{\bf =\!=\!=}};
		\node [style=port] (4) at (-1, 4.25) {};
		\node [style=codifferential] (5) at (1, 3.25) {{\bf =\!=\!=}};
		\node [style=port] (6) at (2.25, 0) {};
		\node [style=port] (7) at (1, 4.25) {};
	\end{pgfonlayer}
	\begin{pgfonlayer}{edgelayer}
		\draw [style=wire] (0) to (1);
		\draw [style=wire, in=90, out=-28, looseness=1.25] (3) to (2);
		\draw [style=wire] (4) to (3);
		\draw [style=wire, in=90, out=-54, looseness=1.25] (5) to (6);
		\draw [style=wire] (7) to (5);
		\draw [style=wire, in=0, out=-135, looseness=1.25] (5) to (0);
		\draw [style=wire, in=180, out=-135, looseness=1.50] (3) to (0);
	\end{pgfonlayer}
\end{tikzpicture}
  }%
\end{array}
\end{align*}
\end{proof} 

\section{Integral Categories}\label{intcatsec}

Integral categories are integral analogues of differential categories and therefore there are two equivalent definitions of an integral category: one in terms of the integral combinator and another in terms of the integral transformation. While having two different ways of illustrating a concept is always important, there is a reason why we wish to have both. The combinator description of the integral category structure is more intuitive while the transformation description is much more practical (in particular for examples and calculations). For this reason, we give both definitions and prove the equivalence between an integral combinator and an integral transformation. This way of introducing integral categories mirrors precisely the way differential categories were originally introduced in  \cite{dblute2006differential}. 

\subsection{Additive Categories}

The basic setting of integral and differential categories: additive symmetric monoidal categories with a coalgebra modality.  Here ``additive'' simply means commutative monoid enriched: we do not assume negatives nor do we assume biproducts (this differs from the usage in \cite{mac2013categories} for example). This allows many important examples such as the category of sets and relation or the category of modules for a commutative rig (also known as semirings: they are ri{\em n\/}gs without {\em n\/}egatives.).

\begin{definition} \normalfont An \textbf{additive category} is a commutative monoid enriched category, that is, a category in which each hom-set is a commutative monoid -- with an addition operation $+$ and zero $0$ -- such that composition preserves the additive structure, that is: $k(f\!+\!g)h\!=kfh\!+\!kgh$, $0f=0$ and $f0=0$. An \textbf{additive symmetric monoidal category} is a symmetric monoidal category which is also an additive category in which the tensor product is compatible with the additive structure in the sense that $k \otimes (f\!+\!g)\otimes h\!= \!k\otimes\!f\otimes h \!+ \!k\otimes\!g\otimes h$, $0\otimes h\!=\!0$, and $h \otimes 0\!=\!0$. 
\end{definition}

For the maps of any additive category, there is a notion of ``scalar multiplication'' by the natural numbers $\mathbb{N}$. The scalar multiplication of a map $f: A \to B$ by $n \in \mathbb{N}$ is the map $n \cdot f: A \to B$ defined by summing $n$ copies of $f$ together.
\[n \cdot f \coloneqq \begin{cases} 
 \underbrace{f + ... + f}_{n\text{-times}} & \text{if } n \geq 1 \\
 0 & \text{if } n=0 
\end{cases}\]
Furthermore, in additive symmetric monoidal categories, one has that $(n \cdot f ) \otimes g= n \cdot (f \otimes g)=f \otimes (n \cdot g)$. For every object $A$ and every natural number $n \in \mathbb{N}$ we can also define the map $n_A: A \to A$ as scalar multiplying the identity by $n$, that is, $n_A=n \cdot 1_A$. For every map $f: A \to B$, we have that $n_Af=n \cdot f=fn_B$, which implies that $n_A$ commutes (with respect to composition) with all endomorphisms $f: A \to A$. 

Since we are not assuming any negatives, it is possible to have $1=2$: this is the case in the category of sets and relations where addition is union. 

\begin{definition} \normalfont In an additive category, an object $A$ is said to be \textbf{additively idempotent} if $1_A + 1_A = 1_A$. \end{definition}

It is important to note that if $A$ is additively idempotent then $A \otimes B$ is additively idempotent for each object $B$. In particular, if the monoidal unit is additively idempotent then every object is as well. An additive category in which every object is additively idempotent is, as we shall see shortly, always an integral category.

\subsection{Integral Combinators}

An integral category is an additive symmetric monoidal category with a coalgebra modality and an integral combinator. An integral combinator is an operator on maps with the type $\oc A \to B$, called integrable maps, which produces the integral of that map. 

\begin{definition}\label{intcombdef} \normalfont For an additive symmetric monoidal category with a coalgebra modality $(!, \delta, \varepsilon, \Delta, e)$, an \textbf{integral combinator} $\mathsf{S}$ produces for each map $f: \oc A \otimes A \to B$ a map $\mathsf{S}[f]: \oc A \to B$: 
\[\infer[\mathsf{S}]{\mathsf{S}[f]: \oc A \to B}{f: \oc A \otimes A \to B}\]
such that $\mathsf{S}$ satisfies the following properties: 
\begin{enumerate}[{\bf [S.A]}]
\item Additivity: $\mathsf{S}[f+g]=\mathsf{S}[f]+\mathsf{S}[g]$
\end{enumerate}
\begin{enumerate}[{\bf [S.N]}]
\item Naturality/Linear Substitution: If the square on the left commutes, then the square on the right commutes:
  \[ \begin{array}[c]{c} \xymatrixcolsep{5pc}\xymatrix{  \oc A \otimes A \ar[r]^{f} \ar[d]_{!h \otimes h} & B \ar[d]^{k} \\
 !C \otimes C  \ar[r]_{g} & D} \end{array} \Rightarrow  
 \begin{array}[c]{c}\xymatrixcolsep{5pc}\xymatrix{  \oc A  \ar[r]^{\mathsf{S}[f]} \ar[d]_{!h} & B \ar[d]^{k} \\  
  !C \ar[r]_{\mathsf{S}[g]} & D} \end{array}\]
\end{enumerate}
\begin{enumerate}[{\bf [S.1]}]
\item Integral of Constants Rule: $\mathsf{S}[(e \otimes 1)]=\varepsilon$; 
\item Rota-Baxter Rule: 
\[\Delta (\mathsf{S}[f]\otimes \mathsf{S}[g]) =\mathsf{S}[(\Delta \otimes 1)(\mathsf{S}[f] \otimes g)]+\mathsf{S}[(\Delta \otimes 1)(1 \otimes \sigma)(f \otimes \mathsf{S}[g])]\] 
\item Interchange rule: $\mathsf{S}[\mathsf{S}[f \otimes 1] \otimes 1]=\mathsf{S}[(\mathsf{S}[f \otimes 1] \otimes 1)(1 \otimes \sigma)]$. 
\end{enumerate}
An \textbf{integral category} is an additive symmetric monoidal category with a coalgebra modality and an integral combinator. In an integral category, a map is \textbf{integrable} if it is of the form $f: \oc A \otimes A \to B$ and the \textbf{integral} of $f$ is the map $\mathsf{S}[f]: \oc A \to B$. 
\end{definition}

It might be useful for the reader to have some intuition regarding integral combinators, integrable maps and the axioms. The integral of $f: \oc A \otimes A \to B$,  $\mathsf{S}[f]: \oc A \to B$, should be thought of as the classical integral of $f$ evaluated from $0$ to $x$ as a function of $x$:
\[ S[f](x) := \int_0^x f(t)~{\sf d}t\]
To interpret this as $\mathsf{S}[f]$ one must regard $f$ as being a function of two variables $t$ and ${\sf d} t$, which is linear in ${\sf d} t$.  Classically, $f$ is regarded as a function of one (one dimensional) variable, $t$, and to obtain the interpretation as a function of two arguments one simply multiplies by the variable ${\sf d}t$.  This allows a simple interpretation of the integral notation for one dimensional functions: it leaves open the interpretation for multidimensional functions -- an issue to which we shall return.

The first axiom \textbf{[S.A]} simply formalizes the fact that the integral of a sum of functions is the sum of each integral. The second axiom \textbf{[S.N]} is a naturality condition with respect to linear maps which amounts to saying that the integral preserves linear substitution. The next axiom \textbf{[S.1]} asks our integral combinator to satisfy that the integral of the constant function $1$, or simply $\mathsf{d}t$, is the linear function $x$ (which is not the standard identity but rather the identity in the coKleisli category). The next axiom \textbf{[S.2]} is the so called Rota-Baxter rule, as discussed in our introduction. This is the fundamental axiom of our integral combinator. Finally {\bf [S.3]} ensures the independence of the order of integration -- the interchange law -- that is integral with respect to $u$ then $t$ is the same as integral with respect to $t$ then $u$.  It may be tempting to think this is related to Fubini's theorem. In fact, it is not closely related at all: we discuss this at the end of this section. 

Naturality of the integral combinator also gives a general way of describing the integrals of maps. Notice since the square on the left commutes trivially for every integrable map then the square on the right commutes:
  \[ \begin{array}[c]{c} \xymatrixcolsep{5pc}\xymatrix{  \oc A \otimes A \ar@{=}[r]^{1_{\oc A \otimes A}} \ar@{=}[d]_{\oc(1_A) \otimes 1_A} & \oc A \otimes A \ar[d]^{f} \\
 \oc A \otimes A  \ar[r]_{f} & B} \end{array} \Rightarrow  
 \begin{array}[c]{c}\xymatrixcolsep{5pc}\xymatrix{  \oc A  \ar[r]^{\mathsf{S}[1_{\oc A \otimes A}]} \ar[d]_{\oc(1_A)} & \oc A \otimes A \ar[d]^{f} \\  
  \oc A \ar[r]_{\mathsf{S}[f]} & B} \end{array}\]
Therefore, for every integrable map $f: \oc A \otimes A \to B$, we have that $\mathsf{S}[f]=\mathsf{S}[1_{\oc A \otimes A}]f$. If we define $\mathsf{s}_A=\mathsf{S}[1_{\oc A \otimes A}]: \oc A \to \oc A \otimes A$, then $\mathsf{S}[f]=\mathsf{s}_Af$. In fact, the integral combinator $\mathsf{S}$ can be completely re-expressed in terms of $\mathsf{s}$, which we discuss next.

\subsection{Integral Transformation}

Integral transformations give an equivalent (and extremely useful) definition of an integral category -- we integrate maps by precomposing with the integral transformation. Integral transformations are the analogue of deriving transformations for differential categories  \cite{dblute2006differential} and are obtained in a similar fashion through naturality of the combinator as discussed above.

\begin{definition}\label{inttran}\normalfont For an additive symmetric monoidal category with a coalgebra modality $(!, \delta, \varepsilon, \Delta, e)$, an \textbf{integral transformation} is a natural transformation $\mathsf{s}$ with components $\mathsf{s}_A: \oc A \to \oc A \otimes A$ which is represented in the graphical calculus as:
\[\mathsf{s} \coloneqq 
\begin{array}[c]{c}\resizebox{!}{1.5cm}{%
\begin{tikzpicture}
	\begin{pgfonlayer}{nodelayer}
		\node [style=integral] (0) at (0, 2) {{\bf -----}};
		\node [style=port] (1) at (0, 3) {};
		\node [style=port] (2) at (0.75, 0.75) {};
		\node [style=port] (3) at (-0.75, 0.75) {};
	\end{pgfonlayer}
	\begin{pgfonlayer}{edgelayer}
		\draw [style=wire, bend left, looseness=1.00] (0) to (2);
		\draw [style=wire] (1) to (0);
		\draw [style=wire, bend right, looseness=1.00] (0) to (3);
	\end{pgfonlayer}
\end{tikzpicture}}\end{array}\]
 such that $\mathsf{s}$ satisfies the following equations: 
\begin{enumerate}[{\bf [s.1]}]
\item Integral of Constants: $\mathsf{s}(e \otimes 1)=\varepsilon$
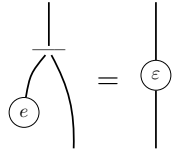
\[\begin{array}[c]{c}\resizebox{!}{2cm}{%
\begin{tikzpicture}
	\begin{pgfonlayer}{nodelayer}
		\node [style=port] (0) at (0, 3) {};
		\node [style=integral] (1) at (0, 2) {{\bf -----}};
		\node [style={circle, draw}] (2) at (-0.5, 0.75) {$e$};
		\node [style=port] (3) at (0.5, 0) {};
	\end{pgfonlayer}
	\begin{pgfonlayer}{edgelayer}
		\draw [style=wire, bend right=15, looseness=1.25] (1) to (2);
		\draw [style=wire, bend left=15, looseness=1.00] (1) to (3);
		\draw [style=wire] (0) to (1);
	\end{pgfonlayer}
\end{tikzpicture}}
   \end{array} =
   \begin{array}[c]{c}\resizebox{!}{2cm}{%
\begin{tikzpicture}
	\begin{pgfonlayer}{nodelayer}
		\node [style=port] (0) at (0, 3) {};
		\node [style=port] (1) at (0, 0) {};
		\node [style={circle, draw}] (2) at (0, 1.5) {$\varepsilon$};
	\end{pgfonlayer}
	\begin{pgfonlayer}{edgelayer}
		\draw [style=wire] (0) to (2);
		\draw [style=wire] (2) to (1);
	\end{pgfonlayer}
\end{tikzpicture}}
   \end{array}
   \]
\item Rota-Baxter Rule: $\Delta (\mathsf{s} \otimes \mathsf{s}) =\mathsf{s}(\Delta \otimes 1)(\mathsf{s} \otimes 1 \otimes 1)+\mathsf{s}(\Delta \otimes 1)(1 \otimes \sigma)(1 \otimes 1 \otimes \mathsf{s})$
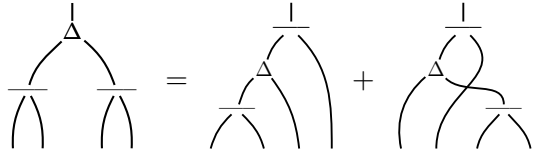
    $$\begin{array}[c]{c} \resizebox{!}{2cm}{%
 \begin{tikzpicture}
	\begin{pgfonlayer}{nodelayer}
		\node [style=integral] (0) at (-0.75, 1) {{\bf -----}};
		\node [style=port] (1) at (-0.5, 0) {};
		\node [style=port] (2) at (-1, 0) {};
		\node [style=integral] (3) at (0.75, 1) {{\bf -----}};
		\node [style=port] (4) at (1, 0) {};
		\node [style=port] (5) at (0.5, 0) {};
		\node [style=duplicate] (6) at (0, 2) {$\Delta$};
		\node [style=duplicate] (7) at (0, 2) {$\Delta$};
		\node [style=port] (8) at (0, 2.5) {};
	\end{pgfonlayer}
	\begin{pgfonlayer}{edgelayer}
		\draw [style=wire, bend right=15, looseness=1.00] (0) to (2);
		\draw [style=wire, bend left=15, looseness=1.00] (0) to (1);
		\draw [style=wire, bend right=15, looseness=1.00] (3) to (5);
		\draw [style=wire, bend left=15, looseness=1.00] (3) to (4);
		\draw [style=wire, in=74, out=-136, looseness=1.00] (6) to (0);
		\draw [style=wire, bend left=15, looseness=1.00] (6) to (3);
		\draw [style=wire] (8) to (7);
	\end{pgfonlayer}
\end{tikzpicture} }
   \end{array} =
   \begin{array}[c]{c} \resizebox{!}{2cm}{%
\begin{tikzpicture}
	\begin{pgfonlayer}{nodelayer}
		\node [style=duplicate] (0) at (-0.5, 1.5) {$\Delta$};
		\node [style=port] (1) at (-0.5, 0) {};
		\node [style=port] (2) at (-1.5, 0) {};
		\node [style=port] (3) at (0.15, 0) {};
		\node [style=duplicate] (4) at (0, 2.25) {{\bf -----}};
		\node [style=port] (5) at (0.75, 0) {};
		\node [style=duplicate] (6) at (0, 2.25) {{\bf -----}};
		\node [style=integral] (7) at (-1, 0.75) {{\bf -----}};
		\node [style=port] (8) at (0, 2.75) {};
	\end{pgfonlayer}
	\begin{pgfonlayer}{edgelayer}
		\draw [style=wire, bend left=15, looseness=1.00] (0) to (3);
		\draw [style=wire, in=71, out=-139, looseness=1.00] (0) to (7);
		\draw [style=wire, bend right=15, looseness=1.00] (7) to (2);
		\draw [style=wire, in=105, out=-45, looseness=1.00] (7) to (1);
		\draw [style=wire, in=90, out=-45, looseness=1.00] (4) to (5);
		\draw [style=wire, bend right=15, looseness=1.00] (4) to (0);
		\draw [style=wire] (8) to (6);
	\end{pgfonlayer}
\end{tikzpicture}}
   \end{array}+
   \begin{array}[c]{c} \resizebox{!}{2cm}{%
\begin{tikzpicture}
	\begin{pgfonlayer}{nodelayer}
		\node [style=port] (0) at (-1.15, 0) {};
		\node [style=port] (1) at (-0.5, 0) {};
		\node [style=duplicate] (2) at (-0.5, 1.5) {$\Delta$};
		\node [style=integral] (3) at (0.75, 0.75) {{\bf -----}};
		\node [style=port] (4) at (1.25, 0) {};
		\node [style=duplicate] (5) at (0, 2.25) {{\bf -----}};
		\node [style=duplicate] (6) at (0, 2.25) {{\bf -----}};
		\node [style=port] (7) at (0.25, 0) {};
		\node [style=port] (8) at (0, 2.75) {};
	\end{pgfonlayer}
	\begin{pgfonlayer}{edgelayer}
		\draw [style=wire, bend right=15, looseness=1.00] (6) to (2);
		\draw [style=wire, in=78, out=-132, looseness=1.00] (3) to (7);
		\draw [style=wire, in=102, out=-48, looseness=1.00] (3) to (4);
		\draw [style=wire, in=-45, out=90, looseness=1.25] (1) to (5);
		\draw [style=wire] (8) to (6);
		\draw [style=wire, bend right, looseness=1.00] (2) to (0);
		\draw [style=wire, in=90, out=-45, looseness=1.00] (2) to (3);
	\end{pgfonlayer}
\end{tikzpicture}}
   \end{array}$$
\item Interchange rule: $\mathsf{s}(\mathsf{s} \otimes 1)=\mathsf{s}(\mathsf{s} \otimes 1)(1 \otimes \sigma)$
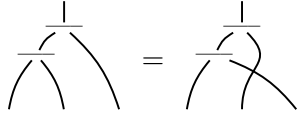
      $$\begin{array}[c]{c} \resizebox{!}{1.5cm}{%
   \begin{tikzpicture}
	\begin{pgfonlayer}{nodelayer}
		\node [style=port] (0) at (0, 2.5) {};
		\node [style=integral] (1) at (-0.5, 1.5) {{\bf -----}};
		\node [style=integral] (2) at (0, 2) {{\bf -----}};
		\node [style=port] (3) at (1, 0.5) {};
		\node [style=port] (4) at (0, 0.5) {};
		\node [style=port] (5) at (-1, 0.5) {};
	\end{pgfonlayer}
	\begin{pgfonlayer}{edgelayer}
		\draw [style=wire, bend right=15, looseness=1.25] (2) to (1);
		\draw [style=wire, bend left=15, looseness=1.00] (2) to (3);
		\draw [style=wire] (0) to (2);
		\draw [style=wire, bend right=15, looseness=1.00] (1) to (5);
		\draw [style=wire, bend left=15, looseness=1.00] (1) to (4);
	\end{pgfonlayer}
\end{tikzpicture}}
   \end{array} =
   \begin{array}[c]{c} \resizebox{!}{1.5cm}{%
\begin{tikzpicture}
	\begin{pgfonlayer}{nodelayer}
		\node [style=port] (0) at (0, 2.5) {};
		\node [style=port] (1) at (1, 0.5) {};
		\node [style=integral] (2) at (-0.5, 1.5) {{\bf -----}};
		\node [style=port] (3) at (0, 0.5) {};
		\node [style=integral] (4) at (0, 2) {{\bf -----}};
		\node [style=port] (5) at (-1, 0.5) {};
	\end{pgfonlayer}
	\begin{pgfonlayer}{edgelayer}
		\draw [style=wire, bend right=15, looseness=1.25] (4) to (2);
		\draw [style=wire] (0) to (4);
		\draw [style=wire, bend right=15, looseness=1.00] (2) to (5);
		\draw [style=wire, in=-45, out=90, looseness=1.50] (3) to (4);
		\draw [style=wire, bend right=15, looseness=1.00] (1) to (2);
	\end{pgfonlayer}
\end{tikzpicture}}
   \end{array}
   $$
\end{enumerate}
\end{definition}

The axioms for the integral transformation are simple re-expressions of the axioms for the integral combinator. Therefore every integral transformation induces an integral combinator by pre-composition $\mathsf{S}[f]=\mathsf{s}_Af$. Conversly, every integral combinator induces an integral transformation by defining $\mathsf{s}_A=\mathsf{S}[1_{\oc A \otimes A}]$.

\begin{proposition}\label{strancomb} For an additive symmetric monoidal category with a coalgebra modality, integral combinators are in bijective correspondence to integral transformations by: 
\begin{align*}
&&\mathsf{S} &\longmapsto \mathsf{s}\coloneqq \mathsf{S}[1]
&&\mathsf{s}  \longmapsto  \mathsf{S}[f] \coloneqq \mathsf{s}f
\end{align*}
Therefore, the following are equivalent: 
\begin{enumerate}[{\em (i)}]
\item An integral category;
\item An additive symmetric monoidal category with a coalgebra modality and an integral transformation. 
\end{enumerate}
\end{proposition} 

Using the integral transformation, we can obtain sufficient conditions for when two maps have the same integral. This will be important in future work when we take the coKleisli category of an integral category to obtain a cartesian integral category. 

\begin{proposition} In an integral category, if $f,g: \oc A \otimes A \to B$ are integrable maps such that $(1 \otimes \varepsilon)f=(1 \otimes \varepsilon)g$, then $\mathsf{S}[f]=\mathsf{S}[g]$. 
\end{proposition}
\begin{proof} Suppose that $(1\otimes \varepsilon)f=(1 \otimes \varepsilon)g$. By the comonad triangle equalities and the naturality of the integral transformation, we have that:
\begin{align*}
\mathsf{s}f = \delta \oc (\varepsilon) \mathsf{s} f = \delta \mathsf{s} (\oc(\varepsilon) \otimes \varepsilon) f = \delta \mathsf{s} (\oc(\varepsilon) \otimes \varepsilon) g = \delta \oc (\varepsilon) \mathsf{s} g = \mathsf{s} g
\end{align*}
\end{proof} 

The coderiving transformation would probably be one's first attempt at constructing an integral transformation. Notice that {\bf [cd.1]} and {\bf[cd.6]} from Proposition \ref{codevprop} are precisely the same as {\bf [s.1]} and {\bf [s.3]}. However, {\bf [cd.3]} and {\bf [cd.4]} give us that: 
\begin{align*}
\begin{array}[c]{c}\resizebox{!}{2cm}{%
\begin{tikzpicture}
	\begin{pgfonlayer}{nodelayer}
		\node [style=duplicate] (0) at (-0.5, 1.5) {$\Delta$};
		\node [style=port] (1) at (-0.5, -0.25) {};
		\node [style=port] (2) at (-1.5, -0.25) {};
		\node [style=port] (3) at (0.25, -0.25) {};
		\node [style=duplicate] (4) at (0, 2.25) {{\bf =\!=\!=}};
		\node [style=port] (5) at (1.25, -0.25) {};
		\node [style=duplicate] (6) at (0, 2.25) {{\bf =\!=\!=}};
		\node [style=integral] (7) at (-1, 0.75) {{\bf =\!=\!=}};
		\node [style=port] (8) at (0, 3) {};
	\end{pgfonlayer}
	\begin{pgfonlayer}{edgelayer}
		\draw [style=wire, bend left=15, looseness=1.00] (0) to (3);
		\draw [style=wire, in=71, out=-139, looseness=1.00] (0) to (7);
		\draw [style=wire, bend right=15, looseness=1.00] (7) to (2);
		\draw [style=wire, in=105, out=-45, looseness=1.00] (7) to (1);
		\draw [style=wire, in=90, out=-45, looseness=1.00] (4) to (5);
		\draw [style=wire, bend right=15, looseness=1.00] (4) to (0);
		\draw [style=wire] (8) to (6);
	\end{pgfonlayer}
\end{tikzpicture}}%
   \end{array}+
   \begin{array}[c]{c}\resizebox{!}{2cm}{%
\begin{tikzpicture}
	\begin{pgfonlayer}{nodelayer}
		\node [style=port] (0) at (-1.5, -0.25) {};
		\node [style=port] (1) at (-0.5, -0.25) {};
		\node [style=duplicate] (2) at (-0.5, 1.5) {$\Delta$};
		\node [style=integral] (3) at (0.75, 0.75) {{\bf =\!=\!=}};
		\node [style=port] (4) at (1.25, -0.25) {};
		\node [style=duplicate] (5) at (0, 2.25) {{\bf =\!=\!=}};
		\node [style=duplicate] (6) at (0, 2.25) {{\bf =\!=\!=}};
		\node [style=port] (7) at (0.25, -0.25) {};
		\node [style=port] (8) at (0, 3) {};
	\end{pgfonlayer}
	\begin{pgfonlayer}{edgelayer}
		\draw [style=wire, bend right=15, looseness=1.00] (6) to (2);
		\draw [style=wire, in=78, out=-132, looseness=1.00] (3) to (7);
		\draw [style=wire, in=102, out=-48, looseness=1.00] (3) to (4);
		\draw [style=wire, in=-45, out=90, looseness=1.25] (1) to (5);
		\draw [style=wire] (8) to (6);
		\draw [style=wire, bend right, looseness=1.00] (2) to (0);
		\draw [style=wire, in=90, out=-45, looseness=1.00] (2) to (3);
	\end{pgfonlayer}
\end{tikzpicture}}%
   \end{array}=\begin{array}[c]{c}\resizebox{!}{2cm}{%
 \begin{tikzpicture}
	\begin{pgfonlayer}{nodelayer}
		\node [style=integral] (0) at (-0.75, 1) {{\bf =\!=\!=}};
		\node [style=port] (1) at (-0.25, -0.25) {};
		\node [style=port] (2) at (-1.25, -0.25) {};
		\node [style=integral] (3) at (0.75, 1) {{\bf =\!=\!=}};
		\node [style=port] (4) at (1.25, -0.25) {};
		\node [style=port] (5) at (0.25, -0.25) {};
		\node [style=duplicate] (6) at (0, 2.25) {$\Delta$};
		\node [style=duplicate] (7) at (0, 2.25) {$\Delta$};
		\node [style=port] (8) at (0, 3) {};
	\end{pgfonlayer}
	\begin{pgfonlayer}{edgelayer}
		\draw [style=wire, bend right=15, looseness=1.00] (0) to (2);
		\draw [style=wire, bend left=15, looseness=1.00] (0) to (1);
		\draw [style=wire, bend right=15, looseness=1.00] (3) to (5);
		\draw [style=wire, bend left=15, looseness=1.00] (3) to (4);
		\draw [style=wire, in=74, out=-136, looseness=1.00] (6) to (0);
		\draw [style=wire, bend left=15, looseness=1.00] (6) to (3);
		\draw [style=wire] (8) to (7);
	\end{pgfonlayer}
\end{tikzpicture}}%
   \end{array}+\begin{array}[c]{c}\resizebox{!}{2cm}{%
 \begin{tikzpicture}
	\begin{pgfonlayer}{nodelayer}
		\node [style=integral] (0) at (-0.75, 1) {{\bf =\!=\!=}};
		\node [style=port] (1) at (-0.25, -0.25) {};
		\node [style=port] (2) at (-1.25, -0.25) {};
		\node [style=integral] (3) at (0.75, 1) {{\bf =\!=\!=}};
		\node [style=port] (4) at (1.25, -0.25) {};
		\node [style=port] (5) at (0.25, -0.25) {};
		\node [style=duplicate] (6) at (0, 2.25) {$\Delta$};
		\node [style=duplicate] (7) at (0, 2.25) {$\Delta$};
		\node [style=port] (8) at (0, 3) {};
	\end{pgfonlayer}
	\begin{pgfonlayer}{edgelayer}
		\draw [style=wire, bend right=15, looseness=1.00] (0) to (2);
		\draw [style=wire, bend left=15, looseness=1.00] (0) to (1);
		\draw [style=wire, bend right=15, looseness=1.00] (3) to (5);
		\draw [style=wire, bend left=15, looseness=1.00] (3) to (4);
		\draw [style=wire, in=74, out=-136, looseness=1.00] (6) to (0);
		\draw [style=wire, bend left=15, looseness=1.00] (6) to (3);
		\draw [style=wire] (8) to (7);
	\end{pgfonlayer}
\end{tikzpicture}}%
   \end{array} = 2 \cdot \begin{array}[c]{c}\resizebox{!}{2cm}{%
 \begin{tikzpicture}
	\begin{pgfonlayer}{nodelayer}
		\node [style=integral] (0) at (-0.75, 1) {{\bf =\!=\!=}};
		\node [style=port] (1) at (-0.25, -0.25) {};
		\node [style=port] (2) at (-1.25, -0.25) {};
		\node [style=integral] (3) at (0.75, 1) {{\bf =\!=\!=}};
		\node [style=port] (4) at (1.25, -0.25) {};
		\node [style=port] (5) at (0.25, -0.25) {};
		\node [style=duplicate] (6) at (0, 2.25) {$\Delta$};
		\node [style=duplicate] (7) at (0, 2.25) {$\Delta$};
		\node [style=port] (8) at (0, 3) {};
	\end{pgfonlayer}
	\begin{pgfonlayer}{edgelayer}
		\draw [style=wire, bend right=15, looseness=1.00] (0) to (2);
		\draw [style=wire, bend left=15, looseness=1.00] (0) to (1);
		\draw [style=wire, bend right=15, looseness=1.00] (3) to (5);
		\draw [style=wire, bend left=15, looseness=1.00] (3) to (4);
		\draw [style=wire, in=74, out=-136, looseness=1.00] (6) to (0);
		\draw [style=wire, bend left=15, looseness=1.00] (6) to (3);
		\draw [style=wire] (8) to (7);
	\end{pgfonlayer}
\end{tikzpicture}}%
   \end{array} 
\end{align*}
Therefore, the coderiving transformation fails to satisfy {\bf [s.2]}, the Rota-Baxter rule, since there is an extra factor of $2$. Of course, this problem is solved if $1+1=1$. Thus as advertised, we obtain: 

\begin{proposition} Every additive symmetric monoidal category with a coalgebra modality such that for each object $A$, $\oc A$ is additively idempotent is an integral category with respect to the coderiving transformation, that is, the coderiving transformation is an integral transformation. \end{proposition} 

\subsection{Integration by Substitution}

Arguably the most important axiom of differential categories is the chain rule -- the axiom which describes the derivative of a composition of functions. The chain rule is what separates differential categories from differential modules and differential algebras. The integration rule analogous to the chain rule is known as integration by substitution or $U$-substitution. As for integration by parts, integration by substitution is expressed using derivatives but can also be expressed in terms of only integrals: 
\[\int_0^x f(\int_0^{t}g(u)~{\sf d}u) \cdot g(t)~{\sf d}t = \int_{0}^{\int_0^{x}g(u)~{\sf d}u} f(t) ~{\sf d}t\]
Unlike differential categories, integration by substitution is not an axiom of integral categories but rather a property that only certain integrable maps satisfy (similar to a situation we will explore later in Subsection \ref{FT1subsec} with the First Fundamental Theorem of Calculus). 
 
\begin{definition} \normalfont In an integral category, an integrable map $g: \oc A \otimes A \to B$ can be \textbf{substituted} into an integrable map $f: \oc B \otimes B \to C$ if the following equality holds: 
 \[\delta\oc(\mathsf{S}[g])\mathsf{S}[f]= \mathsf{S}[(\Delta \otimes 1)(\delta \otimes 1 \otimes 1)(\oc{S}[g] \otimes g)f]\]
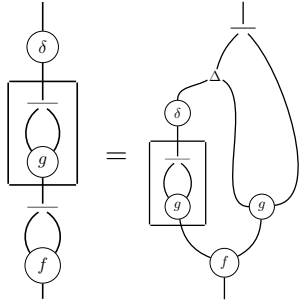
\[   \begin{array}[c]{c}\resizebox{!}{4cm}{%
\begin{tikzpicture}
	\begin{pgfonlayer}{nodelayer}
		\node [style={circle, draw}] (0) at (0, 3) {$\delta$};
		\node [style=port] (1) at (0, 4) {};
		\node [style=port] (2) at (0, -2.5) {};
		\node [style=port] (3) at (-0.75, 0) {};
		\node [style=port] (4) at (0.75, 0) {};
		\node [style=port] (5) at (0.75, 2.25) {};
		\node [style=port] (6) at (-0.75, 2.25) {};
		\node [style={circle, draw}] (7) at (0, 0.5) {$g$};
		\node [style=integral] (8) at (0, 1.75) {{\bf -----}};
		\node [style={circle, draw}] (9) at (0, -1.75) {$f$};
		\node [style=integral] (10) at (0, -0.5) {{\bf -----}};
	\end{pgfonlayer}
	\begin{pgfonlayer}{edgelayer}
		\draw [style=wire] (1) to (0);
		\draw [style=wire] (6) to (5);
		\draw [style=wire] (6) to (3);
		\draw [style=wire] (5) to (4);
		\draw [style=wire] (3) to (4);
		\draw [style=wire, bend right=45, looseness=1.00] (8) to (7);
		\draw [style=wire, bend left=45, looseness=1.00] (8) to (7);
		\draw [style=wire, bend right=45, looseness=1.00] (10) to (9);
		\draw [style=wire, bend left=45, looseness=1.00] (10) to (9);
		\draw [style=wire] (0) to (8);
		\draw [style=wire] (7) to (10);
		\draw [style=wire] (9) to (2);
	\end{pgfonlayer}
\end{tikzpicture}
  }%
\end{array}=   \begin{array}[c]{c}\resizebox{!}{4cm}{%
\begin{tikzpicture}
	\begin{pgfonlayer}{nodelayer}
		\node [style=duplicate] (0) at (5.75, 7.25) {{\bf -----}};
		\node [style=port] (1) at (5.75, 8) {};
		\node [style=duplicate] (2) at (5.75, 7.25) {{\bf -----}};
		\node [style=duplicate] (3) at (5, 6) {$\Delta$};
		\node [style=integral] (4) at (4, 3.75) {{\bf -----}};
		\node [style=port] (5) at (3.25, 2) {};
		\node [style=port] (6) at (4.75, 2) {};
		\node [style=port] (7) at (4.75, 4.25) {};
		\node [style={circle, draw}] (8) at (4, 5) {$\delta$};
		\node [style={circle, draw}] (9) at (4, 2.5) {$g$};
		\node [style=port] (10) at (3.25, 4.25) {};
		\node [style={circle, draw}] (11) at (6.25, 2.5) {$g$};
		\node [style={circle, draw}] (12) at (5.25, 1) {$f$};
		\node [style=port] (13) at (5.25, 0) {};
	\end{pgfonlayer}
	\begin{pgfonlayer}{edgelayer}
		\draw [style=wire, bend right=15, looseness=1.00] (0) to (3);
		\draw [style=wire] (1) to (2);
		\draw [style=wire] (8) to (4);
		\draw [style=wire, bend right=45, looseness=1.00] (4) to (9);
		\draw [style=wire, bend left=45, looseness=1.00] (4) to (9);
		\draw [style=wire] (10) to (7);
		\draw [style=wire] (10) to (5);
		\draw [style=wire] (7) to (6);
		\draw [style=wire] (5) to (6);
		\draw [style=wire, in=6, out=-45, looseness=0.75] (2) to (11);
		\draw [style=wire, in=90, out=-150, looseness=1.00] (3) to (8);
		\draw [style=wire, in=180, out=-30, looseness=0.75] (3) to (11);
		\draw [style=wire, bend right, looseness=1.00] (9) to (12);
		\draw [style=wire, bend left, looseness=1.00] (11) to (12);
		\draw [style=wire] (12) to (13);
	\end{pgfonlayer}
\end{tikzpicture}  }%
\end{array}\]
An integral map $f: \oc A \otimes A \to B$ satisfies \textbf{integration by substitution} if every integrable map with codomain $A$ can be substituted into $f$. 
\end{definition}

\begin{proposition} In an integral category whose integral transformation is the coderiving transformation, all maps satisfy integration by substitution. 
\end{proposition} 
\begin{proof} Let $g: \oc A \otimes A \to B$ and $f: \oc B \otimes B \to C$ be any integrable map. Using \textbf{[cd.3]}, the comonad triangle identities, that $\delta$ is a comonoid morphism, and the naturality of the coderiving transformation we have that: 
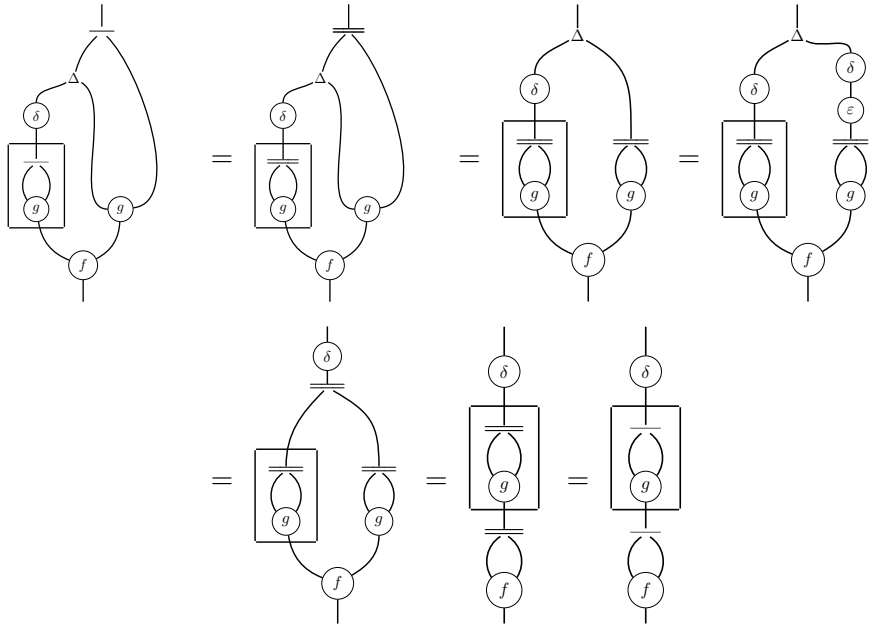
\begin{align*}
\begin{array}[c]{c}\resizebox{!}{4cm}{%
\begin{tikzpicture}
	\begin{pgfonlayer}{nodelayer}
		\node [style=duplicate] (0) at (5.75, 7.25) {{\bf -----}};
		\node [style=port] (1) at (5.75, 8) {};
		\node [style=duplicate] (2) at (5.75, 7.25) {{\bf -----}};
		\node [style=duplicate] (3) at (5, 6) {$\Delta$};
		\node [style=integral] (4) at (4, 3.75) {{\bf -----}};
		\node [style=port] (5) at (3.25, 2) {};
		\node [style=port] (6) at (4.75, 2) {};
		\node [style=port] (7) at (4.75, 4.25) {};
		\node [style={circle, draw}] (8) at (4, 5) {$\delta$};
		\node [style={circle, draw}] (9) at (4, 2.5) {$g$};
		\node [style=port] (10) at (3.25, 4.25) {};
		\node [style={circle, draw}] (11) at (6.25, 2.5) {$g$};
		\node [style={circle, draw}] (12) at (5.25, 1) {$f$};
		\node [style=port] (13) at (5.25, 0) {};
	\end{pgfonlayer}
	\begin{pgfonlayer}{edgelayer}
		\draw [style=wire, bend right=15, looseness=1.00] (0) to (3);
		\draw [style=wire] (1) to (2);
		\draw [style=wire] (8) to (4);
		\draw [style=wire, bend right=45, looseness=1.00] (4) to (9);
		\draw [style=wire, bend left=45, looseness=1.00] (4) to (9);
		\draw [style=wire] (10) to (7);
		\draw [style=wire] (10) to (5);
		\draw [style=wire] (7) to (6);
		\draw [style=wire] (5) to (6);
		\draw [style=wire, in=6, out=-45, looseness=0.75] (2) to (11);
		\draw [style=wire, in=90, out=-150, looseness=1.00] (3) to (8);
		\draw [style=wire, in=180, out=-30, looseness=0.75] (3) to (11);
		\draw [style=wire, bend right, looseness=1.00] (9) to (12);
		\draw [style=wire, bend left, looseness=1.00] (11) to (12);
		\draw [style=wire] (12) to (13);
	\end{pgfonlayer}
\end{tikzpicture}  }%
\end{array} &= \begin{array}[c]{c}\resizebox{!}{4cm}{%
\begin{tikzpicture}
	\begin{pgfonlayer}{nodelayer}
		\node [style=codifferential] (0) at (5.75, 7.25) {{\bf =\!=\!=}};
		\node [style=port] (1) at (5.75, 8) {};
		\node [style=codifferential] (2) at (5.75, 7.25) {{\bf =\!=\!=}};
		\node [style=duplicate] (3) at (5, 6) {$\Delta$};
		\node [style=codifferential] (4) at (4, 3.75) {{\bf =\!=\!=}};
		\node [style=port] (5) at (3.25, 2) {};
		\node [style=port] (6) at (4.75, 2) {};
		\node [style=port] (7) at (4.75, 4.25) {};
		\node [style={circle, draw}] (8) at (4, 5) {$\delta$};
		\node [style={circle, draw}] (9) at (4, 2.5) {$g$};
		\node [style=port] (10) at (3.25, 4.25) {};
		\node [style={circle, draw}] (11) at (6.25, 2.5) {$g$};
		\node [style={circle, draw}] (12) at (5.25, 1) {$f$};
		\node [style=port] (13) at (5.25, 0) {};
	\end{pgfonlayer}
	\begin{pgfonlayer}{edgelayer}
		\draw [style=wire, bend right=15, looseness=1.00] (0) to (3);
		\draw [style=wire] (1) to (2);
		\draw [style=wire] (8) to (4);
		\draw [style=wire, bend right=45, looseness=1.00] (4) to (9);
		\draw [style=wire, bend left=45, looseness=1.00] (4) to (9);
		\draw [style=wire] (10) to (7);
		\draw [style=wire] (10) to (5);
		\draw [style=wire] (7) to (6);
		\draw [style=wire] (5) to (6);
		\draw [style=wire, in=6, out=-45, looseness=0.75] (2) to (11);
		\draw [style=wire, in=90, out=-150, looseness=1.00] (3) to (8);
		\draw [style=wire, in=180, out=-30, looseness=0.75] (3) to (11);
		\draw [style=wire, bend right, looseness=1.00] (9) to (12);
		\draw [style=wire, bend left, looseness=1.00] (11) to (12);
		\draw [style=wire] (12) to (13);
	\end{pgfonlayer}
\end{tikzpicture}  }%
\end{array}=    \begin{array}[c]{c}\resizebox{!}{4cm}{%
\begin{tikzpicture}
	\begin{pgfonlayer}{nodelayer}
		\node [style=port] (0) at (5, 7) {};
		\node [style=duplicate] (1) at (5, 6.25) {$\Delta$};
		\node [style=codifferential] (2) at (4, 3.75) {{\bf =\!=\!=}};
		\node [style=port] (3) at (3.25, 2) {};
		\node [style=port] (4) at (4.75, 2) {};
		\node [style=port] (5) at (4.75, 4.25) {};
		\node [style={circle, draw}] (6) at (4, 5) {$\delta$};
		\node [style={circle, draw}] (7) at (4, 2.5) {$g$};
		\node [style=port] (8) at (3.25, 4.25) {};
		\node [style={circle, draw}] (9) at (5.25, 1) {$f$};
		\node [style=port] (10) at (5.25, 0) {};
		\node [style={circle, draw}] (11) at (6.25, 2.5) {$g$};
		\node [style=codifferential] (12) at (6.25, 3.75) {{\bf =\!=\!=}};
	\end{pgfonlayer}
	\begin{pgfonlayer}{edgelayer}
		\draw [style=wire] (0) to (1);
		\draw [style=wire] (6) to (2);
		\draw [style=wire, bend right=45, looseness=1.00] (2) to (7);
		\draw [style=wire, bend left=45, looseness=1.00] (2) to (7);
		\draw [style=wire] (8) to (5);
		\draw [style=wire] (8) to (3);
		\draw [style=wire] (5) to (4);
		\draw [style=wire] (3) to (4);
		\draw [style=wire, bend right, looseness=1.00] (7) to (9);
		\draw [style=wire] (9) to (10);
		\draw [style=wire, in=90, out=-150, looseness=1.00] (1) to (6);
		\draw [style=wire, bend right=45, looseness=1.00] (12) to (11);
		\draw [style=wire, bend left=45, looseness=1.00] (12) to (11);
		\draw [style=wire, in=30, out=-90, looseness=1.00] (11) to (9);
		\draw [style=wire, in=90, out=-30, looseness=1.25] (1) to (12);
	\end{pgfonlayer}
\end{tikzpicture}
  }%
\end{array}=    \begin{array}[c]{c}\resizebox{!}{4cm}{%
\begin{tikzpicture}
	\begin{pgfonlayer}{nodelayer}
		\node [style=port] (0) at (5, 7) {};
		\node [style=duplicate] (1) at (5, 6.25) {$\Delta$};
		\node [style=codifferential] (2) at (4, 3.75) {{\bf =\!=\!=}};
		\node [style=port] (3) at (3.25, 2) {};
		\node [style=port] (4) at (4.75, 2) {};
		\node [style=port] (5) at (4.75, 4.25) {};
		\node [style={circle, draw}] (6) at (4, 5) {$\delta$};
		\node [style={circle, draw}] (7) at (4, 2.5) {$g$};
		\node [style=port] (8) at (3.25, 4.25) {};
		\node [style={circle, draw}] (9) at (5.25, 1) {$f$};
		\node [style=port] (10) at (5.25, 0) {};
		\node [style={circle, draw}] (11) at (6.25, 2.5) {$g$};
		\node [style=codifferential] (12) at (6.25, 3.75) {{\bf =\!=\!=}};
		\node [style={circle, draw}] (13) at (6.25, 5.5) {$\delta$};
		\node [style={circle, draw}] (14) at (6.25, 4.5) {$\varepsilon$};
	\end{pgfonlayer}
	\begin{pgfonlayer}{edgelayer}
		\draw [style=wire] (0) to (1);
		\draw [style=wire] (6) to (2);
		\draw [style=wire, bend right=45, looseness=1.00] (2) to (7);
		\draw [style=wire, bend left=45, looseness=1.00] (2) to (7);
		\draw [style=wire] (8) to (5);
		\draw [style=wire] (8) to (3);
		\draw [style=wire] (5) to (4);
		\draw [style=wire] (3) to (4);
		\draw [style=wire, bend right, looseness=1.00] (7) to (9);
		\draw [style=wire] (9) to (10);
		\draw [style=wire, in=90, out=-150, looseness=1.00] (1) to (6);
		\draw [style=wire, bend right=45, looseness=1.00] (12) to (11);
		\draw [style=wire, bend left=45, looseness=1.00] (12) to (11);
		\draw [style=wire, in=30, out=-90, looseness=1.00] (11) to (9);
		\draw [style=wire, in=90, out=-31, looseness=1.00] (1) to (13);
		\draw [style=wire] (14) to (12);
		\draw [style=wire] (13) to (14);
	\end{pgfonlayer}
\end{tikzpicture}
  }%
\end{array}\\&=   \begin{array}[c]{c}\resizebox{!}{4cm}{%
\begin{tikzpicture}
	\begin{pgfonlayer}{nodelayer}
		\node [style=port] (0) at (5, 7.25) {};
		\node [style=codifferential] (1) at (4, 3.75) {{\bf =\!=\!=}};
		\node [style=port] (2) at (3.25, 2) {};
		\node [style=port] (3) at (4.75, 2) {};
		\node [style=port] (4) at (4.75, 4.25) {};
		\node [style={circle, draw}] (5) at (4, 2.5) {$g$};
		\node [style=port] (6) at (3.25, 4.25) {};
		\node [style={circle, draw}] (7) at (5.25, 1) {$f$};
		\node [style=port] (8) at (5.25, 0) {};
		\node [style={circle, draw}] (9) at (6.25, 2.5) {$g$};
		\node [style=codifferential] (10) at (6.25, 3.75) {{\bf =\!=\!=}};
		\node [style=codifferential] (11) at (5, 5.75) {{\bf =\!=\!=}};
		\node [style={circle, draw}] (12) at (5, 6.5) {$\delta$};
	\end{pgfonlayer}
	\begin{pgfonlayer}{edgelayer}
		\draw [style=wire, bend right=45, looseness=1.00] (1) to (5);
		\draw [style=wire, bend left=45, looseness=1.00] (1) to (5);
		\draw [style=wire] (6) to (4);
		\draw [style=wire] (6) to (2);
		\draw [style=wire] (4) to (3);
		\draw [style=wire] (2) to (3);
		\draw [style=wire, bend right, looseness=1.00] (5) to (7);
		\draw [style=wire] (7) to (8);
		\draw [style=wire, bend right=45, looseness=1.00] (10) to (9);
		\draw [style=wire, bend left=45, looseness=1.00] (10) to (9);
		\draw [style=wire, in=30, out=-90, looseness=1.00] (9) to (7);
		\draw [style=wire] (0) to (12);
		\draw [style=wire] (12) to (11);
		\draw [style=wire, in=90, out=-135, looseness=1.00] (11) to (1);
		\draw [style=wire, in=90, out=-30, looseness=1.25] (11) to (10);
	\end{pgfonlayer}
\end{tikzpicture}
  }%
\end{array}= \begin{array}[c]{c}\resizebox{!}{4cm}{%
\begin{tikzpicture}
	\begin{pgfonlayer}{nodelayer}
		\node [style={circle, draw}] (0) at (0, 3) {$\delta$};
		\node [style=port] (1) at (0, 4) {};
		\node [style=port] (2) at (0, -2.5) {};
		\node [style=port] (3) at (-0.75, 0) {};
		\node [style=port] (4) at (0.75, 0) {};
		\node [style=port] (5) at (0.75, 2.25) {};
		\node [style=port] (6) at (-0.75, 2.25) {};
		\node [style={circle, draw}] (7) at (0, 0.5) {$g$};
		\node [style=codifferential] (8) at (0, 1.75) {{\bf =\!=\!=}};
		\node [style={circle, draw}] (9) at (0, -1.75) {$f$};
		\node [style=codifferential] (10) at (0, -0.5) {{\bf =\!=\!=}};
	\end{pgfonlayer}
	\begin{pgfonlayer}{edgelayer}
		\draw [style=wire] (1) to (0);
		\draw [style=wire] (6) to (5);
		\draw [style=wire] (6) to (3);
		\draw [style=wire] (5) to (4);
		\draw [style=wire] (3) to (4);
		\draw [style=wire, bend right=45, looseness=1.00] (8) to (7);
		\draw [style=wire, bend left=45, looseness=1.00] (8) to (7);
		\draw [style=wire, bend right=45, looseness=1.00] (10) to (9);
		\draw [style=wire, bend left=45, looseness=1.00] (10) to (9);
		\draw [style=wire] (0) to (8);
		\draw [style=wire] (7) to (10);
		\draw [style=wire] (9) to (2);
	\end{pgfonlayer}
\end{tikzpicture}
  }%
\end{array}= \begin{array}[c]{c}\resizebox{!}{4cm}{%
\begin{tikzpicture}
	\begin{pgfonlayer}{nodelayer}
		\node [style={circle, draw}] (0) at (0, 3) {$\delta$};
		\node [style=port] (1) at (0, 4) {};
		\node [style=port] (2) at (0, -2.5) {};
		\node [style=port] (3) at (-0.75, 0) {};
		\node [style=port] (4) at (0.75, 0) {};
		\node [style=port] (5) at (0.75, 2.25) {};
		\node [style=port] (6) at (-0.75, 2.25) {};
		\node [style={circle, draw}] (7) at (0, 0.5) {$g$};
		\node [style=integral] (8) at (0, 1.75) {{\bf -----}};
		\node [style={circle, draw}] (9) at (0, -1.75) {$f$};
		\node [style=integral] (10) at (0, -0.5) {{\bf -----}};
	\end{pgfonlayer}
	\begin{pgfonlayer}{edgelayer}
		\draw [style=wire] (1) to (0);
		\draw [style=wire] (6) to (5);
		\draw [style=wire] (6) to (3);
		\draw [style=wire] (5) to (4);
		\draw [style=wire] (3) to (4);
		\draw [style=wire, bend right=45, looseness=1.00] (8) to (7);
		\draw [style=wire, bend left=45, looseness=1.00] (8) to (7);
		\draw [style=wire, bend right=45, looseness=1.00] (10) to (9);
		\draw [style=wire, bend left=45, looseness=1.00] (10) to (9);
		\draw [style=wire] (0) to (8);
		\draw [style=wire] (7) to (10);
		\draw [style=wire] (9) to (2);
	\end{pgfonlayer}
\end{tikzpicture}
  }%
\end{array}
\end{align*}

\end{proof} 

\subsection{Polynomial Integration and Positive Rationals}
Perhaps the first formula learned in first year calculus is the integral of monomials: 
\[\int_0^x x^n \mathsf{d}x = \frac{1}{n+1}x^{n+1}\] 
However this formula cannot be expressed in a general additive category -- simply because there may not be fractions. That said, we will soon see that in every integral category there is a notion of scalar multiplication by positive rationals, that is, certain hom-sets are $\mathbb{Q}_{\geq 0}$-modules, where $\mathbb{Q}_{\geq 0}$ is the rig of non-negative rationals. However, the integral of monomials can be re-expressed as: 
\[(n+1) \int_0^x x^n ~\mathsf{d}x = x^{n+1}\]
This re-expression of the identity does hold in any integral category!  The proof of the polynomial integration formula is actually quite elegant and demonstrates the utility of the graphical calculus. The beauty of the proof of the polynomial integration formula is that it uses every axiom: integration of constants, the Rota-Baxter rule and the interchange rule. In fact, polynomial integration in classical calculus also uses all three axioms but the use of the interchange rule goes unnoticed, since, in classical calculus, the interchange rule is simply a change of variables. 

To express the polynomial integration identity in integral categories, we will need the $n$-fold comultiplication. For every natural number $n \in \mathbb{N}$, define the $n$-fold comultiplication: 
\[\Delta_{n}: \oc A \to \underbrace{\oc A \otimes ... \otimes \oc A}_{n\text{-times}}\]
by comultiplying $\oc A$ into $n$ copies of $\oc A$. Explicitly one can write  $\Delta_{n}$ as follows:
\[\Delta_{n}=\Delta(\Delta \otimes 1)(\Delta \otimes 1 \otimes 1)...(\Delta \otimes \underbrace{1 \otimes ... \otimes 1}_{n-2\text{-times}})\]
However by co-associativity, we could have comultiplied in any other order. For example, we could have expressed $\Delta_4$ as follows:
\[\Delta_4=\Delta(1 \otimes \Delta)(1 \otimes \Delta \otimes 1)\] 
Note the following for the cases $n=0,1$ and $2$:
 \begin{enumerate}[{\em (i)}]
 \item  When $n=0$, $\Delta_{0}=e$ (since zero copies of $\oc A$ is the monoidal unit);
 \item When $n=1$, $\Delta_{1}=1$;
 \item When $n=2$, $\Delta_{2}=\Delta$. 
\end{enumerate}

With this $n$-fold comultiplication, we can now properly express and prove a rule for integration of monomials in integral categories which works properly.
\begin{proposition}\label{polyint} For every $n \in \mathbb{N}$, the integral transformation satisfies the following:
 \begin{description}
\item{\bf [s. Poly]} $(n+1) \cdot \mathsf{s}(\Delta_{n} \otimes 1)(\underbrace{\varepsilon \otimes ... \otimes \varepsilon}_{\text{$n$-times}} \otimes 1) = \Delta_{n+1}(\underbrace{\varepsilon \otimes ... \otimes \varepsilon}_{(n+1)\text{-times}})$ 
\[
(n+1) \cdot ~~    
\begin{array}[c]{c}\resizebox{!}{2cm}{%
\begin{tikzpicture}
	\begin{pgfonlayer}{nodelayer}
		\node [style=port] (0) at (-1.25, -0.25) {};
		\node [style=port] (1) at (1.25, -0.25) {};
		\node [style={circle, draw}] (2) at (-1.25, 0.5) {$\varepsilon$};
		\node [style={circle, draw}] (3) at (0.25, 0.5) {$\varepsilon$};
		\node [style=port] (4) at (0.25, -0.25) {};
		\node [style=port] (5) at (0, 3) {};
		\node [style=duplicate] (6) at (0, 2.25) {{\bf -----}};
		\node [style=duplicate] (7) at (-0.5, 1.5) {$\Delta_n$};
		\node [style=duplicate] (8) at (0, 2.25) {{\bf -----}};
		\node [style=port] (9) at (-0.5, 0.75) {$\hdots$};
	\end{pgfonlayer}
	\begin{pgfonlayer}{edgelayer}
		\draw [style=wire, in=90, out=-45, looseness=1.00] (8) to (1);
		\draw [style=wire, bend right=15, looseness=1.00] (8) to (7);
		\draw [style=wire] (5) to (6);
		\draw [style=wire] (2) to (0);
		\draw [style=wire] (3) to (4);
		\draw [style=wire, bend right=15, looseness=1.00] (7) to (2);
		\draw [style=wire, in=109, out=-41, looseness=1.00] (7) to (3);
	\end{pgfonlayer}
\end{tikzpicture} }
   \end{array} = \begin{array}[c]{c}\resizebox{!}{2cm}{%
   \begin{tikzpicture}
	\begin{pgfonlayer}{nodelayer}
		\node [style=duplicate] (0) at (-0.75, -0.25) {};
		\node [style=duplicate] (1) at (0.75, -0.25) {};
		\node [style=port] (2) at (0, 3) {};
		\node [style=duplicate] (3) at (0, 2.25) {$\Delta_{n+1}$};
		\node [style={circle, draw}] (4) at (-0.75, 1) {$\varepsilon$};
		\node [style={circle, draw}] (5) at (0.75, 1) {$\varepsilon$};
		\node [style=port] (6) at (0, 1) {$\hdots$ };
	\end{pgfonlayer}
	\begin{pgfonlayer}{edgelayer}
		\draw [style=wire, in=74, out=-136, looseness=1.00] (3) to (4);
		\draw [style=wire, in=106, out=-44, looseness=1.00] (3) to (5);
		\draw [style=wire] (2) to (3);
		\draw [style=wire] (4) to (0);
		\draw [style=wire] (5) to (1);
	\end{pgfonlayer}
\end{tikzpicture}}
   \end{array}
\]
\end{description}
\end{proposition} 
\begin{proof} We will prove the equality for the integral transformation by induction on $n$. The proof is much smoother using the graphical calculus. For the base case of $n=0$, this equality holds directly by the constant rule {\bf [s.1]}. Assume the induction hypothesis {\bf [s. Poly]} holds for $n$, we now show it for $n+1$:
 {\footnotesize \begin{align*}
  & 
  \begin{array}[c]{c} \resizebox{!}{2cm}{%
\begin{tikzpicture}
	\begin{pgfonlayer}{nodelayer}
		\node [style=duplicate] (0) at (-0.75, -0.25) {};
		\node [style=duplicate] (1) at (0.75, -0.25) {};
		\node [style=port] (2) at (0, 3) {};
		\node [style=duplicate] (3) at (0, 2.25) {$\Delta_{n+2}$};
		\node [style={circle, draw}] (4) at (-0.75, 1) {$\varepsilon$};
		\node [style={circle, draw}] (5) at (0.75, 1) {$\varepsilon$};
		\node [style=port] (6) at (0, 1) {$\hdots$ };
	\end{pgfonlayer}
	\begin{pgfonlayer}{edgelayer}
		\draw [style=wire, in=74, out=-136, looseness=1.00] (3) to (4);
		\draw [style=wire, in=106, out=-44, looseness=1.00] (3) to (5);
		\draw [style=wire] (2) to (3);
		\draw [style=wire] (4) to (0);
		\draw [style=wire] (5) to (1);
	\end{pgfonlayer}
\end{tikzpicture}}  \end{array} =
   \begin{array}[c]{c} \resizebox{!}{2cm}{%
\begin{tikzpicture}
	\begin{pgfonlayer}{nodelayer}
		\node [style={circle, draw}] (0) at (0.25, 0.5) {$\varepsilon$};
		\node [style=duplicate] (1) at (-1.25, -0.25) {};
		\node [style=duplicate] (2) at (-0.5, 1.5) {$\Delta_{n+1}$};
		\node [style={circle, draw}] (3) at (0.75, 1.5) {$\varepsilon$};
		\node [style=duplicate] (4) at (0.75, -0.25) {};
		\node [style=port] (5) at (0, 3) {};
		\node [style={circle, draw}] (6) at (-1.25, 0.5) {$\varepsilon$};
		\node [style=duplicate] (7) at (0.25, -0.25) {};
		\node [style=duplicate] (8) at (0, 2.5) {$\Delta$};
		\node [style=port] (9) at (-0.5, 0.5) {$\hdots$};
	\end{pgfonlayer}
	\begin{pgfonlayer}{edgelayer}
		\draw [style=wire] (6) to (1);
		\draw [style=wire] (0) to (7);
		\draw [style=wire] (3) to (4);
		\draw [style=wire, bend right, looseness=1.00] (2) to (6);
		\draw [style=wire, bend left, looseness=1.00] (2) to (0);
		\draw [style=wire, in=112, out=-38, looseness=1.00] (8) to (3);
		\draw [style=wire, bend right=15, looseness=1.00] (8) to (2);
		\draw [style=wire] (5) to (8);
	\end{pgfonlayer}
\end{tikzpicture}}
   \end{array}=(n+1)\!\cdot\!
   \begin{array}[c]{c} \resizebox{!}{2cm}{%
\begin{tikzpicture}
	\begin{pgfonlayer}{nodelayer}
		\node [style=integral] (0) at (-0.75, 1.5) {{\bf -----}};
		\node [style=port] (1) at (1.25, -1) {};
		\node [style=integral] (2) at (0.75, 1.5) {{\bf -----}};
		\node [style={circle, draw}] (3) at (0.25, 0.5) {$e$};
		\node [style=port] (4) at (0, 3) {};
		\node [style=duplicate] (5) at (0, 2.5) {$\Delta$};
		\node [style=port] (6) at (0, -1) {};
		\node [style=duplicate] (7) at (-1.5, 0.75) {$\Delta_{n}$};
		\node [style={circle, draw}] (8) at (-0.75, -0.25) {$\varepsilon$};
		\node [style=duplicate] (9) at (-0.75, -1) {};
		\node [style=port] (10) at (-1.5, -0.25) {$\hdots$};
		\node [style={circle, draw}] (11) at (-2.25, -0.25) {$\varepsilon$};
		\node [style=duplicate] (12) at (-2.25, -1) {};
	\end{pgfonlayer}
	\begin{pgfonlayer}{edgelayer}
		\draw [style=wire] (4) to (5);
		\draw [style=wire, bend right=15, looseness=1.25] (2) to (3);
		\draw [style=wire, bend left=15, looseness=1.00] (2) to (1);
		\draw [style=wire, bend left=15, looseness=1.00] (0) to (6);
		\draw [style=wire, in=75, out=-150, looseness=1.00] (5) to (0);
		\draw [style=wire, in=105, out=-30, looseness=1.00] (5) to (2);
		\draw [style=wire] (11) to (12);
		\draw [style=wire] (8) to (9);
		\draw [style=wire, bend right, looseness=1.00] (7) to (11);
		\draw [style=wire, bend left, looseness=1.00] (7) to (8);
		\draw [style=wire, in=-150, out=60, looseness=1.00] (7) to (0);
	\end{pgfonlayer}
\end{tikzpicture}}   \end{array}\\&=(n+1)\!\cdot\! \Bigg(
      \begin{array}[c]{c}
      \resizebox{!}{2.25cm}{%
\begin{tikzpicture}
	\begin{pgfonlayer}{nodelayer}
		\node [style=duplicate] (0) at (-0.5, 1.5) {$\Delta$};
		\node [style=duplicate] (1) at (0, 2.25) {{\bf -----}};
		\node [style=port] (2) at (0.25, -1.5) {};
		\node [style=integral] (3) at (-1, 0.75) {{\bf -----}};
		\node [style=duplicate] (4) at (0, 2.25) {{\bf -----}};
		\node [style={circle, draw}] (5) at (0.25, 0.5) {$e$};
		\node [style=port] (6) at (0, 3) {};
		\node [style=port] (7) at (1.25, -1.5) {};
		\node [style={circle, draw}] (8) at (-0.75, -0.75) {$\varepsilon$};
		\node [style=duplicate] (9) at (-2.25, -1.5) {};
		\node [style=duplicate] (10) at (-0.75, -1.5) {};
		\node [style=port] (11) at (-1.5, -0.75) {$\hdots$};
		\node [style=duplicate] (12) at (-1.5, 0.25) {$\Delta_{n}$};
		\node [style={circle, draw}] (13) at (-2.25, -0.75) {$\varepsilon$};
	\end{pgfonlayer}
	\begin{pgfonlayer}{edgelayer}
		\draw [style=wire, bend left=15, looseness=1.00] (0) to (5);
		\draw [style=wire, in=71, out=-139, looseness=1.00] (0) to (3);
		\draw [style=wire, in=105, out=-45, looseness=1.00] (3) to (2);
		\draw [style=wire, in=90, out=-45, looseness=1.00] (1) to (7);
		\draw [style=wire, bend right=15, looseness=1.00] (1) to (0);
		\draw [style=wire] (6) to (4);
		\draw [style=wire] (13) to (9);
		\draw [style=wire] (8) to (10);
		\draw [style=wire, bend right, looseness=1.00] (12) to (13);
		\draw [style=wire, bend left, looseness=1.00] (12) to (8);
		\draw [style=wire, bend left=15, looseness=1.00] (12) to (3);
	\end{pgfonlayer}
\end{tikzpicture}  } \end{array}+
   \begin{array}[c]{c}
 \resizebox{!}{2.25cm}{%
 \begin{tikzpicture}
	\begin{pgfonlayer}{nodelayer}
		\node [style=integral] (0) at (1, 0.5) {{\bf -----}};
		\node [style=port] (1) at (-0.25, -1) {};
		\node [style={circle, draw}] (2) at (0.5, -0.25) {$e$};
		\node [style=duplicate] (3) at (-0.5, 1.5) {$\Delta$};
		\node [style=port] (4) at (1.5, -1) {};
		\node [style=duplicate] (5) at (0, 2.25) {{\bf -----}};
		\node [style=port] (6) at (0, 3) {};
		\node [style=duplicate] (7) at (0, 2.25) {{\bf -----}};
		\node [style=port] (8) at (-1.5, -0.25) {$\hdots$};
		\node [style={circle, draw}] (9) at (-0.75, -0.25) {$\varepsilon$};
		\node [style=duplicate] (10) at (-0.75, -1) {};
		\node [style=duplicate] (11) at (-1.5, 0.75) {$\Delta_{n}$};
		\node [style={circle, draw}] (12) at (-2.25, -0.25) {$\varepsilon$};
		\node [style=duplicate] (13) at (-2.25, -1) {};
	\end{pgfonlayer}
	\begin{pgfonlayer}{edgelayer}
		\draw [style=wire, in=134, out=-16, looseness=1.00] (3) to (0);
		\draw [style=wire, bend right=15, looseness=1.00] (5) to (3);
		\draw [style=wire, in=86, out=-154, looseness=1.00] (0) to (2);
		\draw [style=wire, in=102, out=-48, looseness=1.00] (0) to (4);
		\draw [style=wire, in=-45, out=90, looseness=1.25] (1) to (7);
		\draw [style=wire] (6) to (5);
		\draw [style=wire] (12) to (13);
		\draw [style=wire] (9) to (10);
		\draw [style=wire, bend right, looseness=1.00] (11) to (12);
		\draw [style=wire, bend left, looseness=1.00] (11) to (9);
		\draw [style=wire, in=-158, out=52, looseness=1.00] (11) to (3);
	\end{pgfonlayer}
\end{tikzpicture}}
   \end{array} \Bigg)\\&=(n+1)\!\cdot\! \Bigg(
      \begin{array}[c]{c}
 \resizebox{!}{2.25cm}{%
 \begin{tikzpicture}
	\begin{pgfonlayer}{nodelayer}
		\node [style=port] (0) at (-1.5, -0.75) {$\hdots$};
		\node [style=duplicate] (1) at (-0.75, -1.5) {};
		\node [style=port] (2) at (0.25, -1.5) {};
		\node [style=duplicate] (3) at (0, 2.25) {{\bf -----}};
		\node [style=duplicate] (4) at (-1.5, 0.25) {$\Delta_{n}$};
		\node [style={circle, draw}] (5) at (-2.25, -0.75) {$\varepsilon$};
		\node [style=duplicate] (6) at (-2.25, -1.5) {};
		\node [style=port] (7) at (0, 3) {};
		\node [style=port] (8) at (1.25, -1.5) {};
		\node [style=integral] (9) at (-1, 0.75) {{\bf -----}};
		\node [style={circle, draw}] (10) at (-0.75, -0.75) {$\varepsilon$};
		\node [style=duplicate] (11) at (0, 2.25) {{\bf -----}};
	\end{pgfonlayer}
	\begin{pgfonlayer}{edgelayer}
		\draw [style=wire, in=105, out=-45, looseness=1.00] (9) to (2);
		\draw [style=wire, in=90, out=-45, looseness=1.00] (11) to (8);
		\draw [style=wire] (7) to (3);
		\draw [style=wire] (5) to (6);
		\draw [style=wire] (10) to (1);
		\draw [style=wire, bend right, looseness=1.00] (4) to (5);
		\draw [style=wire, bend left, looseness=1.00] (4) to (10);
		\draw [style=wire, bend left=15, looseness=1.00] (4) to (9);
		\draw [style=wire, bend left=15, looseness=1.00] (9) to (3);
	\end{pgfonlayer}
\end{tikzpicture}}
   \end{array}+
   \begin{array}[c]{c}
 \resizebox{!}{2.25cm}{%
 \begin{tikzpicture}
	\begin{pgfonlayer}{nodelayer}
		\node [style=integral] (0) at (1, 0.5) {{\bf -----}};
		\node [style=port] (1) at (-0.25, -1) {};
		\node [style={circle, draw}] (2) at (0.5, -0.25) {$e$};
		\node [style=duplicate] (3) at (-0.5, 1.5) {$\Delta$};
		\node [style=port] (4) at (1.5, -1) {};
		\node [style=duplicate] (5) at (0, 2.25) {{\bf -----}};
		\node [style=port] (6) at (0, 3) {};
		\node [style=duplicate] (7) at (0, 2.25) {{\bf -----}};
		\node [style=port] (8) at (-1.5, -0.25) {$\hdots$};
		\node [style={circle, draw}] (9) at (-0.75, -0.25) {$\varepsilon$};
		\node [style=duplicate] (10) at (-0.75, -1) {};
		\node [style=duplicate] (11) at (-1.5, 0.75) {$\Delta_{n}$};
		\node [style={circle, draw}] (12) at (-2.25, -0.25) {$\varepsilon$};
		\node [style=duplicate] (13) at (-2.25, -1) {};
	\end{pgfonlayer}
	\begin{pgfonlayer}{edgelayer}
		\draw [style=wire, in=134, out=-16, looseness=1.00] (3) to (0);
		\draw [style=wire, bend right=15, looseness=1.00] (5) to (3);
		\draw [style=wire, in=86, out=-154, looseness=1.00] (0) to (2);
		\draw [style=wire, in=102, out=-48, looseness=1.00] (0) to (4);
		\draw [style=wire, in=-45, out=90, looseness=1.25] (1) to (7);
		\draw [style=wire] (6) to (5);
		\draw [style=wire] (12) to (13);
		\draw [style=wire] (9) to (10);
		\draw [style=wire, bend right, looseness=1.00] (11) to (12);
		\draw [style=wire, bend left, looseness=1.00] (11) to (9);
		\draw [style=wire, in=-158, out=52, looseness=1.00] (11) to (3);
	\end{pgfonlayer}
\end{tikzpicture}}
   \end{array}\Bigg)\\
    &=
      \begin{array}[c]{c}
 \resizebox{!}{2cm}{%
 \begin{tikzpicture}
	\begin{pgfonlayer}{nodelayer}
		\node [style=port] (0) at (-1.25, -0.25) {};
		\node [style=port] (1) at (1.25, -0.25) {};
		\node [style={circle, draw}] (2) at (-1.25, 0.5) {$\varepsilon$};
		\node [style={circle, draw}] (3) at (0.25, 0.5) {$\varepsilon$};
		\node [style=port] (4) at (0.25, -0.25) {};
		\node [style=port] (5) at (0, 3) {};
		\node [style=duplicate] (6) at (0, 2.25) {{\bf -----}};
		\node [style=duplicate] (7) at (-0.5, 1.5) {$\Delta_{n+1}$};
		\node [style=duplicate] (8) at (0, 2.25) {{\bf -----}};
		\node [style=port] (9) at (-0.5, 0.75) {$\hdots$};
	\end{pgfonlayer}
	\begin{pgfonlayer}{edgelayer}
		\draw [style=wire, in=90, out=-45, looseness=1.00] (8) to (1);
		\draw [style=wire, bend right=15, looseness=1.00] (8) to (7);
		\draw [style=wire] (5) to (6);
		\draw [style=wire] (2) to (0);
		\draw [style=wire] (3) to (4);
		\draw [style=wire, bend right=15, looseness=1.00] (7) to (2);
		\draw [style=wire, in=109, out=-41, looseness=1.00] (7) to (3);
	\end{pgfonlayer}
\end{tikzpicture}}
   \end{array}+(n+1)\! \cdot\!
   \begin{array}[c]{c} \resizebox{!}{2cm}{%
\begin{tikzpicture}
	\begin{pgfonlayer}{nodelayer}
		\node [style=duplicate] (0) at (0, 2.25) {{\bf -----}};
		\node [style=port] (1) at (0, 3) {};
		\node [style=port] (2) at (-0.5, 0.5) {$\hdots$};
		\node [style=port] (3) at (-1.25, -0.25) {};
		\node [style=duplicate] (4) at (-0.5, 1.5) {$\Delta_{n+1}$};
		\node [style={circle, draw}] (5) at (0.25, 0.5) {$\varepsilon$};
		\node [style=port] (6) at (0.6, -0.25) {};
		\node [style=duplicate] (7) at (0, 2.25) {{\bf -----}};
		\node [style=port] (8) at (1, -0.25) {};
		\node [style={circle, draw}] (9) at (-1.25, 0.5) {$\varepsilon$};
	\end{pgfonlayer}
	\begin{pgfonlayer}{edgelayer}
		\draw [style=wire, in=90, out=-45, looseness=1.00] (0) to (6);
		\draw [style=wire, bend right=15, looseness=1.00] (0) to (4);
		\draw [style=wire] (1) to (7);
		\draw [style=wire] (9) to (3);
		\draw [style=wire, in=75, out=-30, looseness=1.00] (5) to (8);
		\draw [style=wire, bend right=15, looseness=1.00] (4) to (9);
		\draw [style=wire, in=109, out=-41, looseness=1.00] (4) to (5);
	\end{pgfonlayer}
\end{tikzpicture}}
   \end{array}=
      \begin{array}[c]{c}
    \resizebox{!}{2cm}{%
    \begin{tikzpicture}
	\begin{pgfonlayer}{nodelayer}
		\node [style=port] (0) at (-1.25, -0.25) {};
		\node [style=port] (1) at (1.25, -0.25) {};
		\node [style={circle, draw}] (2) at (-1.25, 0.5) {$\varepsilon$};
		\node [style={circle, draw}] (3) at (0.25, 0.5) {$\varepsilon$};
		\node [style=port] (4) at (0.25, -0.25) {};
		\node [style=port] (5) at (0, 3) {};
		\node [style=duplicate] (6) at (0, 2.25) {{\bf -----}};
		\node [style=duplicate] (7) at (-0.5, 1.5) {$\Delta_{n+1}$};
		\node [style=duplicate] (8) at (0, 2.25) {{\bf -----}};
		\node [style=port] (9) at (-0.5, 0.75) {$\hdots$};
	\end{pgfonlayer}
	\begin{pgfonlayer}{edgelayer}
		\draw [style=wire, in=90, out=-45, looseness=1.00] (8) to (1);
		\draw [style=wire, bend right=15, looseness=1.00] (8) to (7);
		\draw [style=wire] (5) to (6);
		\draw [style=wire] (2) to (0);
		\draw [style=wire] (3) to (4);
		\draw [style=wire, bend right=15, looseness=1.00] (7) to (2);
		\draw [style=wire, in=109, out=-41, looseness=1.00] (7) to (3);
	\end{pgfonlayer}
\end{tikzpicture} }  \end{array}+(n+1)^2\!\cdot\!
   \begin{array}[c]{c}    \resizebox{!}{2.5cm}{%
\begin{tikzpicture}
	\begin{pgfonlayer}{nodelayer}
		\node [style=port] (0) at (-1.5, -0.75) {$\hdots$};
		\node [style=duplicate] (1) at (-0.75, -1.5) {};
		\node [style=port] (2) at (1, -1.5) {};
		\node [style=duplicate] (3) at (0, 2.25) {{\bf -----}};
		\node [style=duplicate] (4) at (-1.5, 0.25) {$\Delta_{n}$};
		\node [style={circle, draw}] (5) at (-2.25, -0.75) {$\varepsilon$};
		\node [style=duplicate] (6) at (-2.25, -1.5) {};
		\node [style=port] (7) at (0, 3) {};
		\node [style=port] (8) at (0, -1.5) {};
		\node [style=integral] (9) at (-1, 1.25) {{\bf -----}};
		\node [style={circle, draw}] (10) at (-0.75, -0.75) {$\varepsilon$};
		\node [style=duplicate] (11) at (0, 2.25) {{\bf -----}};
	\end{pgfonlayer}
	\begin{pgfonlayer}{edgelayer}
		\draw [style=wire, in=105, out=-45, looseness=1.00] (9) to (2);
		\draw [style=wire, in=90, out=-45, looseness=1.00] (11) to (8);
		\draw [style=wire] (7) to (3);
		\draw [style=wire] (5) to (6);
		\draw [style=wire] (10) to (1);
		\draw [style=wire, bend right, looseness=1.00] (4) to (5);
		\draw [style=wire, bend left, looseness=1.00] (4) to (10);
		\draw [style=wire, bend left=15, looseness=1.00] (4) to (9);
		\draw [style=wire, bend left=15, looseness=1.00] (9) to (3);
	\end{pgfonlayer}
\end{tikzpicture}}
   \end{array}\\&=
      \begin{array}[c]{c}
           \resizebox{!}{2cm}{%
\begin{tikzpicture}
	\begin{pgfonlayer}{nodelayer}
		\node [style=port] (0) at (-1.25, -0.25) {};
		\node [style=port] (1) at (1.25, -0.25) {};
		\node [style={circle, draw}] (2) at (-1.25, 0.5) {$\varepsilon$};
		\node [style={circle, draw}] (3) at (0.25, 0.5) {$\varepsilon$};
		\node [style=port] (4) at (0.25, -0.25) {};
		\node [style=port] (5) at (0, 3) {};
		\node [style=duplicate] (6) at (0, 2.25) {{\bf -----}};
		\node [style=duplicate] (7) at (-0.5, 1.5) {$\Delta_{n+1}$};
		\node [style=duplicate] (8) at (0, 2.25) {{\bf -----}};
		\node [style=port] (9) at (-0.5, 0.75) {$\hdots$};
	\end{pgfonlayer}
	\begin{pgfonlayer}{edgelayer}
		\draw [style=wire, in=90, out=-45, looseness=1.00] (8) to (1);
		\draw [style=wire, bend right=15, looseness=1.00] (8) to (7);
		\draw [style=wire] (5) to (6);
		\draw [style=wire] (2) to (0);
		\draw [style=wire] (3) to (4);
		\draw [style=wire, bend right=15, looseness=1.00] (7) to (2);
		\draw [style=wire, in=109, out=-41, looseness=1.00] (7) to (3);
	\end{pgfonlayer}
\end{tikzpicture} }  \end{array}+(n+1)^2\!\cdot\!
   \begin{array}[c]{c}
   \resizebox{!}{3cm}{%
     \begin{tikzpicture}
	\begin{pgfonlayer}{nodelayer}
		\node [style={circle, draw}] (0) at (-2.25, -0.75) {$\varepsilon$};
		\node [style=duplicate] (1) at (-2.25, -1.5) {};
		\node [style=duplicate] (2) at (-0.75, -1.5) {};
		\node [style=port] (3) at (0.25, -1.5) {};
		\node [style=duplicate] (4) at (0, 2.25) {{\bf -----}};
		\node [style={circle, draw}] (5) at (-0.75, -0.75) {$\varepsilon$};
		\node [style=integral] (6) at (-1, 1.25) {{\bf -----}};
		\node [style=port] (7) at (0, 3) {};
		\node [style=duplicate] (8) at (-1.5, 0.25) {$\Delta_{n}$};
		\node [style=port] (9) at (-1.5, -0.75) {$\hdots$};
		\node [style=duplicate] (10) at (0, 2.25) {{\bf -----}};
		\node [style=port] (11) at (1, -1.5) {};
	\end{pgfonlayer}
	\begin{pgfonlayer}{edgelayer}
		\draw [style=wire, in=90, out=-45, looseness=1.00] (4) to (11);
		\draw [style=wire] (7) to (10);
		\draw [style=wire] (0) to (1);
		\draw [style=wire] (5) to (2);
		\draw [style=wire, bend right, looseness=1.00] (8) to (0);
		\draw [style=wire, bend left, looseness=1.00] (8) to (5);
		\draw [style=wire, bend left=15, looseness=1.00] (8) to (6);
		\draw [style=wire, bend left=15, looseness=1.00] (6) to (10);
		\draw [style=wire, bend left, looseness=1.00] (6) to (3);
	\end{pgfonlayer}
\end{tikzpicture}}
   \end{array}=
      \begin{array}[c]{c} \resizebox{!}{2cm}{%
\begin{tikzpicture}
	\begin{pgfonlayer}{nodelayer}
		\node [style=port] (0) at (-1.25, -0.25) {};
		\node [style=port] (1) at (1.25, -0.25) {};
		\node [style={circle, draw}] (2) at (-1.25, 0.5) {$\varepsilon$};
		\node [style={circle, draw}] (3) at (0.25, 0.5) {$\varepsilon$};
		\node [style=port] (4) at (0.25, -0.25) {};
		\node [style=port] (5) at (0, 3) {};
		\node [style=duplicate] (6) at (0, 2.25) {{\bf -----}};
		\node [style=duplicate] (7) at (-0.5, 1.5) {$\Delta_{n+1}$};
		\node [style=duplicate] (8) at (0, 2.25) {{\bf -----}};
		\node [style=port] (9) at (-0.5, 0.75) {$\hdots$};
	\end{pgfonlayer}
	\begin{pgfonlayer}{edgelayer}
		\draw [style=wire, in=90, out=-45, looseness=1.00] (8) to (1);
		\draw [style=wire, bend right=15, looseness=1.00] (8) to (7);
		\draw [style=wire] (5) to (6);
		\draw [style=wire] (2) to (0);
		\draw [style=wire] (3) to (4);
		\draw [style=wire, bend right=15, looseness=1.00] (7) to (2);
		\draw [style=wire, in=109, out=-41, looseness=1.00] (7) to (3);
	\end{pgfonlayer}
\end{tikzpicture}}   \end{array}+(n+1) \cdot
   \begin{array}[c]{c} \resizebox{!}{2cm}{%
\begin{tikzpicture}
	\begin{pgfonlayer}{nodelayer}
		\node [style=port] (0) at (-1.25, -0.25) {};
		\node [style=port] (1) at (1.25, -0.25) {};
		\node [style={circle, draw}] (2) at (-1.25, 0.5) {$\varepsilon$};
		\node [style={circle, draw}] (3) at (0.25, 0.5) {$\varepsilon$};
		\node [style=port] (4) at (0.25, -0.25) {};
		\node [style=port] (5) at (0, 3) {};
		\node [style=duplicate] (6) at (0, 2.25) {{\bf -----}};
		\node [style=duplicate] (7) at (-0.5, 1.5) {$\Delta_{n+1}$};
		\node [style=duplicate] (8) at (0, 2.25) {{\bf -----}};
		\node [style=port] (9) at (-0.5, 0.75) {$\hdots$};
	\end{pgfonlayer}
	\begin{pgfonlayer}{edgelayer}
		\draw [style=wire, in=90, out=-45, looseness=1.00] (8) to (1);
		\draw [style=wire, bend right=15, looseness=1.00] (8) to (7);
		\draw [style=wire] (5) to (6);
		\draw [style=wire] (2) to (0);
		\draw [style=wire] (3) to (4);
		\draw [style=wire, bend right=15, looseness=1.00] (7) to (2);
		\draw [style=wire, in=109, out=-41, looseness=1.00] (7) to (3);
	\end{pgfonlayer}
\end{tikzpicture} }
   \end{array}
\end{align*}}%
The induction hypothesis was used in the 3rd, 6th, and last of the above equalities. 

 \end{proof} 
 
 In classical calculus notation, the $\mathsf{d}t$ always appears at the end of the expression of the integral. However, for integral categories the position of the linear part $\mathsf{d}t$ is unimportant when integrating polynomials. Intuitively, this can be expressed using classical calculus notation as follows: 
\[\int_0^x t^n ~\mathsf{d}t = \int_0^x t^{k}~\mathsf{d}t~t^{n-k}\]
To express these formulas in integral categories, we need to introduce notation for a specific twist map. In any symmetric monoidal category $\mathbb{X}$, for each finite family of objects $\lbrace A_1, ..., A_n \rbrace$ and for each $k \leq n$, define the following permutation isomorphism:
\[\omega_{(k \ n)}: A_1 \otimes ... \otimes A_n \to A_{1} \otimes ... \otimes A_n \otimes ... \otimes A_{k}\]
which swaps the last tensor factor and the $k$-th tensor factor. Notice that $\omega_{(n \ n)}$ is the identity and when $k=0$, define $\omega_{(0 \ n)}$ as follows: 
\[\omega_{(0 \ n)}: A_1 \otimes ... \otimes A_n \to A_{n} \otimes A_1 \otimes ... \otimes A_{n-1}\]
For each finite family of objects of size $n \geq 1$ there are $n+1$ of these kinds of permutation isomorphisms. 
\begin{proposition}\label{swappoly} For every $n \in \mathbb{N}$ and every $k \leq n$, the following equality holds:
\[\mathsf{s}(\Delta_{n} \otimes 1)(\underbrace{\varepsilon \otimes ... \otimes \varepsilon}_{\text{$n$-times}} \otimes 1)\omega_{(k \ n)} = \mathsf{s}(\Delta_{n} \otimes 1)(\underbrace{\varepsilon \otimes ... \otimes \varepsilon}_{\text{$n$-times}} \otimes 1)\] 
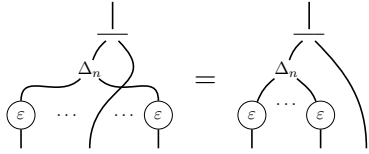
\[\begin{array}[c]{c}\resizebox{!}{2cm}{%
\begin{tikzpicture}
	\begin{pgfonlayer}{nodelayer}
		\node [style=port] (0) at (-2, -0.25) {};
		\node [style={circle, draw}] (1) at (1, 0.5) {$\varepsilon$};
		\node [style=port] (2) at (0, 3) {};
		\node [style=port] (3) at (-1, 0.5) {$\hdots$};
		\node [style=port] (4) at (1, -0.25) {};
		\node [style={circle, draw}] (5) at (-2, 0.5) {$\varepsilon$};
		\node [style=duplicate] (6) at (-0.5, 1.5) {$\Delta_n$};
		\node [style=duplicate] (7) at (0, 2.25) {{\bf -----}};
		\node [style=port] (8) at (-0.5, -0.25) {};
		\node [style=duplicate] (9) at (0, 2.25) {{\bf -----}};
		\node [style=port] (10) at (0.25, 0.5) {$\hdots$};
	\end{pgfonlayer}
	\begin{pgfonlayer}{edgelayer}
		\draw [style=wire, in=90, out=-45, looseness=1.50] (7) to (8);
		\draw [style=wire, bend right=15, looseness=1.00] (7) to (6);
		\draw [style=wire] (2) to (9);
		\draw [style=wire] (5) to (0);
		\draw [style=wire] (1) to (4);
		\draw [style=wire, in=90, out=-135, looseness=1.25] (6) to (5);
		\draw [style=wire, in=90, out=-45, looseness=1.25] (6) to (1);
	\end{pgfonlayer}
\end{tikzpicture}}
   \end{array}=
   \begin{array}[c]{c}\resizebox{!}{2cm}{%
\begin{tikzpicture}
	\begin{pgfonlayer}{nodelayer}
		\node [style=port] (0) at (-1.25, -0.25) {};
		\node [style=port] (1) at (1.25, -0.25) {};
		\node [style={circle, draw}] (2) at (-1.25, 0.5) {$\varepsilon$};
		\node [style={circle, draw}] (3) at (0.25, 0.5) {$\varepsilon$};
		\node [style=port] (4) at (0.25, -0.25) {};
		\node [style=port] (5) at (0, 3) {};
		\node [style=duplicate] (6) at (0, 2.25) {{\bf -----}};
		\node [style=duplicate] (7) at (-0.5, 1.5) {$\Delta_n$};
		\node [style=duplicate] (8) at (0, 2.25) {{\bf -----}};
		\node [style=port] (9) at (-0.5, 0.75) {$\hdots$};
	\end{pgfonlayer}
	\begin{pgfonlayer}{edgelayer}
		\draw [style=wire, in=90, out=-45, looseness=1.00] (8) to (1);
		\draw [style=wire, bend right=15, looseness=1.00] (8) to (7);
		\draw [style=wire] (5) to (6);
		\draw [style=wire] (2) to (0);
		\draw [style=wire] (3) to (4);
		\draw [style=wire, bend right=15, looseness=1.00] (7) to (2);
		\draw [style=wire, in=109, out=-41, looseness=1.00] (7) to (3);
	\end{pgfonlayer}
\end{tikzpicture}}
   \end{array}\]
\end{proposition} 
\begin{proof} We prove this by induction on $n$ using both the interchange rule {\bf [s.3]} and polynomial integration rule {\bf [s.poly]}. For the base case of $n=0$, this equality holds directly by the constant rule {\bf [s.1]} since $\Delta_{0}=e$.  Assume the induction hypothesis holds for $n$, we now show it for $n+1$ and for any $k \leq n+1$: 
{\footnotesize \begin{align*}
&\begin{array}[c]{c}\resizebox{!}{2cm}
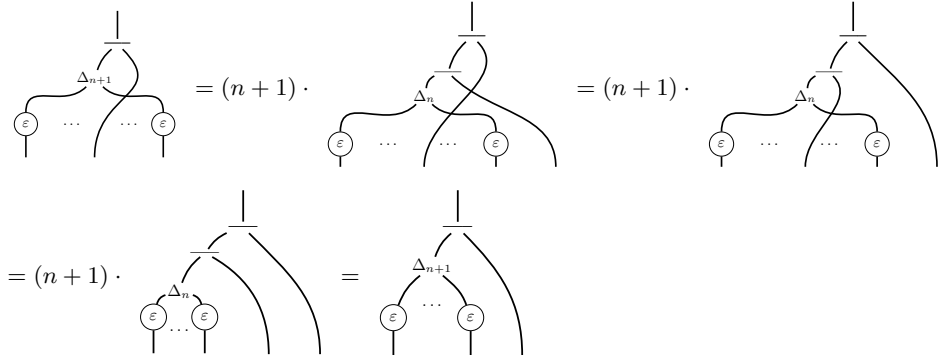
{%
\begin{tikzpicture}
	\begin{pgfonlayer}{nodelayer}
		\node [style=port] (0) at (-2, -0.25) {};
		\node [style={circle, draw}] (1) at (1, 0.5) {$\varepsilon$};
		\node [style=port] (2) at (0, 3) {};
		\node [style=port] (3) at (-1, 0.5) {$\hdots$};
		\node [style=port] (4) at (1, -0.25) {};
		\node [style={circle, draw}] (5) at (-2, 0.5) {$\varepsilon$};
		\node [style=duplicate] (6) at (-0.5, 1.5) {$\Delta_{n+1}$};
		\node [style=duplicate] (7) at (0, 2.25) {{\bf -----}};
		\node [style=port] (8) at (-0.5, -0.25) {};
		\node [style=duplicate] (9) at (0, 2.25) {{\bf -----}};
		\node [style=port] (10) at (0.25, 0.5) {$\hdots$};
	\end{pgfonlayer}
	\begin{pgfonlayer}{edgelayer}
		\draw [style=wire, in=90, out=-45, looseness=1.50] (7) to (8);
		\draw [style=wire, bend right=15, looseness=1.00] (7) to (6);
		\draw [style=wire] (2) to (9);
		\draw [style=wire] (5) to (0);
		\draw [style=wire] (1) to (4);
		\draw [style=wire, in=90, out=-135, looseness=1.25] (6) to (5);
		\draw [style=wire, in=90, out=-45, looseness=1.25] (6) to (1);
	\end{pgfonlayer}
\end{tikzpicture}}
   \end{array}=(n+1)\cdot 
   \begin{array}[c]{c}\resizebox{!}{2.25cm}{%
\begin{tikzpicture}
	\begin{pgfonlayer}{nodelayer}
		\node [style=duplicate] (0) at (-1, 1) {$\Delta_n$};
		\node [style=duplicate] (1) at (0, 2.25) {{\bf -----}};
		\node [style=port] (2) at (0.5, -0.5) {};
		\node [style={circle, draw}] (3) at (0.5, 0) {$\varepsilon$};
		\node [style=port] (4) at (-1, -0.5) {};
		\node [style=port] (5) at (1.75, -0.5) {};
		\node [style=port] (6) at (-0.5, 0) {$\hdots$};
		\node [style=integral] (7) at (-0.5, 1.5) {{\bf -----}};
		\node [style=port] (8) at (-1.75, 0) {$\hdots$};
		\node [style=port] (9) at (-2.75, -0.5) {};
		\node [style={circle, draw}] (10) at (-2.75, 0) {$\varepsilon$};
		\node [style=duplicate] (11) at (0, 2.25) {{\bf -----}};
		\node [style=port] (12) at (0, 3) {};
	\end{pgfonlayer}
	\begin{pgfonlayer}{edgelayer}
		\draw [style=wire, in=90, out=-45, looseness=1.00] (7) to (5);
		\draw [style=wire, in=90, out=-45, looseness=1.00] (1) to (4);
		\draw [style=wire] (12) to (11);
		\draw [style=wire, bend left=15, looseness=1.00] (7) to (11);
		\draw [style=wire] (10) to (9);
		\draw [style=wire] (3) to (2);
		\draw [style=wire, in=90, out=-135, looseness=1.25] (0) to (10);
		\draw [style=wire, in=90, out=-45, looseness=1.25] (0) to (3);
		\draw [style=wire, bend right, looseness=1.00] (7) to (0);
	\end{pgfonlayer}
\end{tikzpicture}}
   \end{array}=    (n+1)\cdot   \begin{array}[c]{c}\resizebox{!}{2.25cm}{%
\begin{tikzpicture}
	\begin{pgfonlayer}{nodelayer}
		\node [style={circle, draw}] (0) at (0.5, 0) {$\varepsilon$};
		\node [style=port] (1) at (0.5, -0.5) {};
		\node [style=duplicate] (2) at (0, 2.25) {{\bf -----}};
		\node [style=port] (3) at (1.75, -0.5) {};
		\node [style=port] (4) at (0, 3) {};
		\node [style={circle, draw}] (5) at (-2.75, 0) {$\varepsilon$};
		\node [style=port] (6) at (-0.5, 0) {$\hdots$};
		\node [style=port] (7) at (-2.75, -0.5) {};
		\node [style=duplicate] (8) at (0, 2.25) {{\bf -----}};
		\node [style=integral] (9) at (-0.5, 1.5) {{\bf -----}};
		\node [style=port] (10) at (-1.75, 0) {$\hdots$};
		\node [style=duplicate] (11) at (-1, 1) {$\Delta_n$};
		\node [style=port] (12) at (-1, -0.5) {};
	\end{pgfonlayer}
	\begin{pgfonlayer}{edgelayer}
		\draw [style=wire, in=90, out=-45, looseness=1.00] (9) to (12);
		\draw [style=wire, in=90, out=-45, looseness=1.00] (8) to (3);
		\draw [style=wire] (4) to (2);
		\draw [style=wire, bend left=15, looseness=1.00] (9) to (2);
		\draw [style=wire] (5) to (7);
		\draw [style=wire] (0) to (1);
		\draw [style=wire, in=90, out=-135, looseness=1.25] (11) to (5);
		\draw [style=wire, in=90, out=-45, looseness=1.25] (11) to (0);
		\draw [style=wire, bend right, looseness=1.00] (9) to (11);
	\end{pgfonlayer}
\end{tikzpicture}}
   \end{array}
\\&=  (n+1)\cdot \begin{array}[c]{c}\resizebox{!}{2.25cm}{%
\begin{tikzpicture}
	\begin{pgfonlayer}{nodelayer}
		\node [style={circle, draw}] (0) at (-0.75, 0.5) {$\varepsilon$};
		\node [style=port] (1) at (1.5, -0.25) {};
		\node [style=duplicate] (2) at (-1.75, -0.25) {};
		\node [style={circle, draw}] (3) at (-1.75, 0.5) {$\varepsilon$};
		\node [style=duplicate] (4) at (0, 2.25) {{\bf -----}};
		\node [style=duplicate] (5) at (-1.25, 1) {$\Delta_n$};
		\node [style=port] (6) at (-1.25, 0.25) {$\hdots$};
		\node [style=integral] (7) at (-0.75, 1.75) {{\bf -----}};
		\node [style=port] (8) at (0.5, -0.25) {};
		\node [style=duplicate] (9) at (-0.75, -0.25) {};
		\node [style=duplicate] (10) at (0, 2.25) {{\bf -----}};
		\node [style=port] (11) at (0, 3) {};
	\end{pgfonlayer}
	\begin{pgfonlayer}{edgelayer}
		\draw [style=wire, in=90, out=-45, looseness=1.00] (10) to (1);
		\draw [style=wire] (11) to (4);
		\draw [style=wire] (3) to (2);
		\draw [style=wire] (0) to (9);
		\draw [style=wire, bend right, looseness=1.00] (5) to (3);
		\draw [style=wire, bend left, looseness=1.00] (5) to (0);
		\draw [style=wire, bend left=15, looseness=1.00] (5) to (7);
		\draw [style=wire, bend left=15, looseness=1.00] (7) to (4);
		\draw [style=wire, bend left, looseness=1.00] (7) to (8);
	\end{pgfonlayer}
\end{tikzpicture}}
   \end{array}=
      \begin{array}[c]{c}\resizebox{!}{2.25cm}{%
\begin{tikzpicture}
	\begin{pgfonlayer}{nodelayer}
		\node [style=port] (0) at (-1.25, -0.25) {};
		\node [style=port] (1) at (1.25, -0.25) {};
		\node [style={circle, draw}] (2) at (-1.25, 0.5) {$\varepsilon$};
		\node [style={circle, draw}] (3) at (0.25, 0.5) {$\varepsilon$};
		\node [style=port] (4) at (0.25, -0.25) {};
		\node [style=port] (5) at (0, 3) {};
		\node [style=duplicate] (6) at (0, 2.25) {{\bf -----}};
		\node [style=duplicate] (7) at (-0.5, 1.5) {$\Delta_{n+1}$};
		\node [style=duplicate] (8) at (0, 2.25) {{\bf -----}};
		\node [style=port] (9) at (-0.5, 0.75) {$\hdots$};
	\end{pgfonlayer}
	\begin{pgfonlayer}{edgelayer}
		\draw [style=wire, in=90, out=-45, looseness=1.00] (8) to (1);
		\draw [style=wire, bend right=15, looseness=1.00] (8) to (7);
		\draw [style=wire] (5) to (6);
		\draw [style=wire] (2) to (0);
		\draw [style=wire] (3) to (4);
		\draw [style=wire, bend right=15, looseness=1.00] (7) to (2);
		\draw [style=wire, in=109, out=-41, looseness=1.00] (7) to (3);
	\end{pgfonlayer}
\end{tikzpicture}}
   \end{array}
\end{align*}}%
\end{proof} 

An important consequence of polynomial integration is that certain hom-sets are $\mathbb{Q}_{\geq 0}$-modules. 

\begin{theorem}\label{rationals} In an integral category, for every $n \in \mathbb{N}$, $n \geq 1$, and every object $A$, the map $n_{\oc A}: \oc A \to \oc A$ is an isomorphism. 
\end{theorem}
\begin{proof} Recall that for any $n \in \mathbb{N}$, we defined $n_{\oc A} = n \cdot 1_{\oc A} = \underbrace{1_{\oc A}+ \hdots + 1_{\oc A}}_{\text{$n$-times}}$. Then for $n \geq 1$, define $n^{-1}_{\oc A}: \oc A \to \oc A$ as follows:
\[n^{-1}_{\oc A}=\delta_A\mathsf{s}_{\oc A}(\Delta_{n-1} \otimes 1)(\underbrace{\varepsilon_{\oc A} \otimes ... \otimes \varepsilon_{\oc A}}_{\text{$n-1$-times}} \otimes 1_{\oc A})(1_{\oc A} \otimes \underbrace{e \otimes ... \otimes e}_{\text{$n-1$-times}}) \]
In string diagrams, $n^{-1}_{\oc A}$ is expressed as follows: 
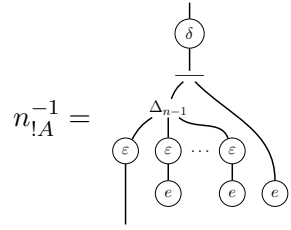
\[n^{-1}_{\oc A}=\begin{array}[c]{c}\resizebox{!}{3cm}{%
\begin{tikzpicture}
	\begin{pgfonlayer}{nodelayer}
		\node [style=port] (0) at (-1.5, -1.25) {};
		\node [style={circle, draw}] (1) at (2, -0.5) {$e$};
		\node [style={circle, draw}] (2) at (-1.5, 0.5) {$\varepsilon$};
		\node [style={circle, draw}] (3) at (1, 0.5) {$\varepsilon$};
		\node [style={circle, draw}] (4) at (1, -0.5) {$e$};
		\node [style={circle, draw}] (5) at (0, 3.25) {$\delta$};
		\node [style=duplicate] (6) at (0, 2.25) {{\bf -----}};
		\node [style=duplicate] (7) at (-0.5, 1.5) {$\Delta_{n-1}$};
		\node [style=duplicate] (8) at (0, 2.25) {{\bf -----}};
		\node [style=port] (9) at (0.25, 0.5) {$\hdots$};
			\node [style={circle, draw}] (10) at (-0.5, 0.5) {$\varepsilon$};
				\node [style={circle, draw}] (11) at (-0.5, -0.5) {$e$};
					\node [style=port] (12) at (0, 4) {};
		
	\end{pgfonlayer}
	\begin{pgfonlayer}{edgelayer}
		\draw [style=wire, in=90, out=-45, looseness=1.00] (8) to (1);
		\draw [style=wire, bend right=15, looseness=1.00] (8) to (7);
		\draw [style=wire] (5) to (6);
		\draw [style=wire] (2) to (0);
		\draw [style=wire] (3) to (4);
		\draw [style=wire, bend right=15, looseness=1.00] (7) to (2);
		\draw [style=wire, in=109, out=-41, looseness=1.00] (7) to (3);
		\draw [style=wire] (7) to (10);
		\draw [style=wire] (11) to (10);
			\draw [style=wire] (12) to (5);
	\end{pgfonlayer}
\end{tikzpicture} }
   \end{array}\]
Recall that by definition, we have $n^{-1}_{\oc A}n_{\oc A}=n \cdot n^{-1}_{\oc A}= n_{\oc A}n^{-1}_{\oc A}$. So it suffices to show that $n \cdot n^{-1}_{\oc A}=1_{\oc A}$. To prove this we use the additive structure, \textbf{[s.Poly]}, the comonad triangle identities, that $\delta$ is a comonoid morphism and the counit laws for the comultiplication.
{\footnotesize \begin{align*}
&n \cdot n^{-1}_{\oc A}=n \cdot \begin{array}[c]{c}\resizebox{!}{3cm}
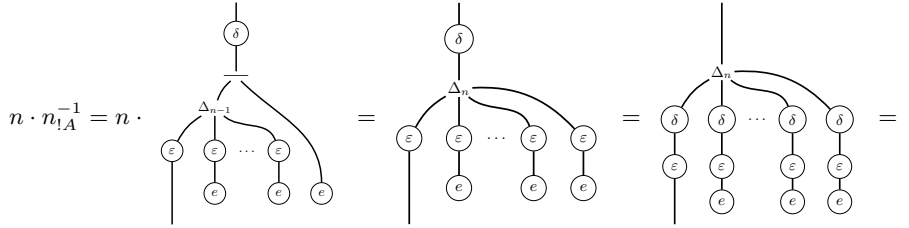
{%
\begin{tikzpicture}
	\begin{pgfonlayer}{nodelayer}
		\node [style=port] (0) at (-1.5, -1.25) {};
		\node [style={circle, draw}] (1) at (2, -0.5) {$e$};
		\node [style={circle, draw}] (2) at (-1.5, 0.5) {$\varepsilon$};
		\node [style={circle, draw}] (3) at (1, 0.5) {$\varepsilon$};
		\node [style={circle, draw}] (4) at (1, -0.5) {$e$};
		\node [style={circle, draw}] (5) at (0, 3.25) {$\delta$};
		\node [style=duplicate] (6) at (0, 2.25) {{\bf -----}};
		\node [style=duplicate] (7) at (-0.5, 1.5) {$\Delta_{n-1}$};
		\node [style=duplicate] (8) at (0, 2.25) {{\bf -----}};
		\node [style=port] (9) at (0.25, 0.5) {$\hdots$};
			\node [style={circle, draw}] (10) at (-0.5, 0.5) {$\varepsilon$};
				\node [style={circle, draw}] (11) at (-0.5, -0.5) {$e$};
					\node [style=port] (12) at (0, 4) {};
	\end{pgfonlayer}
	\begin{pgfonlayer}{edgelayer}
		\draw [style=wire, in=90, out=-45, looseness=1.00] (8) to (1);
		\draw [style=wire, bend right=15, looseness=1.00] (8) to (7);
		\draw [style=wire] (5) to (6);
		\draw [style=wire] (2) to (0);
		\draw [style=wire] (3) to (4);
		\draw [style=wire, bend right=15, looseness=1.00] (7) to (2);
		\draw [style=wire, in=109, out=-41, looseness=1.00] (7) to (3);
		\draw [style=wire] (7) to (10);
		\draw [style=wire] (11) to (10);
			\draw [style=wire] (12) to (5);
	\end{pgfonlayer}
\end{tikzpicture} }
   \end{array}=\begin{array}[c]{c}\resizebox{!}{3cm}{%
   \begin{tikzpicture}
	\begin{pgfonlayer}{nodelayer}
		\node [style=port] (0) at (-1.5, -1.25) {};
		\node [style={circle, draw}] (1) at (2, -0.5) {$e$};
		\node [style={circle, draw}] (2) at (-1.5, 0.5) {$\varepsilon$};
		\node [style={circle, draw}] (3) at (1, 0.5) {$\varepsilon$};
		\node [style={circle, draw}] (4) at (1, -0.5) {$e$};
		\node [style={circle, draw}] (5) at (-0.5, 2.5) {$\delta$};
		\node [style=duplicate] (7) at (-0.5, 1.5) {$\Delta_{n}$};
		\node [style=port] (9) at (0.25, 0.5) {$\hdots$};
			\node [style={circle, draw}] (10) at (-0.5, 0.5) {$\varepsilon$};
				\node [style={circle, draw}] (11) at (-0.5, -0.5) {$e$};
					\node [style=port] (12) at (-0.5, 3.25) {};
							\node [style={circle, draw}] (13) at (2, 0.5) {$\varepsilon$};
	\end{pgfonlayer}
	\begin{pgfonlayer}{edgelayer}
		\draw [style=wire] (5) to (7);
		\draw [style=wire] (2) to (0);
		\draw [style=wire] (3) to (4);
		\draw [style=wire, bend right=15, looseness=1.00] (7) to (2);
		\draw [style=wire, in=109, out=-41, looseness=1.00] (7) to (3);
				\draw [style=wire, bend left=25, looseness=1.00] (7) to (13);
		\draw [style=wire] (7) to (10);
		\draw [style=wire] (11) to (10);
			\draw [style=wire] (1) to (13);
			\draw [style=wire] (12) to (5);
	\end{pgfonlayer}
\end{tikzpicture} }
   \end{array}=
   \begin{array}[c]{c}\resizebox{!}{3cm}{%
   \begin{tikzpicture}
	\begin{pgfonlayer}{nodelayer}
		\node [style={circle, draw}] (0) at (-1.5, -0.5) {$\varepsilon$};
		\node [style={circle, draw}] (1) at (2, -0.5) {$\varepsilon$};
		\node [style={circle, draw}] (2) at (-1.5, 0.5) {$\delta$};
		\node [style={circle, draw}] (3) at (1, 0.5) {$\delta$};
		\node [style={circle, draw}] (4) at (1, -0.5) {$\varepsilon$};
		\node [style=duplicate] (7) at (-0.5, 1.5) {$\Delta_{n}$};
		\node [style=port] (9) at (0.25, 0.5) {$\hdots$};
			\node [style={circle, draw}] (10) at (-0.5, 0.5) {$\delta$};
				\node [style={circle, draw}] (11) at (-0.5, -0.5) {$\varepsilon$};
					\node [style=port] (12) at (-0.5, 3) {};
							\node [style={circle, draw}] (13) at (2, 0.5) {$\delta$};
							\node [style={circle, draw}] (14) at (2, -1.25) {$e$};
							\node [style={circle, draw}] (15) at (1, -1.25) {$e$};
							\node [style={circle, draw}] (16) at (-0.5, -1.25) {$e$};
							\node [style=port] (17) at (-1.5, -1.75) {};
	\end{pgfonlayer}
	\begin{pgfonlayer}{edgelayer}
		\draw [style=wire] (2) to (0);
		\draw [style=wire] (3) to (4);
		\draw [style=wire, bend right=15, looseness=1.00] (7) to (2);
		\draw [style=wire, in=109, out=-41, looseness=1.00] (7) to (3);
				\draw [style=wire, bend left=25, looseness=1.00] (7) to (13);
		\draw [style=wire] (7) to (10);
		\draw [style=wire] (11) to (10);
			\draw [style=wire] (1) to (13);
			\draw [style=wire] (12) to (7);
					\draw [style=wire] (1) to (14);
							\draw [style=wire] (4) to (15);
									\draw [style=wire] (11) to (16);
											\draw [style=wire] (0) to (17);
	\end{pgfonlayer}
\end{tikzpicture} }
   \end{array}= \begin{array}[c]{c}\resizebox{!}{2cm}{%
\begin{tikzpicture}
	\begin{pgfonlayer}{nodelayer}
		\node [style=port] (0) at (0, 3) {};
		\node [style=port] (1) at (0, 0) {};
	\end{pgfonlayer}
	\begin{pgfonlayer}{edgelayer}
		\draw [style=wire] (0) to (1);
	\end{pgfonlayer}
\end{tikzpicture}}
   \end{array}
\end{align*}}%
\end{proof} 

This implies that in an integral category, hom-sets with domain $\oc A$ are $\mathbb{Q}_{\geq 0}$-modules. The scalar multiplication of a map $f: \oc A \to B$ with a non-negative rational $\frac{p}{q} \in \mathbb{Q}_{\geq 0}$ is the map $\frac{p}{q} \cdot f: \oc A \to B$ defined as $\frac{p}{q} \cdot f = q^{-1}_{\oc A} (p\cdot f)$. 

\begin{corollary} The coKleisli category of an integral category is skew enriched or left additive over $\mathbb{Q}_{\geq 0}$-modules. 
\end{corollary}

In certain integral categories, it is possible to prove that, for every object $A$ and $n \geq 1$, $n_A$ is an isomorphism. There are two simple ways of obtaining this: the first when $\varepsilon$ is a natural retraction and the second when the coalgebra modality is monoidal. Here we focus on $\varepsilon$ being a natural retraction, that is, there exists a natural transformation $\eta$ with components $\eta_A: A \to \oc A$ such that $\eta_A\varepsilon_A=1_A$. 

\begin{proposition}\label{rationalscor} In an integral category such that $\varepsilon$ is a natural retraction, for every $n \in \mathbb{N}$, $n \geq 1$, and every object $A$, the map $n_{A}: A \to A$ is an isomorphism. 
\end{proposition}
\begin{proof} Let $\eta$ be a natural section of $\varepsilon$. For every $n \geq 1$, define $n^{-1}_{A}: A \to A$ as $n^{-1}_{A}=\eta_A n^{-1}_{\oc A} \varepsilon_A$, where $n^{-1}_{\oc A}$ is defined as in the proof of Theorem \ref{rationals}. As before, it suffices to show that $n \cdot n^{-1}_{A}=1_{A}$. Here we use the additive structure, that $n \cdot n^{-1}_{\oc A}=1_{\oc A}$ and that $\eta_A\varepsilon_A=1_A$:   
\[n\cdot n^{-1}_{A}= n\cdot (\eta_A n^{-1}_{\oc A} \varepsilon_A)= \eta_A (n \cdot n^{-1}_{\oc A})\varepsilon_A=\eta_A 1_{\oc A}\varepsilon_A= \eta_A\varepsilon_A=1_A \]
\end{proof} 

\begin{corollary} An integral category such that $\varepsilon$ is a natural retraction is enriched over $\mathbb{Q}_{\geq 0}$-modules. 
\end{corollary}

\subsection{Integral Categories with Monoidal Coalgebra Modalities and Fubini's Theorem}

\begin{proposition} In an integral category with a monoidal coalgebra modality, for every $n \in \mathbb{N}$, $n \geq 1$, and every object $A$, the map $n_{A}: A \to A$ is an isomorphism. 
\end{proposition} 
\begin{proof} For every $n \geq 1$, define $n^{-1}_{A}: A \to A$ as follows: 
  \[  \xymatrixcolsep{3pc}\xymatrix{A \ar[r]^-{m_K \otimes 1} & \oc K \otimes A \ar[r]^-{n^{-1}_{\oc K} \otimes 1} & \oc K \otimes A \ar[r]^-{e \otimes 1} & A  
  } \]
where $n^{-1}_{\oc A}$ is defined as in the proof of Theorem \ref{rationals}. As before, it suffices to show that $n \cdot n^{-1}_{A}=1_{A}$. Here we use the additive structure, that $n \cdot n^{-1}_{\oc A}=1_{\oc A}$ and that $e$ is a monoidal transformation:   
\begin{align*}
n \cdot n^{-1}_{A}&= ~ n \cdot (m_K \otimes 1)(n^{-1}_{\oc K} \otimes 1)(e\otimes 1)\\
&=~(m_K \otimes 1)((n \cdot n^{-1}_{\oc K}) \otimes 1)(e\otimes 1)\\
&=~(m_K \otimes 1)(e \otimes 1)\\
&=~1
\end{align*}
\end{proof} 

\begin{corollary}\label{monoidalrationals} An integral category with a monoidal coalgebra modality is enriched over $\mathbb{Q}_{\geq 0}$-modules. 
\end{corollary}

\begin{definition} \normalfont A \textbf{monoidal integral transformation} is an integral transformation $\mathsf{s}$ of a monoidal coalgebra modality which satisfies the following \textbf{integral monoidal rule}:
\begin{description}
\item[{\bf [s.m]}] Monoidal Rule: 
\[m_\otimes \mathsf{s}=(\mathsf{s} \otimes \mathsf{d}^\circ)(1 \otimes \sigma \otimes 1)(m \otimes 1 \otimes 1)=(\mathsf{d}^\circ \otimes \mathsf{s})(1 \otimes \sigma \otimes 1)(m \otimes 1 \otimes 1)\]
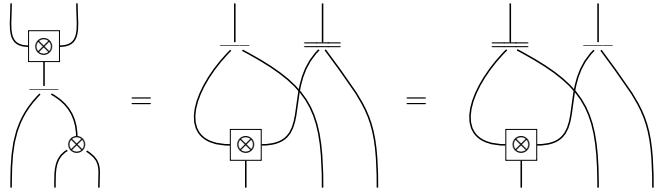
\[   \begin{array}[c]{c}\resizebox{!}{2.5cm}{%
\begin{tikzpicture}
	\begin{pgfonlayer}{nodelayer}
		\node [style=differential] (0) at (2, -2.25) {$\bigotimes$};
		\node [style=port] (1) at (0.5, -3.25) {};
		\node [style={regular polygon,regular polygon sides=4, draw, inner sep=1pt,minimum size=1pt}] (2) at (1.25, 0) {$\bigotimes$};
		\node [style=port] (3) at (2, 1) {};
		\node [style=port] (4) at (0.5, 1) {};
		\node [style=codifferential] (5) at (1.25, -1) {{\bf -----}};
		\node [style=port] (6) at (1.5, -3.25) {};
		\node [style=port] (7) at (2.5, -3.25) {};
	\end{pgfonlayer}
	\begin{pgfonlayer}{edgelayer}
		\draw [style=wire, in=-90, out=180, looseness=1.25] (2) to (4);
		\draw [style=wire, in=0, out=-90, looseness=1.25] (3) to (2);
		\draw [style=wire, bend left, looseness=1.00] (5) to (0);
		\draw [style=wire, in=90, out=-135, looseness=1.00] (5) to (1);
		\draw [style=wire] (2) to (5);
		\draw [style=wire, in=90, out=-33, looseness=1.25] (0) to (7);
		\draw [style=wire, in=90, out=-150, looseness=1.00] (0) to (6);
	\end{pgfonlayer}
\end{tikzpicture}
  }%
\end{array}=    \begin{array}[c]{c}\resizebox{!}{2.5cm}{%
\begin{tikzpicture}
	\begin{pgfonlayer}{nodelayer}
		\node [style={regular polygon,regular polygon sides=4, draw, inner sep=1pt,minimum size=1pt}] (0) at (-0.75, 1) {$\bigotimes$};
		\node [style=port] (1) at (-0.75, 0) {};
		\node [style=port] (2) at (1, 0) {};
		\node [style=codifferential] (3) at (-1, 3.25) {{\bf -----}};
		\node [style=port] (4) at (-1, 4.25) {};
		\node [style=codifferential] (5) at (1, 3.25) {{\bf =\!=\!=}};
		\node [style=port] (6) at (2.25, 0) {};
		\node [style=port] (7) at (1, 4.25) {};
	\end{pgfonlayer}
	\begin{pgfonlayer}{edgelayer}
		\draw [style=wire] (0) to (1);
		\draw [style=wire, in=90, out=-28, looseness=1.25] (3) to (2);
		\draw [style=wire] (4) to (3);
		\draw [style=wire, in=90, out=-54, looseness=1.25] (5) to (6);
		\draw [style=wire] (7) to (5);
		\draw [style=wire, in=0, out=-135, looseness=1.25] (5) to (0);
		\draw [style=wire, in=180, out=-135, looseness=1.50] (3) to (0);
	\end{pgfonlayer}
\end{tikzpicture}
  }%
\end{array}=  \begin{array}[c]{c}\resizebox{!}{2.5cm}{%
\begin{tikzpicture}
	\begin{pgfonlayer}{nodelayer}
		\node [style={regular polygon,regular polygon sides=4, draw, inner sep=1pt,minimum size=1pt}] (0) at (-0.75, 1) {$\bigotimes$};
		\node [style=port] (1) at (-0.75, 0) {};
		\node [style=port] (2) at (1, 0) {};
		\node [style=codifferential] (3) at (-1, 3.25) {{\bf =\!=\!=}};
		\node [style=port] (4) at (-1, 4.25) {};
		\node [style=codifferential] (5) at (1, 3.25) {{\bf -----}};
		\node [style=port] (6) at (2.25, 0) {};
		\node [style=port] (7) at (1, 4.25) {};
	\end{pgfonlayer}
	\begin{pgfonlayer}{edgelayer}
		\draw [style=wire] (0) to (1);
		\draw [style=wire, in=90, out=-28, looseness=1.25] (3) to (2);
		\draw [style=wire] (4) to (3);
		\draw [style=wire, in=90, out=-54, looseness=1.25] (5) to (6);
		\draw [style=wire] (7) to (5);
		\draw [style=wire, in=0, out=-135, looseness=1.25] (5) to (0);
		\draw [style=wire, in=180, out=-135, looseness=1.50] (3) to (0);
	\end{pgfonlayer}
\end{tikzpicture}
  }%
\end{array}\]
\end{description}
\end{definition}

An alternative description of a monoidal integral transformation is given as follows: 

\begin{proposition}\label{altintmon} For an integral transformation $\mathsf{s}$ of a monoidal coalgebra modality, the following are equivalent: 
\begin{enumerate}[{\em (i)}]
\item $\mathsf{s}$ is monoidal; 
\item The following equality holds: $\mathsf{s}=(m_K \otimes \mathsf{d}^\circ)(\mathsf{s}_K \otimes 1 \otimes 1)(m_ \otimes \otimes 1)$
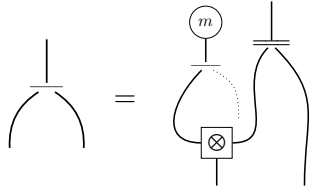
\[   \begin{array}[c]{c}\resizebox{!}{1.5cm}{%
\begin{tikzpicture}
	\begin{pgfonlayer}{nodelayer}
		\node [style=integral] (0) at (0, 2) {{\bf -----}};
		\node [style=port] (1) at (0, 3) {};
		\node [style=port] (2) at (0.75, 0.75) {};
		\node [style=port] (3) at (-0.75, 0.75) {};
	\end{pgfonlayer}
	\begin{pgfonlayer}{edgelayer}
		\draw [style=wire, bend left, looseness=1.00] (0) to (2);
		\draw [style=wire] (1) to (0);
		\draw [style=wire, bend right, looseness=1.00] (0) to (3);
	\end{pgfonlayer}
\end{tikzpicture}
  }%
\end{array}= \begin{array}[c]{c}\resizebox{!}{2.5cm}{%
\begin{tikzpicture}
	\begin{pgfonlayer}{nodelayer}
		\node [style=codifferential] (0) at (-2.5, 1.25) {{\bf =\!=\!=}};
		\node [style=port] (1) at (-3.75, -2) {};
		\node [style=port] (2) at (-3.25, -0.5) {};
		\node [style=port] (3) at (-2.5, 2.25) {};
		\node [style={regular polygon,regular polygon sides=4, draw, inner sep=1pt,minimum size=1pt}] (4) at (-3.75, -1) {$\bigotimes$};
		\node [style=codifferential] (5) at (-4, 0.75) {{\bf -----}};
		\node [style=port] (6) at (-1.75, -2) {};
		\node [style={circle, draw}] (7) at (-4, 1.75) {$m$};
	\end{pgfonlayer}
	\begin{pgfonlayer}{edgelayer}
		\draw [style=wire] (4) to (1);
		\draw [style=dwire, in=90, out=-28, looseness=1.25] (5) to (2);
		\draw [style=wire] (7) to (5);
		\draw [style=wire, in=90, out=-45, looseness=1.25] (0) to (6);
		\draw [style=wire] (3) to (0);
		\draw [style=wire, in=0, out=-135, looseness=1.25] (0) to (4);
		\draw [style=wire, in=180, out=-135, looseness=1.50] (5) to (4);
	\end{pgfonlayer}
\end{tikzpicture}
  }%
\end{array}\]
where the dotted line represents the monoidal unit. 
\end{enumerate}
\end{proposition} 
\begin{proof} $i) \Rightarrow ii)$: Using that $\mathsf{s}$ is monoidal we obtain the following equality 
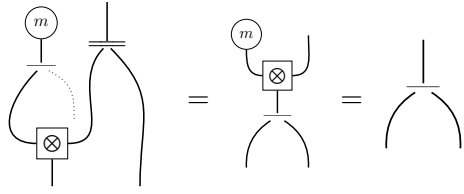
\begin{align*}
\begin{array}[c]{c}\resizebox{!}{2.5cm}{%
\begin{tikzpicture}
	\begin{pgfonlayer}{nodelayer}
		\node [style=codifferential] (0) at (-2.5, 1.25) {{\bf =\!=\!=}};
		\node [style=port] (1) at (-3.75, -2) {};
		\node [style=port] (2) at (-3.25, -0.5) {};
		\node [style=port] (3) at (-2.5, 2.25) {};
		\node [style={regular polygon,regular polygon sides=4, draw, inner sep=1pt,minimum size=1pt}] (4) at (-3.75, -1) {$\bigotimes$};
		\node [style=codifferential] (5) at (-4, 0.75) {{\bf -----}};
		\node [style=port] (6) at (-1.75, -2) {};
		\node [style={circle, draw}] (7) at (-4, 1.75) {$m$};
	\end{pgfonlayer}
	\begin{pgfonlayer}{edgelayer}
		\draw [style=wire] (4) to (1);
		\draw [style=dwire, in=90, out=-28, looseness=1.25] (5) to (2);
		\draw [style=wire] (7) to (5);
		\draw [style=wire, in=90, out=-45, looseness=1.25] (0) to (6);
		\draw [style=wire] (3) to (0);
		\draw [style=wire, in=0, out=-135, looseness=1.25] (0) to (4);
		\draw [style=wire, in=180, out=-135, looseness=1.50] (5) to (4);
	\end{pgfonlayer}
\end{tikzpicture}
  }%
\end{array}=    \begin{array}[c]{c}\resizebox{!}{2cm}{%
\begin{tikzpicture}
	\begin{pgfonlayer}{nodelayer}
		\node [style={regular polygon,regular polygon sides=4, draw, inner sep=1pt,minimum size=1pt}] (0) at (1.25, 0) {$\bigotimes$};
		\node [style=port] (1) at (2, 1) {};
		\node [style={circle, draw}] (2) at (0.5, 1) {$m$};
		\node [style=port] (3) at (2, -2.25) {};
		\node [style=port] (4) at (0.5, -2.25) {};
		\node [style=codifferential] (5) at (1.25, -1) {{\bf -----}};
	\end{pgfonlayer}
	\begin{pgfonlayer}{edgelayer}
		\draw [style=wire, in=-90, out=180, looseness=1.25] (0) to (2);
		\draw [style=wire, in=0, out=-90, looseness=1.25] (1) to (0);
		\draw [style=wire, bend left, looseness=1.00] (5) to (3);
		\draw [style=wire, bend right, looseness=1.00] (5) to (4);
		\draw [style=wire] (0) to (5);
	\end{pgfonlayer}
\end{tikzpicture}
  }%
\end{array} =  \begin{array}[c]{c}\resizebox{!}{1.5cm}{%
\begin{tikzpicture}
	\begin{pgfonlayer}{nodelayer}
		\node [style=integral] (0) at (0, 2) {{\bf -----}};
		\node [style=port] (1) at (0, 3) {};
		\node [style=port] (2) at (0.75, 0.75) {};
		\node [style=port] (3) at (-0.75, 0.75) {};
	\end{pgfonlayer}
	\begin{pgfonlayer}{edgelayer}
		\draw [style=wire, bend left, looseness=1.00] (0) to (2);
		\draw [style=wire] (1) to (0);
		\draw [style=wire, bend right, looseness=1.00] (0) to (3);
	\end{pgfonlayer}
\end{tikzpicture}
  }%
\end{array}
\end{align*}
$i) \Rightarrow ii)$: For the equality with $\mathsf{d}^\circ$ on the right, using \textbf{[cd.m]} and the associativity of $m_\otimes$, we have that: 
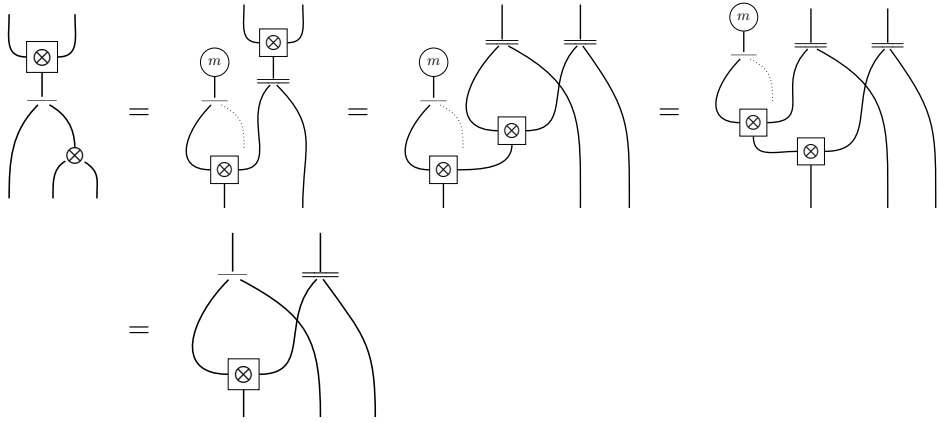
\begin{align*}
\begin{array}[c]{c}\resizebox{!}{2.5cm}{%
\begin{tikzpicture}
	\begin{pgfonlayer}{nodelayer}
		\node [style=differential] (0) at (2, -2.25) {$\bigotimes$};
		\node [style=port] (1) at (0.5, -3.25) {};
		\node [style={regular polygon,regular polygon sides=4, draw, inner sep=1pt,minimum size=1pt}] (2) at (1.25, 0) {$\bigotimes$};
		\node [style=port] (3) at (2, 1) {};
		\node [style=port] (4) at (0.5, 1) {};
		\node [style=codifferential] (5) at (1.25, -1) {{\bf -----}};
		\node [style=port] (6) at (1.5, -3.25) {};
		\node [style=port] (7) at (2.5, -3.25) {};
	\end{pgfonlayer}
	\begin{pgfonlayer}{edgelayer}
		\draw [style=wire, in=-90, out=180, looseness=1.25] (2) to (4);
		\draw [style=wire, in=0, out=-90, looseness=1.25] (3) to (2);
		\draw [style=wire, bend left, looseness=1.00] (5) to (0);
		\draw [style=wire, in=90, out=-135, looseness=1.00] (5) to (1);
		\draw [style=wire] (2) to (5);
		\draw [style=wire, in=90, out=-33, looseness=1.25] (0) to (7);
		\draw [style=wire, in=90, out=-150, looseness=1.00] (0) to (6);
	\end{pgfonlayer}
\end{tikzpicture}
  }%
\end{array}&=
   \begin{array}[c]{c}\resizebox{!}{2.75cm}{%
\begin{tikzpicture}
	\begin{pgfonlayer}{nodelayer}
		\node [style=codifferential] (0) at (0.5, 0.75) {{\bf -----}};
		\node [style=codifferential] (1) at (2, 1.25) {{\bf =\!=\!=}};
		\node [style=port] (2) at (2.75, -2) {};
		\node [style={circle, draw}] (3) at (0.5, 1.75) {$m$};
		\node [style=port] (4) at (0.75, -2) {};
		\node [style=port] (5) at (1.25, -0.5) {};
		\node [style={regular polygon,regular polygon sides=4, draw, inner sep=1pt,minimum size=1pt}] (6) at (0.75, -1) {$\bigotimes$};
		\node [style={regular polygon,regular polygon sides=4, draw, inner sep=1pt,minimum size=1pt}] (7) at (2, 2.25) {$\bigotimes$};
		\node [style=port] (8) at (2.75, 3.25) {};
		\node [style=port] (9) at (1.25, 3.25) {};
	\end{pgfonlayer}
	\begin{pgfonlayer}{edgelayer}
		\draw [style=wire] (6) to (4);
		\draw [style=dwire, in=90, out=-28, looseness=1.25] (0) to (5);
		\draw [style=wire] (3) to (0);
		\draw [style=wire, in=90, out=-45, looseness=1.25] (1) to (2);
		\draw [style=wire, in=0, out=-135, looseness=1.25] (1) to (6);
		\draw [style=wire, in=180, out=-135, looseness=1.50] (0) to (6);
		\draw [style=wire, in=-90, out=180, looseness=1.25] (7) to (9);
		\draw [style=wire, in=0, out=-90, looseness=1.25] (8) to (7);
		\draw [style=wire] (7) to (1);
	\end{pgfonlayer}
\end{tikzpicture}
  }%
\end{array}=    \begin{array}[c]{c}\resizebox{!}{2.75cm}{%
\begin{tikzpicture}
	\begin{pgfonlayer}{nodelayer}
		\node [style=codifferential] (0) at (-3, 0.75) {{\bf -----}};
		\node [style={circle, draw}] (1) at (-3, 1.75) {$m$};
		\node [style=port] (2) at (-2.75, -2) {};
		\node [style=port] (3) at (-2.25, -0.5) {};
		\node [style={regular polygon,regular polygon sides=4, draw, inner sep=1pt,minimum size=1pt}] (4) at (-2.75, -1) {$\bigotimes$};
		\node [style=codifferential] (5) at (0.75, 2.25) {{\bf =\!=\!=}};
		\node [style=port] (6) at (0.75, -2) {};
		\node [style=port] (7) at (0.75, 3.25) {};
		\node [style={regular polygon,regular polygon sides=4, draw, inner sep=1pt,minimum size=1pt}] (8) at (-1, 0) {$\bigotimes$};
		\node [style=codifferential] (9) at (-1.25, 2.25) {{\bf =\!=\!=}};
		\node [style=port] (10) at (2, -2) {};
		\node [style=port] (11) at (-1.25, 3.25) {};
	\end{pgfonlayer}
	\begin{pgfonlayer}{edgelayer}
		\draw [style=wire] (4) to (2);
		\draw [style=dwire, in=90, out=-28, looseness=1.25] (0) to (3);
		\draw [style=wire] (1) to (0);
		\draw [style=wire, in=180, out=-135, looseness=1.50] (0) to (4);
		\draw [style=wire, in=90, out=-28, looseness=1.25] (9) to (6);
		\draw [style=wire] (11) to (9);
		\draw [style=wire, in=90, out=-54, looseness=1.25] (5) to (10);
		\draw [style=wire] (7) to (5);
		\draw [style=wire, in=0, out=-135, looseness=1.25] (5) to (8);
		\draw [style=wire, in=180, out=-135, looseness=1.50] (9) to (8);
		\draw [style=wire, in=-90, out=0, looseness=1.00] (4) to (8);
	\end{pgfonlayer}
\end{tikzpicture}
  }%
\end{array}=    \begin{array}[c]{c}\resizebox{!}{2.75cm}{%
\begin{tikzpicture}
	\begin{pgfonlayer}{nodelayer}
		\node [style=codifferential] (0) at (-3, 2) {{\bf -----}};
		\node [style={circle, draw}] (1) at (-3, 3) {$m$};
		\node [style=port] (2) at (-1.25, -2) {};
		\node [style=port] (3) at (-2.25, 0.75) {};
		\node [style={regular polygon,regular polygon sides=4, draw, inner sep=1pt,minimum size=1pt}] (4) at (-2.75, 0.25) {$\bigotimes$};
		\node [style=codifferential] (5) at (0.75, 2.25) {{\bf =\!=\!=}};
		\node [style=port] (6) at (0.75, -2) {};
		\node [style=port] (7) at (0.75, 3.25) {};
		\node [style={regular polygon,regular polygon sides=4, draw, inner sep=1pt,minimum size=1pt}] (8) at (-1.25, -0.5) {$\bigotimes$};
		\node [style=codifferential] (9) at (-1.25, 2.25) {{\bf =\!=\!=}};
		\node [style=port] (10) at (2, -2) {};
		\node [style=port] (11) at (-1.25, 3.25) {};
	\end{pgfonlayer}
	\begin{pgfonlayer}{edgelayer}
		\draw [style=dwire, in=90, out=-28, looseness=1.25] (0) to (3);
		\draw [style=wire] (1) to (0);
		\draw [style=wire, in=180, out=-135, looseness=1.50] (0) to (4);
		\draw [style=wire, in=90, out=-28, looseness=1.25] (9) to (6);
		\draw [style=wire] (11) to (9);
		\draw [style=wire, in=90, out=-54, looseness=1.25] (5) to (10);
		\draw [style=wire] (7) to (5);
		\draw [style=wire, in=0, out=-135, looseness=1.25] (5) to (8);
		\draw [style=wire, in=180, out=-90, looseness=1.25] (4) to (8);
		\draw [style=wire] (8) to (2);
		\draw [style=wire, in=-135, out=0, looseness=1.50] (4) to (9);
	\end{pgfonlayer}
\end{tikzpicture}
  }%
\end{array}\\&= \begin{array}[c]{c}\resizebox{!}{2.5cm}{%
\begin{tikzpicture}
	\begin{pgfonlayer}{nodelayer}
		\node [style={regular polygon,regular polygon sides=4, draw, inner sep=1pt,minimum size=1pt}] (0) at (-0.75, 1) {$\bigotimes$};
		\node [style=port] (1) at (-0.75, 0) {};
		\node [style=port] (2) at (1, 0) {};
		\node [style=codifferential] (3) at (-1, 3.25) {{\bf -----}};
		\node [style=port] (4) at (-1, 4.25) {};
		\node [style=codifferential] (5) at (1, 3.25) {{\bf =\!=\!=}};
		\node [style=port] (6) at (2.25, 0) {};
		\node [style=port] (7) at (1, 4.25) {};
	\end{pgfonlayer}
	\begin{pgfonlayer}{edgelayer}
		\draw [style=wire] (0) to (1);
		\draw [style=wire, in=90, out=-28, looseness=1.25] (3) to (2);
		\draw [style=wire] (4) to (3);
		\draw [style=wire, in=90, out=-54, looseness=1.25] (5) to (6);
		\draw [style=wire] (7) to (5);
		\draw [style=wire, in=0, out=-135, looseness=1.25] (5) to (0);
		\draw [style=wire, in=180, out=-135, looseness=1.50] (3) to (0);
	\end{pgfonlayer}
\end{tikzpicture}
  }%
\end{array}
\end{align*}
For the other equality with $\mathsf{d}^\circ$ on the left, one must simply observe that: 
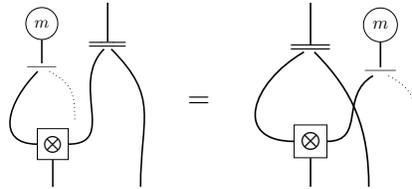
\[ \begin{array}[c]{c}\resizebox{!}{2.5cm}{%
\begin{tikzpicture}
	\begin{pgfonlayer}{nodelayer}
		\node [style=codifferential] (0) at (-2.5, 1.25) {{\bf =\!=\!=}};
		\node [style=port] (1) at (-3.75, -2) {};
		\node [style=port] (2) at (-3.25, -0.5) {};
		\node [style=port] (3) at (-2.5, 2.25) {};
		\node [style={regular polygon,regular polygon sides=4, draw, inner sep=1pt,minimum size=1pt}] (4) at (-3.75, -1) {$\bigotimes$};
		\node [style=codifferential] (5) at (-4, 0.75) {{\bf -----}};
		\node [style=port] (6) at (-1.75, -2) {};
		\node [style={circle, draw}] (7) at (-4, 1.75) {$m$};
	\end{pgfonlayer}
	\begin{pgfonlayer}{edgelayer}
		\draw [style=wire] (4) to (1);
		\draw [style=dwire, in=90, out=-28, looseness=1.25] (5) to (2);
		\draw [style=wire] (7) to (5);
		\draw [style=wire, in=90, out=-45, looseness=1.25] (0) to (6);
		\draw [style=wire] (3) to (0);
		\draw [style=wire, in=0, out=-135, looseness=1.25] (0) to (4);
		\draw [style=wire, in=180, out=-135, looseness=1.50] (5) to (4);
	\end{pgfonlayer}
\end{tikzpicture}
  }%
  \end{array}=  \begin{array}[c]{c}\resizebox{!}{2.5cm}{%
\begin{tikzpicture}
	\begin{pgfonlayer}{nodelayer}
		\node [style=codifferential] (0) at (-3, 1) {{\bf =\!=\!=}};
		\node [style=port] (1) at (-3, -2) {};
		\node [style=port] (2) at (-0.75, -0.75) {};
		\node [style=port] (3) at (-3, 2) {};
		\node [style={regular polygon,regular polygon sides=4, draw, inner sep=1pt,minimum size=1pt}] (4) at (-3, -1) {$\bigotimes$};
		\node [style=codifferential] (5) at (-1.5, 0.5) {{\bf -----}};
		\node [style=port] (6) at (-1.75, -2) {};
		\node [style={circle, draw}] (7) at (-1.5, 1.5) {$m$};
	\end{pgfonlayer}
	\begin{pgfonlayer}{edgelayer}
		\draw [style=wire] (4) to (1);
		\draw [style=dwire, in=90, out=-28, looseness=1.25] (5) to (2);
		\draw [style=wire] (7) to (5);
		\draw [style=wire] (3) to (0);
		\draw [style=wire, in=0, out=-150, looseness=1.25] (5) to (4);
		\draw [style=wire, in=180, out=-150, looseness=1.75] (0) to (4);
		\draw [style=wire, in=90, out=-45, looseness=1.00] (0) to (6);
	\end{pgfonlayer}
\end{tikzpicture}
  }%
\end{array}\]
And the remainder of the proof is similar. 
\end{proof} 

We discuss briefly the interpretation of Fubini's theorem in integral categories. Fubini's theorem states that for a function $f$ in two variables the following equality holds:
\[\int_Y (\int_X f(x,y) ~\mathsf{d}x) ~\mathsf{d}y = \int_X (\int_Y f(x,y) ~\mathsf{d}y) ~\mathsf{d}x\]
that is, integral $f$ with respect to $x$ first then $y$ is the same as integral $f$ with respect to $y$ then $x$. Under our intuition of integration for integral categories, this implies we are dealing with a function in four variables $x$, $y$, $\mathsf{d}x$ and $\mathsf{d}y$ which is bilinear in $\mathsf{d}x$ and $\mathsf{d}y$. The reader may notice the similarities with the axiom \textbf{[S.3]}, the independence rule, for integral categories. However there are crucial differences. First, Fubini's theorem concerns to the double integral of a function in two different variables while the axiom \textbf{[S.3]} explains the process of integral the same variable twice. Secondly, the double integral in Fubini's theorem is an iterated integral, where the function is integrated as a function of a single variable by holding the other variable constant. On the other hand, \textbf{[S.3]} takes all variables into consideration at once. Therefore, we will need a different approach in the interpretation of Fubini's theorem.

To interpret Fubini's theorem in an integral category we will require that our additive symmetric monoidal category has finite biproducts and that our coalgebra modality has Seely isomorphisms \cite{bierman1995categorical,blute2015cartesian,dblute2006differential}. A coalgebra modality on a symmetric monoidal with finite products has \textbf{Seely isomorphisms} if the natural transformations $\chi$ and $\chi_{0}$ defined respectively as:
    $$  \xymatrixcolsep{2pc}\xymatrix{\oc(A \times B) \ar[r]^-{\Delta} & \oc(A \times B) \otimes \oc(A \times B) \ar[rr]^-{\oc(\pi_0) \otimes \oc(\pi_1)} && \oc A \otimes \oc B  & \oc(0) \ar[r]^-{e} & K
  } $$
are isomorphisms, so $\oc(A \times B) \cong \oc A \otimes \oc B$ and $\oc(0) \cong K$. As explained in \cite{blute2015cartesian} (Theorem 3.1.6), requiring that a coalgebra modality has Seely isomorphisms is equivalent to asking that it be monoidal. Every coalgebra modality having the Seely isomorphisms is a monoidal coalgebra modality, where $m_\otimes$ is 
$$\xymatrixcolsep{2.5pc}\xymatrix{\oc A \otimes \oc B \ar[r]^-{\chi^{-1}} & \oc(A \times B) \ar[r]^-{\delta} &   \oc \oc(A \times B) \ar[r]^-{\oc(\chi)} & \oc (\oc A \otimes \oc B) \ar[r]^-{\oc(\varepsilon \otimes \varepsilon)} & \oc(A \otimes B) 
  }$$
  and $m_K$ is defined as
  $$\xymatrixcolsep{2.5pc}\xymatrix{K \ar[r]^-{\chi^{-1}_{\mathsf{T}}} & \oc(\mathsf{T}) \ar[r]^-{\delta} & \oc \oc(\mathsf{T}) \ar[r]^-{\oc(\chi_{\mathsf{T}})} & \oc(K)
  }$$
  Conversly, in the presence of finite products, every monoidal coalgebra modality satisfies the Seely isomorphisms \cite{bierman1995categorical} where the inverse of $\chi$ is 
  $$ \xymatrixcolsep{2.5pc}\xymatrix{\oc A \otimes \oc B \ar[r]^-{\delta \otimes \delta} & \oc \oc A \otimes \oc \oc B \ar[r]^-{m_\otimes} & \oc(\oc A \otimes \oc B) \ar[rr]^-{\oc(\langle \varepsilon \otimes e, e \otimes \varepsilon \rangle)} && \oc(A \times B)  
  } $$
  while the inverse of $\chi_{\mathsf{T}}$ is
    $$  \xymatrixcolsep{2.5pc}\xymatrix{ K \ar[r]^-{m_K} & \oc(K) \ar[r]^-{\oc(\mathsf{t})} & \oc(\mathsf{T})
  } $$
where $\mathsf{t}: K \to \mathsf{T}$ is the unique map to the terminal object. 

Suppose now that we have an integral category with finite biproducts and a monoidal coalgebra modality, which implies it has Seely isomorphisms. Fubini's theorem concerns maps of the form $f: \oc(A \times B) \otimes A \otimes B \to C$ whose type ensures it is bilinear in the second two occurrences of $A$ and $B$. Maps of this form can be integrated with respect to either $A$ or $B$, or both $A$ and $B$ simultaneously: the latter, the double integral of $f$, is obtained as follows:
  \[  \xymatrixcolsep{3pc}\xymatrix{ \oc(A \times B) \ar[r]^-{\chi^{-1}} & \oc A \otimes \oc B \ar[r]^-{\mathsf{s} \otimes \mathsf{s}} & \oc A \otimes A \otimes \oc B \otimes B \ar[r]^-{1 \otimes \sigma \otimes 1} & \\
   \oc A \otimes \oc B \otimes A \otimes B \ar[r]^-{\chi \otimes 1 \otimes 1} & \oc(A \times B) \otimes A \otimes B \ar[r]^-{f} & C
  } \]
Fubini's theorem asserts that the order of integration in this double integral does not matter.   At this level of generality this order independence is an immediate consequence 
of the bifunctoriality of $\_\otimes\_$. 

\section{Differential Categories}\label{dblute2006differentialec}

\subsection{Differential Categories}

We briefly recall the definition of a differential category  \cite{dblute2006differential}. We omit the full definition of the differential combinator and directly give the definition of a deriving transformation (of course both are equivalent).  At the same time, we introduce the graphical calculus for differential categories.  

\begin{definition} \normalfont A \textbf{differential category} is an additive symmetric monoidal category with a coalgebra modality which comes equipped with a \textbf{deriving transformation}  \cite{dblute2006differential}, that is, a natural transformation $\mathsf{d}$ with components $\mathsf{d}_A: \oc A \otimes A \to \oc A$, which is represented in the graphical calculus as:
$$ \mathsf{d}\coloneqq \begin{array}[c]{c} \resizebox{!}{1.5cm}{%
\begin{tikzpicture}
	\begin{pgfonlayer}{nodelayer}
		\node [style=port] (0) at (2, 3) {};
		\node [style=port] (1) at (1.25, 0.75) {};
		\node [style=integral] (2) at (1.25, 1.75) {{\bf =\!=\!=}};
		\node [style=port] (3) at (0.5, 3) {};
	\end{pgfonlayer}
	\begin{pgfonlayer}{edgelayer}
		\draw [style=wire, bend right, looseness=1.00] (2) to (0);
		\draw [style=wire] (1) to (2);
		\draw [style=wire, bend left, looseness=1.00] (2) to (3);
	\end{pgfonlayer}
\end{tikzpicture}}
   \end{array}$$
 such that $\mathsf{d}$ satisfies the following equations: 
\begin{description}
\item[{\bf [d.1]}] Constant Rule: $\mathsf{d}e= 0$
$$\begin{array}[c]{c} \resizebox{!}{1.7cm}{%
\begin{tikzpicture}
	\begin{pgfonlayer}{nodelayer}
		\node [style={circle, draw}] (0) at (0, 0) {$e$};
		\node [style=port] (1) at (0.75, 2) {};
		\node [style=differential] (2) at (0, 1) {{\bf =\!=\!=\!=}};
		\node [style=port] (3) at (-0.75, 2) {};
	\end{pgfonlayer}
	\begin{pgfonlayer}{edgelayer}
		\draw [style=wire, bend right, looseness=1.00] (2) to (1);
		\draw [style=wire, bend left, looseness=1.00] (2) to (3);
		\draw [style=wire] (2) to (0);
	\end{pgfonlayer}
\end{tikzpicture}}
   \end{array}=
   \begin{array}[c]{c}
0
   \end{array}$$
\item[{\bf [d.2]}] Leibniz Rule:  $\mathsf{d}\Delta=(\Delta \otimes 1)(1 \otimes \sigma)(\mathsf{d} \otimes 1)+ (\Delta \otimes 1)(1\otimes \mathsf{d})$
\[\begin{array}[c]{c} \resizebox{!}{1.7cm}{%
\begin{tikzpicture}
	\begin{pgfonlayer}{nodelayer}
		\node [style=differential] (0) at (0, 1) {{\bf =\!=\!=\!=}};
		\node [style=port] (1) at (0.75, 2) {};
		\node [style=port] (2) at (-0.75, 2) {};
		\node [style=port] (3) at (-0.75, -0.75) {};
		\node [style=port] (4) at (0, 0.25) {$\Delta$};
		\node [style=port] (5) at (0.75, -0.75) {};
	\end{pgfonlayer}
	\begin{pgfonlayer}{edgelayer}
		\draw [style=wire, bend right, looseness=1.00] (0) to (1);
		\draw [style=wire, bend left, looseness=1.00] (0) to (2);
		\draw [style=wire, bend right, looseness=1.00] (4) to (3);
		\draw [style=wire, bend left, looseness=1.00] (4) to (5);
		\draw [style=wire] (0) to (4);
	\end{pgfonlayer}
\end{tikzpicture}}
   \end{array}=
      \begin{array}[c]{c} \resizebox{!}{1.7cm}{%
\begin{tikzpicture}
	\begin{pgfonlayer}{nodelayer}
		\node [style=port] (0) at (-0.25, 2) {};
		\node [style=differential] (1) at (-0.75, 0) {{\bf =\!=\!=\!=}};
		\node [style=port] (2) at (1, 2) {};
		\node [style=port] (3) at (-0.25, 1) {$\Delta$};
		\node [style=port] (4) at (-0.75, -0.75) {};
		\node [style=port] (5) at (0.75, -0.75) {};
	\end{pgfonlayer}
	\begin{pgfonlayer}{edgelayer}
		\draw [style=wire, in=-90, out=30, looseness=1.00] (1) to (2);
		\draw [style=wire, in=150, out=-150, looseness=1.50] (3) to (1);
		\draw [style=wire] (0) to (3);
		\draw [style=wire] (1) to (4);
		\draw [style=wire, bend left, looseness=1.25] (3) to (5);
	\end{pgfonlayer}
\end{tikzpicture}}
   \end{array}+ \begin{array}[c]{c} \resizebox{!}{1.7cm}{%
\begin{tikzpicture}
	\begin{pgfonlayer}{nodelayer}
		\node [style=port] (0) at (1, 2) {};
		\node [style=differential] (1) at (0.75, 0) {{\bf =\!=\!=\!=}};
		\node [style=port] (2) at (0.75, -0.75) {};
		\node [style=port] (3) at (-0.75, -0.75) {};
		\node [style=port] (4) at (-0.25, 2) {};
		\node [style=port] (5) at (-0.25, 1) {$\Delta$};
	\end{pgfonlayer}
	\begin{pgfonlayer}{edgelayer}
		\draw [style=wire, bend right, looseness=1.00] (1) to (0);
		\draw [style=wire, in=91, out=-135, looseness=0.75] (5) to (3);
		\draw [style=wire, in=150, out=-30, looseness=1.00] (5) to (1);
		\draw [style=wire] (4) to (5);
		\draw [style=wire] (1) to (2);
	\end{pgfonlayer}
\end{tikzpicture}}
   \end{array}\]
\item[{\bf [d.3]}] Linear Rule: $\mathsf{d}\varepsilon=e \otimes 1$
\[\begin{array}[c]{c} \resizebox{!}{2.0cm}{%
\begin{tikzpicture}
	\begin{pgfonlayer}{nodelayer}
		\node [style=port] (0) at (0.75, 2) {};
		\node [style=port] (1) at (-0.75, 2) {};
		\node [style={circle, draw}] (2) at (0, 0) {$\varepsilon$};
		\node [style=differential] (3) at (0, 1) {{\bf =\!=\!=\!=}};
		\node [style=port] (4) at (0, -1) {};
	\end{pgfonlayer}
	\begin{pgfonlayer}{edgelayer}
		\draw [style=wire, bend right, looseness=1.00] (3) to (0);
		\draw [style=wire, bend left, looseness=1.00] (3) to (1);
		\draw [style=wire] (3) to (2);
		\draw [style=wire] (2) to (4);
	\end{pgfonlayer}
\end{tikzpicture}}
   \end{array}=
   \begin{array}[c]{c} \resizebox{!}{2.0cm}{%
\begin{tikzpicture}
	\begin{pgfonlayer}{nodelayer}
		\node [style=port] (0) at (-0.75, 2) {};
		\node [style=port] (1) at (0, 2) {};
		\node [style=port] (2) at (0, -1) {};
		\node [style={circle, draw}] (3) at (-0.75, 0) {$e$};
	\end{pgfonlayer}
	\begin{pgfonlayer}{edgelayer}
		\draw [style=wire] (0) to (3);
		\draw [style=wire] (1) to (2);
	\end{pgfonlayer}
\end{tikzpicture}}
   \end{array}\]
\item[{\bf [d.4]}] Chain Rule: $\mathsf{d}\delta=(\Delta \otimes 1)(\delta \otimes 1 \otimes 1)(1 \otimes \mathsf{d})\mathsf{d}$
\[\begin{array}[c]{c} \resizebox{!}{2.0cm}{%
\begin{tikzpicture}
	\begin{pgfonlayer}{nodelayer}
		\node [style=port] (0) at (0.75, 2) {};
		\node [style=differential] (1) at (0, 1) {{\bf =\!=\!=\!=}};
		\node [style=port] (2) at (0, -1) {};
		\node [style=port] (3) at (-0.75, 2) {};
		\node [style={circle, draw}] (4) at (0, 0) {$\delta$};
	\end{pgfonlayer}
	\begin{pgfonlayer}{edgelayer}
		\draw [style=wire, bend right, looseness=1.00] (1) to (0);
		\draw [style=wire, bend left, looseness=1.00] (1) to (3);
		\draw [style=wire] (1) to (4);
		\draw [style=wire] (4) to (2);
	\end{pgfonlayer}
\end{tikzpicture}}
   \end{array}=
   \begin{array}[c]{c} \resizebox{!}{2.0cm}{%
\begin{tikzpicture}
	\begin{pgfonlayer}{nodelayer}
		\node [style={circle, draw}] (0) at (-1, 0.25) {$\delta$};
		\node [style=port] (1) at (-0.5, 1.25) {$\Delta$};
		\node [style=circle] (2) at (-0.5, 2) {};
		\node [style=differential] (3) at (0.25, 0.5) {{\bf =\!=\!=\!=}};
		\node [style=port] (4) at (0.5, 2) {};
		\node [style=differential] (5) at (-0.25, -0.75) {{\bf =\!=\!=\!=}};
		\node [style=port] (6) at (-0.25, -1.5) {};
	\end{pgfonlayer}
	\begin{pgfonlayer}{edgelayer}
		\draw [style=wire, bend right, looseness=1.00] (1) to (0);
		\draw [style=wire] (2) to (1);
		\draw [style=wire, bend right, looseness=1.00] (3) to (4);
		\draw [style=wire, in=150, out=-30, looseness=1.25] (1) to (3);
		\draw [style=wire] (5) to (6);
		\draw [style=wire, in=30, out=-90, looseness=1.00] (3) to (5);
		\draw [style=wire, in=150, out=-90, looseness=1.00] (0) to (5);
	\end{pgfonlayer}
\end{tikzpicture}}
   \end{array}\]
     \item[{\bf [d.5]}] Interchange Rule: $(\mathsf{d} \otimes 1)\mathsf{d}=(1 \otimes \sigma)(\mathsf{d} \otimes 1)\mathsf{d}$
$$\begin{array}[c]{c} \resizebox{!}{1.75cm}{%
\begin{tikzpicture}
	\begin{pgfonlayer}{nodelayer}
		\node [style=port] (0) at (0.5, 2.5) {};
		\node [style=port] (1) at (0, 2.5) {};
		\node [style=port] (2) at (0, 0) {};
		\node [style=port] (3) at (-1, 2.5) {};
		\node [style=codifferential] (4) at (0, 0.75) {{\bf =\!=\!=\!=}};
		\node [style=codifferential] (5) at (-0.5, 1.75) {{\bf =\!=\!=\!=}};
	\end{pgfonlayer}
	\begin{pgfonlayer}{edgelayer}
		\draw [style=wire, bend left=15, looseness=1.25] (4) to (5);
		\draw [style=wire, bend right, looseness=1.00] (4) to (0);
		\draw [style=wire] (2) to (4);
		\draw [style=wire, bend left, looseness=1.00] (5) to (3);
		\draw [style=wire, bend right, looseness=1.00] (5) to (1);
	\end{pgfonlayer}
\end{tikzpicture}}
   \end{array}=
      \begin{array}[c]{c} \resizebox{!}{1.75cm}{%
\begin{tikzpicture}
	\begin{pgfonlayer}{nodelayer}
		\node [style=port] (0) at (-1, 2.5) {};
		\node [style=codifferential] (1) at (-0.5, 1.75) {{\bf =\!=\!=\!=}};
		\node [style=codifferential] (2) at (0, 0.75) {{\bf =\!=\!=\!=}};
		\node [style=port] (3) at (0, 0) {};
		\node [style=port] (4) at (0, 2.5) {};
		\node [style=port] (5) at (0.5, 2.5) {};
	\end{pgfonlayer}
	\begin{pgfonlayer}{edgelayer}
		\draw [style=wire, bend left=15, looseness=1.25] (2) to (1);
		\draw [style=wire] (3) to (2);
		\draw [style=wire, bend left, looseness=1.00] (1) to (0);
		\draw [style=wire, in=45, out=-90, looseness=1.00] (5) to (1);
		\draw [style=wire, in=45, out=-90, looseness=1.50] (4) to (2);
	\end{pgfonlayer}
\end{tikzpicture}}
   \end{array}$$
      \end{description}
\end{definition}

As for integral categories, it might be useful for the reader to have some intuition regarding the axioms of a deriving transformation. The first axiom, {\bf [d.1]}, states that the derivative of a constant map is zero. The second axiom {\bf [d.2]} is the Leibniz rule for differentiation -- also called the product rule.  The third axiom {\bf [d.3]} says that the derivative of a linear map is a constant. The fourth axiom {\bf [d.4]} is the chain rule. The last axiom {\bf [d.5]} is the independence of differentiation or the interchange law, which naively states that differentiating with respect to $x$ then $y$ is the same as differentiation with respect to $y$ then $x$. It should be noted that {\bf [d.5]} was not a requirement in  \cite{dblute2006differential} but was later added to the definition \cite{blute2015cartesian, blute2009cartesian} to ensure that the coKleisli category of a differential category was a Cartesian differential category. 

The coderiving transformation is related to the deriving transformation by:  
\begin{proposition}\label{Wprop} The deriving transformation and coderiving transformations satisfy the following equality:
$$\mathsf{d}\mathsf{d}^\circ= \mathsf{W}+1$$ 
where $\mathsf{W}=(\mathsf{d}^\circ \otimes 1)(1 \otimes \sigma)(\mathsf{d} \otimes 1)$. 
$$   \begin{array}[c]{c}\resizebox{!}{2cm}
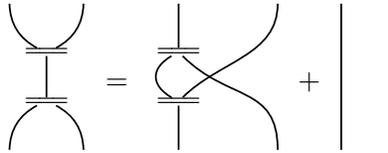
{%
\begin{tikzpicture}
	\begin{pgfonlayer}{nodelayer}
		\node [style=differential] (0) at (0, 2) {{\bf =\!=\!=}};
		\node [style=port] (1) at (-0.75, 3) {};
		\node [style=port] (2) at (0.75, 3) {};
		\node [style=port] (3) at (-0.75, 0) {};
		\node [style=port] (4) at (0.75, 0) {};
		\node [style=codifferential] (5) at (0, 1) {{\bf =\!=\!=}};
	\end{pgfonlayer}
	\begin{pgfonlayer}{edgelayer}
		\draw [style=wire, bend right, looseness=1.00] (0) to (2);
		\draw [style=wire, bend left, looseness=1.00] (0) to (1);
		\draw [style=wire, bend left, looseness=1.00] (5) to (4);
		\draw [style=wire, bend right, looseness=1.00] (5) to (3);
		\draw [style=wire] (0) to (5);
	\end{pgfonlayer}
\end{tikzpicture}}
   \end{array}
=
   \begin{array}[c]{c}\resizebox{!}{2cm}{%
\begin{tikzpicture}
	\begin{pgfonlayer}{nodelayer}
		\node [style=differential] (0) at (-1, 1) {{\bf =\!=\!=}};
		\node [style=port] (1) at (-1, 0) {};
		\node [style=codifferential] (2) at (-1, 2) {{\bf =\!=\!=}};
		\node [style=port] (3) at (-1, 3) {};
		\node [style=port] (4) at (1, 3) {};
		\node [style=port] (5) at (1, 0) {};
	\end{pgfonlayer}
	\begin{pgfonlayer}{edgelayer}
		\draw [style=wire] (3) to (2);
		\draw [style=wire] (1) to (0);
		\draw [style=wire, bend right=60, looseness=1.50] (2) to (0);
		\draw [style=wire, in=90, out=-45, looseness=1.25] (2) to (5);
		\draw [style=wire, in=45, out=-90, looseness=1.00] (4) to (0);
	\end{pgfonlayer}
\end{tikzpicture}}
   \end{array}
   +
   \begin{array}[c]{c}\resizebox{!}{2cm}{%
\begin{tikzpicture}
	\begin{pgfonlayer}{nodelayer}
		\node [style=port] (0) at (-0.5, 3) {};
		\node [style=port] (1) at (0.5, 3) {};
		\node [style=port] (2) at (-0.5, 0) {};
		\node [style=port] (3) at (0.5, 0) {};
	\end{pgfonlayer}
	\begin{pgfonlayer}{edgelayer}
		\draw [style=wire] (0) to (2);
		\draw [style=wire] (1) to (3);
	\end{pgfonlayer}
\end{tikzpicture}}
   \end{array}
   $$
\end{proposition} 

The notation $\mathsf{W}$ was introduced by Ehrhard in \cite{ehrhard2017introduction}. 

\begin{proof} Here we use the Leibniz rule {\bf [d.2]}, that the derivative of a linear map is a constant {\bf [d.3]} and the counit of the comultiplication:  
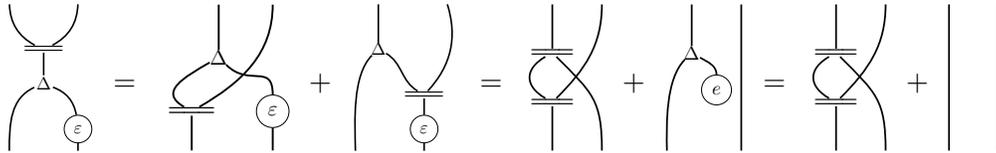
\begin{align*}
\begin{array}[c]{c}\resizebox{!}{2cm}{%
\begin{tikzpicture}
	\begin{pgfonlayer}{nodelayer}
		\node [style=port] (0) at (0.5, 2) {};
		\node [style=differential] (1) at (1.25, 1) {{\bf =\!=\!=}};
		\node [style=port] (2) at (2, 2) {};
		\node [style={circle, draw}] (3) at (2, -0.75) {$\varepsilon$};
		\node [style=port] (4) at (2, -1.25) {};
		\node [style=duplicate] (5) at (1.25, 0.25) {$\Delta$};
		\node [style=port] (6) at (0.5, -1.25) {};
	\end{pgfonlayer}
	\begin{pgfonlayer}{edgelayer}
		\draw [style=wire, bend right, looseness=1.00] (1) to (2);
		\draw [style=wire, bend left, looseness=1.00] (1) to (0);
		\draw [style=wire, bend left, looseness=1.00] (5) to (3);
		\draw [style=wire] (3) to (4);
		\draw [style=wire, in=90, out=-156, looseness=1.00] (5) to (6);
		\draw [style=wire] (5) to (1);
	\end{pgfonlayer}
\end{tikzpicture}}%
 \end{array}=   \begin{array}[c]{c}\resizebox{!}{2cm}{%
\begin{tikzpicture}
	\begin{pgfonlayer}{nodelayer}
		\node [style=differential] (0) at (-1, 0) {{\bf =\!=\!=}};
		\node [style=port] (1) at (-0.5, 2) {};
		\node [style=port] (2) at (-0.5, 1) {$\Delta$};
		\node [style=port] (3) at (0.5, -0.75) {};
		\node [style=port] (4) at (0.5, 2) {};
		\node [style=port] (5) at (-1, -0.75) {};
		\node [style={circle, draw}] (6) at (0.5, 0) {$\varepsilon$};
	\end{pgfonlayer}
	\begin{pgfonlayer}{edgelayer}
		\draw [style=wire, in=-90, out=30, looseness=1.00] (0) to (4);
		\draw [style=wire, in=150, out=-150, looseness=1.50] (2) to (0);
		\draw [style=wire] (1) to (2);
		\draw [style=wire] (0) to (5);
		\draw [style=wire, in=90, out=-45, looseness=1.50] (2) to (6);
		\draw [style=wire] (6) to (3);
	\end{pgfonlayer}
\end{tikzpicture}}%
  \end{array}+  \begin{array}[c]{c}\resizebox{!}{2cm}{%
\begin{tikzpicture}
	\begin{pgfonlayer}{nodelayer}
		\node [style=differential] (0) at (2, 0) {{\bf =\!=\!=}};
		\node [style=port] (1) at (0.5, -1.25) {};
		\node [style={circle, draw}] (2) at (2, -0.75) {$\varepsilon$};
		\node [style=port] (3) at (1, 2) {};
		\node [style=port] (4) at (2.5, 2) {};
		\node [style=port] (5) at (1, 1) {$\Delta$};
		\node [style=port] (6) at (2, -1.25) {};
	\end{pgfonlayer}
	\begin{pgfonlayer}{edgelayer}
		\draw [style=wire, bend right, looseness=1.00] (0) to (4);
		\draw [style=wire, in=91, out=-135, looseness=0.75] (5) to (1);
		\draw [style=wire, in=150, out=-30, looseness=1.00] (5) to (0);
		\draw [style=wire] (3) to (5);
		\draw [style=wire] (0) to (2);
		\draw [style=wire] (2) to (6);
	\end{pgfonlayer}
\end{tikzpicture}}%
   \end{array}=\begin{array}[c]{c}\resizebox{!}{2cm}{%
\begin{tikzpicture}
	\begin{pgfonlayer}{nodelayer}
		\node [style=differential] (0) at (-1, 1) {{\bf =\!=\!=}};
		\node [style=port] (1) at (-1, 0) {};
		\node [style=codifferential] (2) at (-1, 2) {{\bf =\!=\!=}};
		\node [style=port] (3) at (-1, 3) {};
		\node [style=port] (4) at (0, 3) {};
		\node [style=port] (5) at (0, 0) {};
	\end{pgfonlayer}
	\begin{pgfonlayer}{edgelayer}
		\draw [style=wire] (3) to (2);
		\draw [style=wire] (1) to (0);
		\draw [style=wire, bend right=60, looseness=1.50] (2) to (0);
		\draw [style=wire, in=90, out=-45, looseness=1.25] (2) to (5);
		\draw [style=wire, in=45, out=-90, looseness=1.00] (4) to (0);
	\end{pgfonlayer}
\end{tikzpicture}}%
   \end{array}+
   \begin{array}[c]{c}\resizebox{!}{2cm}{%
\begin{tikzpicture}
	\begin{pgfonlayer}{nodelayer}
		\node [style=port] (0) at (0.5, -1) {};
		\node [style=port] (1) at (1, 2) {};
		\node [style=port] (2) at (2, 2) {};
		\node [style=port] (3) at (1, 1) {$\Delta$};
		\node [style=port] (4) at (2, -1) {};
		\node [style={circle, draw}] (5) at (1.5, 0.25) {$e$};
	\end{pgfonlayer}
	\begin{pgfonlayer}{edgelayer}
		\draw [style=wire, in=91, out=-135, looseness=0.75] (3) to (0);
		\draw [style=wire] (1) to (3);
		\draw [style=wire] (2) to (4);
		\draw [style=wire, in=90, out=-30, looseness=1.00] (3) to (5);
	\end{pgfonlayer}
\end{tikzpicture}}%
   \end{array}=   \begin{array}[c]{c}\resizebox{!}{2cm}{%
\begin{tikzpicture}
	\begin{pgfonlayer}{nodelayer}
		\node [style=differential] (0) at (-1, 1) {{\bf =\!=\!=}};
		\node [style=port] (1) at (-1, 0) {};
		\node [style=codifferential] (2) at (-1, 2) {{\bf =\!=\!=}};
		\node [style=port] (3) at (-1, 3) {};
		\node [style=port] (4) at (0, 3) {};
		\node [style=port] (5) at (0, 0) {};
	\end{pgfonlayer}
	\begin{pgfonlayer}{edgelayer}
		\draw [style=wire] (3) to (2);
		\draw [style=wire] (1) to (0);
		\draw [style=wire, bend right=60, looseness=1.50] (2) to (0);
		\draw [style=wire, in=90, out=-45, looseness=1.25] (2) to (5);
		\draw [style=wire, in=45, out=-90, looseness=1.00] (4) to (0);
	\end{pgfonlayer}
\end{tikzpicture}}%
   \end{array}
   +
   \begin{array}[c]{c}\resizebox{!}{2cm}{%
\begin{tikzpicture}
	\begin{pgfonlayer}{nodelayer}
		\node [style=port] (0) at (-0.5, 3) {};
		\node [style=port] (1) at (0.5, 3) {};
		\node [style=port] (2) at (-0.5, 0) {};
		\node [style=port] (3) at (0.5, 0) {};
	\end{pgfonlayer}
	\begin{pgfonlayer}{edgelayer}
		\draw [style=wire] (0) to (2);
		\draw [style=wire] (1) to (3);
	\end{pgfonlayer}
\end{tikzpicture}}%
   \end{array}
   \end{align*}
\end{proof} 

One might expect the definition of a {\em monoidal\/} deriving transformation, that is a deriving transformation $\mathsf{d}$ for a monoidal coalgebra modality, to require that 
the \textbf{differential monoidal rule} should hold:
\begin{description}
\item[{\bf [d.m]}] Monoidal Rule: $(1 \otimes \mathsf{d})m_\otimes = (\mathsf{d}^\circ \otimes 1 \otimes 1)(1 \otimes \sigma \otimes 1)(m_ \otimes \otimes 1 \otimes 1)\mathsf{d}$
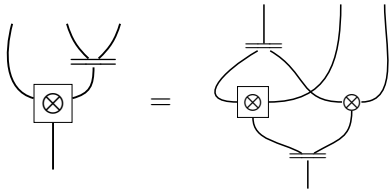
\[\begin{array}[c]{c}\resizebox{!}{2cm}{%
\begin{tikzpicture}
	\begin{pgfonlayer}{nodelayer}
		\node [style=port] (0) at (0, -1.25) {};
		\node [style=port] (1) at (-0.75, 1.5) {};
		\node [style={regular polygon,regular polygon sides=4, draw, inner sep=1pt,minimum size=1pt}] (2) at (0, 0) {$\bigotimes$};
		\node [style=codifferential] (3) at (0.75, 0.75) {{\bf =\!=\!=}};
		\node [style=port] (4) at (1.25, 1.5) {};
		\node [style=port] (5) at (0.25, 1.5) {};
	\end{pgfonlayer}
	\begin{pgfonlayer}{edgelayer}
		\draw [style=wire, in=-105, out=165, looseness=1.00] (2) to (1);
		\draw [style=wire] (2) to (0);
		\draw [style=wire, bend right=15, looseness=1.00] (3) to (4);
		\draw [style=wire, bend left=15, looseness=1.00] (3) to (5);
		\draw [style=wire, in=15, out=-90, looseness=1.25] (3) to (2);
	\end{pgfonlayer}
\end{tikzpicture}
  }%
\end{array}= \begin{array}[c]{c}\resizebox{!}{2.5cm}{%
\begin{tikzpicture}
	\begin{pgfonlayer}{nodelayer}
		\node [style=port] (0) at (-1.25, 1.25) {};
		\node [style=differential] (1) at (-2, -1) {$\bigotimes$};
		\node [style={regular polygon,regular polygon sides=4, draw, inner sep=1pt,minimum size=1pt}] (2) at (-4.25, -1) {$\bigotimes$};
		\node [style=port] (3) at (-2.25, 1.25) {};
		\node [style=port] (4) at (-4, 1.25) {};
		\node [style=port] (5) at (-3, -3) {};
		\node [style=codifferential] (6) at (-3, -2.25) {{\bf =\!=\!=}};
		\node [style=codifferential] (7) at (-4, 0.25) {{\bf =\!=\!=}};
	\end{pgfonlayer}
	\begin{pgfonlayer}{edgelayer}
		\draw [style=wire] (6) to (5);
		\draw [style=wire, in=150, out=-90, looseness=1.00] (2) to (6);
		\draw [style=wire, in=30, out=-90, looseness=1.25] (1) to (6);
		\draw [style=wire, in=0, out=-90, looseness=1.25] (3) to (2);
		\draw [style=wire, in=0, out=-90, looseness=1.00] (0) to (1);
		\draw [style=wire] (4) to (7);
		\draw [style=wire, in=180, out=-30, looseness=1.25] (7) to (1);
		\draw [style=wire, in=180, out=-150, looseness=2.00] (7) to (2);
	\end{pgfonlayer}
\end{tikzpicture}
  }%
\end{array}\]
\end{description}
However, it turns out that this is already implied (see \cite{cockett2017there}): 

\begin{proposition} For a monoidal coalgebra modality, all deriving transformations satisfy the monoidal rule \textbf{[d.m]}. 
\end{proposition}

\subsection{$\mathsf{L}$, $\mathsf{K}$ and $\mathsf{J}$}\label{LJKsec}

Here we introduce $\mathsf{L}$, $\mathsf{K}$ and $\mathsf{J}$, which are natural transformations occurring in all differential categories, constructed using the deriving transformation and coderiving transformation. We provide long lists of rather technical properties which will be used to develop the theory of antiderivatives.

\begin{definition} \normalfont In a differential category, define three natural transformations $\mathsf{L}$, $\mathsf{K}$ and $\mathsf{J}$ of the same type $\oc \Rightarrow \oc$ as follows:
\begin{align*}
\mathsf{L}&=\mathsf{d}^\circ\mathsf{d} & \mathsf{K}&=\mathsf{L}+\oc 0& \mathsf{J}&=\mathsf{L}+1 \\
\begin{array}[c]{c}\resizebox{!}{2cm}{%
\begin{tikzpicture}
	\begin{pgfonlayer}{nodelayer}
		\node [style=port] (0) at (0, 3) {};
		\node [style=port] (1) at (0, 0) {};
		\node [style={circle, draw}] (2) at (0, 1.5) {$\mathsf{L}$};
	\end{pgfonlayer}
	\begin{pgfonlayer}{edgelayer}
		\draw [style=wire] (0) to (2);
		\draw [style=wire] (2) to (1);
	\end{pgfonlayer}
\end{tikzpicture}}%
   \end{array} &=
   \begin{array}[c]{c}\resizebox{!}{2cm}{%
\begin{tikzpicture}
	\begin{pgfonlayer}{nodelayer}
		\node [style=codifferential] (0) at (0, 2) {{\bf =\!=\!=}};
		\node [style=port] (1) at (0, 3) {};
		\node [style=differential] (2) at (0, 1) {{\bf =\!=\!=}};
		\node [style=port] (3) at (0, 0) {};
	\end{pgfonlayer}
	\begin{pgfonlayer}{edgelayer}
		\draw [style=wire] (1) to (0);
		\draw [style=wire] (3) to (2);
		\draw [style=wire, bend right=60, looseness=1.50] (0) to (2);
		\draw [style=wire, bend left=60, looseness=1.50] (0) to (2);
	\end{pgfonlayer}
\end{tikzpicture}}%
   \end{array}&\begin{array}[c]{c}\resizebox{!}{2cm}{%
\begin{tikzpicture}
	\begin{pgfonlayer}{nodelayer}
		\node [style=port] (0) at (0, 3) {};
		\node [style=port] (1) at (0, 0) {};
		\node [style={circle, draw}] (2) at (0, 1.5) {$\mathsf{K}$};
	\end{pgfonlayer}
	\begin{pgfonlayer}{edgelayer}
		\draw [style=wire] (0) to (2);
		\draw [style=wire] (2) to (1);
	\end{pgfonlayer}
\end{tikzpicture}}%
   \end{array} &=
   \begin{array}[c]{c}\resizebox{!}{2cm}{%
\begin{tikzpicture}
	\begin{pgfonlayer}{nodelayer}
		\node [style=codifferential] (0) at (0, 2) {{\bf =\!=\!=}};
		\node [style=port] (1) at (0, 3) {};
		\node [style=differential] (2) at (0, 1) {{\bf =\!=\!=}};
		\node [style=port] (3) at (0, 0) {};
	\end{pgfonlayer}
	\begin{pgfonlayer}{edgelayer}
		\draw [style=wire] (1) to (0);
		\draw [style=wire] (3) to (2);
		\draw [style=wire, bend right=60, looseness=1.50] (0) to (2);
		\draw [style=wire, bend left=60, looseness=1.50] (0) to (2);
	\end{pgfonlayer}
\end{tikzpicture}}%
   \end{array}+
      \begin{array}[c]{c}\resizebox{!}{2cm}{%
\begin{tikzpicture}
	\begin{pgfonlayer}{nodelayer}
		\node [style=port] (0) at (0, 0) {};
		\node [style=port] (1) at (0, 3) {};
		\node [style={regular polygon,regular polygon sides=4, draw}] (2) at (0, 1.5) {$0$};
	\end{pgfonlayer}
	\begin{pgfonlayer}{edgelayer}
		\draw [style=wire] (1) to (2);
		\draw [style=wire] (2) to (0);
	\end{pgfonlayer}
\end{tikzpicture}}%
   \end{array} &\begin{array}[c]{c}\resizebox{!}{2cm}{%
\begin{tikzpicture}
	\begin{pgfonlayer}{nodelayer}
		\node [style=port] (0) at (0, 3) {};
		\node [style=port] (1) at (0, 0) {};
		\node [style={circle, draw}] (2) at (0, 1.5) {$\mathsf{J}$};
	\end{pgfonlayer}
	\begin{pgfonlayer}{edgelayer}
		\draw [style=wire] (0) to (2);
		\draw [style=wire] (2) to (1);
	\end{pgfonlayer}
\end{tikzpicture}}%
   \end{array} &=
   \begin{array}[c]{c}\resizebox{!}{2cm}{%
\begin{tikzpicture}
	\begin{pgfonlayer}{nodelayer}
		\node [style=codifferential] (0) at (0, 2) {{\bf =\!=\!=}};
		\node [style=port] (1) at (0, 3) {};
		\node [style=differential] (2) at (0, 1) {{\bf =\!=\!=}};
		\node [style=port] (3) at (0, 0) {};
	\end{pgfonlayer}
	\begin{pgfonlayer}{edgelayer}
		\draw [style=wire] (1) to (0);
		\draw [style=wire] (3) to (2);
		\draw [style=wire, bend right=60, looseness=1.50] (0) to (2);
		\draw [style=wire, bend left=60, looseness=1.50] (0) to (2);
	\end{pgfonlayer}
\end{tikzpicture}}%
   \end{array}+
   \begin{array}[c]{c}\resizebox{!}{2cm}{%
\begin{tikzpicture}
	\begin{pgfonlayer}{nodelayer}
		\node [style=port] (0) at (0, 3) {};
		\node [style=port] (1) at (0, 0) {};
	\end{pgfonlayer}
	\begin{pgfonlayer}{edgelayer}
		\draw [style=wire] (0) to (1);
	\end{pgfonlayer}
\end{tikzpicture}}%
   \end{array}
\end{align*}
\end{definition}
The main purpose of $\mathsf{L}$ will be to simplify proofs of the properties for $\mathsf{K}$ and $\mathsf{J}$.  The properties of $\mathsf{K}$ and $\mathsf{J}$ will be used to establish properties of $\mathsf{K}^{-1}$ and $\mathsf{J}^{-1}$ which will, in turn, be used in the definition of antiderivatives. It is convenient also to have a special  notation for the following families of maps which will become 
important:
$$\mathsf{L}_n := \mathsf{L} + n \cdot 1 ~~~~~\mbox{and}~~~~~ \mathsf{J}_n := \mathsf{J} + n \cdot 1$$
Clearly $\mathsf{L} = \mathsf{L}_0$, $\mathsf{J} = \mathsf{J}_0 = \mathsf{L}_1$, and $\mathsf{J}_n = \mathsf{L}_{n+1}$.  Thus, by establishing properties of  $\mathsf{L}_{n}$, we will obtain as 
a corollary properties of $\mathsf{J}_n$.  We shall begin therefore with the basic properties of $\mathsf{L}_n$:

\begin{proposition}\label{Lprop} $\mathsf{L}$ satisfies the following properties:
\begin{enumerate}[{\bf [L.1]}]
\item $\mathsf{L}_n \oc(0)=n \cdot \oc(0) =\oc(0)\mathsf{L}_n$;
\item $\mathsf{L}_n e=n \cdot e$;
\item $\mathsf{L}_n \varepsilon= (n+1) \cdot \varepsilon$;
\item $\mathsf{L}_n \Delta=\Delta(\mathsf{L}_r \otimes 1)+\Delta(1 \otimes \mathsf{L}_s)$ where $r+s=n$;
\item $\mathsf{L}_n\delta=\delta \mathsf{d}^\circ (1 \otimes \mathsf{L})\mathsf{d} + n \cdot \delta$ ;
\item $\mathsf{d}\mathsf{L}_n=(\mathsf{L}_{n+1} \otimes 1)\mathsf{d}$;
\item $\mathsf{L}_n\mathsf{d}^\circ=\mathsf{d}^\circ(\mathsf{L}_{n+1} \otimes 1)$;
\item $\mathsf{L}_n\mathsf{d}^\circ (1 \otimes f) \mathsf{d} =  \mathsf{d}^\circ (1 \otimes f) \mathsf{d} \mathsf{L}_n$;
\item $(\mathsf{d}^\circ \otimes 1)(\mathsf{L}_n \otimes 1)\mathsf{d}= (1 \otimes \sigma)(\mathsf{d}^\circ \otimes 1)(\mathsf{L}_n \otimes 1)\mathsf{d}$;
\item $\mathsf{d}^\circ(\mathsf{L}_n \otimes 1)(\mathsf{d}^\circ \otimes 1)= \mathsf{d}^\circ (\mathsf{L}_n \otimes 1)(\mathsf{d}^\circ \otimes 1)(1 \otimes \sigma)$;
\item $(\mathsf{L}_n \otimes 1)\mathsf{W}=\mathsf{W}(\mathsf{L}_n \otimes 1)$;
\item $(\mathsf{L}_n \otimes 1)\mathsf{d}\mathsf{d}^\circ=\mathsf{d}\mathsf{d}^\circ(\mathsf{L}_n \otimes 1)$
\end{enumerate}
\end{proposition} 

\begin{proof} These are mostly straightforward calculations: \\ 

\noindent
{\bf [L.1]}: The naturality of the deriving transformation and coderiving transformation gives:
\begin{align*}
\begin{array}[c]{c}\resizebox{!}{2cm}{%
\begin{tikzpicture}
	\begin{pgfonlayer}{nodelayer}
		\node [style=port] (0) at (0, 0) {};
		\node [style={regular polygon,regular polygon sides=4, draw}] (1) at (0, 0.75) {$0$};
		\node [style=port] (2) at (0, 3) {};
		\node [style={circle, draw}] (3) at (0, 2.25) {$\mathsf{L}$};
	\end{pgfonlayer}
	\begin{pgfonlayer}{edgelayer}
		\draw [style=wire] (2) to (3);
		\draw [style=wire] (3) to (1);
		\draw [style=wire] (1) to (0);
	\end{pgfonlayer}
\end{tikzpicture}}%
   \end{array}=
\begin{array}[c]{c}\resizebox{!}{2cm}{%
\begin{tikzpicture}
	\begin{pgfonlayer}{nodelayer}
		\node [style=port] (0) at (0, 2.5) {};
		\node [style=codifferential] (1) at (0, 1.5) {{\bf =\!=\!=}};
		\node [style=differential] (2) at (0, 0.5) {{\bf =\!=\!=}};
		\node [style={regular polygon,regular polygon sides=4, draw}] (3) at (0, -0.25) {$0$};
		\node [style=port] (4) at (0, -1) {};
	\end{pgfonlayer}
	\begin{pgfonlayer}{edgelayer}
		\draw [style=wire] (0) to (1);
		\draw [style=wire, bend right=60, looseness=1.50] (1) to (2);
		\draw [style=wire, bend left=60, looseness=1.50] (1) to (2);
		\draw [style=wire] (2) to (3);
		\draw [style=wire] (3) to (4);
	\end{pgfonlayer}
\end{tikzpicture}}%
   \end{array}=
   \begin{array}[c]{c}\resizebox{!}{2cm}{%
\begin{tikzpicture}
	\begin{pgfonlayer}{nodelayer}
		\node [style=port] (0) at (0, 2.5) {};
		\node [style=codifferential] (1) at (0, 1.5) {{\bf =\!=\!=}};
		\node [style=differential] (2) at (0, 0) {{\bf =\!=\!=}};
		\node [style={regular polygon,regular polygon sides=4, draw}] (3) at (-0.5, 0.75) {$0$};
		\node [style={circle, draw}] (4) at (0.5, 0.75) {$0$};
		\node [style=port] (5) at (0, -1) {};
	\end{pgfonlayer}
	\begin{pgfonlayer}{edgelayer}
		\draw [style=wire] (0) to (1);
		\draw [style=wire, bend left, looseness=1.25] (4) to (2);
		\draw [style=wire, bend right, looseness=1.25] (3) to (2);
		\draw [style=wire, bend right=15, looseness=1.00] (1) to (3);
		\draw [style=wire, bend left=15, looseness=1.25] (1) to (4);
		\draw [style=wire] (2) to (5);
	\end{pgfonlayer}
\end{tikzpicture}}%
   \end{array}=0= \begin{array}[c]{c}\resizebox{!}{2cm}{%
\begin{tikzpicture}
	\begin{pgfonlayer}{nodelayer}
		\node [style=port] (0) at (0, 2.5) {};
		\node [style=codifferential] (1) at (0, 1.5) {{\bf =\!=\!=}};
		\node [style=differential] (2) at (0, 0) {{\bf =\!=\!=}};
		\node [style={regular polygon,regular polygon sides=4, draw}] (3) at (-0.5, 0.75) {$0$};
		\node [style={circle, draw}] (4) at (0.5, 0.75) {$0$};
		\node [style=port] (5) at (0, -1) {};
	\end{pgfonlayer}
	\begin{pgfonlayer}{edgelayer}
		\draw [style=wire] (0) to (1);
		\draw [style=wire, bend left, looseness=1.25] (4) to (2);
		\draw [style=wire, bend right, looseness=1.25] (3) to (2);
		\draw [style=wire, bend right=15, looseness=1.00] (1) to (3);
		\draw [style=wire, bend left=15, looseness=1.25] (1) to (4);
		\draw [style=wire] (2) to (5);
	\end{pgfonlayer}
\end{tikzpicture}}%
   \end{array}=
   \begin{array}[c]{c}\resizebox{!}{2cm}{%
\begin{tikzpicture}
	\begin{pgfonlayer}{nodelayer}
		\node [style=port] (0) at (0, -1) {};
		\node [style=codifferential] (1) at (0, 0) {{\bf =\!=\!=}};
		\node [style=differential] (2) at (0, 1) {{\bf =\!=\!=}};
		\node [style={regular polygon,regular polygon sides=4, draw}] (3) at (0, 1.75) {$0$};
		\node [style=port] (4) at (0, 2.5) {};
	\end{pgfonlayer}
	\begin{pgfonlayer}{edgelayer}
		\draw [style=wire] (0) to (1);
		\draw [style=wire, bend left=60, looseness=1.50] (1) to (2);
		\draw [style=wire, bend right=60, looseness=1.50] (1) to (2);
		\draw [style=wire] (2) to (3);
		\draw [style=wire] (3) to (4);
	\end{pgfonlayer}
\end{tikzpicture}}%
   \end{array}=\begin{array}[c]{c}\resizebox{!}{2cm}{%
\begin{tikzpicture}
	\begin{pgfonlayer}{nodelayer}
		\node [style=port] (0) at (0, 0) {};
		\node [style={circle, draw}] (1) at (0, 0.75) {$\mathsf{L}$};
		\node [style=port] (2) at (0, 3) {};
		\node [style={regular polygon,regular polygon sides=4, draw}] (3) at (0, 2.25) {$0$};
	\end{pgfonlayer}
	\begin{pgfonlayer}{edgelayer}
		\draw [style=wire] (2) to (3);
		\draw [style=wire] (3) to (1);
		\draw [style=wire] (1) to (0);
	\end{pgfonlayer}
\end{tikzpicture}}%
   \end{array}
\end{align*}
By adding $n \cdot 1$ to each side we obtain the result. \\ 

\noindent
{\bf [L.2]}: Use that the derivative of a constant is zero {\bf [d.1]} to obtain
\begin{align*}
\begin{array}[c]{c}\resizebox{!}{1.5cm}{%
\begin{tikzpicture}
	\begin{pgfonlayer}{nodelayer}
		\node [style={circle, draw}] (0) at (0, 2) {$\mathsf{L}$};
		\node [style=port] (1) at (0, 3) {};
		\node [style={circle, draw}] (2) at (0, 1) {$e$};
	\end{pgfonlayer}
	\begin{pgfonlayer}{edgelayer}
		\draw [style=wire] (1) to (0);
		\draw [style=wire] (0) to (2);
	\end{pgfonlayer}
\end{tikzpicture}}%
   \end{array}=
\begin{array}[c]{c}\resizebox{!}{2cm}{%
\begin{tikzpicture}
	\begin{pgfonlayer}{nodelayer}
		\node [style={circle, draw}] (0) at (0, 0) {$e$};
		\node [style=codifferential] (1) at (0, 0.75) {{\bf =\!=\!=}};
		\node [style=differential] (2) at (0, 1.75) {{\bf =\!=\!=}};
		\node [style=port] (3) at (0, 2.5) {};
	\end{pgfonlayer}
	\begin{pgfonlayer}{edgelayer}
		\draw [style=wire] (0) to (1);
		\draw [style=wire, bend left=60, looseness=1.50] (1) to (2);
		\draw [style=wire, bend right=60, looseness=1.50] (1) to (2);
		\draw [style=wire] (2) to (3);
	\end{pgfonlayer}
\end{tikzpicture}}%
 \end{array}=0\end{align*}
 By adding $n \cdot e$ to each side we obtain the result.\\ 
 
 \noindent
{\bf [L.3]} Here we use that the derivative of a linear map is a constant {\bf [d.3]} and {\bf [cd.1]}
\begin{align*}
\begin{array}[c]{c}\resizebox{!}{2cm}{%
\begin{tikzpicture}
	\begin{pgfonlayer}{nodelayer}
		\node [style={circle, draw}] (0) at (0, 2) {$\mathsf{L}$};
		\node [style=port] (1) at (0, 3) {};
		\node [style={circle, draw}] (2) at (0, 1) {$\varepsilon$};
		\node [style=port] (3) at (0, 0) {};
	\end{pgfonlayer}
	\begin{pgfonlayer}{edgelayer}
		\draw [style=wire] (1) to (0);
		\draw [style=wire] (0) to (2);
		\draw [style=wire] (2) to (3);
	\end{pgfonlayer}
\end{tikzpicture}}%
   \end{array}=
\begin{array}[c]{c}\resizebox{!}{2cm}{%
\begin{tikzpicture}
	\begin{pgfonlayer}{nodelayer}
		\node [style=codifferential] (0) at (0, 0.75) {{\bf =\!=\!=}};
		\node [style=differential] (1) at (0, 1.75) {{\bf =\!=\!=}};
		\node [style=port] (2) at (0, 2.5) {};
		\node [style={circle, draw}] (3) at (0, 0) {$\varepsilon$};
		\node [style=port] (4) at (0, -0.75) {};
	\end{pgfonlayer}
	\begin{pgfonlayer}{edgelayer}
		\draw [style=wire, bend left=60, looseness=1.50] (0) to (1);
		\draw [style=wire, bend right=60, looseness=1.50] (0) to (1);
		\draw [style=wire] (1) to (2);
		\draw [style=wire] (3) to (4);
		\draw [style=wire] (0) to (3);
	\end{pgfonlayer}
\end{tikzpicture}}%
   \end{array}=
\begin{array}[c]{c}\resizebox{!}{2cm}{%
\begin{tikzpicture}
	\begin{pgfonlayer}{nodelayer}
		\node [style=port] (0) at (0, 3) {};
		\node [style=codifferential] (1) at (0, 2) {{\bf =\!=\!=}};
		\node [style={circle, draw}] (2) at (-0.5, 0.75) {$e$};
		\node [style=port] (3) at (0.5, 0) {};
	\end{pgfonlayer}
	\begin{pgfonlayer}{edgelayer}
		\draw [style=wire, bend right=15, looseness=1.25] (1) to (2);
		\draw [style=wire, bend left=15, looseness=1.00] (1) to (3);
		\draw [style=wire] (0) to (1);
	\end{pgfonlayer}
\end{tikzpicture}}%
   \end{array}%
\end{align*}
and add $n \cdot \varepsilon$ to both sides.\\ 

\noindent
{\bf [L.4]}: Here we use the Leibniz rule {\bf [d.2]}, {\bf [cd.3]}, {\bf [cd.4]} and the cocommutativity of the comultiplication: 
\begin{align*}
%
\end{align*}
By adding $n \cdot \Delta= r \cdot \Delta + s \cdot \Delta$  to both sides we obtain the result. \\ 

\noindent
{\bf [L.5]}: Here we use the chain rule {\bf [d.4]}, {\bf [cd.3]}, the comonad triangle identity and that $\delta$ is a coalgebra map to obtain: 
\begin{align*}
%
\end{align*}
Adding $n \cdot \delta$ to each side gives the result. \\ 

\noindent
{\bf [L.6]}: Here we use Proposition \ref{Wprop} and {\bf [d.5]} to obtain:
 \begin{align*}%
   \end{align*}
 Adding $n \cdot \mathsf{d}$ then gives the result. \\
 
 \noindent
{\bf [L.7]}: Here we use Proposition \ref{Wprop}, {\bf [cd.7]} and then add $n \cdot \mathsf{d}^\circ$ to each side:
 \begin{align*}%
   \end{align*} \\
   
 \noindent
 {\bf [L.8]}: This is immediate from {\bf [L.6]} and {\bf [L.7]}. \\
   
\noindent
{\bf [L.9]}: Here we use {\bf [L.6]} and {\bf [cd.7]} and then add $n$ copies of the differential interchange {\bf [d.5]} equality: 
   \begin{align*}
%
   \end{align*}
   
\noindent
{\bf [L.10]}: Here we use {\bf [L.7]} and {\bf [cd.5]} to obtain:
 \begin{align*}
%
   \end{align*}
   Then we add $n$ copies of the identity {\bf [cd.5]} to this to obtain the result. \\ 

\noindent
{\bf [L.11]} Here we use {\bf [L.9]} and {\bf [L.10]} and add in $n$ copies of $\mathsf{W}$: 
\begin{align*}
   %
\end{align*} \\ 

\noindent
{\bf [L.12]} Here we use {\bf [L.11]} and Proposition \ref{Wprop}: 
\begin{align*}
 %
\end{align*}
\end{proof} 

As $\mathsf{J}_n = \mathsf{L}_{n+1}$ we immediately have:

\begin{corollary}\label{Jprop} $\mathsf{J}$ satisfies the following properties:
\begin{enumerate}[{\bf [J.1]}]
\item $\mathsf{J}_n\oc(0)=(n+1) \cdot \oc(0)=\oc(0)\mathsf{J}_n$;
\item $\mathsf{J}_n e= (n+1) \cdot e$;
\item $\mathsf{J}_n \varepsilon= (n+2) \cdot \varepsilon$; 
\item $\mathsf{J}_n\Delta=\Delta(\mathsf{L}_r \otimes 1)+\Delta(1 \otimes \mathsf{L}_s)$ where $r + s = n+1$;
\item $\mathsf{J}_n\delta=\delta \mathsf{d}^\circ (1 \otimes \mathsf{L})\mathsf{d} + (n+1) \cdot \delta$;
\item $\mathsf{d} \mathsf{J}_n = (\mathsf{J}_{n+1} \otimes 1)\mathsf{d}$; 
\item $\mathsf{J}_n\mathsf{d}^\circ = \mathsf{d}^\circ (\mathsf{J}_{n+1} \otimes 1)$;
\item $\mathsf{J}_n\mathsf{d}^\circ (1 \otimes f) \mathsf{d} = \mathsf{d}^\circ (1 \otimes f) \mathsf{d} \mathsf{J}_n$;
\item $(\mathsf{d}^\circ \otimes 1)(\mathsf{J}_n \otimes 1)\mathsf{d}= (1 \otimes \sigma)(\mathsf{d}^\circ \otimes 1)(\mathsf{J}_n \otimes 1)\mathsf{d}$;
\item $\mathsf{d}^\circ(\mathsf{J}_n \otimes 1)(\mathsf{d}^\circ \otimes 1)= \mathsf{d}^\circ (\mathsf{J}_n \otimes 1)(\mathsf{d}^\circ \otimes 1)(1 \otimes \sigma)$;
\item $(\mathsf{J}_n \otimes 1)\mathsf{W}=\mathsf{W}(\mathsf{J}_n \otimes 1)$;
\item $(\mathsf{J}_n \otimes 1)\mathsf{d}\mathsf{d}^\circ=\mathsf{d}\mathsf{d}^\circ(\mathsf{J}_n \otimes 1)$;
\end{enumerate}
\end{corollary}

We now give properties of $\mathsf{K}$. Notice that most of the properties of $\mathsf{K}$ are extremely similar, indeed sometimes identical, to the corresponding property of $\mathsf{J}$.

\begin{proposition}\label{Kprop} $\mathsf{K}$ satisfies the following properties:
\begin{enumerate}[{\bf [K.1]}]
\item $\mathsf{K}\oc(0)=\oc(0)=\oc(0)\mathsf{K}$; 
\item $\mathsf{K}e=e$; 
\item $\mathsf{K}\varepsilon= \varepsilon$; 
\item $\mathsf{K}\Delta=\Delta(\mathsf{L} \otimes 1)+\Delta(1 \otimes \mathsf{L})+\Delta(\oc(0) \otimes \oc(0))$;
\item $\mathsf{K}\delta=\delta \mathsf{d}^\circ (1 \otimes \mathsf{L})\mathsf{d}+ \delta\oc(\oc(0))$;
\item $\mathsf{d}\mathsf{K}=\mathsf{d}\mathsf{L}$;
\item $\mathsf{K}\mathsf{d}^\circ=\mathsf{L}\mathsf{d}^\circ$;
\item $\mathsf{K}\mathsf{d}^\circ (1 \otimes f) \mathsf{d} = \mathsf{d}^\circ (1 \otimes f) \mathsf{d} \mathsf{K}$;
\item $(\mathsf{d}^\circ \otimes 1)(\mathsf{K} \otimes 1)\mathsf{d}= (1 \otimes \sigma)(\mathsf{d}^\circ \otimes 1)(\mathsf{K} \otimes 1)\mathsf{d}$;
\item $\mathsf{d}^\circ(\mathsf{K} \otimes 1)(\mathsf{d}^\circ \otimes 1)= \mathsf{d}^\circ (\mathsf{K} \otimes 1)(\mathsf{d}^\circ \otimes 1)(1 \otimes \sigma)$;
\item $(\mathsf{K} \otimes 1)\mathsf{W}=\mathsf{W}(\mathsf{K} \otimes 1)$;
\item $(\mathsf{K} \otimes 1)\mathsf{d}\mathsf{d}^\circ=\mathsf{d}\mathsf{d}^\circ(\mathsf{K} \otimes 1)$.
\end{enumerate}
\end{proposition} 

\begin{proof} These are straightforward calculations using the properties of $\mathsf{L}$: \\ \\
{\bf [K.1]}: Here we use {\bf [L.1]} and that $\oc(0)\oc(0)=\oc(0)$: 
\begin{align*}

\end{align*}
{\bf [K.2]}: Here we use {\bf [L.2]} and the naturality of $e$: 
\begin{align*}
%
\end{align*}
{\bf [K.3]}: Here we use {\bf [L.3]} and the naturality of $\varepsilon$: 
\begin{align*}
%
\end{align*}
{\bf [K.4]}: Here we use {\bf [L.4]} and the naturality of $\Delta$:
\begin{align*}
%
\end{align*}
{\bf [K.5]}: Here we use {\bf [L.5]} and the naturality of $\delta$:
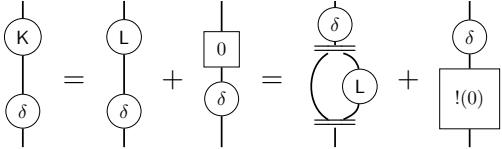
\begin{align*}
%
\end{align*}
{\bf [K.6]} Here we use the naturality of the differential transformation: 
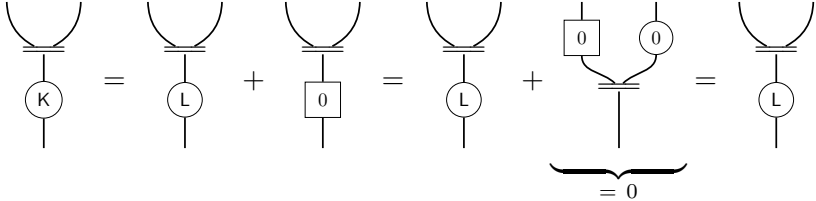
\begin{align*}
   %
\end{align*}
{\bf [K.7]} Here we use the naturality of the coderivative: 
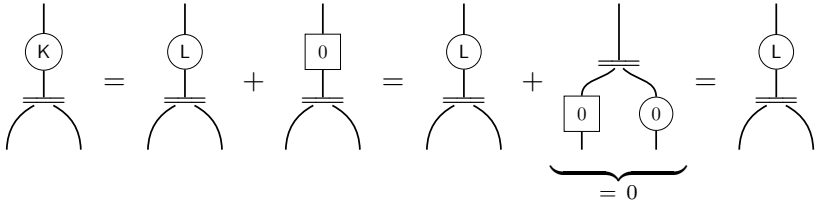
\begin{align*}
   %
   \end{align*}
{\bf [K.8]}: Use  {\bf [K.7]},  {\bf [L.8]}, and {\bf [K.6]}. \\
   
\noindent
{\bf [K.8]}: Here we use {\bf [K.6]} and {\bf [L.9]}: 
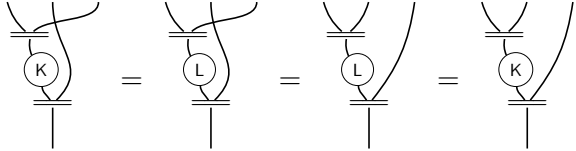
\begin{align*}%
\end{align*}
{\bf [K.10]}: Here we use {\bf [K.7]} and {\bf [L.10]}: 
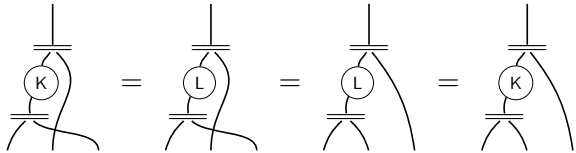
\begin{align*}%
\end{align*}
{\bf [K.11]}: Here we use {\bf [K.6]}, {\bf [K.7]} and {\bf [L.11]}: 
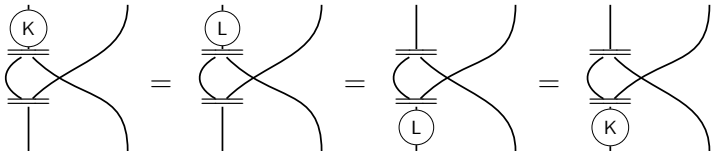
\begin{align*}
 %
\end{align*}
{\bf [K.12]}: Here we use {\bf [K.11]}: 
\begin{align*}
 %
\end{align*}
\end{proof}

$\mathsf{L}$, $\mathsf{K}$, and $\mathsf{J}$ are compatible with the monoidal coalgebra modality as well. 

\begin{proposition} In a differential category with a monoidal coalgebra modality, the following equalities hold: 
\end{proposition}
\begin{description}
\item[\textbf{[L.m]}] $m_\otimes \mathsf{L}_n = (\mathsf{L}_n \otimes 1)m_\otimes= (1 \otimes \mathsf{L}_n)m_\otimes$;
\item[\textbf{[J.m]}] $m_\otimes \mathsf{J}_n= (\mathsf{J}_n \otimes 1)m_\otimes= (1 \otimes \mathsf{J}_n)m_\otimes$.
\item[\textbf{[K.m]}] $m_\otimes \mathsf{K}= (\mathsf{K} \otimes 1)m_\otimes= (1 \otimes \mathsf{K})m_\otimes$;
\end{description}
\begin{proof} Straightforward using the monoidal rules of the deriving transformation and coderiving transformation. \\ \\
\textbf{[L.m]}: Here, for $\mathsf{L}_0$, we use \textbf{[cd.m]} and \textbf{[d.m]}:
\begin{align*}
\begin{array}[c]{c}\resizebox{!}{2cm}{%
\begin{tikzpicture}
	\begin{pgfonlayer}{nodelayer}
		\node [style={regular polygon,regular polygon sides=4, draw, inner sep=1pt,minimum size=1pt}] (0) at (0, 4) {$\bigotimes$};
		\node [style=port] (1) at (0.75, 5) {};
		\node [style=port] (2) at (-0.75, 5) {};
		\node [style={circle, draw}] (3) at (0, 3) {$\mathsf{L}$};
		\node [style=port] (4) at (0, 2) {};
	\end{pgfonlayer}
	\begin{pgfonlayer}{edgelayer}
		\draw [style=wire, in=-90, out=180, looseness=1.25] (0) to (2);
		\draw [style=wire, in=0, out=-90, looseness=1.25] (1) to (0);
		\draw [style=wire] (3) to (4);
		\draw [style=wire] (0) to (3);
	\end{pgfonlayer}
\end{tikzpicture}
  }%
\end{array}= 
   \begin{array}[c]{c}\resizebox{!}{2.5cm}{%
\begin{tikzpicture}
	\begin{pgfonlayer}{nodelayer}
		\node [style={regular polygon,regular polygon sides=4, draw, inner sep=1pt,minimum size=1pt}] (0) at (1.25, 3.25) {$\bigotimes$};
		\node [style=port] (1) at (2, 4.25) {};
		\node [style=port] (2) at (0.5, 4.25) {};
		\node [style=differential] (3) at (1.25, 1.25) {{\bf =\!=\!=}};
		\node [style=codifferential] (4) at (1.25, 2.25) {{\bf =\!=\!=}};
		\node [style=port] (5) at (1.25, 0.25) {};
	\end{pgfonlayer}
	\begin{pgfonlayer}{edgelayer}
		\draw [style=wire, in=-90, out=180, looseness=1.25] (0) to (2);
		\draw [style=wire, in=0, out=-90, looseness=1.25] (1) to (0);
		\draw [style=wire, bend right=60, looseness=1.50] (4) to (3);
		\draw [style=wire, bend left=60, looseness=1.50] (4) to (3);
		\draw [style=wire] (0) to (4);
		\draw [style=wire] (3) to (5);
	\end{pgfonlayer}
\end{tikzpicture}
  }%
\end{array}=    \begin{array}[c]{c}\resizebox{!}{2.5cm}{%
\begin{tikzpicture}
	\begin{pgfonlayer}{nodelayer}
		\node [style=codifferential] (0) at (-3, 2.75) {{\bf =\!=\!=}};
		\node [style=differential] (1) at (-1, 1.5) {$\bigotimes$};
		\node [style=port] (2) at (-3, 3.75) {};
		\node [style=port] (3) at (-2, -0.5) {};
		\node [style=codifferential] (4) at (-2, 0.25) {{\bf =\!=\!=}};
		\node [style={regular polygon,regular polygon sides=4, draw, inner sep=1pt,minimum size=1pt}] (5) at (-3.25, 1.5) {$\bigotimes$};
		\node [style=port] (6) at (-1, 3.75) {};
		\node [style=codifferential] (7) at (-1, 2.75) {{\bf =\!=\!=}};
	\end{pgfonlayer}
	\begin{pgfonlayer}{edgelayer}
		\draw [style=wire] (4) to (3);
		\draw [style=wire, in=150, out=-90, looseness=1.00] (5) to (4);
		\draw [style=wire, in=30, out=-90, looseness=1.25] (1) to (4);
		\draw [style=wire] (2) to (0);
		\draw [style=wire, in=180, out=-30, looseness=1.25] (0) to (1);
		\draw [style=wire, in=180, out=-150, looseness=2.00] (0) to (5);
		\draw [style=wire] (6) to (7);
		\draw [style=wire, in=0, out=-150, looseness=1.25] (7) to (5);
		\draw [style=wire, in=0, out=-45, looseness=1.75] (7) to (1);
	\end{pgfonlayer}
\end{tikzpicture}
  }%
\end{array}=    \begin{array}[c]{c}\resizebox{!}{2.5cm}{%
\begin{tikzpicture}
	\begin{pgfonlayer}{nodelayer}
		\node [style=port] (0) at (-3, 3.5) {};
		\node [style={regular polygon,regular polygon sides=4, draw, inner sep=1pt,minimum size=1pt}] (1) at (-2, 0.5) {$\bigotimes$};
		\node [style=port] (2) at (-1, 3.5) {};
		\node [style=differential] (3) at (-1, 1.5) {{\bf =\!=\!=}};
		\node [style=codifferential] (4) at (-1, 2.5) {{\bf =\!=\!=}};
		\node [style=port] (5) at (-2, -0.5) {};
	\end{pgfonlayer}
	\begin{pgfonlayer}{edgelayer}
		\draw [style=wire, bend right=60, looseness=1.50] (4) to (3);
		\draw [style=wire, bend left=60, looseness=1.50] (4) to (3);
		\draw [style=wire] (2) to (4);
		\draw [style=wire, in=0, out=-90, looseness=1.50] (3) to (1);
		\draw [style=wire] (1) to (5);
		\draw [style=wire, in=180, out=-90, looseness=1.00] (0) to (1);
	\end{pgfonlayer}
\end{tikzpicture}
  }%
\end{array}= \begin{array}[c]{c}\resizebox{!}{2cm}{%
\begin{tikzpicture}
	\begin{pgfonlayer}{nodelayer}
		\node [style={regular polygon,regular polygon sides=4, draw, inner sep=1pt,minimum size=1pt}] (0) at (0, 2.25) {$\bigotimes$};
		\node [style=port] (1) at (-0.75, 4.25) {};
		\node [style=port] (2) at (0, 1) {};
		\node [style={circle, draw}] (3) at (0.75, 3.25) {$\mathsf{L}$};
		\node [style=port] (4) at (0.75, 4.25) {};
	\end{pgfonlayer}
	\begin{pgfonlayer}{edgelayer}
		\draw [style=wire, in=180, out=-90, looseness=1.25] (1) to (0);
		\draw [style=wire] (0) to (2);
		\draw [style=wire] (4) to (3);
		\draw [style=wire, in=0, out=-90, looseness=1.25] (3) to (0);
	\end{pgfonlayer}
\end{tikzpicture}
  }%
\end{array}
\end{align*}
To obtain the other equality with $\mathsf{L}$ on left, one must simply use the left versions of \textbf{[d.m]} and \textbf{[cd.m]} instead.   Now to 
obtain the general result for $\mathsf{L}_n$ we have:

\begin{align*}
 \begin{array}[c]{c}\resizebox{!}{2cm}{%
\begin{tikzpicture}
	\begin{pgfonlayer}{nodelayer}
		\node [style={regular polygon,regular polygon sides=4, draw, inner sep=1pt,minimum size=1pt}] (0) at (0, 4) {$\bigotimes$};
		\node [style=port] (1) at (0.75, 5) {};
		\node [style=port] (2) at (-0.75, 5) {};
		\node [style={circle, draw, inner sep=1pt,minimum size=1pt}] (3) at (0, 3) {$\mathsf{L}_n$};
		\node [style=port] (4) at (0, 2) {};
	\end{pgfonlayer}
	\begin{pgfonlayer}{edgelayer}
		\draw [style=wire, in=-90, out=180, looseness=1.25] (0) to (2);
		\draw [style=wire, in=0, out=-90, looseness=1.25] (1) to (0);
		\draw [style=wire] (3) to (4);
		\draw [style=wire] (0) to (3);
	\end{pgfonlayer}
\end{tikzpicture}
  }%
\end{array}=  \begin{array}[c]{c}\resizebox{!}{2cm}{%
\begin{tikzpicture}
	\begin{pgfonlayer}{nodelayer}
		\node [style={regular polygon,regular polygon sides=4, draw, inner sep=1pt,minimum size=1pt}] (0) at (0, 4) {$\bigotimes$};
		\node [style=port] (1) at (0.75, 5) {};
		\node [style=port] (2) at (-0.75, 5) {};
		\node [style={circle, draw}] (3) at (0, 3) {$\mathsf{L}$};
		\node [style=port] (4) at (0, 2) {};
	\end{pgfonlayer}
	\begin{pgfonlayer}{edgelayer}
		\draw [style=wire, in=-90, out=180, looseness=1.25] (0) to (2);
		\draw [style=wire, in=0, out=-90, looseness=1.25] (1) to (0);
		\draw [style=wire] (3) to (4);
		\draw [style=wire] (0) to (3);
	\end{pgfonlayer}
\end{tikzpicture}
  }%
\end{array} +  n \cdot \begin{array}[c]{c}\resizebox{!}{1.5cm}{%
\begin{tikzpicture}
	\begin{pgfonlayer}{nodelayer}
		\node [style=port] (0) at (2, 1) {};
		\node [style=port] (1) at (1.25, -1) {};
		\node [style=port] (2) at (0.5, 1) {};
		\node [style={regular polygon,regular polygon sides=4, draw, inner sep=1pt,minimum size=1pt}] (3) at (1.25, 0) {$\bigotimes$};
	\end{pgfonlayer}
	\begin{pgfonlayer}{edgelayer}
		\draw [style=wire, in=-90, out=180, looseness=1.25] (3) to (2);
		\draw [style=wire] (3) to (1);
		\draw [style=wire, in=0, out=-90, looseness=1.25] (0) to (3);
	\end{pgfonlayer}
\end{tikzpicture}
  }%
\end{array} = \begin{array}[c]{c}\resizebox{!}{2cm}{%
\begin{tikzpicture}
	\begin{pgfonlayer}{nodelayer}
		\node [style={regular polygon,regular polygon sides=4, draw, inner sep=1pt,minimum size=1pt}] (0) at (0, 2.25) {$\bigotimes$};
		\node [style=port] (1) at (-0.75, 4.25) {};
		\node [style=port] (2) at (0, 1) {};
		\node [style={circle, draw}] (3) at (0.75, 3.25) {$\mathsf{L}$};
		\node [style=port] (4) at (0.75, 4.25) {};
	\end{pgfonlayer}
	\begin{pgfonlayer}{edgelayer}
		\draw [style=wire, in=180, out=-90, looseness=1.25] (1) to (0);
		\draw [style=wire] (0) to (2);
		\draw [style=wire] (4) to (3);
		\draw [style=wire, in=0, out=-90, looseness=1.25] (3) to (0);
	\end{pgfonlayer}
\end{tikzpicture}
  }%
\end{array}+ n \cdot \begin{array}[c]{c}\resizebox{!}{1.5cm}{%
\begin{tikzpicture}
	\begin{pgfonlayer}{nodelayer}
		\node [style=port] (0) at (2, 1) {};
		\node [style=port] (1) at (1.25, -1) {};
		\node [style=port] (2) at (0.5, 1) {};
		\node [style={regular polygon,regular polygon sides=4, draw, inner sep=1pt,minimum size=1pt}] (3) at (1.25, 0) {$\bigotimes$};
	\end{pgfonlayer}
	\begin{pgfonlayer}{edgelayer}
		\draw [style=wire, in=-90, out=180, looseness=1.25] (3) to (2);
		\draw [style=wire] (3) to (1);
		\draw [style=wire, in=0, out=-90, looseness=1.25] (0) to (3);
	\end{pgfonlayer}
\end{tikzpicture}
  }%
\end{array}= \begin{array}[c]{c}\resizebox{!}{2cm}{%
\begin{tikzpicture}
	\begin{pgfonlayer}{nodelayer}
		\node [style={regular polygon,regular polygon sides=4, draw, inner sep=1pt,minimum size=1pt}] (0) at (0, 2.25) {$\bigotimes$};
		\node [style=port] (1) at (-0.75, 4.25) {};
		\node [style=port] (2) at (0, 1) {};
		\node [style={circle, draw, inner sep=1pt,minimum size=1pt}] (3) at (0.75, 3.25) {$\mathsf{L}_n$};
		\node [style=port] (4) at (0.75, 4.25) {};
	\end{pgfonlayer}
	\begin{pgfonlayer}{edgelayer}
		\draw [style=wire, in=180, out=-90, looseness=1.25] (1) to (0);
		\draw [style=wire] (0) to (2);
		\draw [style=wire] (4) to (3);
		\draw [style=wire, in=0, out=-90, looseness=1.25] (3) to (0);
	\end{pgfonlayer}
\end{tikzpicture}
  }%
\end{array}
\end{align*} \\

\noindent
\textbf{[J.m]} This is a direct consequence of \textbf{[L.m]}.

\noindent
\textbf{[K.m]}: By the naturality of $m_\otimes$ and the additive structure we first observe that: 
\[(\oc(0) \otimes 1)m_\otimes= m_\otimes \oc(0 \otimes 1) = m_\otimes \oc(0) = m_\otimes \oc(1 \otimes 0)= (1 \otimes \oc(0)) m_\otimes\]
Then using \textbf{[L.m]}, the additive structure and the naturality of $m_\otimes$ we get that: 
\begin{align*}
 \begin{array}[c]{c}\resizebox{!}{2cm}{%
\begin{tikzpicture}
	\begin{pgfonlayer}{nodelayer}
		\node [style={regular polygon,regular polygon sides=4, draw, inner sep=1pt,minimum size=1pt}] (0) at (0, 4) {$\bigotimes$};
		\node [style=port] (1) at (0.75, 5) {};
		\node [style=port] (2) at (-0.75, 5) {};
		\node [style={circle, draw}] (3) at (0, 3) {$\mathsf{K}$};
		\node [style=port] (4) at (0, 2) {};
	\end{pgfonlayer}
	\begin{pgfonlayer}{edgelayer}
		\draw [style=wire, in=-90, out=180, looseness=1.25] (0) to (2);
		\draw [style=wire, in=0, out=-90, looseness=1.25] (1) to (0);
		\draw [style=wire] (3) to (4);
		\draw [style=wire] (0) to (3);
	\end{pgfonlayer}
\end{tikzpicture}
  }%
\end{array}=  \begin{array}[c]{c}\resizebox{!}{2cm}{%
\begin{tikzpicture}
	\begin{pgfonlayer}{nodelayer}
		\node [style={regular polygon,regular polygon sides=4, draw, inner sep=1pt,minimum size=1pt}] (0) at (0, 4) {$\bigotimes$};
		\node [style=port] (1) at (0.75, 5) {};
		\node [style=port] (2) at (-0.75, 5) {};
		\node [style={circle, draw}] (3) at (0, 3) {$\mathsf{L}$};
		\node [style=port] (4) at (0, 2) {};
	\end{pgfonlayer}
	\begin{pgfonlayer}{edgelayer}
		\draw [style=wire, in=-90, out=180, looseness=1.25] (0) to (2);
		\draw [style=wire, in=0, out=-90, looseness=1.25] (1) to (0);
		\draw [style=wire] (3) to (4);
		\draw [style=wire] (0) to (3);
	\end{pgfonlayer}
\end{tikzpicture}
  }%
\end{array} +  \begin{array}[c]{c}\resizebox{!}{2cm}{%
\begin{tikzpicture}
	\begin{pgfonlayer}{nodelayer}
		\node [style={regular polygon,regular polygon sides=4, draw, inner sep=1pt,minimum size=1pt}] (0) at (0, 4) {$\bigotimes$};
		\node [style=port] (1) at (0.75, 5) {};
		\node [style=port] (2) at (-0.75, 5) {};
		\node [style={regular polygon,regular polygon sides=4, draw}] (3) at (0, 3) {$0$};
		\node [style=port] (4) at (0, 2) {};
	\end{pgfonlayer}
	\begin{pgfonlayer}{edgelayer}
		\draw [style=wire, in=-90, out=180, looseness=1.25] (0) to (2);
		\draw [style=wire, in=0, out=-90, looseness=1.25] (1) to (0);
		\draw [style=wire] (3) to (4);
		\draw [style=wire] (0) to (3);
	\end{pgfonlayer}
\end{tikzpicture}
  }%
\end{array} = \begin{array}[c]{c}\resizebox{!}{2cm}{%
\begin{tikzpicture}
	\begin{pgfonlayer}{nodelayer}
		\node [style={regular polygon,regular polygon sides=4, draw, inner sep=1pt,minimum size=1pt}] (0) at (0, 2.25) {$\bigotimes$};
		\node [style=port] (1) at (-0.75, 4.25) {};
		\node [style=port] (2) at (0, 1) {};
		\node [style={circle, draw}] (3) at (0.75, 3.25) {$\mathsf{L}$};
		\node [style=port] (4) at (0.75, 4.25) {};
	\end{pgfonlayer}
	\begin{pgfonlayer}{edgelayer}
		\draw [style=wire, in=180, out=-90, looseness=1.25] (1) to (0);
		\draw [style=wire] (0) to (2);
		\draw [style=wire] (4) to (3);
		\draw [style=wire, in=0, out=-90, looseness=1.25] (3) to (0);
	\end{pgfonlayer}
\end{tikzpicture}
  }%
\end{array}+ \begin{array}[c]{c}\resizebox{!}{2cm}{%
\begin{tikzpicture}
	\begin{pgfonlayer}{nodelayer}
		\node [style={regular polygon,regular polygon sides=4, draw, inner sep=1pt,minimum size=1pt}] (0) at (0, 2.25) {$\bigotimes$};
		\node [style=port] (1) at (-0.75, 4.25) {};
		\node [style=port] (2) at (0, 1) {};
		\node [style={regular polygon,regular polygon sides=4, draw}] (3) at (0.75, 3.25) {$0$};
		\node [style=port] (4) at (0.75, 4.25) {};
	\end{pgfonlayer}
	\begin{pgfonlayer}{edgelayer}
		\draw [style=wire, in=180, out=-90, looseness=1.25] (1) to (0);
		\draw [style=wire] (0) to (2);
		\draw [style=wire] (4) to (3);
		\draw [style=wire, in=0, out=-90, looseness=1.25] (3) to (0);
	\end{pgfonlayer}
\end{tikzpicture}
  }%
\end{array}= \begin{array}[c]{c}\resizebox{!}{2cm}{%
\begin{tikzpicture}
	\begin{pgfonlayer}{nodelayer}
		\node [style={regular polygon,regular polygon sides=4, draw, inner sep=1pt,minimum size=1pt}] (0) at (0, 2.25) {$\bigotimes$};
		\node [style=port] (1) at (-0.75, 4.25) {};
		\node [style=port] (2) at (0, 1) {};
		\node [style={circle, draw}] (3) at (0.75, 3.25) {$\mathsf{K}$};
		\node [style=port] (4) at (0.75, 4.25) {};
	\end{pgfonlayer}
	\begin{pgfonlayer}{edgelayer}
		\draw [style=wire, in=180, out=-90, looseness=1.25] (1) to (0);
		\draw [style=wire] (0) to (2);
		\draw [style=wire] (4) to (3);
		\draw [style=wire, in=0, out=-90, looseness=1.25] (3) to (0);
	\end{pgfonlayer}
\end{tikzpicture}
  }%
\end{array}
\end{align*}
Similarly to obtain $\mathsf{K}$ on the left. 
\end{proof}

\section{Calculus Categories}\label{calccatsec}

One of the main goals of developing integral categories, apart from axiomatizing integration, was to obtain a categorical setting for calculus. Now that we have established both integral categories and differential categories, we will combine them to obtain calculus categories. However, a calculus category is not simply a category which is both an integral category and a differential category: we require integration and differentiation to be compatible with each other. In particular, we would like integration to be anti-differentiation. In the classical calculus setting, the relationship between integration and anti-differentiation is precisely expressed by the two fundamental theorems of calculus. The fundamental theorems explain to what extent integration and differentiation are inverse processes. Therefore, the extra axioms required for calculus categories will be to impose a similar relation between the deriving transformation and integral transformations by axiomatizing the two fundamental theorems. 

\subsection{Second Fundamental Theorem of Calculus}

\begin{definition}\label{FT2DEF} \normalfont A deriving transformation $\mathsf{d}$ and an integral transformation $\mathsf{s}$ for the same coalgebra modality satisfy the \textbf{Second Fundamental Theorem of Calculus} if: $\mathsf{s}\mathsf{d} +\oc 0=1$. 
$$ \begin{array}[c]{c} \resizebox{!}{1.5cm}{%
\begin{tikzpicture}
	\begin{pgfonlayer}{nodelayer}
		\node [style=port] (0) at (0, 0.5) {};
		\node [style=integral] (1) at (0, 2) {{\bf ------}};
		\node [style=port] (2) at (0, 2.5) {};
		\node [style=differential] (3) at (0, 1) {{\bf =\!=\!=\!=}};
	\end{pgfonlayer}
	\begin{pgfonlayer}{edgelayer}
		\draw [style=wire] (2) to (1);
		\draw [style=wire] (0) to (3);
		\draw [style=wire, bend right=60, looseness=1.50] (1) to (3);
		\draw [style=wire, bend left=60, looseness=1.50] (1) to (3);
	\end{pgfonlayer}
\end{tikzpicture}}
   \end{array}+
   \begin{array}[c]{c} \resizebox{!}{2cm}{%
\begin{tikzpicture}
	\begin{pgfonlayer}{nodelayer}
		\node [style=port] (0) at (0, 0) {};
		\node [style=port] (1) at (0, 3) {};
		\node [style={regular polygon,regular polygon sides=4, draw}] (2) at (0, 1.5) {$0$};
	\end{pgfonlayer}
	\begin{pgfonlayer}{edgelayer}
		\draw [style=wire] (1) to (2);
		\draw [style=wire] (2) to (0);
	\end{pgfonlayer}
\end{tikzpicture}}
   \end{array}=
   \begin{array}[c]{c} \resizebox{!}{2cm}{%
\begin{tikzpicture}
	\begin{pgfonlayer}{nodelayer}
		\node [style=port] (0) at (0, 2.5) {};
		\node [style=port] (1) at (0, 0.5) {};
	\end{pgfonlayer}
	\begin{pgfonlayer}{edgelayer}
		\draw [style=wire] (0) to (1);
	\end{pgfonlayer}
\end{tikzpicture}}
   \end{array}$$
Therefore, for every differentiable map $f: \oc A \to B$: $\mathsf{S}[\mathsf{D}[f]]+\oc(0)f=f$.\end{definition}

Two important concepts related to the Second Fundamental Theorem of Calculus are the Compatibility and Taylor conditions:

\begin{definition} \normalfont A deriving transformation $\mathsf{d}$ and an integral transformation $\mathsf{s}$ for the same coalgebra modality are said to be \textbf{compatible} if: $\mathsf{d}\mathsf{s}\mathsf{d}=\mathsf{d}$ 
$$ \begin{array}[c]{c} \resizebox{!}{2cm}{%
\begin{tikzpicture}
	\begin{pgfonlayer}{nodelayer}
		\node [style=port] (0) at (0, 0) {};
		\node [style=integral] (1) at (0, 1.5) {{\bf ------}};
		\node [style=differential] (2) at (0, 0.5) {{\bf =\!=\!=\!=}};
		\node [style=port] (3) at (0.5, 2.75) {};
		\node [style=differential] (4) at (0, 2) {{\bf =\!=\!=\!=}};
		\node [style=port] (5) at (-0.5, 2.75) {};
	\end{pgfonlayer}
	\begin{pgfonlayer}{edgelayer}
		\draw [style=wire] (0) to (2);
		\draw [style=wire, bend right=60, looseness=1.50] (1) to (2);
		\draw [style=wire, bend left=60, looseness=1.50] (1) to (2);
		\draw [style=wire, bend right, looseness=1.00] (4) to (3);
		\draw [style=wire, bend left, looseness=1.00] (4) to (5);
		\draw [style=wire] (4) to (1);
	\end{pgfonlayer}
\end{tikzpicture}}
   \end{array}=
   \begin{array}[c]{c} \resizebox{!}{1.5cm}{%
\begin{tikzpicture}
	\begin{pgfonlayer}{nodelayer}
		\node [style=differential] (0) at (0, 1.75) {{\bf =\!=\!=\!=}};
		\node [style=port] (1) at (0.75, 3) {};
		\node [style=port] (2) at (-0.75, 3) {};
		\node [style=port] (3) at (0, 0.75) {};
	\end{pgfonlayer}
	\begin{pgfonlayer}{edgelayer}
		\draw [style=wire, bend right, looseness=1.00] (0) to (1);
		\draw [style=wire] (3) to (0);
		\draw [style=wire, bend left, looseness=1.00] (0) to (2);
	\end{pgfonlayer}
\end{tikzpicture}}
   \end{array}
   $$
\end{definition}

Simply put, compatibility is a weaker version of the Second Fundamental Theorem of Calculus. In fact, if one differentiates the equation of the Second Fundamental Theorem of Calculus, one obtains the expression for compatibility. We next define the Taylor condition, which unlike compatibility and the fundamental theorems, is a property only of the deriving transformation. 

\begin{definition} \normalfont A deriving transformation $\mathsf{d}$ is said to be \textbf{Taylor} \cite{ehrhard2017introduction} if for every pair of maps $f,g: C \otimes \oc A  \to B$, such that $(1 \otimes \mathsf{d})f=(1 \otimes \mathsf{d})g$, then: 
\[f+(1 \otimes \oc(0))g=g+(1 \otimes \oc(0)) f\]
\end{definition}

The Taylor property is analogous to the statement that if the derivatives of two maps are equal, then those two maps differ by constants. In this case, the constants are the maps evaluated at zero. If one assumes that we can subtract maps, that is, if we have negatives, then the Taylor property is equivalent to the statement that, if the derivative of a map is zero, then it is a constant. The Taylor condition was the extra condition Ehrhard required of his differential transformation to obtain the Second Fundamental Theorem of Calculus for his integral transformation \cite{ehrhard2017introduction}.

Being Compatible and Taylor is equivalent to satisfying the second fundamental theorem of calculus:

\begin{proposition}\label{FTC2CompTaylor} For a deriving transformation $\mathsf{d}$ and an integral transformation $\mathsf{s}$ on the same coalgebra modality, the following are equivalent:
\begin{enumerate}[{\em (i)}]
\item $\mathsf{d}$ and $\mathsf{s}$ satisfy the Second Fundamental Theorem of Calculus;
\item $\mathsf{d}$ and $\mathsf{s}$ are Compatible and $\mathsf{d}$ is Taylor. 
\end{enumerate}
\end{proposition} 
\begin{proof} $(i) \Rightarrow (ii)$: Suppose $\mathsf{d}$ and $\mathsf{s}$ satisfy the Second Fundamental Theorem of Calculus. For Taylor, suppose that $(1 \otimes \mathsf{d})f=(1 \otimes \mathsf{d})g$. Then we have the following equality:
\begin{align*}
f+(1 \otimes \oc 0)g &= (1 \otimes \mathsf{s})(1 \otimes \mathsf{d})f + (1 \otimes \oc 0)f +(1 \otimes \oc 0)g\\
& = (1 \otimes \mathsf{s})(1 \otimes \mathsf{d})g + (1 \otimes \oc 0)f +(1 \otimes \oc 0)g\\
& = g+(1 \otimes \oc 0) f
\end{align*}
For compatibility, by naturality, we have the following equality: 
$$\mathsf{d}=\mathsf{d}\mathsf{s}\mathsf{d}+\mathsf{d}\oc(0)=\mathsf{d}\mathsf{s}\mathsf{d}+(\oc(0) \otimes 0)\mathsf{d}=\mathsf{d}\mathsf{s}\mathsf{d}+0=\mathsf{d}\mathsf{s}\mathsf{d}$$
$(i) \Rightarrow (ii)$: Suppose $\mathsf{d}$ and $\mathsf{s}$ are Compatible and $\mathsf{d}$ is Taylor. Notice by Compatibility we have: $\mathsf{d}\mathsf{s}\mathsf{d}=\mathsf{d}$, and then by Taylor (where $f=\mathsf{s}\mathsf{d}$ and $g=1$) we have the following equality:
$\mathsf{s}\mathsf{d}+\oc(0)=1+\oc(0)\mathsf{s}\mathsf{d}$.
However, using naturality, we have:
$$\mathsf{s}\mathsf{d}+\oc(0)=1+\oc(0)\mathsf{s}\mathsf{d}=1+\mathsf{s}(\oc(0) \otimes 0)\mathsf{d}=1+0=1$$
\end{proof} 

\subsection{First Fundamental Theorem of Calculus and the Poincar\' e Condition}\label{FT1subsec}

Unlike the Second Fundamental Theorem of Calculus, the First Fundamental Theorem of Calculus holds for only certain integrable maps. 
\begin{definition} 
\normalfont Let $\mathsf{d}$ be a deriving transformation and $\mathsf{s}$ be an integral transformation on the same coalgebra modality. A map $f: C \otimes \oc A \otimes A \to B$ satisfies the \textbf{First Fundamental Theorem} (in the last two arguments) if:
\[(1 \otimes (\mathsf{d}_A\mathsf{s}_A))f=f\]
 $$\begin{array}[c]{c}
   \resizebox{!}{2.5cm}{%
\begin{tikzpicture}
	\begin{pgfonlayer}{nodelayer}
		\node [style={circle, draw}] (0) at (-0.75, -0.5) {$f$};
		\node [style=port] (1) at (0.25, 2.5) {};
		\node [style=port] (2) at (-0.75, 2.5) {};
		\node [style=codifferential] (3) at (-0.25, 0.75) {{\bf --------}};
		\node [style=integral] (4) at (-0.25, 1.5) {{\bf =\!=\!=\!=}};
		\node [style=port] (5) at (-1.5, 2.5) {};
		\node [style=port] (6) at (-0.75, -1.25) {};
	\end{pgfonlayer}
	\begin{pgfonlayer}{edgelayer}
		\draw [style=wire, bend right, looseness=1.00] (4) to (1);
		\draw [style=wire, bend left, looseness=1.00] (4) to (2);
		\draw [style=wire] (4) to (3);
		\draw [style=wire, bend right, looseness=1.00] (3) to (0);
		\draw [style=wire, in=30, out=-45, looseness=1.50] (3) to (0);
		\draw [style=wire, in=150, out=-91, looseness=0.75] (5) to (0);
		\draw [style=wire] (0) to (6);
	\end{pgfonlayer}
\end{tikzpicture}}
   \end{array}=  \begin{array}[c]{c}
     \resizebox{!}{2.5cm}{%
\begin{tikzpicture}
	\begin{pgfonlayer}{nodelayer}
		\node [style={circle, draw}] (0) at (-0.75, -0.5) {$f$};
		\node [style=port] (1) at (0, 2.5) {};
		\node [style=port] (2) at (-0.75, 2.5) {};
		\node [style=port] (3) at (-1.5, 2.5) {};
		\node [style=port] (4) at (-0.75, -1.25) {};
	\end{pgfonlayer}
	\begin{pgfonlayer}{edgelayer}
		\draw [style=wire, in=150, out=-91, looseness=0.75] (3) to (0);
		\draw [style=wire] (0) to (4);
		\draw [style=wire] (2) to (0);
		\draw [style=wire, in=30, out=-93, looseness=0.75] (1) to (0);
	\end{pgfonlayer}
\end{tikzpicture}}
   \end{array}$$
\end{definition}

Thus, if $f$ satisfies the First Fundamental theorem, it may be viewed as the differential of a map -- namely the differential of its integral.  Clearly not all maps will satisfy the First Fundamental theorem calculus, but a necessary condition is:

\begin{lemma}\label{Poincarebackwards} For a deriving transformation $\mathsf{d}$ and an integral transformation $\mathsf{s}$ on the same coalgebra modality, if $f: C \otimes \oc A \otimes A \to B$, satisfies the First Fundamental Theorem, then: 
\[(1 \otimes 1 \otimes \sigma)(1 \otimes \mathsf{d} \otimes 1)f=(1 \otimes \mathsf{d} \otimes 1)f\]
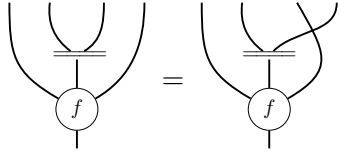
\[\begin{array}[c]{c}\resizebox{!}{2cm}{%
\begin{tikzpicture}
	\begin{pgfonlayer}{nodelayer}
		\node [style=differential] (0) at (-1, 1) {${\bf =\!=\!=\!=}$};
		\node [style=port] (1) at (-1.5, 2) {};
		\node [style=port] (2) at (-0.5, 2) {};
		\node [style=port] (3) at (-1, -0.75) {};
		\node [style={circle, draw}] (4) at (-1, 0) {$f$};
		\node [style=port] (5) at (-2.25, 2) {};
		\node [style=port] (6) at (0.25, 2) {};
	\end{pgfonlayer}
	\begin{pgfonlayer}{edgelayer}
		\draw [style=wire, in=30, out=-90, looseness=1.00] (2) to (0);
		\draw [style=wire, in=90, out=-90, looseness=1.25] (0) to (4);
		\draw [style=wire] (4) to (3);
		\draw [style=wire, in=150, out=-90, looseness=1.25] (5) to (4);
		\draw [style=wire, bend right, looseness=1.00] (1) to (0);
		\draw [style=wire, in=30, out=-90, looseness=1.00] (6) to (4);
	\end{pgfonlayer}
\end{tikzpicture}}
   \end{array} \!=\!   \begin{array}[c]{c}\resizebox{!}{2cm}{%
\begin{tikzpicture}
	\begin{pgfonlayer}{nodelayer}
		\node [style=differential] (0) at (-1, 1) {${\bf =\!=\!=\!=}$};
		\node [style=port] (1) at (-1.5, 2) {};
		\node [style=port] (2) at (-0.5, 2) {};
		\node [style=port] (3) at (-1, -0.75) {};
		\node [style={circle, draw}] (4) at (-1, 0) {$f$};
		\node [style=port] (5) at (-2.25, 2) {};
		\node [style=port] (6) at (0.25, 2) {};
	\end{pgfonlayer}
	\begin{pgfonlayer}{edgelayer}
		\draw [style=wire, in=90, out=-90, looseness=1.25] (0) to (4);
		\draw [style=wire] (4) to (3);
		\draw [style=wire, in=150, out=-90, looseness=1.25] (5) to (4);
		\draw [style=wire, bend right, looseness=1.00] (1) to (0);
		\draw [style=wire, in=45, out=-90, looseness=1.00] (6) to (0);
		\draw [style=wire, in=30, out=-60, looseness=1.50] (2) to (4);
	\end{pgfonlayer}
\end{tikzpicture}}
   \end{array}\]
\end{lemma} 
\begin{proof}
As $(1 \otimes \mathsf{d}\mathsf{s})f=f$, the interchange rule for the deriving transformation {\bf [d.5]} gives: 
\begin{align*}
(1 \otimes 1 \otimes \sigma)(1 \otimes \mathsf{d} \otimes 1)f  &=   (1 \otimes 1 \otimes \sigma)(1 \otimes \mathsf{d} \otimes 1)(1 \otimes (\mathsf{d}\mathsf{s}))f\\
&= (1 \otimes \mathsf{d} \otimes 1)(1 \otimes (\mathsf{d}\mathsf{s}))f\\
&= (1 \otimes \mathsf{d} \otimes 1)f
\end{align*}
\end{proof} 

The converse of this lemma will be used as an axiom for calculus categories called the Poincar\' e condition.

The Poincar\' e Condition states that if an integrable map satisfies a certain pre-condition, then it satisfies the First Fundamental Theorem of Calculus. The name of the condition comes from the Poincar\' e Lemma from cohomology  \cite{weibel1995introduction} and differential topology \cite{bott2013differential}, which states an analoguous result giving criteria for a map to be an antiderivative. 

\begin{definition}\label{Poincaredef} \normalfont A deriving transformation $\mathsf{d}$ and an integral transformation $\mathsf{s}$ for the same coalgebra modality are said to satisfy the \textbf{Poincar\' e condition} if any map 
$f: C \otimes \oc A \otimes A \to B$ for which: $(1 \otimes  1 \otimes \sigma)(1 \otimes \mathsf{d} \otimes 1)f=(1 \otimes \mathsf{d} \otimes 1)f$, 
satisfies the First Fundamental Theorem -- that is: $(1 \otimes \mathsf{d}\mathsf{s})f=f$. 
\end{definition}

The Poincar\' e condition and Lemma \ref{Poincarebackwards} imply the following equivalence:
$$\begin{array}[c]{c}\resizebox{!}{2cm}
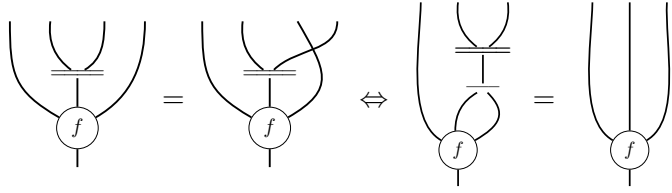
{%
\begin{tikzpicture}
	\begin{pgfonlayer}{nodelayer}
		\node [style=differential] (0) at (-1, 1) {${\bf =\!=\!=\!=}$};
		\node [style=port] (1) at (-1.5, 2) {};
		\node [style=port] (2) at (-0.5, 2) {};
		\node [style=port] (3) at (-1, -0.75) {};
		\node [style={circle, draw}] (4) at (-1, 0) {$f$};
		\node [style=port] (5) at (-2.25, 2) {};
		\node [style=port] (6) at (0.25, 2) {};
	\end{pgfonlayer}
	\begin{pgfonlayer}{edgelayer}
		\draw [style=wire, in=30, out=-90, looseness=1.00] (2) to (0);
		\draw [style=wire, in=90, out=-90, looseness=1.25] (0) to (4);
		\draw [style=wire] (4) to (3);
		\draw [style=wire, in=150, out=-90, looseness=1.25] (5) to (4);
		\draw [style=wire, bend right, looseness=1.00] (1) to (0);
		\draw [style=wire, in=30, out=-90, looseness=1.00] (6) to (4);
	\end{pgfonlayer}
\end{tikzpicture}}
   \end{array} \!=\!   \begin{array}[c]{c}\resizebox{!}{2cm}{%
\begin{tikzpicture}
	\begin{pgfonlayer}{nodelayer}
		\node [style=differential] (0) at (-1, 1) {${\bf =\!=\!=\!=}$};
		\node [style=port] (1) at (-1.5, 2) {};
		\node [style=port] (2) at (-0.5, 2) {};
		\node [style=port] (3) at (-1, -0.75) {};
		\node [style={circle, draw}] (4) at (-1, 0) {$f$};
		\node [style=port] (5) at (-2.25, 2) {};
		\node [style=port] (6) at (0.25, 2) {};
	\end{pgfonlayer}
	\begin{pgfonlayer}{edgelayer}
		\draw [style=wire, in=90, out=-90, looseness=1.25] (0) to (4);
		\draw [style=wire] (4) to (3);
		\draw [style=wire, in=150, out=-90, looseness=1.25] (5) to (4);
		\draw [style=wire, bend right, looseness=1.00] (1) to (0);
		\draw [style=wire, in=45, out=-90, looseness=1.00] (6) to (0);
		\draw [style=wire, in=30, out=-60, looseness=1.50] (2) to (4);
	\end{pgfonlayer}
\end{tikzpicture}}
   \end{array} \Leftrightarrow\!  \begin{array}[c]{c}\resizebox{!}{2.5cm}{%
\begin{tikzpicture}
	\begin{pgfonlayer}{nodelayer}
		\node [style={circle, draw}] (0) at (-0.75, -0.5) {$f$};
		\node [style=port] (1) at (0.25, 2.5) {};
		\node [style=port] (2) at (-0.75, 2.5) {};
		\node [style=codifferential] (3) at (-0.25, 0.75) {{\bf -----}};
		\node [style=integral] (4) at (-0.25, 1.5) {{\bf =\!=\!=\!=}};
		\node [style=port] (5) at (-1.5, 2.5) {};
		\node [style=port] (6) at (-0.75, -1.25) {};
	\end{pgfonlayer}
	\begin{pgfonlayer}{edgelayer}
		\draw [style=wire, bend right, looseness=1.00] (4) to (1);
		\draw [style=wire, bend left, looseness=1.00] (4) to (2);
		\draw [style=wire] (4) to (3);
		\draw [style=wire, bend right, looseness=1.00] (3) to (0);
		\draw [style=wire, in=30, out=-45, looseness=1.50] (3) to (0);
		\draw [style=wire, in=150, out=-91, looseness=0.75] (5) to (0);
		\draw [style=wire] (0) to (6);
	\end{pgfonlayer}
\end{tikzpicture}}
   \end{array}=  \begin{array}[c]{c}\resizebox{!}{2.5cm}{%
\begin{tikzpicture}
	\begin{pgfonlayer}{nodelayer}
		\node [style={circle, draw}] (0) at (-0.75, -0.5) {$f$};
		\node [style=port] (1) at (0, 2.5) {};
		\node [style=port] (2) at (-0.75, 2.5) {};
		\node [style=port] (3) at (-1.5, 2.5) {};
		\node [style=port] (4) at (-0.75, -1.25) {};
	\end{pgfonlayer}
	\begin{pgfonlayer}{edgelayer}
		\draw [style=wire, in=150, out=-91, looseness=0.75] (3) to (0);
		\draw [style=wire] (0) to (4);
		\draw [style=wire] (2) to (0);
		\draw [style=wire, in=30, out=-93, looseness=0.75] (1) to (0);
	\end{pgfonlayer}
\end{tikzpicture}}
   \end{array}$$

The Poincar\' e condition also implies compatibility of the deriving transformation and integral transformation. 

\begin{proposition} A deriving transformation $\mathsf{d}$ and an integral transformation $\mathsf{s}$ which satisfy the Poincar\' e condition are Compatible. 
\end{proposition}
\begin{proof} By {\bf [d.5]}, the deriving transformation $\mathsf{d}$ satisfies the Poincar\' e pre-condition that $(1 \otimes \sigma)(\mathsf{d} \otimes 1)\mathsf{d}=(\mathsf{d} \otimes 1)\mathsf{d}$. Therefore, $\mathsf{d}$ satisfies the First fundamental theorem of Calculus, which is simply the statement of Compatibility: $\mathsf{d}\mathsf{s}\mathsf{d}=\mathsf{d}$. 
\end{proof} 

\begin{corollary}\label{PoincareFTC2} A deriving transformation $\mathsf{d}$ and an integral transformation $\mathsf{s}$ which satisfy the Poincar\' e condition such that $\mathsf{d}$ is Taylor, satisfies the Second Fundamental Theorem of Calculus. 
\end{corollary}
\begin{proof} The Poincar\' e condition implies Compatibility. The Second Fundamental Theorem of Calculus follows since it is equivalent to Compatibility and Taylor by Proposition \ref{FTC2CompTaylor}. 
\end{proof} 

\subsection{Calculus Categories and Calculus Objects}

We now give the definition of a calculus category. 

\begin{definition} \normalfont A \textbf{calculus category} is a differential category and an integral category on the same coalgebra modality such that the deriving transformation and the integral transformation satisfy the Second Fundamental Theorem of Calculus and the Poincar\' e condition. 
\end{definition}

It is important to note that by Corollary \ref{PoincareFTC2}, an equivalent definition of a calculus category would simply require the Poincar\' e Condition and that the deriving transformation is Taylor. In the next section, we will show that a differential category for which $\mathsf{K}$ is an isomorphism is a calculus category. 

We now turn our attention to calculus objects: these are objects where the identity of $\oc A \otimes A$ satisfies the First Fundamental Theorem of Calculus, which is another way of saying that the deriving transformation and integral transformation satisfies it. 

\begin{definition} \normalfont In a calculus category, a \textbf{calculus object} is an object $A$ such that the deriving transformation $\mathsf{d}_A$ and the integral transformation $\mathsf{s}_A$ satisfy the \textbf{First Fundamental Theorem of Calculus}, $\mathsf{d}_A\mathsf{s}_A=1_{\oc A \otimes A}$. 
\end{definition}

This ensures that any integrable map with domain $C \otimes \oc A \otimes A$, where $A$ is a calculus object, will satisfy the First Fundamental Theorem of Calculus. Every calculus category has at least one calculus object: the monoidal unit! 

\begin{proposition}\label{Poincarecalcobj} In a calculus category, if $A$ is an object such that $\sigma_{A,A}=1_{A \otimes A}$, then $A$ is a calculus object. Conversely, if $A$ is a calculus object and $\varepsilon_A$ is a retraction then $\sigma_{A,A}=1_{A \otimes A}$. 
\end{proposition} 
\begin{proof} Suppose that for an object $A$, $\sigma_{A,A}=1_{A \otimes A}$. Then the Poincar\' e pre-condition is true trivially for the identity $1_{\oc A \otimes A}$. Therefore, $\mathsf{d}_A\mathsf{s}_A=1_{\oc A \otimes A}$. Conversly, suppose $A$ is a calculus object and let $\eta_A$ be a section of $\varepsilon_A$. To show that $\sigma_{A,A}=1_{A \otimes A}$ we will need to use the case $n=1$ of Proposition \ref{swappoly} (recall that $\omega_{(0 \ 1)}=\sigma$):
\begin{align*}
1_{A \otimes A}&= (\eta_A \otimes 1_A)(\varepsilon_A \otimes 1_A) \\
&= (\eta_A \otimes 1_A)\mathsf{d}_A\mathsf{s}_A(\varepsilon_A \otimes 1_A)  \\
&= (\eta_A \otimes 1_A)\mathsf{d}_A\mathsf{s}_A(\varepsilon_A \otimes 1_A)\sigma_{A,A}  \\
&= (\eta_A \otimes 1_A)(\varepsilon_A \otimes 1_A)\sigma_{A,A}\\
&= \sigma_{A,A}
\end{align*}
\end{proof} 

\begin{corollary}\label{Kcalcobj} The monoidal unit, $K$, of a calculus category is always a calculus object. 
\end{corollary}
\begin{proof} By coherence of symmetric monoidal categories, $\sigma_{K,K}=1_{K \otimes K}$. Therefore, $K$ is a calculus object. 
\end{proof} 

\section{Antiderivatives}\label{antidiffsec}

In this section we explore what it means for a differential category to have antiderivatives.  We shall equate this with the requirement  that $\mathsf{K}$ be a natural isomorphism.  Recall, however, that Ehrhard's original idea was to obtain antiderivatives by requiring that $\mathsf{J}$ be a natural isomorphism \cite{ehrhard2017introduction}. Requiring that just $\mathsf{J}$ is invertible does not imply that $\mathsf{K}$ is invertible, and, more importantly, it does not imply the second fundamental theorem of calculus.  This means that the antiderivative built from the inverse of $\mathsf{J}$  does not deliver a calculus category. This is why asking that $\mathsf{K}$ be a natural isomorphism is important.  Proposition \ref{K-1J-1equiv} shows that $\mathsf{J}$ being invertible {\em and\/} the Taylor Property holding implies $\mathsf{K}$ is invertible.  In Theorem \ref{anticalc} we show that $\mathsf{K}$ being invertible implies that one has 
a calculus category (with respect to the antiderivative) and so the Taylor property holds.  Thus, as advertised, assuming $\mathsf{K}$ is invertible removes the necessity for the non-equational Taylor requirement.

As shown by Example \ref{JnotKexample} in Section \ref{examplesec},  the invertibility of $\mathsf{J}$ may not deliver a calculus category.  However, there is no reason why it should not still deliver an integral category and, in fact, we shall show that the invertibility of $\mathsf{J}$ {\em does\/} deliver an integral category.  Furthermore the Poincar\' e Condition will hold in for this integral, thus, the only defect is that this integration may not satisfy the second fundamental theorem of calculus. 

\subsection{$\mathsf{K}^{-1}$ and $\mathsf{J}_n^{-1}$}\label{antisubsec}

\begin{proposition}\label{K-1J-1equiv}In a differential category:
\begin{enumerate}[{\em (i)}]
\item $\mathsf{K}$ is a natural isomorphism implies $\mathsf{J}$ is a natural isomorphism;
\item $\mathsf{J}_n$ is a natural isomorphism implies $\mathsf{J}_m$ is a natural isomorphism for $m \geq  n$;
\item If $\mathsf{J}$ is a natural isomorphism and the Taylor property holds then $\mathsf{K}$ is a natural isomorphism.
\end{enumerate}
\end{proposition} 

Before proving Proposition \ref{K-1J-1equiv} it is useful to have the following observation in hand:

\begin{lemma} \label{L-lemma} In any differential category, the following equality holds: 
$$\mathsf{d}^\circ (1 \otimes \mathsf{L}) \mathsf{d}\mathsf{d}^\circ (\varepsilon \otimes e) = \mathsf{d}^\circ (\varepsilon \otimes e) \mathsf{L}.$$
\end{lemma}

\begin{proof} Here we use Proposition \ref{Wprop}, {\bf [L.2]}, naturality of $\mathsf{d}$ and $\mathsf{d}^\circ$, and {\bf [cd.6]}: 
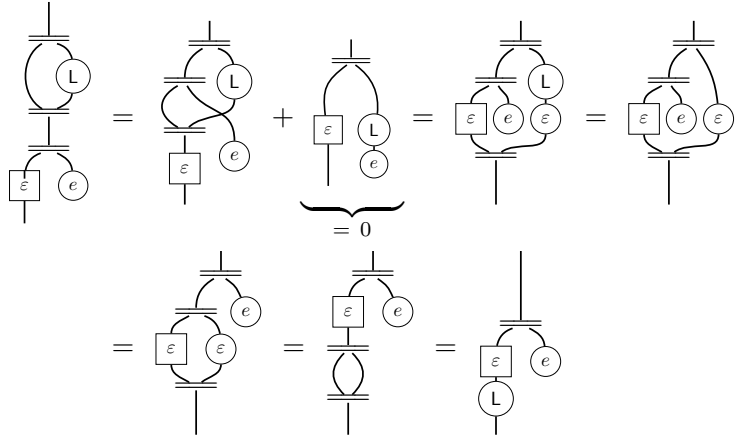
\begin{align*}

\end{align*}
\end{proof}

\begin{proof}[of Proposition \ref{K-1J-1equiv}] \\

\noindent
$(i)$ Suppose that $\mathsf{K}$ is a natural isomorphism. Define $\mathsf{J}^{-1}$ to be: 
 \[%
 \]
We need to prove that $\mathsf{J}\mathsf{J}^{-1}=1$, $\mathsf{J}^{-1}\mathsf{J}=1$. We begin by proving that $\mathsf{J}^{-1}\mathsf{J}=1$. Here we use naturality of $\mathsf{J}$, \textbf{[J.7]}, and {\bf[cd.7]} and that the differentials of $\mathsf{L}$ and $\mathsf{K}$ are the same: 
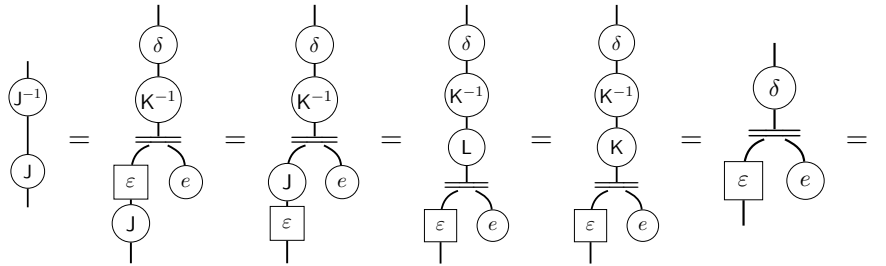
\begin{align*} %
\end{align*}
Next we show that $\mathsf{J}\mathsf{J}^{-1}=\mathsf{J}^{-1}\mathsf{J}$ which, by the above. means $\mathsf{J}\mathsf{J}^{-1}=1$. Here we use \textbf{[J.5]}, \textbf{[K.8]}, and Lemma  \ref{L-lemma}. 
\begin{align*}
&   %
 
\end{align*} \\

\noindent
 $(ii)$ We prove that if $\mathsf{J}_n$ is invertible that $\mathsf{J}_{n+1}$ is invertible.  The proof follows the form of $(i)$: we show that 
 $$\mathsf{J}^{-1}_{n+1} := \delta \mathsf{J}_n \mathsf{d}^\circ (\varepsilon \otimes e)$$
 by first showing, with this definition, that $\mathsf{J}_{n+1}^{-1} \mathsf{J}_{n+1} = 1$ and then showing that $\mathsf{J}_{n+1} \mathsf{J}_{n+1}^{-1} = \mathsf{J}_{n+1}^{-1} \mathsf{J}_{n+1}$. First, by naturality of $\mathsf{J}_{n}$, \textbf{[J.7]}, and \textbf{[cd.7]} we have:
 \begin{align*} %
\end{align*}
Next, we show $\mathsf{J}_{n+1} \mathsf{J}_{n+1}^{-1} = \mathsf{J}_{n+1}^{-1} \mathsf{J}_{n+1}$. Here we use \textbf{[J.5]}, \textbf{[J.8]}, and Lemma  \ref{L-lemma}. 
 \begin{align*}
&   %
 
\end{align*} \\

 \noindent
 $(iii)$ Suppose that $\mathsf{J}$ is a natural isomorphism and the deriving transformation is Taylor. Let $\mathsf{F} := \mathsf{J}^{-1} \otimes 1: \oc A \otimes A \to \oc A \otimes A$ be the inverse of $\mathsf{J} \otimes 1$. Define $\mathsf{K}^{-1}$ as $\mathsf{K}^{-1}:=\mathsf{d}^\circ\mathsf{F}\mathsf{F}\mathsf{d}+\oc 0$. We need to prove that $\mathsf{K}\mathsf{K}^{-1}=1$ and $\mathsf{K}^{-1}\mathsf{K}=1$. First we have the following equality by \textbf{[K.1]}, \textbf{[J.6]} and \textbf{[J.7]}:
 \begin{align*}
\mathsf{K}\mathsf{K}^{-1} &= \mathsf{K}\mathsf{d}^\circ\mathsf{F}\mathsf{F}\mathsf{d}+\mathsf{K}\oc 0 \\
&= \mathsf{d}^\circ(\mathsf{J} \otimes 1)\mathsf{F}\mathsf{F}\mathsf{d}+\oc 0 \\
&= \mathsf{d}^\circ\mathsf{F}\mathsf{d}+\oc 0\\
&=\mathsf{d}^\circ\mathsf{F}\mathsf{F}(\mathsf{J} \otimes 1)\mathsf{d}+\oc 0 \\
&=\mathsf{d}^\circ\mathsf{F}\mathsf{F}\mathsf{d}\mathsf{K}+\oc 0\mathsf{K} \\
&= \mathsf{K}^{-1}\mathsf{K}
\end{align*}
 Therefore, $\mathsf{K}\mathsf{K}^{-1}=\mathsf{K}^{-1}\mathsf{K}=\mathsf{d}^\circ\mathsf{F}\mathsf{d}+\oc 0$. Consider now the derivative of $\mathsf{d}^\circ\mathsf{F}\mathsf{d}$. By \textbf{[J.6]} and \textbf{[J.13]} we have that: 
 \begin{align*}
\mathsf{d}\mathsf{d}^\circ\mathsf{F}\mathsf{d} &= \mathsf{F}(\mathsf{J} \otimes 1)\mathsf{d}\mathsf{d}^\circ\mathsf{F}\mathsf{d} \\
&= \mathsf{F}\mathsf{d}\mathsf{d}^\circ(\mathsf{J} \otimes 1)\mathsf{F}\mathsf{d} \\
&=  \mathsf{F}\mathsf{d}\mathsf{d}^\circ\mathsf{d} \\
&= \mathsf{F}\mathsf{d}\mathsf{L} \\
&= \mathsf{F}(\mathsf{J} \otimes 1)\mathsf{d}\\
&= \mathsf{d}
\end{align*}
Since the deriving transformation is Taylor, we then have: 
\[\mathsf{d}^\circ\mathsf{F}\mathsf{d}+\oc 0=1+\oc 0 \mathsf{d}^\circ\mathsf{F}\mathsf{d}\]
However by naturality we have that $\oc 0 \mathsf{d}^\circ\mathsf{F}\mathsf{d}=\mathsf{d}^\circ( \oc 0 \otimes 0)\mathsf{F}\mathsf{d}=0$. Finally, we conclude:
\[\mathsf{K}\mathsf{K}^{-1}=\mathsf{K}^{-1}\mathsf{K}=\mathsf{d}^\circ(\mathsf{J}^{-1} \otimes 1)\mathsf{d}+\oc 0 = 1\]
\end{proof} 

 \begin{definition} \normalfont A differential category is said to have \textbf{antiderivatives} if $\mathsf{K}$ is a natural isomorphism. \end{definition}
 
By Proposition \ref{K-1J-1equiv}, an equivalent definition of having antiderivatives is if $\mathsf{J}$ is a natural isomorphism and the deriving transformation is Taylor. We now give properties of $\mathsf{K}^{-1}$. 

\begin{proposition}\label{K-1prop} $\mathsf{K}^{-1}$ satisfies the following properties:
\begin{enumerate}[{\bf [K$^{-1}$.1]}]
\item $\mathsf{K}^{-1}\oc(0)=\oc(0)=\oc(0)\mathsf{K}^{-1}$; 
\item $\mathsf{K}^{-1}e=e$; 
\item $\mathsf{K}^{-1}\varepsilon= \varepsilon$; 
\item $\Delta(\mathsf{K}^{-1} \otimes \mathsf{K}^{-1})+\Delta(\mathsf{K}^{-1} \otimes \oc(0))+\Delta(\oc(0) \otimes \mathsf{K}^{-1})=\mathsf{K}^{-1}\Delta(\mathsf{K}^{-1} \otimes 1)+\mathsf{K}^{-1}\Delta(1 \otimes \mathsf{K}^{-1})+\Delta(\oc(0) \otimes \oc(0))$
\item $(\mathsf{K}^{-1} \otimes 1)\mathsf{W}=\mathsf{W}(\mathsf{K}^{-1} \otimes 1)$;
\item $(\mathsf{K}^{-1} \otimes 1)\mathsf{d}\mathsf{d}^\circ=\mathsf{d}\mathsf{d}^\circ(\mathsf{K}^{-1} \otimes 1)$;
\item $(\mathsf{d} \otimes 1)(\mathsf{K}^{-1} \otimes 1)\mathsf{d}= (1 \otimes \sigma)(\mathsf{d} \otimes 1)(\mathsf{K}^{-1} \otimes 1)\mathsf{d}$; 
\item $\mathsf{d}^\circ(\mathsf{K}^{-1} \otimes 1)(\mathsf{d}^\circ \otimes 1)= \mathsf{d}^\circ (\mathsf{K}^{-1} \otimes 1)(\mathsf{d}^\circ \otimes 1)(1 \otimes \sigma)$. 
\end{enumerate}
\end{proposition} 

\begin{proof} These are mostly straightforward calculations by using the properties of $\mathsf{K}$ and that $\mathsf{K}$ is an isomorphism: \\ \\
{\bf [K$^{-1}$.1]}: Here we use {\bf [K.1]}:
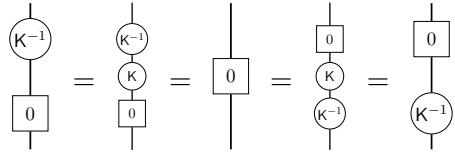
\begin{align*}

\end{align*}
{\bf [K$^{-1}$.2]}: Here we use {\bf [K.2]}: 
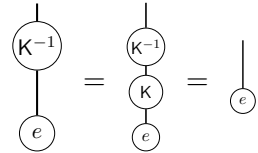
\begin{align*}
%
\end{align*}
{\bf [K$^{-1}$.3]}: Here we use {\bf [K.3]}:
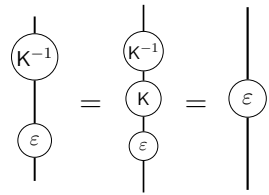
\begin{align*}
%
\end{align*}
{\bf [K$^{-1}$.4]}: Here we use {\bf [K.4]}, {\bf [K$^{-1}$.1]} and naturality of $\Delta$ to first obtain the following equality:
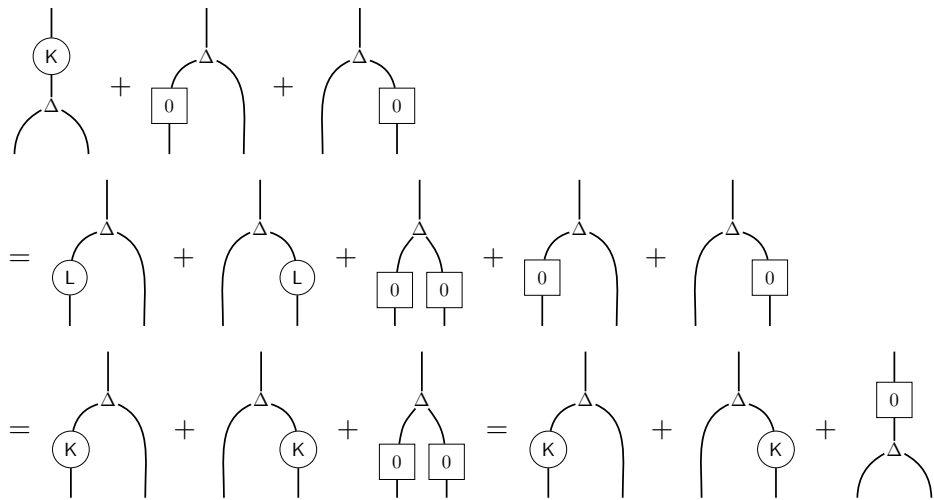
\begin{align*}
&%
\end{align*}
Then we obtain the following equality: 
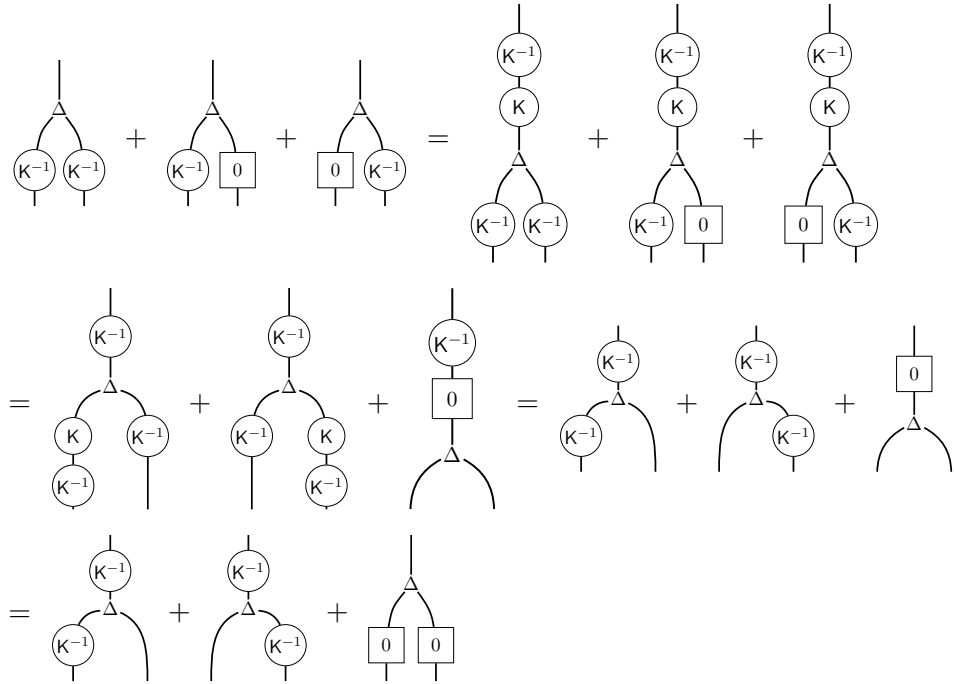
\begin{align*}&%
\end{align*}
{\bf [K$^{-1}$.5]}: Here we use {\bf [K.10]}: 
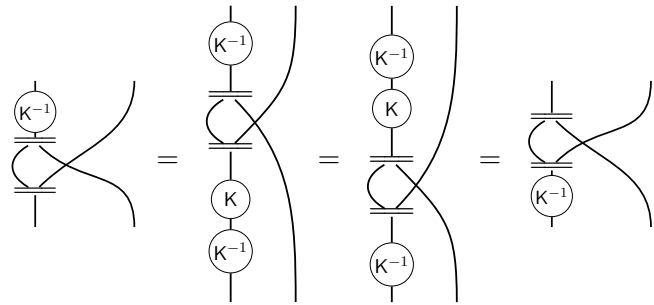
\begin{align*}
  %
\end{align*}
{\bf [K$^{-1}$.6]}: Here we use {\bf [K$^{-1}$.5]}: 
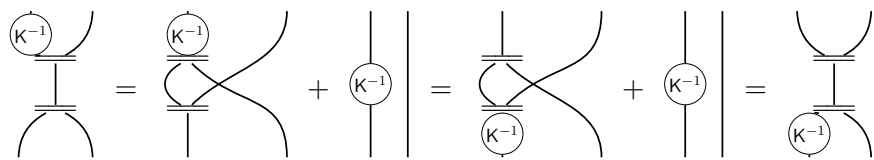
\begin{align*}
 %
\end{align*}
{\bf [K$^{-1}$.7]}: Here we use {\bf [K.6]}, {\bf [K$^{-1}$.5]} and {\bf [d.5]}:
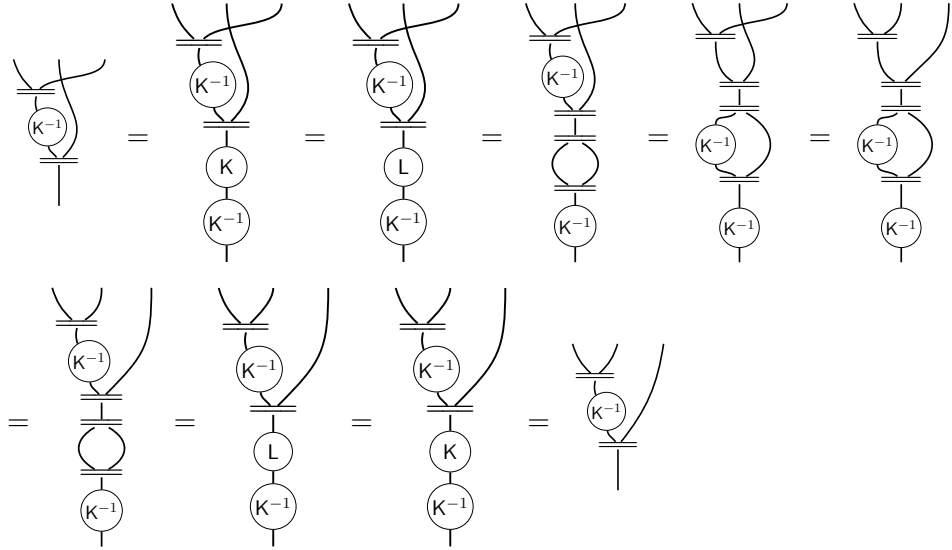
\begin{align*}
 &  %
\end{align*}
{\bf [K$^{-1}$.8]}: Here we use {\bf [K.7]}, {\bf [K$^{-1}$.6]} and {\bf [cd.7]}:
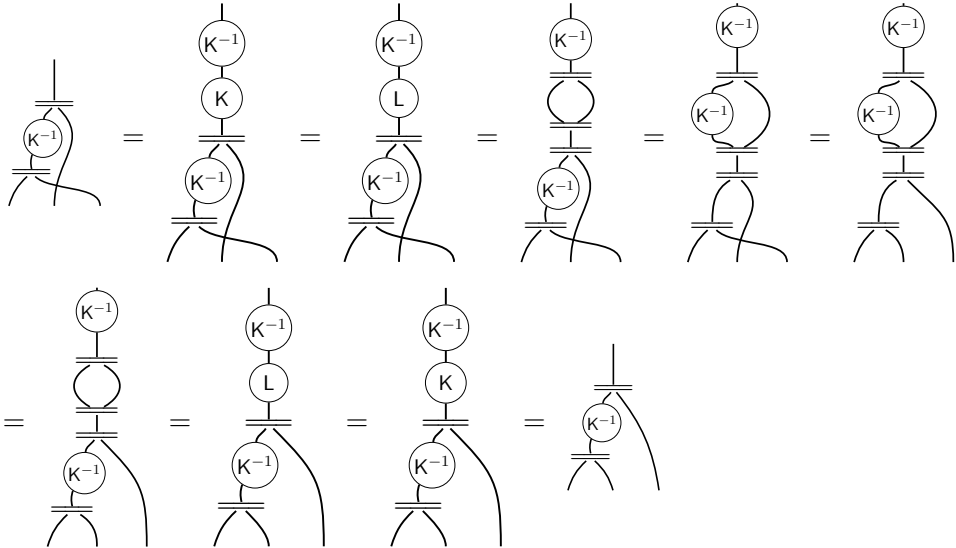
\begin{align*}
 &  %
\end{align*}
\end{proof} 

We now give properties of $\mathsf{J}^{-1}$. Notice that the first seven points below hold without the assumption that $\mathsf{K}$ is also an isomorphism. However {\bf [J$^{-1}$.8]} and {\bf [J$^{-1}$.9]} require $\mathsf{K}^{-1}$. In particular, the {\bf [J$^{-1}$.9]} will be important for the following section when we discuss our constructed integral transformation for differential category with antiderivatives. 

\begin{proposition}\label{J-1prop} $\mathsf{J}_n^{-1}$ satisfies the following properties:
\begin{enumerate}[{\bf [J$^{-1}$.1]}]
\item $(n+1) \cdot \mathsf{J}_n^{-1}\oc(0)=\oc(0)=(n+1) \cdot \oc(0)\mathsf{J}^{-1}$; 
\item $(n+1) \cdot \mathsf{J}_n^{-1}e=e$;
\item $2 \cdot \mathsf{J}^{-1}\varepsilon= \varepsilon$;
\item $\mathsf{d}(\mathsf{J}_n)^{-1} = ((\mathsf{J}_{n+1})^{-1} \otimes 1)\mathsf{d}$; 
\item $(\mathsf{J}_n)^{-1}\mathsf{d}^\circ = \mathsf{d}^\circ ((\mathsf{J}_{n+1})^{-1} \otimes 1)$;
\item $(\mathsf{J}_n^{-1} \otimes 1)\mathsf{W}=\mathsf{W}(\mathsf{J}_n^{-1} \otimes 1)$; 
\item $(\mathsf{J}_n^{-1} \otimes 1)\mathsf{d}\mathsf{d}^\circ=\mathsf{d}\mathsf{d}^\circ(\mathsf{J}_n^{-1} \otimes 1)$;
\item $(\mathsf{J}^{-1} \otimes 1)\mathsf{d}=\mathsf{d}\mathsf{K}^{-1}$;
\item $\mathsf{d}^\circ(\mathsf{J}^{-1} \otimes 1)=\mathsf{K}^{-1}\mathsf{d}^\circ$;
\item $(\mathsf{d} \otimes 1)(\mathsf{J}_n^{-1} \otimes 1)\mathsf{d}= (1 \otimes \sigma)(\mathsf{d} \otimes 1)(\mathsf{J}_n^{-1} \otimes 1)\mathsf{d}$; 
\item $\mathsf{d}^\circ(\mathsf{J}_n^{-1} \otimes 1)(\mathsf{d}^\circ \otimes 1)= \mathsf{d}^\circ (\mathsf{J}_n^{-1} \otimes 1)(\mathsf{d}^\circ \otimes 1)(1 \otimes \sigma)$. 
\end{enumerate}
\end{proposition}

\begin{proof} These are mostly straightforward calculations by using the properties of $\mathsf{J}$: \\ \\
{\bf [J$^{-1}$.1]}: Here we use the property {\bf [J.1]}:
\begin{align*} (n+1) \cdot 

\end{align*}
{\bf [J$^{-1}$.2]}: Here we use {\bf [J.2]}:
\begin{align*} (n+1) \cdot
%
\end{align*}
{\bf [J$^{-1}$.3]}: Here we use {\bf [J.3]}:
\begin{align*}(n+2)\cdot
%
\end{align*}
{\bf [J$^{-1}$.4]} Here we use {\bf [J.6]}: 
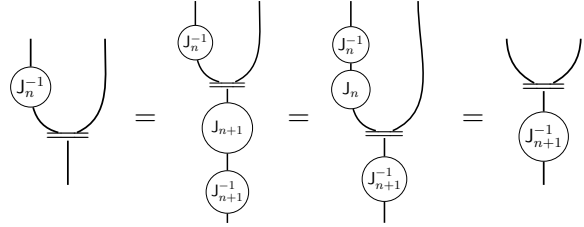
\begin{align*}%
     \end{align*}
{\bf [J$^{-1}$.5]} Here we use {\bf [J.7]}: 
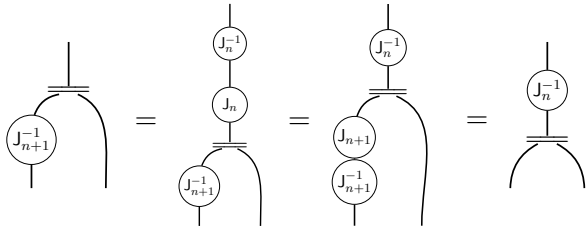
\begin{align*}%
\end{align*}
{\bf [J$^{-1}$.6]}: Here we use {\bf [J.12]}:
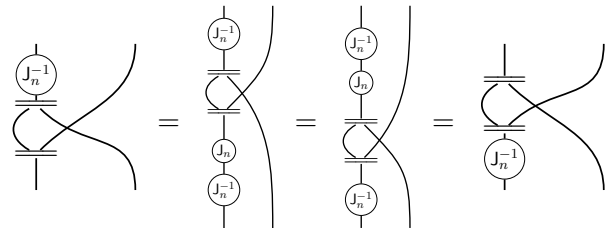
\begin{align*}
  %
\end{align*}
{\bf [J$^{-1}$.7]}: Here we use {\bf [J$^{-1}$.6]}: 
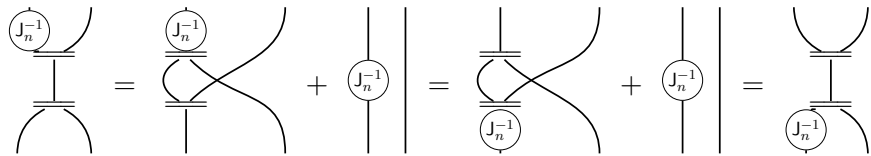
\begin{align*}
 %
\end{align*}
{\bf [J$^{-1}$.8]}: Here we use {\bf [L.6]} and that, as $\mathsf{K} = \mathsf{L} + \oc(0)$, the differential of $\mathsf{K}$ and $\mathsf{L}$ are equal: 
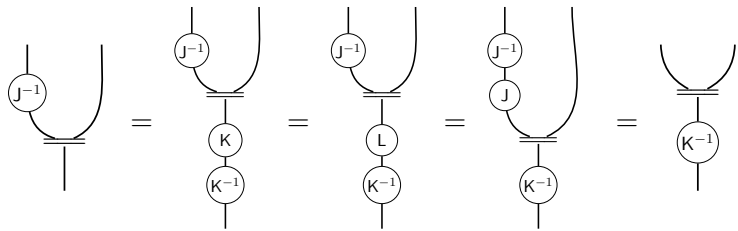
\begin{align*}
%
\end{align*}
{\bf [J$^{-1}$.9]}: Here we use {\bf [J.7]}: 
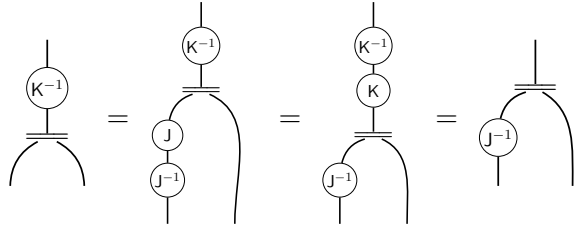
\begin{align*}%
\end{align*}
{\bf [J$^{-1}$.10]}: Here we use {\bf [J$^{-1}$.4]} and {\bf [d.5]}: 
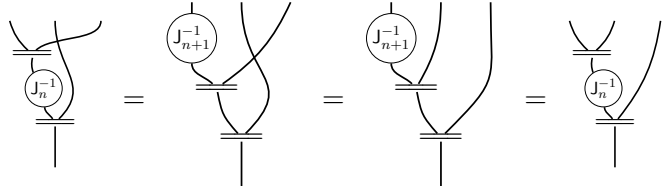
\begin{align*} %
\end{align*}
{\bf [J$^{-1}$.11]}: Here we use {\bf [J$^{-1}$.5]} and {\bf [cd.6]}: 
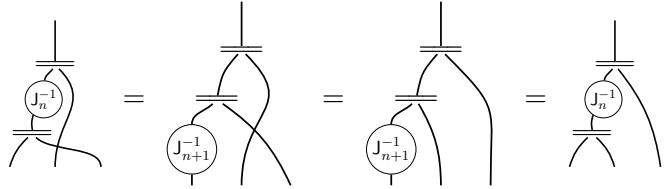
\begin{align*}
 %
 
\end{align*}
\end{proof}

\begin{proposition} In a differential category with antiderivatives and a monoidal coalgebra modality, the following equalities hold: 
\end{proposition} 
\begin{description}
\item[\textbf{[K$^{-1}$.m]}] $m_\otimes \mathsf{K}^{-1}= (\mathsf{K}^{-1} \otimes 1)m_\otimes= (1 \otimes \mathsf{K}^{-1})m_\otimes$;
\item[\textbf{[J$^{-1}$.m]}] $m_\otimes \mathsf{J}^{-1}= (\mathsf{J}^{-1} \otimes 1)m_\otimes= (1 \otimes \mathsf{J}^{-1})m_\otimes$.
\end{description}
\begin{proof} Straightforward using the monoidal rules for $\mathsf{K}$ and $\mathsf{J}$: \\Ê\\
\textbf{[K$^{-1}$.m]}: Here we use \textbf{[K.m]}:
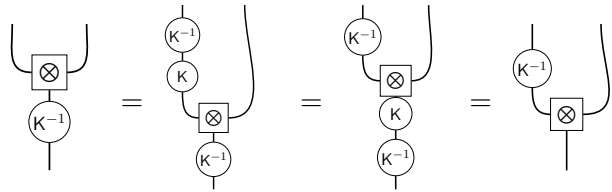
\begin{align*}
\begin{array}[c]{c}\resizebox{!}{2cm}{%
\begin{tikzpicture}
	\begin{pgfonlayer}{nodelayer}
		\node [style={regular polygon,regular polygon sides=4, draw, inner sep=1pt,minimum size=1pt}] (0) at (0, 4) {$\bigotimes$};
		\node [style=port] (1) at (0.75, 5) {};
		\node [style=port] (2) at (-0.75, 5) {};
		\node [style={circle, draw, inner sep=1pt,minimum size=1pt}] (3) at (0, 3) {$\mathsf{K}^{-1}$};
		\node [style=port] (4) at (0, 2) {};
	\end{pgfonlayer}
	\begin{pgfonlayer}{edgelayer}
		\draw [style=wire, in=-90, out=180, looseness=1.25] (0) to (2);
		\draw [style=wire, in=0, out=-90, looseness=1.25] (1) to (0);
		\draw [style=wire] (3) to (4);
		\draw [style=wire] (0) to (3);
	\end{pgfonlayer}
\end{tikzpicture}
  }%
\end{array} =
   \begin{array}[c]{c}\resizebox{!}{2.5cm}{%
\begin{tikzpicture}
	\begin{pgfonlayer}{nodelayer}
		\node [style=port] (0) at (1.25, 0.5) {};
		\node [style={circle, draw, inner sep=1pt,minimum size=1pt}] (1) at (0.5, 4.25) {$\mathsf{K}^{-1}$};
		\node [style={regular polygon,regular polygon sides=4, draw, inner sep=1pt,minimum size=1pt}] (2) at (1.25, 2.25) {$\bigotimes$};
		\node [style={circle, draw}] (3) at (0.5, 3.25) {$\mathsf{K}$};
		\node [style=port] (4) at (2, 5) {};
		\node [style=port] (5) at (0.5, 5) {};
		\node [style={circle, draw, inner sep=1pt,minimum size=1pt}] (6) at (1.25, 1.25) {$\mathsf{K}^{-1}$};
	\end{pgfonlayer}
	\begin{pgfonlayer}{edgelayer}
		\draw [style=wire, in=0, out=-90, looseness=1.00] (4) to (2);
		\draw [style=wire] (1) to (3);
		\draw [style=wire, in=180, out=-90, looseness=1.25] (3) to (2);
		\draw [style=wire] (5) to (1);
		\draw [style=wire] (2) to (6);
		\draw [style=wire] (6) to (0);
	\end{pgfonlayer}
\end{tikzpicture}
  }%
\end{array}=    \begin{array}[c]{c}\resizebox{!}{2.5cm}{%
\begin{tikzpicture}
	\begin{pgfonlayer}{nodelayer}
		\node [style=port] (0) at (1.25, -0.25) {};
		\node [style={regular polygon,regular polygon sides=4, draw, inner sep=1pt,minimum size=1pt}] (1) at (1.25, 2.25) {$\bigotimes$};
		\node [style={circle, draw, inner sep=1pt,minimum size=1pt}] (2) at (0.5, 3.25) {$\mathsf{K}^{-1}$};
		\node [style=port] (3) at (2, 4) {};
		\node [style=port] (4) at (0.5, 4) {};
		\node [style={circle, draw, inner sep=1pt,minimum size=1pt}] (5) at (1.25, 0.5) {$\mathsf{K}^{-1}$};
		\node [style={circle, draw}] (6) at (1.25, 1.5) {$\mathsf{K}$};
	\end{pgfonlayer}
	\begin{pgfonlayer}{edgelayer}
		\draw [style=wire, in=0, out=-90, looseness=1.00] (3) to (1);
		\draw [style=wire, in=180, out=-90, looseness=1.25] (2) to (1);
		\draw [style=wire] (5) to (0);
		\draw [style=wire] (4) to (2);
		\draw [style=wire] (1) to (6);
		\draw [style=wire] (6) to (5);
	\end{pgfonlayer}
\end{tikzpicture}
  }%
\end{array}=  \begin{array}[c]{c}\resizebox{!}{2cm}{%
\begin{tikzpicture}
	\begin{pgfonlayer}{nodelayer}
		\node [style={regular polygon,regular polygon sides=4, draw, inner sep=1pt,minimum size=1pt}] (0) at (0, 2.25) {$\bigotimes$};
		\node [style=port] (1) at (0.75, 4.25) {};
		\node [style=port] (2) at (0, 1) {};
		\node [style={circle, draw, inner sep=1pt,minimum size=1pt}] (3) at (-0.75, 3.25) {$\mathsf{K}^{-1}$};
		\node [style=port] (4) at (-0.75, 4.25) {};
	\end{pgfonlayer}
	\begin{pgfonlayer}{edgelayer}
		\draw [style=wire, in=0, out=-90, looseness=1.25] (1) to (0);
		\draw [style=wire] (0) to (2);
		\draw [style=wire] (4) to (3);
		\draw [style=wire, in=180, out=-90, looseness=1.25] (3) to (0);
	\end{pgfonlayer}
\end{tikzpicture}
  }%
\end{array}
\end{align*}
Similarly to obtain $\mathsf{K}^{-1}$ on the right. \\ \\
\textbf{[J$^{-1}$.m]}: Here we use \textbf{[J.m]}:
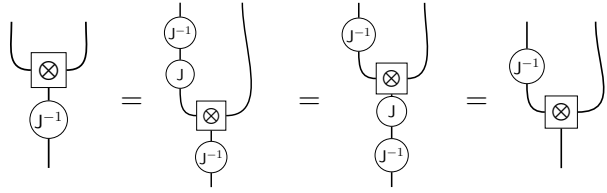
\begin{align*}
\begin{array}[c]{c}\resizebox{!}{2cm}{%
\begin{tikzpicture}
	\begin{pgfonlayer}{nodelayer}
		\node [style={regular polygon,regular polygon sides=4, draw, inner sep=1pt,minimum size=1pt}] (0) at (0, 4) {$\bigotimes$};
		\node [style=port] (1) at (0.75, 5) {};
		\node [style=port] (2) at (-0.75, 5) {};
		\node [style={circle, draw, inner sep=1pt,minimum size=1pt}] (3) at (0, 3) {$\mathsf{J}^{-1}$};
		\node [style=port] (4) at (0, 2) {};
	\end{pgfonlayer}
	\begin{pgfonlayer}{edgelayer}
		\draw [style=wire, in=-90, out=180, looseness=1.25] (0) to (2);
		\draw [style=wire, in=0, out=-90, looseness=1.25] (1) to (0);
		\draw [style=wire] (3) to (4);
		\draw [style=wire] (0) to (3);
	\end{pgfonlayer}
\end{tikzpicture}
  }%
\end{array} =
   \begin{array}[c]{c}\resizebox{!}{2.5cm}{%
\begin{tikzpicture}
	\begin{pgfonlayer}{nodelayer}
		\node [style=port] (0) at (1.25, 0.5) {};
		\node [style={circle, draw, inner sep=1pt,minimum size=1pt}] (1) at (0.5, 4.25) {$\mathsf{J}^{-1}$};
		\node [style={regular polygon,regular polygon sides=4, draw, inner sep=1pt,minimum size=1pt}] (2) at (1.25, 2.25) {$\bigotimes$};
		\node [style={circle, draw}] (3) at (0.5, 3.25) {$\mathsf{J}$};
		\node [style=port] (4) at (2, 5) {};
		\node [style=port] (5) at (0.5, 5) {};
		\node [style={circle, draw, inner sep=1pt,minimum size=1pt}] (6) at (1.25, 1.25) {$\mathsf{J}^{-1}$};
	\end{pgfonlayer}
	\begin{pgfonlayer}{edgelayer}
		\draw [style=wire, in=0, out=-90, looseness=1.00] (4) to (2);
		\draw [style=wire] (1) to (3);
		\draw [style=wire, in=180, out=-90, looseness=1.25] (3) to (2);
		\draw [style=wire] (5) to (1);
		\draw [style=wire] (2) to (6);
		\draw [style=wire] (6) to (0);
	\end{pgfonlayer}
\end{tikzpicture}
  }%
\end{array}=    \begin{array}[c]{c}\resizebox{!}{2.5cm}{%
\begin{tikzpicture}
	\begin{pgfonlayer}{nodelayer}
		\node [style=port] (0) at (1.25, -0.25) {};
		\node [style={regular polygon,regular polygon sides=4, draw, inner sep=1pt,minimum size=1pt}] (1) at (1.25, 2.25) {$\bigotimes$};
		\node [style={circle, draw, inner sep=1pt,minimum size=1pt}] (2) at (0.5, 3.25) {$\mathsf{J}^{-1}$};
		\node [style=port] (3) at (2, 4) {};
		\node [style=port] (4) at (0.5, 4) {};
		\node [style={circle, draw, inner sep=1pt,minimum size=1pt}] (5) at (1.25, 0.5) {$\mathsf{J}^{-1}$};
		\node [style={circle, draw}] (6) at (1.25, 1.5) {$\mathsf{J}$};
	\end{pgfonlayer}
	\begin{pgfonlayer}{edgelayer}
		\draw [style=wire, in=0, out=-90, looseness=1.00] (3) to (1);
		\draw [style=wire, in=180, out=-90, looseness=1.25] (2) to (1);
		\draw [style=wire] (5) to (0);
		\draw [style=wire] (4) to (2);
		\draw [style=wire] (1) to (6);
		\draw [style=wire] (6) to (5);
	\end{pgfonlayer}
\end{tikzpicture}
  }%
\end{array}=  \begin{array}[c]{c}\resizebox{!}{2cm}{%
\begin{tikzpicture}
	\begin{pgfonlayer}{nodelayer}
		\node [style={regular polygon,regular polygon sides=4, draw, inner sep=1pt,minimum size=1pt}] (0) at (0, 2.25) {$\bigotimes$};
		\node [style=port] (1) at (0.75, 4.25) {};
		\node [style=port] (2) at (0, 1) {};
		\node [style={circle, draw, inner sep=1pt,minimum size=1pt}] (3) at (-0.75, 3.25) {$\mathsf{J}^{-1}$};
		\node [style=port] (4) at (-0.75, 4.25) {};
	\end{pgfonlayer}
	\begin{pgfonlayer}{edgelayer}
		\draw [style=wire, in=0, out=-90, looseness=1.25] (1) to (0);
		\draw [style=wire] (0) to (2);
		\draw [style=wire] (4) to (3);
		\draw [style=wire, in=180, out=-90, looseness=1.25] (3) to (0);
	\end{pgfonlayer}
\end{tikzpicture}
  }%
\end{array}
\end{align*}
Similarly to obtain $\mathsf{J}^{-1}$ on the right.
\end{proof}  

\subsection{From Antiderivatives to Calculus Categories}\label{anticalccat}

In this section we prove that a differential category with antiderivative is a calculus category. This is the main result of this paper. We first show that Ehrhard's integral built in \cite{ehrhard2017introduction} is an integral transformation which furthermore satisfies Poincar\'e Condition.

\begin{proposition}\label{Jint} In a differential category for which $\mathsf{J}$ is a natural isomorphism, the natural transformation $\mathsf{s} := \mathsf{d}^\circ (\mathsf{J}^{-1} \otimes 1)$ is an integral transformation which satisfies the Poincar\'e Condition. 
\end{proposition} 
\begin{proof} In string diagram notation, $\mathsf{s}$ is expressed as follows:
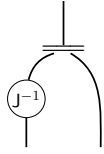
\[
\]
We must show that $\mathsf{s}$ satisfies {\bf [s.1]} to {\bf [s.3]}. \\Ê\\
{\bf [s.1]}: Here we use {\bf [cd.1]} and {\bf [J$^{-1}$.2]}:  
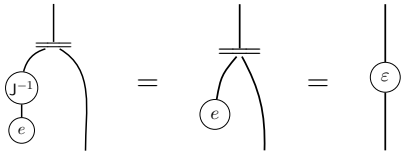
\[ %
\] 
{\bf [s.2]}: First, by using {\bf [J.4]} and {\bf [J.7]}, we have the following equalities: 
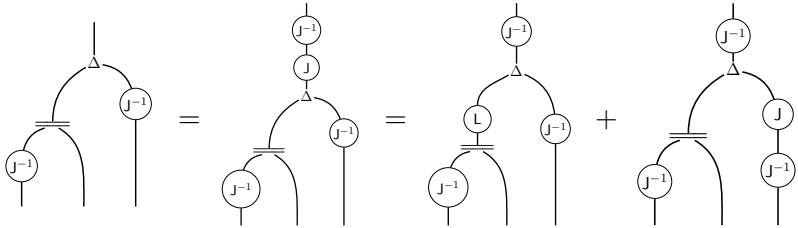
\begin{align*}
  %
\end{align*}
Then by using this above identity, {\bf [J$^{-1}$.11]}, {\bf[cd.3]}, and {\bf[cd.4]} we have that:
\begin{align*}
 &  %
\end{align*}
{\bf [s.3]}: Here we use {\bf [J$^{-1}$.11]}: 
\begin{align*}
   %
\end{align*}
To show that Poincar\' e condition holds we may use the argument in \cite{ehrhard2017introduction}. Let $f: C \otimes \oc A \otimes A \to B$ be an integrable map which satisfies the Poincar\' e pre-condition, that is, $(1 \otimes  1 \otimes \sigma)(1 \otimes \mathsf{d} \otimes 1)f=(1 \otimes  \mathsf{d} \otimes 1)f$. We must show that $f$ satisfies the First Fundamental Theorem of Calculus. Notice first that $f$ satisfies the following identity: 
$$%
$$
Then we have the following equality using {\bf [J$^{-1}$.6]} and the preceding identity: 
$$   %
$$
\end{proof} 

\begin{definition}\label{antiderivativeint} \normalfont  In a differential category with antiderivatives, the \textbf{antiderivative integral transformation} is the natural transformation $\mathsf{K}^{-1}\mathsf{d}^\circ=\mathsf{d}^\circ (\mathsf{J}^{-1} \otimes 1)$. 
\end{definition}
Notice that the equality is {\bf{[J$^{-1}$.9]}} from Proposition \ref{J-1prop}. In string diagram notation, the antiderivative integral transformation is expressed as follows:
\[%
\]

\begin{theorem}\label{anticalc} A differential category with antiderivatives is a calculus category. 
\end{theorem} 
\begin{proof} By {\bf [J$^{-1}$.9]}, Proposition \ref{Jint} implies that the antiderivative is an integral transformation  satisfing the Poincar\'e Condition. Therefore, it remains to show that the antiderivative  satisfies the Second Fundamental Theorem of Calculus: 
\[  %
    \]

Here we have used the definition of the integral transformation,  {\bf [K$^{-1}$.1]}, and that $\mathsf{K}$ is an isomorphism.
\end{proof} 

\begin{proposition} In a differential category with antiderivatives, a map which satisfies the first fundamental theorem of calculus satisfies integration by substitution. 
\end{proposition} 
\begin{proof} Let $f: \oc A \otimes A \to B$ satisfy the first fundamental theorem, that is, $\mathsf{d}\mathsf{K}^{-1}\mathsf{d}^\circ f=f$, and let $g: \oc C \otimes C \to A$ be any integrable map. Using the Second Fundamental Theorem, the chain rule \textbf{[d.5]}, naturality of $\delta$, the deriving transformation and the coderiving transformation, \textbf{[cd.3]}, and {\bf [K$^{-1}$.1]} we have that:
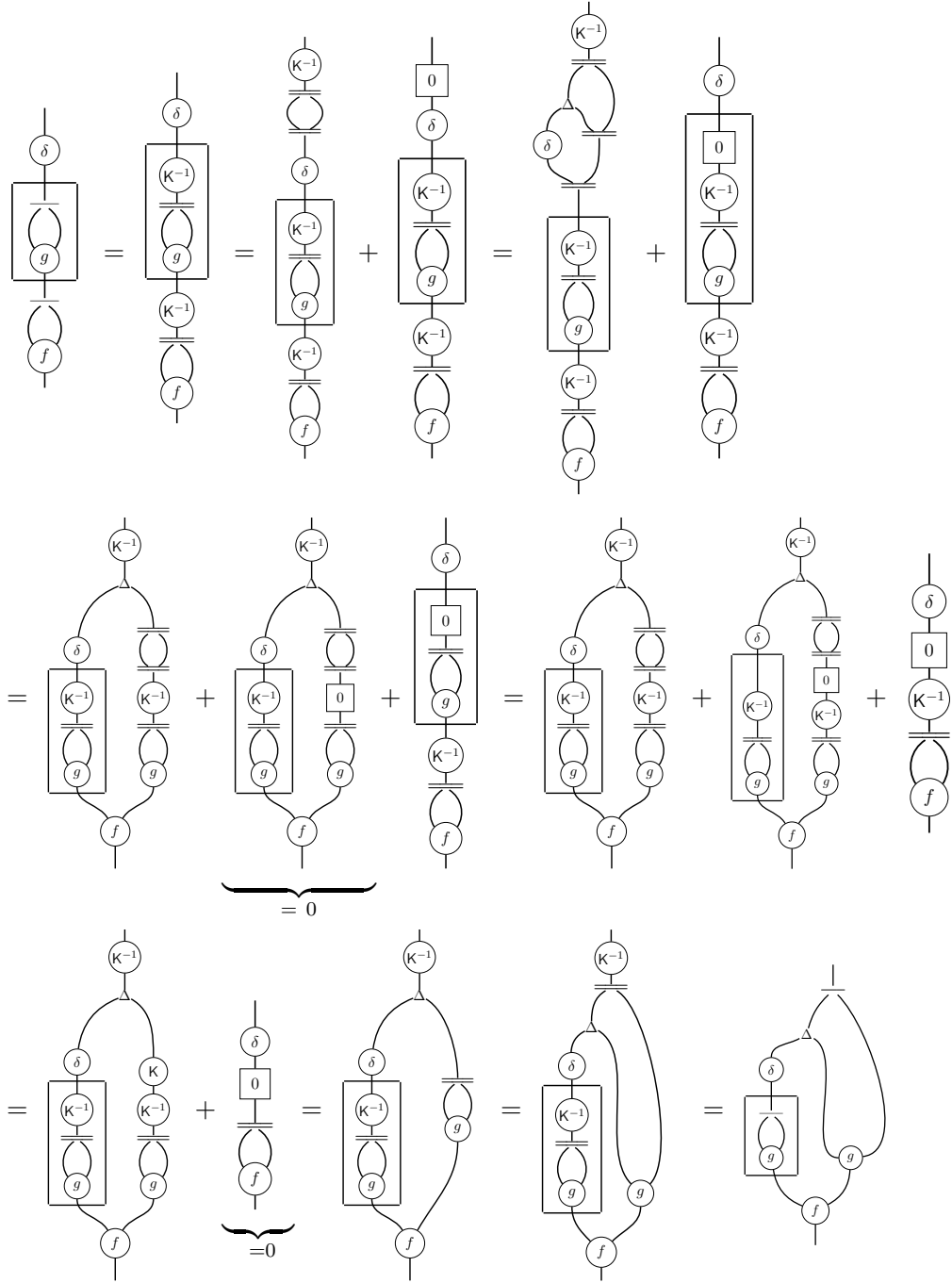
\begin{align*}
  &

\end{align*}
\end{proof} 

\begin{proposition} In a differential category with antiderivatives and a monoidal coalgebra modality, the antiderivative integral transformation is monoidal, that is, satisfies the integral monoidal rule \textbf{[s.m]}. 
\end{proposition} 
\begin{proof} We prove the alternative description of \textbf{[s.m]} from Proposition \ref{altintmon}. Here we use \textbf{[cd.m]} and \textbf{[K$^{-1}$.m]}:
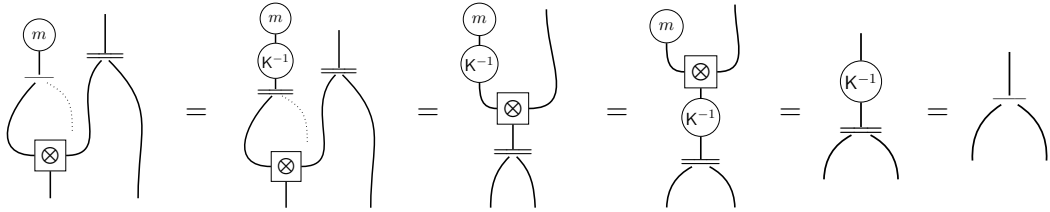
\begin{align*}
%
\end{align*}
\end{proof} 

\subsection{From Calculus Categories to Antiderivatives}\label{monoidalsec}

In this section we prove that a calculus category with a monoidal coalgebra modality has antiderivatives, and provide necessary and sufficient conditions for when the integral is the antiderivative integral. 

\begin{theorem} A differential category  on a monoidal coalgebra modality which is also an integral category satisyfing the Second Fundamental Theorem of Calculus, is also a differential category with antiderivatives. 
\end{theorem} 

It is important to note in this statement that the assumed integral transformation is \emph{not} necessarily the same as the integral transformation constructed as the antiderivative.

\begin{proof} Since the Second Fundamental Theorem of Calculus holds then the deriving transformation is Taylor. Therefore, it suffices to show that $\mathsf{J}$ is a natural isomorphism. Define $\mathsf{J}^{-1}$ as follows:
  \[  \xymatrixcolsep{2pc}\xymatrix{\oc A \ar[r]^-{m_K \otimes 1} & \oc K \otimes \oc A \ar[r]^-{\mathsf{s}_K \otimes 1} & \oc K \otimes K \otimes \oc A \cong \oc K \otimes \oc A \ar[r]^-{m_\otimes} & \oc A
  } \]
expressed in the graphical calculus as:
  \[\begin{array}[c]{c}\resizebox{!}{2cm}
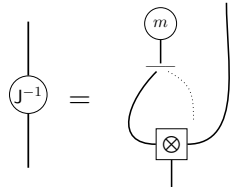
{%
\begin{tikzpicture}
	\begin{pgfonlayer}{nodelayer}
		\node [style=port] (0) at (0, 3) {};
		\node [style=port] (1) at (0, 0) {};
		\node [style={circle, draw, inner sep=1pt,minimum size=1pt}] (2) at (0, 1.5) {$\mathsf{J}^{-1}$};
	\end{pgfonlayer}
	\begin{pgfonlayer}{edgelayer}
		\draw [style=wire] (0) to (2);
		\draw [style=wire] (2) to (1);
	\end{pgfonlayer}
\end{tikzpicture}}
   \end{array} =   \begin{array}[c]{c}\resizebox{!}{2.5cm}{%
\begin{tikzpicture}
	\begin{pgfonlayer}{nodelayer}
		\node [style=port] (0) at (2, 2.25) {};
		\node [style=port] (1) at (0.75, -2) {};
		\node [style=codifferential] (2) at (0.5, 0.75) {{\bf -----}};
		\node [style=port] (3) at (1.25, -0.5) {};
		\node [style={regular polygon,regular polygon sides=4, draw, inner sep=1pt,minimum size=1pt}] (4) at (0.75, -1) {$\bigotimes$};
		\node [style={circle, draw}] (5) at (0.5, 1.75) {$m$};
	\end{pgfonlayer}
	\begin{pgfonlayer}{edgelayer}
		\draw [style=wire] (4) to (1);
		\draw [style=dwire, in=90, out=-28, looseness=1.25] (2) to (3);
		\draw [style=wire] (5) to (2);
		\draw [style=wire, in=0, out=-90, looseness=1.00] (0) to (4);
		\draw [style=wire, in=180, out=-135, looseness=1.50] (2) to (4);
	\end{pgfonlayer}
\end{tikzpicture}
  }%
\end{array} \]
where the dotted line represents the monoidal unit. First notice that by \textbf{[J.m]} we have the following equality: 
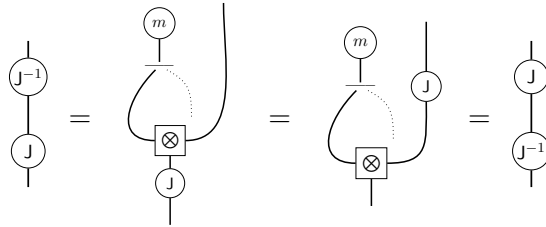
\[\begin{array}[c]{c}\resizebox{!}{2cm}{%
\begin{tikzpicture}
	\begin{pgfonlayer}{nodelayer}
		\node [style=port] (0) at (0, 0) {};
		\node [style={circle, draw}] (1) at (0, 0.75) {$\mathsf{J}$};
		\node [style=port] (2) at (0, 3) {};
		\node [style={circle, draw, inner sep=1pt,minimum size=1pt}] (3) at (0, 2.25) {$\mathsf{J}^{-1}$};
	\end{pgfonlayer}
	\begin{pgfonlayer}{edgelayer}
		\draw [style=wire] (2) to (3);
		\draw [style=wire] (3) to (1);
		\draw [style=wire] (1) to (0);
	\end{pgfonlayer}
\end{tikzpicture}}
   \end{array} =    \begin{array}[c]{c}\resizebox{!}{3cm}{%
\begin{tikzpicture}
	\begin{pgfonlayer}{nodelayer}
		\node [style=port] (0) at (2, 2.25) {};
		\node [style=codifferential] (1) at (0.5, 0.75) {{\bf -----}};
		\node [style=port] (2) at (1.25, -0.5) {};
		\node [style={regular polygon,regular polygon sides=4, draw, inner sep=1pt,minimum size=1pt}] (3) at (0.75, -1) {$\bigotimes$};
		\node [style={circle, draw}] (4) at (0.5, 1.75) {$m$};
		\node [style={circle, draw}] (5) at (0.75, -2) {$\mathsf{J}$};
		\node [style=port] (6) at (0.75, -3) {};
	\end{pgfonlayer}
	\begin{pgfonlayer}{edgelayer}
		\draw [style=dwire, in=90, out=-28, looseness=1.25] (1) to (2);
		\draw [style=wire] (4) to (1);
		\draw [style=wire, in=0, out=-90, looseness=1.00] (0) to (3);
		\draw [style=wire, in=180, out=-135, looseness=1.50] (1) to (3);
		\draw [style=wire] (5) to (6);
		\draw [style=wire] (3) to (5);
	\end{pgfonlayer}
\end{tikzpicture}
  }%
\end{array} =    \begin{array}[c]{c}\resizebox{!}{2.5cm}{%
\begin{tikzpicture}
	\begin{pgfonlayer}{nodelayer}
		\node [style=port] (0) at (2, 2.25) {};
		\node [style=codifferential] (1) at (0.5, 0.75) {{\bf -----}};
		\node [style=port] (2) at (1.25, -0.5) {};
		\node [style={regular polygon,regular polygon sides=4, draw, inner sep=1pt,minimum size=1pt}] (3) at (0.75, -1) {$\bigotimes$};
		\node [style={circle, draw}] (4) at (0.5, 1.75) {$m$};
		\node [style=port] (5) at (0.75, -2) {};
		\node [style={circle, draw}] (6) at (2, 0.75) {$\mathsf{J}$};
	\end{pgfonlayer}
	\begin{pgfonlayer}{edgelayer}
		\draw [style=dwire, in=90, out=-28, looseness=1.25] (1) to (2);
		\draw [style=wire] (4) to (1);
		\draw [style=wire, in=180, out=-135, looseness=1.50] (1) to (3);
		\draw [style=wire] (0) to (6);
		\draw [style=wire, in=0, out=-90, looseness=1.50] (6) to (3);
		\draw [style=wire] (3) to (5);
	\end{pgfonlayer}
\end{tikzpicture}
  }%
\end{array} =    \begin{array}[c]{c}\resizebox{!}{2cm}{%
\begin{tikzpicture}
	\begin{pgfonlayer}{nodelayer}
		\node [style=port] (0) at (0, 0) {};
		\node [style={circle, draw, inner sep=1pt,minimum size=1pt}] (1) at (0, 0.75) {$\mathsf{J}^{-1}$};
		\node [style=port] (2) at (0, 3) {};
		\node [style={circle, draw}] (3) at (0, 2.25) {$\mathsf{J}$};
	\end{pgfonlayer}
	\begin{pgfonlayer}{edgelayer}
		\draw [style=wire] (2) to (3);
		\draw [style=wire] (3) to (1);
		\draw [style=wire] (1) to (0);
	\end{pgfonlayer}
\end{tikzpicture}
  }%
\end{array}\]
Therefore, $\mathsf{J}\mathsf{J}^{-1}=\mathsf{J}^{-1}\mathsf{J}$. To prove this gives the identity, consider $\mathsf{s}_K(\mathsf{J}_K \otimes 1)$. Using that $\sigma_{K,K}=1_K \otimes 1_K$, the $\mathsf{W}$ identity, the Second Fundamental Theorem of Calculus and naturality of the integral transformation, we have that:
 \begin{align*}
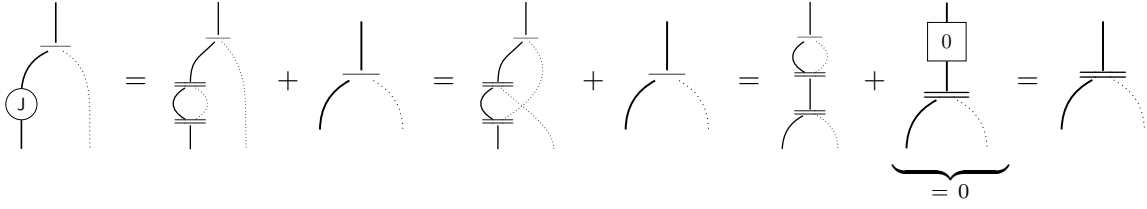
   \begin{array}[c]{c}\resizebox{!}{2cm}{%
\begin{tikzpicture}
	\begin{pgfonlayer}{nodelayer}
		\node [style=port] (0) at (2, -0.25) {};
		\node [style=port] (1) at (1.25, 3) {};
		\node [style=integral] (2) at (1.25, 2) {{\bf -----}};
		\node [style={circle, draw}] (3) at (0.5, 0.75) {$\mathsf{J}$};
		\node [style=port] (4) at (0.5, -0.25) {};
	\end{pgfonlayer}
	\begin{pgfonlayer}{edgelayer}
		\draw [style=dwire, in=93, out=-30, looseness=1.00] (2) to (0);
		\draw [style=wire] (1) to (2);
		\draw [style=wire, bend right, looseness=1.00] (2) to (3);
		\draw [style=wire] (3) to (4);
	\end{pgfonlayer}
\end{tikzpicture}
  }%
\end{array}=    \begin{array}[c]{c}\resizebox{!}{2cm}{%
\begin{tikzpicture}
	\begin{pgfonlayer}{nodelayer}
		\node [style=port] (0) at (2, -1) {};
		\node [style=port] (1) at (1.25, 3) {};
		\node [style=integral] (2) at (1.25, 2) {{\bf -----}};
		\node [style=differential] (3) at (0.5, -0.25) {{\bf =\!=\!=}};
		\node [style=codifferential] (4) at (0.5, 0.75) {{\bf =\!=\!=}};
		\node [style=port] (5) at (0.5, -1) {};
	\end{pgfonlayer}
	\begin{pgfonlayer}{edgelayer}
		\draw [style=dwire, in=93, out=-45, looseness=1.00] (2) to (0);
		\draw [style=wire] (1) to (2);
		\draw [style=wire, bend right=60, looseness=1.50] (4) to (3);
		\draw [style=dwire, bend left=60, looseness=1.50] (4) to (3);
		\draw [style=wire, in=89, out=-135, looseness=1.25] (2) to (4);
		\draw [style=wire] (3) to (5);
	\end{pgfonlayer}
\end{tikzpicture}
  }%
\end{array}+   \begin{array}[c]{c}\resizebox{!}{1.5cm}{%
\begin{tikzpicture}
	\begin{pgfonlayer}{nodelayer}
		\node [style=port] (0) at (-2.5, -1.25) {};
		\node [style=codifferential] (1) at (-3.25, -0.25) {{\bf -----}};
		\node [style=port] (2) at (-4, -1.25) {};
		\node [style=port] (3) at (-3.25, 0.75) {};
	\end{pgfonlayer}
	\begin{pgfonlayer}{edgelayer}
		\draw [style=dwire, bend left, looseness=1.00] (1) to (0);
		\draw [style=wire, bend right, looseness=1.00] (1) to (2);
		\draw [style=wire] (3) to (1);
	\end{pgfonlayer}
\end{tikzpicture}
  }%
\end{array}=    \begin{array}[c]{c}\resizebox{!}{2cm}{%
\begin{tikzpicture}
	\begin{pgfonlayer}{nodelayer}
		\node [style=port] (0) at (2, -1) {};
		\node [style=port] (1) at (1.25, 3) {};
		\node [style=integral] (2) at (1.25, 2) {{\bf -----}};
		\node [style=differential] (3) at (0.5, -0.25) {{\bf =\!=\!=}};
		\node [style=codifferential] (4) at (0.5, 0.75) {{\bf =\!=\!=}};
		\node [style=port] (5) at (0.5, -1) {};
	\end{pgfonlayer}
	\begin{pgfonlayer}{edgelayer}
		\draw [style=wire] (1) to (2);
		\draw [style=wire, bend right=60, looseness=1.50] (4) to (3);
		\draw [style=wire, in=89, out=-135, looseness=1.25] (2) to (4);
		\draw [style=wire] (3) to (5);
		\draw [style=dwire, in=30, out=-45, looseness=1.25] (2) to (3);
		\draw [style=dwire, in=90, out=-45, looseness=0.75] (4) to (0);
	\end{pgfonlayer}
\end{tikzpicture}
  }%
\end{array}+   \begin{array}[c]{c}\resizebox{!}{1.5cm}{%
\begin{tikzpicture}
	\begin{pgfonlayer}{nodelayer}
		\node [style=port] (0) at (-2.5, -1.25) {};
		\node [style=codifferential] (1) at (-3.25, -0.25) {{\bf -----}};
		\node [style=port] (2) at (-4, -1.25) {};
		\node [style=port] (3) at (-3.25, 0.75) {};
	\end{pgfonlayer}
	\begin{pgfonlayer}{edgelayer}
		\draw [style=dwire, bend left, looseness=1.00] (1) to (0);
		\draw [style=wire, bend right, looseness=1.00] (1) to (2);
		\draw [style=wire] (3) to (1);
	\end{pgfonlayer}
\end{tikzpicture}
  }%
\end{array}=    \begin{array}[c]{c}\resizebox{!}{2cm}{%
\begin{tikzpicture}
	\begin{pgfonlayer}{nodelayer}
		\node [style=port] (0) at (-2.5, -1.25) {};
		\node [style=codifferential] (1) at (-3.25, -0.25) {{\bf =\!=\!=}};
		\node [style=port] (2) at (-4, -1.25) {};
		\node [style=integral] (3) at (-3.25, 1.75) {{\bf -----}};
		\node [style=differential] (4) at (-3.25, 0.75) {{\bf =\!=\!=}};
		\node [style=port] (5) at (-3.25, 2.75) {};
	\end{pgfonlayer}
	\begin{pgfonlayer}{edgelayer}
		\draw [style=dwire, bend left, looseness=1.00] (1) to (0);
		\draw [style=wire, bend right, looseness=1.00] (1) to (2);
		\draw [style=wire, bend right=60, looseness=1.50] (3) to (4);
		\draw [style=dwire, in=30, out=-30, looseness=1.50] (3) to (4);
		\draw [style=wire] (4) to (1);
		\draw [style=wire] (5) to (3);
	\end{pgfonlayer}
\end{tikzpicture}
  }%
\end{array} + \underbrace{   \begin{array}[c]{c}\resizebox{!}{2cm}{%
\begin{tikzpicture}
	\begin{pgfonlayer}{nodelayer}
		\node [style=port] (0) at (-2.5, -1.25) {};
		\node [style=codifferential] (1) at (-3.25, -0.25) {{\bf =\!=\!=}};
		\node [style=port] (2) at (-4, -1.25) {};
		\node [style=port] (3) at (-3.25, 1.5) {};
		\node [style={regular polygon,regular polygon sides=4, draw}] (4) at (-3.25, 0.75) {$0$};
	\end{pgfonlayer}
	\begin{pgfonlayer}{edgelayer}
		\draw [style=dwire, bend left, looseness=1.00] (1) to (0);
		\draw [style=wire, bend right, looseness=1.00] (1) to (2);
		\draw [style=wire] (3) to (4);
		\draw [style=wire] (4) to (1);
	\end{pgfonlayer}
\end{tikzpicture}
  }%
\end{array}}_{=~0} =    \begin{array}[c]{c}\resizebox{!}{1.5cm}{%
\begin{tikzpicture}
	\begin{pgfonlayer}{nodelayer}
		\node [style=port] (0) at (-2.5, -1.25) {};
		\node [style=codifferential] (1) at (-3.25, -0.25) {{\bf =\!=\!=}};
		\node [style=port] (2) at (-4, -1.25) {};
		\node [style=port] (3) at (-3.25, 0.75) {};
	\end{pgfonlayer}
	\begin{pgfonlayer}{edgelayer}
		\draw [style=dwire, bend left, looseness=1.00] (1) to (0);
		\draw [style=wire, bend right, looseness=1.00] (1) to (2);
		\draw [style=wire] (3) to (1);
	\end{pgfonlayer}
\end{tikzpicture}
  }%
\end{array}
\end{align*}
Finally using \textbf{[J.m]}, the above identity, that $\Delta$ and $\varepsilon$ are monoidal transformation and the unit laws for the monoidal strength we obtain the following equality: 
\begin{align*}
 \begin{array}[c]{c}\resizebox{!}{3cm}{%
\begin{tikzpicture}
	\begin{pgfonlayer}{nodelayer}
		\node [style=port] (0) at (2, 2.25) {};
		\node [style=codifferential] (1) at (0.5, 0.75) {{\bf -----}};
		\node [style=port] (2) at (1.25, -0.5) {};
		\node [style={regular polygon,regular polygon sides=4, draw, inner sep=1pt,minimum size=1pt}] (3) at (0.75, -1) {$\bigotimes$};
		\node [style={circle, draw}] (4) at (0.5, 1.75) {$m$};
		\node [style={circle, draw}] (5) at (0.75, -2) {$\mathsf{J}$};
		\node [style=port] (6) at (0.75, -3) {};
	\end{pgfonlayer}
	\begin{pgfonlayer}{edgelayer}
		\draw [style=dwire, in=90, out=-28, looseness=1.25] (1) to (2);
		\draw [style=wire] (4) to (1);
		\draw [style=wire, in=0, out=-90, looseness=1.00] (0) to (3);
		\draw [style=wire, in=180, out=-135, looseness=1.50] (1) to (3);
		\draw [style=wire] (5) to (6);
		\draw [style=wire] (3) to (5);
	\end{pgfonlayer}
\end{tikzpicture}
  }%
\end{array}=    \begin{array}[c]{c}\resizebox{!}{2.5cm}{%
\begin{tikzpicture}
	\begin{pgfonlayer}{nodelayer}
		\node [style=port] (0) at (2, 2.25) {};
		\node [style=port] (1) at (1.25, 0) {};
		\node [style={regular polygon,regular polygon sides=4, draw, inner sep=1pt,minimum size=1pt}] (2) at (0.75, -1) {$\bigotimes$};
		\node [style={circle, draw}] (3) at (0.5, 2.25) {$m$};
		\node [style=port] (4) at (0.75, -2) {};
		\node [style=integral] (5) at (0.5, 1.25) {{\bf -----}};
		\node [style={circle, draw}] (6) at (-0.25, 0) {$\mathsf{J}$};
	\end{pgfonlayer}
	\begin{pgfonlayer}{edgelayer}
		\draw [style=wire, in=0, out=-90, looseness=1.00] (0) to (2);
		\draw [style=wire] (2) to (4);
		\draw [style=wire, bend right, looseness=1.00] (5) to (6);
		\draw [style=wire] (3) to (5);
		\draw [style=dwire, bend left=15, looseness=1.25] (5) to (1);
		\draw [style=wire, in=180, out=-90, looseness=1.25] (6) to (2);
	\end{pgfonlayer}
\end{tikzpicture}
  }%
\end{array}=    \begin{array}[c]{c}\resizebox{!}{2cm}{%
\begin{tikzpicture}
	\begin{pgfonlayer}{nodelayer}
		\node [style={circle, draw}] (0) at (0.5, 1.5) {$m$};
		\node [style=codifferential] (1) at (0.5, 0.5) {{\bf =\!=\!=}};
		\node [style=port] (2) at (2.25, 1.75) {};
		\node [style={regular polygon,regular polygon sides=4, draw, inner sep=1pt,minimum size=1pt}] (3) at (1, -1) {$\bigotimes$};
		\node [style=port] (4) at (1, -1.75) {};
		\node [style=port] (5) at (1.25, -0.5) {};
	\end{pgfonlayer}
	\begin{pgfonlayer}{edgelayer}
		\draw [style=wire, in=0, out=-90, looseness=1.25] (2) to (3);
		\draw [style=wire, in=180, out=-150, looseness=2.00] (1) to (3);
		\draw [style=wire] (0) to (1);
		\draw [style=wire] (3) to (4);
		\draw [style=dwire, in=90, out=-30, looseness=1.00] (1) to (5);
	\end{pgfonlayer}
\end{tikzpicture}
  }%
\end{array}=    \begin{array}[c]{c}\resizebox{!}{2cm}{%
\begin{tikzpicture}
	\begin{pgfonlayer}{nodelayer}
		\node [style={circle, draw}] (0) at (0.5, 1.75) {$m$};
		\node [style=port] (1) at (2.25, 1.75) {};
		\node [style={regular polygon,regular polygon sides=4, draw, inner sep=1pt,minimum size=1pt}] (2) at (0.75, -1) {$\bigotimes$};
		\node [style=port] (3) at (0.75, -1.75) {};
		\node [style=duplicate] (4) at (0.5, 1) {$\Delta$};
		\node [style={circle, draw}] (5) at (1.25, 0) {$\varepsilon$};
		\node [style=port] (6) at (1.25, -0.5) {};
	\end{pgfonlayer}
	\begin{pgfonlayer}{edgelayer}
		\draw [style=wire, in=0, out=-90, looseness=1.25] (1) to (2);
		\draw [style=wire] (2) to (3);
		\draw [style=wire, bend left, looseness=1.00] (4) to (5);
		\draw [style=dwire] (5) to (6);
		\draw [style=wire, in=180, out=-150, looseness=1.25] (4) to (2);
		\draw [style=wire] (0) to (4);
	\end{pgfonlayer}
\end{tikzpicture}
  }%
\end{array}=    \begin{array}[c]{c}\resizebox{!}{2cm}{%
\begin{tikzpicture}
	\begin{pgfonlayer}{nodelayer}
		\node [style={circle, draw}] (0) at (1.25, 1.25) {$m$};
		\node [style=port] (1) at (2.25, 1.75) {};
		\node [style={regular polygon,regular polygon sides=4, draw, inner sep=1pt,minimum size=1pt}] (2) at (0.75, -1) {$\bigotimes$};
		\node [style=port] (3) at (0.75, -1.75) {};
		\node [style={circle, draw}] (4) at (1.25, 0.25) {$\varepsilon$};
		\node [style=port] (5) at (1.25, -0.5) {};
		\node [style={circle, draw}] (6) at (0, 0) {$m$};
	\end{pgfonlayer}
	\begin{pgfonlayer}{edgelayer}
		\draw [style=wire, in=0, out=-90, looseness=1.25] (1) to (2);
		\draw [style=wire] (2) to (3);
		\draw [style=dwire] (4) to (5);
		\draw [style=wire] (0) to (4);
		\draw [style=wire, in=180, out=-90, looseness=1.25] (6) to (2);
	\end{pgfonlayer}
\end{tikzpicture}
  }%
\end{array}=    \begin{array}[c]{c}\resizebox{!}{1.5cm}{%
\begin{tikzpicture}
	\begin{pgfonlayer}{nodelayer}
		\node [style=port] (0) at (0, 3) {};
		\node [style=port] (1) at (0, 0) {};
	\end{pgfonlayer}
	\begin{pgfonlayer}{edgelayer}
		\draw [style=wire] (0) to (1);
	\end{pgfonlayer}
\end{tikzpicture}
  }%
\end{array}
\end{align*}
So $\mathsf{J}\mathsf{J}^{-1}=\mathsf{J}^{-1}\mathsf{J}=1$. 
\end{proof} 

\begin{theorem} For a calculus category with a monoidal coalgebra modality, if the integral transformation is monoidal then it is the antiderivative integral transformation.  
\end{theorem} 
\begin{proof} By Corollary \ref{Kcalcobj}, the monoidal unit $K$ is a calculus object for any calculus category structure, that is, $\mathsf{d}_K\mathsf{s}_K=1$. However, since we also have antiderivatives, we also have $\mathsf{d}_K \mathsf{K}^{-1}_K \mathsf{d}^\circ_K=1$, and so we have $\mathsf{d}_K\mathsf{s}_K= \mathsf{d}_K \mathsf{K}^{-1}_K \mathsf{d}^\circ_K$. But since the deriving transformation is Taylor we have that: 
\[\mathsf{s}_K + \oc(0)\mathsf{K}^{-1}_K \mathsf{d}^\circ_K =\mathsf{K}^{-1}_K \mathsf{d}^\circ_K + \oc(0)\mathsf{s}\]
But by naturality, this simplifies to $\mathsf{s}_K=\mathsf{K}^{-1}_K \mathsf{d}^\circ_K$. Therefore, since our integral transformation was assumed to monoidal and the antiderivative integral transformation is monoidal, by Proposition \ref{altintmon}, we obtain the following: 
\begin{align*}
\mathsf{s}&= (m_K \otimes \mathsf{d}^\circ)(\mathsf{s}_K \otimes 1 \otimes 1)(m_ \otimes \otimes 1)\\
&= (m_K \otimes \mathsf{d}^\circ)(\mathsf{K}^{-1}_K \otimes 1 \otimes 1)(\mathsf{d}^\circ_K \otimes 1 \otimes 1)(m_ \otimes \otimes 1)\\
&= \mathsf{K}^{-1} \mathsf{d}^\circ
\end{align*}
\end{proof} 

It is also possible -- though technically harder -- to give necessary and sufficient conditions which ensure the integral transformation of a coalgebra modality, which is not necessarily monoidal, is an antiderivative: 

\begin{proposition} For a differential category $\mathbb{X}$ which is also an integral category on the same coalgebra modality with deriving transformation $\mathsf{d}$ and integral transformation $\mathsf{s}$, the following are equivalent: 
\begin{enumerate}[{\em (i)}]
\item $\mathsf{d}$ and $\mathsf{s}$ satisfy the Second Fundamental Theorem of Calculus and the following equality holds:
$$\mathsf{d}\mathsf{s}= (\delta \otimes 1)(\mathsf{s} \otimes 1)(\oc(\varepsilon) \otimes e \otimes 1)\mathsf{d}\mathsf{d}^\circ
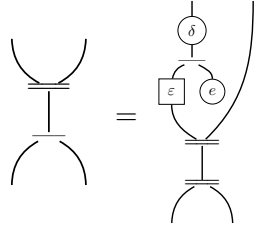
$$
$$  
$$
\item $\mathbb{X}$ has antiderivatives and $\mathsf{s}$ is the antiderivative integral. 
\end{enumerate}
\end{proposition} 
\begin{proof} $(i) \Rightarrow (ii)$: Since the Second Fundamental Theorem of Calculus holds then $\mathsf{d}$ is Taylor. Therefore, it suffices to show that $\mathsf{J}$ is a natural isomorphism. Define $\mathsf{J}^{-1}$ as follows:
  \[  \xymatrixcolsep{5pc}\xymatrix{\oc A \ar[r]^-{\delta} & \oc\oc A \ar[r]^-{\mathsf{s}} & \oc \oc A \otimes \oc A \ar[r]^-{\oc(\varepsilon) \otimes e} & \oc(A)
  } \]
expressed in the graphical calculus as:
$$%
$$
 Before showing that this is the inverse of $\mathsf{J}$, we first observe that $\mathsf{s}$ and the coderiving transformation $\mathsf{d}^\circ$ satisfy a sort of interchange rule. Indeed, by the Second Fundamental Theorem of Calculus, that $\oc(0)\mathsf{s}= 0$ by the naturality of $\mathsf{s}$, the extra coherence between $\mathsf{s}$ and $\mathsf{d}$, and {\bf[cd.6]} we have that: 
 \begin{align*}
&    %
\end{align*}
Now by naturality of $\mathsf{L}$, the interchange rule between $\mathsf{s}$ and $\mathsf{d}^\circ$ above, that $\oc(0)\mathsf{d}^\circ$ by naturality of $\mathsf{d}^\circ$, the Second Fundamental Theorem of Calculus, Proposition \ref{Wprop}, and {\bf[cd.7]} we have that: 
\begin{align*}
&   %
 
\end{align*}
Before proving $\mathsf{J}\mathsf{J}^{-1} = 1$, notice that $\mathsf{s}$ and $\mathsf{d}^\circ$ also satisfy the interchange rule below: this is shown using the interchange rule above and all the integral axioms {\bf [s.1]} to {\bf [s.3]}:
\begin{align*}
&  %
\end{align*}
Then by {\bf [L.5]}, the extra coherence between $\mathsf{s}$ and $\mathsf{d}$, {\bf [L.2]}, Proposition \ref{Wprop}, naturality of $\mathsf{d}$, {\bf [L.3]}, {\bf [cd.5]}, the second interchange rule between $\mathsf{s}$ and $\mathsf{d}^\circ$, {\bf[cd.7]} and the Second Fundamental Theorem of Calculus, we have that: 
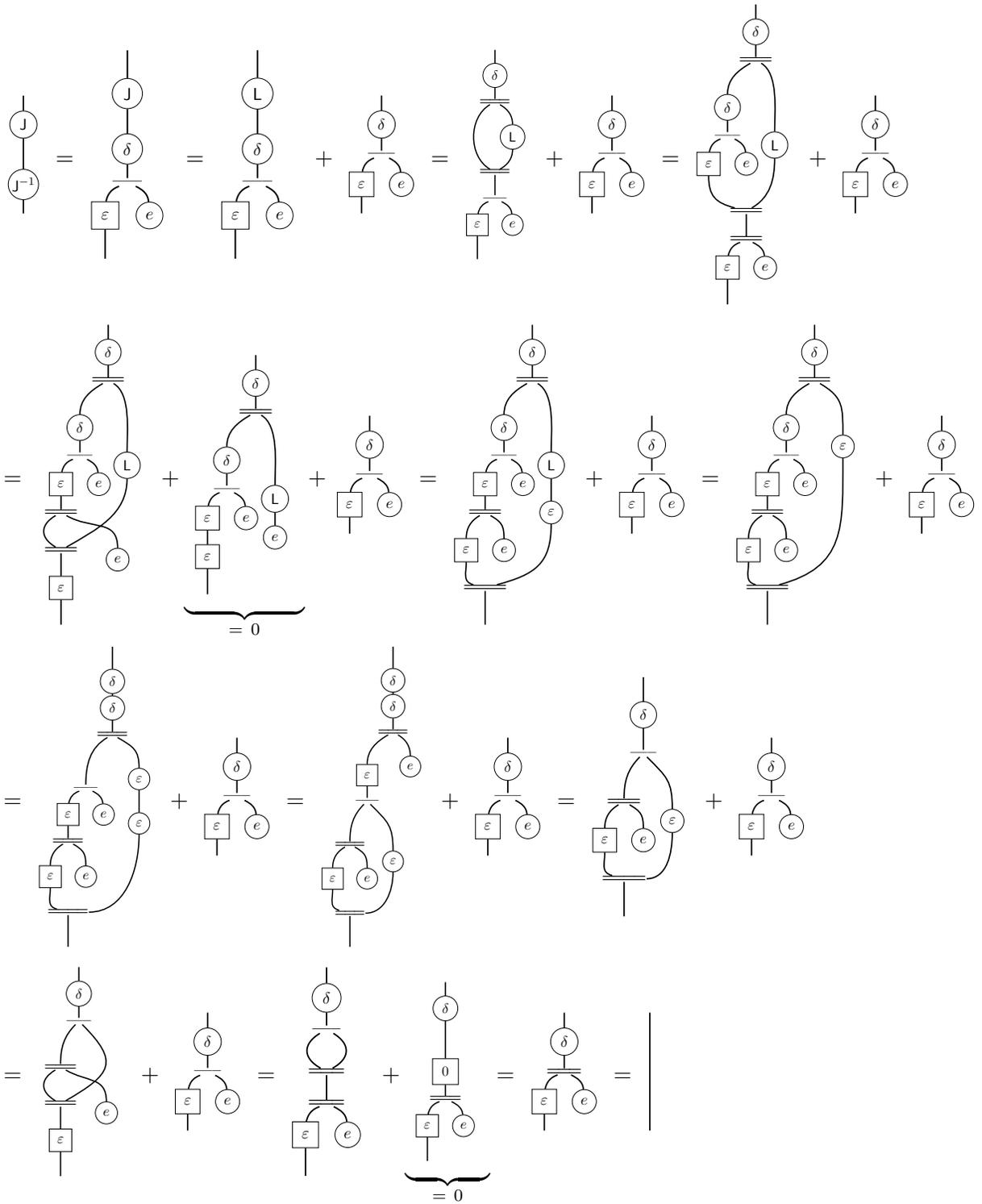
\begin{align*}
& %
\end{align*}
Finally, it remains to show that $\mathsf{s}$ is in fact the antiderivative integral transformation. Here we use {\bf[cd.7]}, naturality of $\mathsf{s}$, and the first interchange rule between $\mathsf{s}$ and $\mathsf{d}^\circ$: 
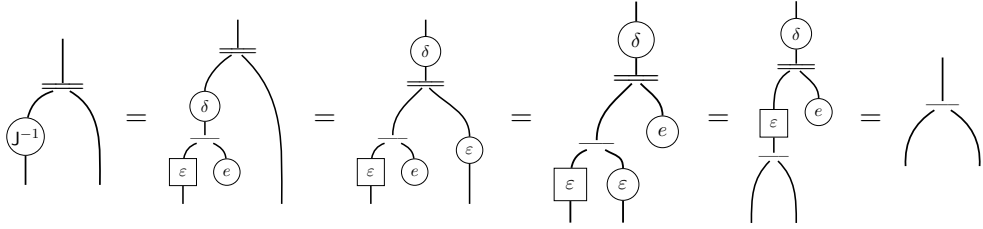
\begin{align*}
& %
\end{align*}
$(ii) \Rightarrow (i)$: Suppose that $\mathbb{X}$ has antiderivatives and that the integral transformation is the antiderivative integral transformation. Therefore $\mathbb{X}$ is a calculus category and the Second Fundamental Theorem of Calculus holds. So it remains to show that the extra coherence identity holds. By the definition of the 
antiderivative using $\mathsf{K}^{-1}$, {\bf [J$^{-1}$.8]}, naturality of $\mathsf{J}^{-1}$, and the (equivalent) definition of the antiderivative using $\mathsf{J}^{-1}$ we have: 
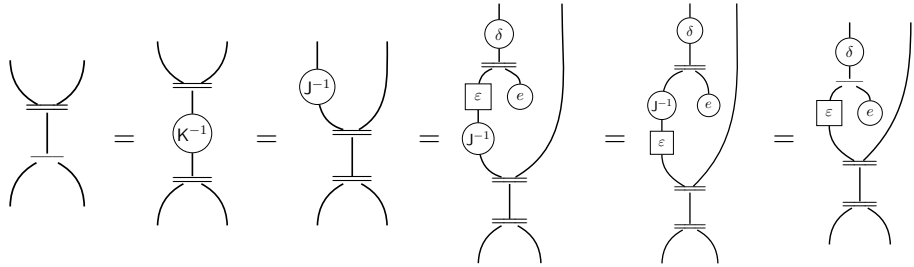
\begin{align*}
&%
\end{align*}
\end{proof} 

\section{Examples}\label{examplesec}

Here we present an overview of separating examples between the various structures defined throughout this paper. To help us give our separating examples, we present a Venn diagram classifying the examples below. 
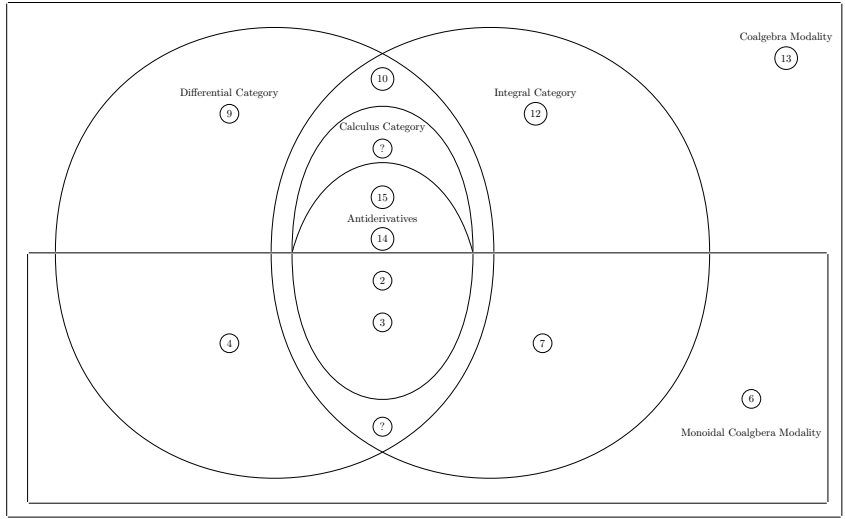
\[%
\]

A well-known example of a differential category (whose coalgebra modality also happens to be monoidal) comes from the free symmetric algebra construction which actually gives a co-differential category. We briefly recall this example (see \cite{dblute2006differential} for more details) by first recalling the free symmetric algebra adjunction: 

\begin{example} \normalfont The category of vector spaces over a field $\mathbb{K}$, $\mathsf{VEC}_{\mathbb{K}}$, is a co-differential category \cite{dblute2006differential}, so that, $\mathsf{VEC}^{\sf op}_{\mathbb{K}}$ is a differential category. The additive symmetric monoidal structure is given by the standard tensor product and additive enrichment of vector spaces. The algebra modality is given by the free symmetric algebra monad where for a vector space $V$, define $\mathsf{Sym}(V)$, called the free symmetric algebra over $V$, as following (see Section 8, Chapter XVI in \cite{lang2002algebra} for more details): 
$$\mathsf{Sym}(V)= \bigoplus^{\infty}_{n=0} \mathsf{Sym}^n(V)= \mathbb{K} \oplus V \oplus \mathsf{Sym}^2(V) \oplus ... $$ 
where $\mathsf{Sym}^n(V)$ is simply the quotient of $V^{\otimes^n}$ by the tensor symmetry equalities: 
$$v_1 \otimes ... \otimes v_i \otimes ... \otimes v_n = v_{\sigma(1)} \otimes ... \otimes v_{\sigma(i)} \otimes ... \otimes v_{\sigma(n)}$$
$\mathsf{Sym}(V)$ is the free commutative algebra over $V$, that is, we obtain an adjunction: 
$$\xymatrixcolsep{2.5pc}\xymatrix{\mathsf{VEC}_{\mathbb{K}} \ar@<+1.1ex>[r]^-{\mathsf{Sym}} & \mathsf{CALG}_{\mathbb{K}} \ar@<+1ex>[l]_-{\bot}^-{U}
  }$$
The induced monad is an algebra modality -- the dual of a coalgebra modality (see \cite{dblute2006differential} for more details on this algebra modality). Furthermore, this algebra modality has the Seely isomorphisms \cite{lang2002algebra}, that is:
$$\mathsf{Sym}(V \times W) \cong \mathsf{Sym}(V) \otimes \mathsf{Sym}(W) \quad \mathsf{Sym}(0)\cong \mathbb{K}$$
and therefore this is also a comonoidal algebra modality. The deriving transformation $\mathsf{d}: \mathsf{Sym}(V) \to ~\mathsf{Sym}(V) \otimes V$ on pure tensors is defined as follows:
$$\mathsf{d}(a_1 \otimes ... \otimes a_n)= \sum_{i=1}^{n} (a_1 \otimes ... \otimes a_{i-1} \otimes a_{i+1} \otimes ... \otimes a_n) \otimes a_i$$ 
which we then extend by linearity (if this map looks backwards, recall that $\mathsf{MOD}_R$ is a co-differential category). An alternative approach of illustrating this differential category structure on $\mathsf{VEC}_{\mathbb{K}}$ is to consider polynomial rings. If $X= \lbrace x_1, x_2, ... \rbrace$ is a basis of $V$, then $\mathsf{Sym} V$ is isomorphic to the polynomial ring over $X$: $\mathsf{Sym}(V) \cong \mathbb{K}[X]$ \cite{lang2002algebra}. Then the deriving transformation $\mathsf{d}_V$ on monomials is given by the sum of partial derivatives of the monomial:
$$\mathsf{d}_V(x_1^{r_1}  ...  x_n^{r_n}) = \sum^n_{i=1} r_i \cdot x_1^{r_1} ...  x_i^{r_i-1}  ...  x^{r_n} \otimes x_i$$
\end{example}

It is important to note that this differential category structure on $\mathsf{VEC}_{\mathbb{K}}$ can be generalized to the category of modules over any ring $R$. In fact, this example can be generalized further. Indeed, the free symmetric algebra construction on appropriate additive symmetric monoidal categories induces a differential category structure. 

We now give two examples of differential categories with monoidal coalgebra modalities and antiderivatives, and therefore, two examples of calculus categories. In fact, these were two of the main examples of differential categories in \cite{dblute2006differential}. 

\begin{example} \normalfont The category of sets and relations, $\mathsf{REL}$, is a differential category \cite{dblute2006differential} with antiderivatives. The symmetric monoidal strucure is given by the cartesian product of sets while the additive structure is given by the union of sets. The coalgebra modality is given by the free symmetric algebra construction on $\mathsf{REL}$, also known as the finite bag comonad. Explicitly, for a set $X$, $\oc X$ is the set of all finite bags (also known as multisets) of $X$ (including the empty bag):
\[\oc(X)=\lbrace \llbracket x_1,...,x_n \rrbracket \vert~ x_i \in X \rbrace\] 
See \cite{dblute2006differential} for more details about this coalgebra modality. The deriving transformation $\mathsf{d}_X: \oc X \times X \to \oc X$ is the relation which adds an extra element to the bag:
$$\mathsf{d}_X = \lbrace ((B, x), B \cup x) \vert ~ x \in X,~ B \in \oc X \rbrace$$
The additive idempotency of $\mathsf{REL}$ makes both $\mathsf{K}$ and $\mathsf{J}$ the identity and thus trivially isomorphisms. Therefore, the integral transformation is the coderiving transformation $\mathsf{d}^\circ_X: \oc X \to \oc X \times X$, which is the relation which removes an element from the bag:
$$\mathsf{d}^\circ_X = \lbrace (B, (B-\lbrace x \rbrace, x)) \vert ~ x \in X,~ B \in \oc X \rbrace$$
This makes $\mathsf{REL}$ into a calculus category. Furthermore, the only calculus objects of this calculus category structure for $\mathsf{REL}$ are one element sets and the empty set. 
\end{example}
 
\begin{example} \normalfont The category of vector spaces over a field $\mathbb{K}$ of characteristic of $0$, $\mathsf{VEC}_{\mathbb{K}}$, is a co-differential category with antiderivatives, so that, $\mathsf{VEC}^{\sf op}_{\mathbb{K}}$ is a calculus category. While having a field of characteristic zero is not required to obtain differential structure, it is required for antiderivatives. On pure tensors, $\mathsf{L}$ scalar multiplies a word by its length:
\[\mathsf{L}(v_1 \otimes ... \otimes v_n)=n \cdot (v_1 \otimes ... \otimes v_n)\]
$\mathsf{K}$ is equal to $\mathsf{L}$ everywhere except on constants (elements of $\mathbb{K}$) since $\mathsf{L}$ on constants is zero while $\mathsf{K}$ map constants is the identity: 
\[\mathsf{K}_V(w)=\begin{cases} w & \text{ if } w\in \mathsf{Sym}^0(V)=\mathbb{K} \\
\mathsf{L}(w) & \text{ o.w. } 
\end{cases}\]
$\mathsf{J}$, on pure tensors, scalar multiplies a word by its length plus one:
\[\mathsf{J}_V(v_1 \otimes ... \otimes v_n)=(n+1) \cdot (v_1 \otimes ... \otimes v_n)\]
Notice that in this case, unlike the previous example, the $\mathsf{J}$ map is not equal to the $\mathsf{K}$ map!  As the rationals are embedded in our field, both $\mathsf{K}$ and $\mathsf{J}$ are isomorphisms, and the resulting integral transformation $\mathsf{s}_V: \mathsf{Sym} V \otimes V \to \mathsf{Sym} V$ is defined on pure tensors as follows: 
\[\mathsf{s}_V((v_1 \otimes ... \otimes v_n) \otimes v)= \frac{1}{n+1} v_1 \otimes ... \otimes v_n \otimes v\]
which we then extend by linearity. Using the alternative description, the integral transformation can be described on the polynomial ring $\mathsf{s}_V: \mathbb{K}[X] \otimes V \to \mathbb{K}[X]$ as follows on monomials:
\[\mathsf{s}_V((x_1^{r_1}  ...  x_n^{r_n}) \otimes x_i)= \frac{1}{1+\sum^n_{j=1} r_j} x_1^{r_1}  ... x_i^{r_i+1}  ...  x^{r_n}\]
At first glance this may seem bizarre. One might expect the integral transformation to integrate a monomial with respect to the variable $x_i$ and thus only multiply by $\frac{1}{1+r_i}$. However, this classical idea of integration fails the Rota-Baxter rule \textbf{[s.2]} for any vector space of dimension greater than one. The only calculus objects of this calculus category structure for $\mathsf{VEC}_{\mathbb{K}}$ are one dimensional vector spaces, in particular the field $\mathbb{K}$, and the zero vector space. 
\end{example}

As seen with Corollary \ref{monoidalrationals}, all integral categories with a monoidal coalgebra modality is enriched over $\mathbb{Q}_{\geq 0}$-modules. Therefore, we do not need to look very far for a differential category which is not an integral category.
 
\begin{example} \normalfont  As shown above, the category of vector spaces for any fixed field with the free symmetric algebra is a differential category. However for a field with non-zero characteristic, this is not an integral category. For a particular example, the free symmetric algebra modality on category of vector spaces over $\mathbb{Z}_2$ does have a deriving transformation but does not have an integral transformation.  
\end{example}

Interestingly, one of the algebraic notions of differentiation -- differential algebras -- does not give a differential category, but does give an integral category (over the appropriate field).

\begin{example} \normalfont  A (commutative) \textbf{differential algebra (of weight $0$)} \cite{guo2008differential} over a field $\mathbb{K}$ is a pair $(A, \mathsf{D})$ consisting of a commutative $\mathbb{K}$-algebra $A$ and a linear map $\mathsf{D}: A \to A$ such that $\mathsf{D}$ satisfies the Leibniz rule, that is, the following equality holds:
$$\mathsf{D}(ab)=\mathsf{D}(a)b+a\mathsf{D}(b) \quad \forall a,b \in A$$
Modifying slightly the construction given in \cite{guo2008differential} of the free differential algebra over a set to instead obtain a free differential algebra over a vector space, we obtain that the forgetful functor from the category of differential algebras over $\mathbb{K}$, $\mathsf{CDA}_\mathbb{K}$, has a left adjoint:
$$\xymatrixcolsep{2.5pc}\xymatrix{\mathsf{VEC}_{\mathbb{K}} \ar@<+1.1ex>[r]^-{\mathsf{DIFF}} & \mathsf{CDA}_{\mathbb{K}} \ar@<+1ex>[l]_-{\bot}^-{U}
  }$$
Briefly, $\mathsf{DIFF}: \mathsf{VEC}_{\mathbb{K}} \to \mathsf{CDA}_{\mathbb{K}}$ is defined on objects as $\mathsf{DIFF}(V)= \mathsf{Sym}(\bigoplus_{n\in \mathbb{N}} V)$ and we refer the reader to \cite{guo2008differential} for the remainder of the construction (as it is similar). The induced monad of this adjunction is an algebra modality. This algebra modality has the Seely isomorphisms since the free symmetric algebra does and the infinite coproduct behaves well with the finite biporduct \cite{lang2002algebra}:
\begin{align*}
\mathsf{DIFF}(M \times N)&= \mathsf{Sym}(\bigoplus_{n\in \mathbb{N}} (M \times N))\\
&\cong \mathsf{Sym}(\bigoplus_{n\in \mathbb{N}} M \times \bigoplus_{n\in \mathbb{N}} N) \\
& \cong \mathsf{Sym}(\bigoplus_{n\in \mathbb{N}} M) \otimes \mathsf{Sym}(\bigoplus_{n\in \mathbb{N}} N)\\
&= \mathsf{DIFF}(M) \otimes \mathsf{DIFF}(N)
\end{align*}
Therefore this is a comonoidal algebra modality. However this algebra modality does not have a deriving transformation, since the chain rule \textbf{[d.4]} cannot be satisfied. 
\end{example}

\begin{example} \normalfont For the category of vector space over $\mathbb{Z}_2$, $\mathsf{DIFF}$ is a comonoidal algebra modality which does not have a deriving transformation or an integral transformation. 
\end{example}

\begin{example} \normalfont If $\mathbb{K}$ is a field of characteristic zero, then $\mathsf{DIFF}$ has an integral transformation defined as: 
  \[  \xymatrixcolsep{3.5pc}\xymatrix{\mathsf{Sym}(\bigoplus_{n\in \mathbb{N}} V) \otimes V \ar[r]^-{1 \otimes \iota_0} & \mathsf{Sym}(\bigoplus_{n\in \mathbb{N}} V) \otimes \bigoplus_{n\in \mathbb{N}} V \ar[r]^-{\mathsf{s}_{\bigoplus_{n\in \mathbb{N}} V}} & \mathsf{Sym}(\bigoplus_{n\in \mathbb{N}} V)  
  } \]
  where $\iota_0: V \to \bigoplus_{n\in \mathbb{N}} V$ is the injection map into the coproduct and $\mathsf{s}_{\bigoplus_{n\in \mathbb{N}} V}$ is the integral transformation of the free symmetric algebra.
\end{example}

Similarly one of the algebraic abstractions of integration -- Rota-Baxter algebras -- always gives a differential category:

\begin{example} \normalfont A (commutative) \textbf{Rota-Baxter algebra (of weight $0$)} \cite{guo2012introduction} over a field $\mathbb{K}$ is a pair $(A, \mathsf{P})$ consisting of a commutative $\mathbb{K}$-algebra $A$ and a linear map $\mathsf{P}: A \to A$ such that $\mathsf{P}$ satisfies the  Rota-Baxter equation, that is, the following equality holds:
$$\mathsf{P}(a)\mathsf{P}(b)=\mathsf{P}(a\mathsf{P}(b))+\mathsf{P}(\mathsf{P}(a)b) \quad \forall ab, \in A$$
 The map $\mathsf{P}$ is called a \textbf{Rota-Baxter operator} (we refer the reader to \cite{guo2012introduction} for more details on Rota-Baxter algebras). It turns out that there is a left adjoint to the forgetful functor between the category of Rota-Baxter algebras, $\mathsf{CRBA}_{\mathbb{K}}$, and the category of commutative algebras over $\mathbb{K}$, $\mathsf{CALG}_{\mathbb{K}}$ (for more details on this adjunction and the induced monad see \cite{zhang2016monads}). 
 $$\xymatrixcolsep{2.5pc}\xymatrix{\mathsf{CALG}_{\mathbb{K}} \ar@<+1.1ex>[r]^-{\mathsf{RB}} & \mathsf{CRBA}_{\mathbb{K}} \ar@<+1ex>[l]_-{\bot}^-{U}
  }$$
 Briefly, on objects the left adjoint $\mathsf{RB}: \mathsf{CALG}_{\mathbb{K}} \to \mathsf{CRBA}_{\mathbb{K}}$ is given by $\mathsf{RB}(A)=\mathsf{Sh}(A) \otimes A$ where $\mathsf{Sh}(A)$ is the shuffle algebra \cite{guo2012introduction}. To obtain an algebra modality on the category of vector spaces over $\mathbb{K}$, we compose the free Rota-Baxter algebra adjunction and the free symmetric algebra adjunction: 
$$\xymatrixcolsep{2.5pc}\xymatrix{\mathsf{VEC}_{\mathbb{K}} \ar@<+1.1ex>[r]^-{\mathsf{Sym}} & \mathsf{CALG}_{\mathbb{K}} \ar@<+1ex>[l]_-{\bot}^-{U} \ar@<+1.1ex>[r]^-{\mathsf{RB}} & \mathsf{CRBA}_{\mathbb{K}} \ar@<+1ex>[l]_-{\bot}^-{U}
  }$$ 
  and take the induced monad of this adjunction. This algebra modality comes equipped with a deriving transformation defined as 
$$1 \otimes \mathsf{d}: \mathsf{Sh}(\mathsf{Sym}(V)) \otimes \mathsf{Sym}(V) \to \mathsf{Sh}(\mathsf{Sym}(V)) \otimes \mathsf{Sym}(V) \otimes V$$
where $\mathsf{d}$ is the deriving transformation of $\mathsf{Sym}$. However, this deriving transformation is not Taylor for the simple reason that there exists non-constants which derive to zero. 
\end{example}

\begin{example} \normalfont For the category of vector space over $\mathbb{Z}_2$, $\mathsf{RB}$ is a non-comonoidal algebra modality with a deriving transformation but not an integral transformation. \end{example}

As usual, in the presence of rationals, Rota-Baxter algebras also give example of an integral category. However, the differential and integral structures do not give a calculus category. 

\begin{example}\label{JnotKexample} \normalfont If the field $\mathbb{K}$ has characteristic zero, then $\mathsf{RB}$ also has an integral transformation defined as
$$1 \otimes \mathsf{s}: \mathsf{Sh}(\mathsf{Sym}(V)) \otimes \mathsf{Sym}(V) \otimes V \to \mathsf{Sh}(\mathsf{Sym}(V)) \otimes \mathsf{Sym}(V)$$
where $\mathsf{s}$ is the integral transformation of $\mathsf{Sym}$. Since the deriving transformation is not Taylor, the Second Fundamental Theorem of Calculus does not hold and therefore this is not a calculus category structure. Furthermore, this is an example where $\mathsf{J}$ is invertible while $\mathsf{K}$ is not. 
\end{example}

Combining the differential algebra and Rota-Baxter algebra examples we obtain an example which is solely an algebra modality.

\begin{example} \normalfont A (commutative) \textbf{differential Rota-Baxter algebra (of weight $0$)} \cite{guo2008differential} over a field $\mathbb{K}$ is a triple $(A, \mathsf{D}, \mathsf{P})$ consisting of a differential algebra $(A, \mathsf{D})$ and a Rota-Baxter algebra $(A, \mathsf{P})$ such that $\mathsf{P}\mathsf{D}=1_A$. It turns out that the free Rota-Baxter algebra over a differential algebra is also its free differential Rota-Baxter algebra, and therefore inducing the following adjunction between the category differential algebras, $\mathsf{CDA}_{\mathbb{K}}$, and the category of differential Rota-Baxter algebras, $\mathsf{CDRBA}_{\mathbb{K}}$: 
$$\xymatrixcolsep{2.5pc}\xymatrix{\mathsf{CDA}_{\mathbb{K}} \ar@<+1.1ex>[r]^-{\mathsf{RB}} & \mathsf{CDRBA}_{\mathbb{K}} \ar@<+1ex>[l]_-{\bot}^-{U}
  }$$
  The full construction can be found in \cite{guo2008differential}. Once again, to obtain an algebra modality we compose this adjunction with the free differential algebra adjunction: 
  $$\xymatrixcolsep{2.5pc}\xymatrix{\mathsf{VEC}_{\mathbb{K}} \ar@<+1.1ex>[r]^-{\mathsf{DIFF}} & \mathsf{CDA}_{\mathbb{K}} \ar@<+1ex>[l]_-{\bot}^-{U} \ar@<+1.1ex>[r]^-{\mathsf{RB}} & \mathsf{CDRBA}_{\mathbb{K}} \ar@<+1ex>[l]_-{\bot}^-{U}
  }$$
  This algebra modality is not comonoidal for the same reasons as the free Rota-Baxter algebra, and is not a differential category failing the chain rule like the free differential algebra. 
\end{example}

\begin{example} \normalfont Continuing the previous example, let $\mathbb{K}$ be a field of characteristic zero. We define an integral transformation for $\mathsf{RB}(\mathsf{DIFF}(V))$ as follows: 
  \[  \xymatrixcolsep{4pc}\xymatrix{\mathsf{Sh}(\mathsf{Sym}(\bigoplus_{n\in \mathbb{N}} V)) \otimes \mathsf{Sym}(\bigoplus_{n\in \mathbb{N}} V) \otimes V \ar[r]^-{1 \otimes 1 \otimes \iota_0} &\\
   \mathsf{Sh}(\mathsf{Sym}(\bigoplus_{n\in \mathbb{N}} V)) \otimes \mathsf{Sym}(\bigoplus_{n\in \mathbb{N}} V) \otimes \bigoplus_{n\in \mathbb{N}} V
   \ar[r]^-{1 \otimes \mathsf{s}_{\bigoplus_{n\in \mathbb{N}} V}} &\\\mathsf{Sh}(\mathsf{Sym}(\bigoplus_{n\in \mathbb{N}} V)) \otimes \mathsf{Sym}(\bigoplus_{n\in \mathbb{N}} V)  
  } \]
Then $\mathsf{VEC}_R$ with the free differential Rota-Baxter algebra monad is an integral category whose algebra modality is not comonoidal. 
\end{example}

\begin{example} \normalfont On the other hand, the category of vector space over $\mathbb{Z}_2$ with free differential Rota-Baxter algebra monad is not a differential category, nor is it an integral category, and whose algebra modality is not comonoidal. 
\end{example}

\begin{example} \normalfont Let $R$ be a non-zero commutative ring. The zero $R$-module, $0$, induces a coalgebra modality on the category of modules over $R$ where the functor maps all objects to $0$ and all maps to zero maps. This coalgebra modality is not monoidal since $\oc(0)=0 \ncong R$. It is however a differential category with antideriviatives since every map in sight would be zero and the necessary equalities hold trivially. This example can be generalized to any additive symmetric monoidal category with a zero object. 
\end{example}

\begin{example} \normalfont Similar to the construction of the free symmetric algebra for the category of vector spaces over a field, one can construct the free additively idempotent (i.e $1+1=1$) commutative algebra in the category of modules over a rig \cite{golan2013semirings}, which gives an algebra modality. The deriving transformation is the same as the free symmetric algebra, and since our algebra is additively idempotent: $\mathsf{K}$ and $\mathsf{J}$ are the identity. Therefore, this is a differential category with antiderivatives. However, if the rig $R$ is not additively idempotent (such as $\mathbb{N}$), this is not a comonoidal algebra modality since $\oc(0) \ncong R$. When $R$ is additively idempotent, then the algebra modlity does have the Seely isomorphisms and therefore would be comonoidal. 
\end{example}

\section{Conclusion}\label{conclusionsec}

In the examples above while we have filled in many of the blanks: however, it is important to note the areas we have {\em not\/} filled.  For example, while we suspect it may be possible, we do not have an example of a differential category on a monoidal coalgebra modality which has integration and yet is not a calculus category.  Nor do we have an example of a calculus category which does not have antiderivatives.

The influence of linear logic  has made monoidal coalgebra modalities central.  However, coalgebra modalities which are not monoidal, while perhaps somewhat overlooked, are certainly mathematically 
important.   A natural example of a differential category with antiderivatives, whose coalgebra modality is not moniodal is the category of real vector spaces with smooth functions determined by the free $\mathcal{C}^\infty$-ring monad.  This example was uncovered by Geoff Cruttwell, Rory Lucyshyn-Wright, and the second author.  The example,  being mathematically central, suggests that non-monoidal coalgebra modalities cannot be ignored.

A reasonable strategy for further developing the theory of integral categories is to follow the same path as that followed for differential categories. The theory of differential categories was developed in four stages with each stage formalizing a different aspect of the theory of differentiation: (tensor) differential categories  \cite{dblute2006differential} formalize basic differentiation, cartesian differential categories axiomatize directional derivatives \cite{blute2009cartesian}, restriction differential categories \cite{cockett2011differential} formalize the theory of differentiation on open subsets, and tangent categories \cite{cockett2014differential} formalize tangent structure and the theory of smooth manifolds. Each stage relates to the next in interesting ways: the coKleisli category of a differential category is a cartesian differential category, while a cartesian differential category is a special case of a restriction differential category. Restriction differential categories, in turn, have an associated manifold completion which is itself a tangent category.

Integral categories are closely related to differential categories. Therefore, it seems worthwhile to develop a parallel approach for integral categories in four corresponding stages. Here we have developed the first step with (tensor) integral categories. The next step -- which is already in progress -- is to develop Cartesian integral categories and Cartesian calculus categories, and in particular, to show the sense in which the coKleisli category of an integral category is a Cartesian integral category.   Cartesian integral categories have a term logic which gives these categories a more ``classical'' feel: in fact, we have borrowed parts of this term logic to help motivate this paper. 

The observation that the coderiving transformation is an integral transformation when addition is idempotent (i.e. $1+1=1$) is interesting as differential categories of interest in computer science often 
have this property.  The properties of these integral categories may, thus, be of special interest to computer science.

\paragraph*{Acknowledgements:} 
The authors would like to thank Rick Blute for drawing both authors' attention to Rota-Baxter algebras. Integral categories simply would not have developed so rapidly without this basic inspiration.   Jonathan Gallagher 
reminded us of Ehrhard's work at exactly the right moment, while Kristine Bauer provided continual constructive criticism during the evolution of our  thoughts. 

\bibliographystyle{plain}      
\bibliography{intcatarchive}   

\end{document}